%% file: alm-plateau.tex
\pdfoutput=1 % for arXiv to use pdflatex instead of latex+dvipdf

\documentclass[11pt,a4paper,twoside,final,leqno]{article}

\usepackage{a4wide}

\usepackage{amsmath}
\usepackage{amssymb}
\usepackage{amsthm}
\usepackage{mathrsfs}

\usepackage{mathabx}

\usepackage[T1]{fontenc}
\usepackage[utf8x]{inputenc}
\usepackage{tgschola}

\usepackage{url}
\usepackage{esint}

\usepackage{enumitem}

\usepackage{hyperref}

\setlist{itemsep=6pt,parsep=0pt,topsep=3pt,partopsep=0pt}
\setlist[enumerate,1]{label=(\alph*), ref=(\alph*)}

\usepackage{dsfont}

\usepackage{mathtools}
\mathtoolsset{showonlyrefs}

\usepackage{CJK}

\input{macros.tex}

\date{\today}

\begin{document}

%%%%%%%%%%%%%%%%%%%%%%%%%%%%%%%%%%%%%%%%%%%%%%%%%%%%%%%%%%%%%%%%%%%%%%%%%%%%%%

\title{Existence of solutions to a general geometric elliptic variational problem}

\author{
  Yangqin Fang
  \and S{\l}awomir Kolasi{\'n}ski}

\maketitle
\begin{abstract}
    We consider the problem of minimising an inhomogeneous anisotropic elliptic
    functional in a class of closed $m$~dimensional subsets of~$\mathbf{R}^n$
    which is stable under taking smooth deformations homotopic to the identity
    and under local Hausdorff limits. We~prove that the minimiser exists inside
    the class and is an $(\mathscr H^m,m)$~rectifiable set in the sense of
    Federer. The class of competitors encodes a notion of spanning
    a~boundary. We admit unrectifiable and non-compact competitors and
    boundaries, and we make no restrictions on the dimension~$m$ and the
    co-dimension $n-m$ other than~\mbox{$1 \le m < n$}. An~important tool for
    the proof is a novel smooth deformation theorem. The skeleton of the proof
    and the main ideas follow Almgren's 1968 paper. In~the end we show that
    classes of sets spanning some closed set~$B$ in homological and
    cohomological sense satisfy our axioms.
\end{abstract}

\tableofcontents

%%%%%%%%%%%%%%%%%%%%%%%%%%%%%%%%%%%%%%%%%%%%%%%%%%%%%%%%%%%%%%%%%%%%%%%%%%%%%%

\input{intro.tex}
\input{notation.tex}
\input{main.tex}
\input{unrect-mapping.tex}
\input{smooth-retr.tex}
\input{central-proj.tex}
\input{smooth-ff.tex}
\input{slicing.tex}
\input{ludrb.tex}
\input{rectifiability.tex}
\input{density.tex}
\input{spans.tex}

\subsection*{Acknowledgement}
Both authors acknowledge the perfect working conditions they had while working
at the Max~Planck Institute for Gravitational Physics (Albert Einstein
Institute) in Potsdam-Golm.

Part of this work has been done while the authors were visiting Xiangyu Liang at
the Institut Camille Jordan, Universit{\'e} Claude Bernard Lyon~1.
We~acknowledge her kind hospitality.

The authors are also thankful to Ulrich Menne for giving the incentive to start
this project, guidance through intricate paths of Almgren's brilliant ideas,
and~constant support along the~way.

The second author was partially supported by NCN Grant no. 2013/10/M/ST1/00416
\emph{Geometric curvature energies for subsets of the Euclidean space.} 

{\small
\addcontentsline{toc}{section}{\numberline{}References}
\bibliography{ohne_duplikate}
\bibliographystyle{halpha}
}

\medskip
{\small \noindent
  Yangqin Fang
  \begin{CJK*}{UTF8}{gbsn}
      (方扬钦)
  \end{CJK*}
  % 方扬钦 (Yangqin Fang)
  \\
  School of Mathematics and Statistics,
  \\
  Huazhong University of Science and Technology,
  \\
  430074, Wuhan, P.R. China
  \\
  \texttt{yangqinfang@hust.edu.cn}
}

\vspace{1em}

{\small \noindent
  S{\l}awomir Kolasi{\'n}ski
  \\
  Instytut Matematyki,
  Uniwersytet Warszawski
  \\
  ul. Banacha 2, 02-097 Warszawa, Poland
  \\
  \texttt{s.kolasinski@mimuw.edu.pl}
}

\end{document}

%% file: macros.tex
\swapnumbers
\theoremstyle{plain}
\newtheorem{theorem}{Theorem}[section]
\newtheorem{lemma}[theorem]{Lemma}

\newtheorem{corollary}[theorem]{Corollary}

\theoremstyle{remark}
\newtheorem{remark}[theorem]{Remark}
\newtheorem*{remark*}{Remark}

\theoremstyle{definition}
\newtheorem{definition}[theorem]{Definition}
\newtheorem{miniremark}[theorem]{}

\newcommand{\ccspace}[1]{\mathscr{K}(#1)}

\newcommand{\dspace}[2]{\mathscr{D}(#1,#2)}

\newcommand{\project}[1]{{#1}_\natural}

\newcommand{\eqproject}[1]{(#1)_\natural}

\newcommand{\perpproject}[1]{#1_\natural^\perp}

\newcommand{\grass}[2]{\mathbf{G}(#1,#2)}

\newcommand{\measureball}[2]{{#1}\,{#2}}

\newcommand{\Var}[1]{\mathbf{V}_{#1}}     % general varifolds
\newcommand{\RVar}[1]{\mathbf{RV}_{#1}}   % rectifiable varifolds
\newcommand{\IVar}[1]{\mathbf{IV}_{#1}}   % integral varifolds
\DeclareMathOperator{\VarTan}{VarTan}     % varifold tangents
\newcommand{\var}[1]{\mathbf{v}_{#1}}     % varifolds generated by rectifiable sets

\newcommand{\oball}[2]{\mathbf{U}(#1,#2)}

\newcommand{\cball}[2]{\mathbf{B}(#1,#2)}

\newcommand{\textint}[2]{{\textstyle\int_{#1}^{#2}}}

\newcommand{\nat}{\mathscr{P}}
\newcommand{\integers}{\mathbf{Z}}

\newcommand{\adm}[1]{\mathfrak{D}({#1})}

\newcommand{\R}{\mathbf{R}}

\newcommand{\CF}{\mathds{1}}

\newcommand{\Ch}{\mathbf{\check H}}

\newcommand{\LM}{\mathscr{L}}

\newcommand{\HM}{\mathscr{H}}

\newcommand{\HD}{{d_{\HM}}}

\newcommand{\HDK}[1]{{d_{\HM,#1}}}

\newcommand{\density}{\boldsymbol{\Theta}}

\newcommand{\BFconst}{\boldsymbol{\beta}}

\newcommand{\unitmeasure}[1]{\boldsymbol{\alpha}(#1)}

\newcommand{\restrict}{ \mathop{ \rule[1pt]{.5pt}{6pt} \rule[1pt]{4pt}{0.5pt} }\nolimits }

\newcommand{\ud}{\ensuremath{\,\mathrm{d}}}

\newcommand{\uD}{\ensuremath{\mathrm{D}}}

\newcommand{\id}[1]{\mathrm{id}_{#1}}
\newcommand{\lIm}{[}
\newcommand{\rIm}{]}

\newcommand{\scale}[1]{\boldsymbol{\mu}_{#1}}
\newcommand{\trans}[1]{\boldsymbol{\tau}_{#1}}

\newcommand{\tbwedge}{{\textstyle \bigwedge}}

\newcommand{\tbcup}{\mathop{{\textstyle \bigcup}}}

\newcommand{\tbcap}{\mathop{{\textstyle \bigcap}}}

\newcommand{\Clos}[1]{\mathop{\mathrm{Clos}}#1}

\newcommand{\side}[1]{\mathbf{l}(#1)}
\newcommand{\centre}[1]{\mathbf{c}(#1)}
\newcommand{\vertex}[1]{\mathbf{o}(#1)}

\newcommand{\cInt}[1]{\operatorname{Int}_{\mathrm{c}}(#1)}
\newcommand{\cBdry}[1]{\operatorname{Bdry}_{\mathrm{c}}(#1)}
\newcommand{\cubes}{\mathbf{K}}

\newcommand{\WF}{\mathbf{WF}}

\newcommand{\CX}{\mathbf{CX}}

\DeclareMathOperator{\trace}{trace}

\DeclareMathOperator{\Hom}{Hom}

\DeclareMathOperator{\Bdry}{Bdry}

\DeclareMathOperator{\ap}{ap}

\DeclareMathOperator{\lin}{span}

\newcommand{\cnt}{\mathscr{C}}

\newcommand{\orthproj}[2]{\mathbf{O}^\ast({#1},{#2})}

\newcommand{\orthgroup}[1]{\mathbf{O}({#1})}

\DeclareMathOperator{\Int}{Int}

\DeclareMathOperator{\dmn}{dmn}

\DeclareMathOperator{\grad}{grad}
\DeclareMathOperator{\spt}{spt}

\DeclareMathOperator{\conv}{conv}

\DeclareMathOperator{\Tan}{Tan}

\DeclareMathOperator{\Lip}{Lip}

\DeclareMathOperator{\dist}{dist}

\DeclareMathOperator{\diam}{diam}

\DeclareMathOperator{\sgn}{sgn}

\DeclareMathOperator{\without}{\sim}

\DeclareMathOperator{\im}{im}

%% file: intro.tex
\section{Introduction}
\label{sec:intro}

The Plateau problem is about finding a minimiser of the area amongst the
surfaces which span a given boundary. The notions of ``area'', ``surface'', and
``spanning a boundary'' of~course need to be precisely defined so this problem
actually has many different incarnations. Its history spans over two centuries
and we have no intention of enumerating numerous prominent developments in~this
field. For the presentation of the classical formulations and solutions, their
drawbacks, and the modern reformulations of the problem we refer the reader to
the excellent expository articles by David~\cite{Dav2014} as well as by Harrison
and Pugh~\cite{HP2016surv,HP2016op}. These sources contain also an extensive list of
references. We shall focus mainly on the comparison of our methods and results
with the ones used in the papers published in the last years.

In this article we~deal with an abstract Plateau's problem which encompasses,
e.g., the formulation of Adams and Reifenberg; cf.~\cite{Rei60}. The notion of
``area'' of a competitor~$S$ is replaced by the value of a functional $\Phi_F$
on~$S$ which is defined by integrating an~\emph{elliptic} integrand $F : \R^n
\times \grass nm \to [0,\infty]$ with respect to the Hausdorff measure~$\HM^m$
over~$S$ -- if $S$ is $(\HM^m,m)$~rectifiable, then we feed $F$ with pairs
$(x,T)$, where $x \in S$ and $T$ is the approximate tangent plane to~$S$
at~$x$. The integrand~$F$ provides an inhomogeneous (depending on the location)
and anisotropic (depending on the tangent direction) weight for the Hausdorff
measure. If $F(x,T) = 1$ for all $(x,T) \in \R^n \times \grass nm$, then we call
$F$ the \emph{area integrand}. Ellipticity means roughly that a~flat
$m$-dimensional disc~$D$ minimises $\Phi_F$ amongst surfaces that cannot be
retracted onto the boundary of~$D$, i.e., $(m-1)$-dimensional sphere;
see~\ref{def:elliptic}. It can be seen as a geometric counterpart of
quasi-convexity; see~\cite{Mor1966}. The~``surfaces'' and ``boundaries'' are, in
our case, quite arbitrary closed subsets of~$\R^n$ -- not necessarily
rectifiable nor compact. The notion of ``spanning a boundary'' does not appear
at all. Our~main theorem (see~\ref{thm:main}) provides existence of an
$(\HM^m,m)$~rectifiable set which minimises $\Phi_F$ (with $F$ elliptic) inside
an axiomatically defined class of competitors (see~\ref{def:goodclass}).

Section~\ref{sec:main} contains all the definitions and the precise statement of
the main theorem. In~section~\ref{sec:cech-spanning} we show that naturally
defined (using \v{C}ech homology and cohomology) classes of sets spanning
a~given boundary (which may be an arbitrary closed set in~$\R^n$) satisfy our
axioms.

Similar results were obtained recently by other authors. Harrison~\cite{Har2014}
suggested a new formulation of the problem, defined spanning employing the
linking number, and used differential chains, developed earlier
in~\cite{Har2015}, to find minimisers of the Hausdorff measure in co-dimension
one. Harrison and Pugh~\cite{HP2016,HP2016rp} proved existence of minimisers for
the area integrand in arbitrary dimension and co-dimension using \v{C}ech
cohomology to define spanning.

The same authors proved independently, in~\cite{HP2016gm}, a very similar result
to ours. They showed existence of minimisers of an elliptic anisotropic and
inhomogeneous functional in an abstractly defined class of competitors of any
dimension and co-dimension. Their result admits non-Euclidean ambient metric
spaces but, on the other hand, is restricted to the case when all competitors
are compact and $(\HM^m,m)$~rectifiable (however, Jenny Harrison informed the
authors that the methods of~\cite{HP2016gm} extend also to the case of
non-compact competitors).

Even though, the main result of the current paper is so similar
to~\cite{HP2016gm}, we emphasis that the method of the proof is different.
In~particular, we make extensive use of varifolds (especially, Almgren's theory
of slicing) and we provide a~\emph{smooth} deformation theorem.

De~Lellis, Ghiraldin, and Maggi~\cite{DGM2017} formulated the problem abstractly
and showed existence of minimisers of the $m$-dimensional Hausdorff measure in
any family of subsets of~$\R^{m+1}$ containing enough competitors;
see~\cite[Definition~1]{DGM2017}. Later De~Philippis, De~Rosa, and
Ghiraldin~\cite{DDG2016} generalised this result to any co-dimension assuming,
roughly, that the set of competitors is closed under taking deformations which
are uniform limits of maps $\cnt^1$~isotopic to identity;
see~\cite[Definition~1.2]{DDG2016}. After that, De~Lellis, De~Rosa, and
Ghiraldin~\cite{DDG2017} obtained also minimisers for an inhomogeneous and
anisotropic problem in co-dimension one. Finally, De~Philippis, De~Rosa,
and~Ghiraldin~\cite{DDG2017b} also solved the problem in full generality.

These works all consider axiomatically defined families of competitors, which
include, e.g., sets that span a~boundary in the sense of Harrison~\cite{Har2014}
and \emph{sliding} competitors of David~\cite{Dav2014}. However, in the latter
case the results of~\cite{DGM2017,DDG2016,DDG2016,DDG2017b} do not ensure that
the minimiser is a sliding deformation of the initial competitor.

The origin of our project was a mini-course that we conducted in~2014.
We~undertook the effort to understand Almgren's existence result presented on
the first 30 pages of~\cite{Alm68}. Enlightened by his brilliant ideas we
decided to present his approach to the Plateau problem in a sequence of
lectures. The~present manuscript was, at first, meant to serve as lecture notes
for the mini-course but, with time, it grew into a~full-fledged research paper
containing new results.

The skeleton of the proof is the same as in~\cite{Alm68} and our proofs of the
intermediate steps are quite similar to Almgren's but we also divert
from~\cite{Alm68} in many places. First of all we separated the abstract
existence result from the application to a specific class of sets which
homologically span a~given boundary. Actually, in the definition of the
\emph{good class} of competitors, we do not use any notion of ``spanning a
boundary'' -- we only assume the class is closed under local Hausdorff
convergence, and under taking images of sets with respect to certain
\emph{admissible} deformations; see~\ref{def:goodclass}. Moreover, we make no
use of currents, flat chains, or $G$-varifolds in this paper. Second, we filled
up most of the details that Almgren left to the reader. In~particular, we had to
develop a new \emph{smooth} deformation theorem, which might be of separate
interest (see~\ref{thm:deformation}) and we proved a perturbation lemma
(see~\ref{lem:reduce-unrect}) which allows to show rectifiability of
minimisers. Third, we improved the main theorem by allowing for non-compact and
unrectifiable competitors and boundaries.

Our deformation theorem~\ref{thm:deformation} takes some $m$~dimensional sets
$\Sigma_1$, \ldots, $\Sigma_l$ and a~finite subfamily $\mathcal A$ of dyadic
Whitney cubes and provides a $\cnt^{\infty}$~smooth homotopy $f : [0,1] \times
\R^n \to \R^n$ between the identity and a map which deforms some
neighbourhood~$G$ of~$\bigcup_{i=1}^l\Sigma_i \cap \bigcup \mathcal A$ onto an
$m$~dimensional skeleton of~$\mathcal A$. Furthermore, for each $t \in [0,1]$
the map~$f(t,\cdot)$ equals the identity outside some neighbourhood of~$\bigcup
\mathcal A$. The main novelty with respect to well known constructions of this
sort is that~$f$ is $\cnt^{\infty}$~smooth. This is important for two
reasons. First, the push-forward by~$f$ defines a \emph{continuous} map $f_{\#}
: \Var{m}(\R^n) \to \Var{m}(\R^n)$ on the space of $m$~dimensional varifolds
in~$\R^n$. This allows to transfer certain estimates valid for the limit
varifold (which, a~priori, is not a competitor) onto elements of the minimising
sequence; see, e.g., ~\ref{lem:drb:aux}. Second, since the image of~$G$
under~$f(1,\cdot)$ is $m$~dimensional, we may use a~perturbation argument
(see~\ref{lem:reduce-unrect}) to find another smooth map which is arbitrarily
close to~$f(1,\cdot)$ in $\cnt^1$~topology and almost kills the measure of the
unrectifiable part of $\Sigma_i \cap \bigcup \mathcal A$. This allows to show
that the minimiser coming from the main theorem~\ref{thm:main} is
$(\HM^m,m)$~rectifiable.

In~contrast to the classical Federer--Fleming deformation
theorem~\cite[4.2.9]{Fed69} and Almgren's deformation
theorem~\cite[Chapter~1]{Alm1986} ours works for quite arbitrary sets $\Sigma$
in~$\R^n$ which may not carry the structure of a rectifiable current. It differs
also from a~similar result of David and~Semmes~\cite[Theorem 3.1]{DS2000}
because we perform the deformation inside Whitney cubes of varying sizes and,
in~case $\Sigma$ is $(\HM^m,m)$~rectifiable, we provide estimates on the
$\HM^{m+1}$~measure of the whole deformation, i.e., on $\HM^{m+1}(f \lIm (0,1)
\times \Sigma \rIm)$. Moreover, our theorem is tailored especially for the use
with varifolds which might not be rectifiable, so we actually prove estimates
not on the $\HM^m$~measure of $f(1,\cdot)\lIm \Sigma \rIm$ but rather on the
integral $\int_{\Sigma} \|\uD f(1,\cdot)\|^m \ud \HM^m$.

In the course of the proof of the main theorem we try to mimic, as much as
possible, Almgren's original ideas. In particular, rectifiability of the
minimiser is proven employing the deformation theorem together with
a~perturbation argument based on the Besicovitch--Federer projection theorem;
see section~\ref{sec:rectifiability}. This point of the proof seems to make
a~lot of trouble in other works. De~Lellis, Ghiraldin, and Maggi
in~\cite{DGM2017} and De~Philippis, De~Rosa, and Ghiraldin in~\cite{DDG2016}
used the famous Preiss' rectifiability theorem~\cite{Pre1987}. De~Lellis,
De~Rosa, and~Ghiraldin in~\cite{DDG2016} employed the theory of Caccioppoli
sets, which is possible in co-dimension one. The first author in~\cite{Fang2013}
used a~very complex construction of Feuvrier~\cite{Feu2012} to modify
a~minimising sequence into a~sequence consisting of quasi-minimal sets;
cf.~\cite{Alm76}. Harrison and Pugh~\cite{HP2016rp,HP2016gm} choose a very
special subsequence, which they call~\emph{Reifenberg regular sequence}, of the
minimising sequence, whose elements enjoy good bounds on density ratios down to
certain scale. We simply follow Almgren's ideas.

In their final paper~\cite{DDG2017b} De Philippis, De~Rosa, and~Ghiraldin make
use their very interesting result~\cite{DDG2016rect} yielding rectifiability for
minima of certain elliptic functionals. Actually, in~\cite{DDG2016rect} they
acquired a sufficient and necessary condition (called the \emph{atomic
  condition}) on the integrand so that varifolds whose first variation with
respect to such~$F$ induces a Radon measure are rectifiable. Later,
De~Rosa~\cite{DeR2016} showed that if~$F$ satisfies the atomic condition, then
an $F$-minimising sequence of integral varifolds contains a sub-sequence
converging to an integral varifold.

% Local Variables:
% coding: utf-8
% eval: (ispell-change-dictionary "british")
% eval: (flyspell-mode)
% End:

%% file: notation.tex
\section{Notation}
\label{sec:notation}

In~principle we shall follow the notation of Federer;
see~\cite[pp.~669--671]{Fed69}. However, we will use the standard abbreviations
for intervals in~$\R$, i.e., $(a,b) = \{ t \in \R : a < t < b\}$, $[a,b) = \{ t
\in \R : a \le t < b\}$ etc. We reserve the symbol~$I = [0,1]$ for the closed
unit interval. We will also write $\{ x \in X : P(x) \}$ rather than $X \cap \{
x : P(x) \}$ to denote the set of those $x \in X$ which satisfy
predicate~$P$. For the identity map on some set~$X$ we use the symbol~$\id{X} :
X \to X$ and the characteristic function of~$X$ is denoted $\CF_X$ and is
defined by $\CF_X(x) = 1$ if $x \in X$ and $\CF_X(x) = 0$ if $x \notin X$. If $U
\subseteq \R^m$ and $V \subseteq \R^n$, the set of maps $f : U \to V$ with
continuous $k^{\mathrm{th}}$~order derivatives is denoted by $\cnt^{k}(U,V)$.
If~$f \in \cnt^k(U,V)$, we say that $f$ is of \emph{class~$\cnt^k$}. In~contrast
to~\cite[2.9.1]{Fed69} given two Radon measures~$\mu$ and~$\nu$ over~$\R^n$ we
write
\begin{equation}
    \label{eq:RNder}
    \mathbf{D}(\mu,\nu,x) = \lim_{r \downarrow 0} \frac{\measureball{\mu}{\cball xr}}{\measureball{\nu}{\cball xr}}
    \quad \text{for $x \in \R^n$} \,.
\end{equation}

Concerning varifolds we shall follow Allard's notation;
see~\cite{All72}. In~particular, if~$U \subset \R^n$ is open, we write
$\Var{k}(U)$, $\IVar{k}(U)$, and $\RVar{k}(U)$ for the space of $k$~dimensional
varifolds, integral varifolds, and rectifiable varifolds in~$U$ following the
definitions~\cite[3.1, 3.5]{All72}. Also $\VarTan(V,x)$ shall denote the set of
varifold tangents as defined in~\cite[3.4]{All72}. In~contrast
to~\cite[3.5]{All72}, we shall write~$\var{k}(S)$ instead of~$\mathbf{v}(S)$ to
highlight the dimension of of the resulting varifold; cf.~\ref{def:var-S}.

We recall some notation of Federer. As in~\cite[2.2.6]{Fed69} we use the symbol
$\nat$ to denote the set of positive integers. The~symbols $\oball ar$
and~$\cball ar$ denote respectively the open and closed ball with centre~$a$ and
radius~$r$; see~\cite[2.8.1]{Fed69}. We use the notation $\trans{a}$
and~$\scale{s}$ for the translation by~$a \in \R^n$ and the homothety with ratio
$s \in \R$ respectively; see~\cite[2.7.16, 4.2.8]{Fed69}. For the Hausdorff
metric on compact subsets of~$\R^n$ we write~$\HD$ and for the $k$~dimensional
Hausdorff measure~$\HM^k$; cf.~\cite[2.10.21 and 2.10.2]{Fed69}. The scalar
product of $u,v \in \R^n$ is denoted $u \bullet v$. The space of maps $p \in
\Hom(\R^n,\R^m)$ such that $p^* u \bullet p^* v = u \bullet v$ for all $u,v \in
\R^m$ (i.e. $p$ is an orthogonal projection) is denoted $\orthproj nm$;
see~\cite[1.7.4]{Fed69}.

Following~\cite{Alm68} and~\cite{Alm00} if $S \in \grass nm$ is an
$m$~dimensional linear subspace of~$\R^n$, then $\project S \in \Hom(\R^n,\R^n)$
shall denote the orthogonal projection onto~$S$. In~particular if $p \in
\orthproj nm$ is such that $\im p^* = S$, then $\project S = p^* \circ p$.

Whenever $\mu$ is a (Radon) measure over some set $U \subseteq \R^n$ we
sometimes use the same symbol~$\mu$ to denote the (not necessarily Radon)
measure $j_{\sharp}\mu$ over~$\R^n$, where $j : U \to \R^n$ is the
inclusion~map. Nonetheless, the support of $\mu$ is always a subset
of~$U$, i.e., $\spt \mu \subseteq U$.

If $A$ and $B$ are subsets of some vectorspace, then we write $A + B$ for the
\emph{algebraic sum} of~$A$ and~$B$, i.e., the set~$\{ a + b : a \in A, b \in
B\}$. If $X$ and $Y$ are vectorspaces, we write $X \oplus Y$ for the the
\emph{direct sum} of~$X$ and~$Y$; see~\cite[Chapter~I, \S{}2]{MR0050886}.

% Local Variables:
% coding: utf-8
% eval: (ispell-change-dictionary "british")
% eval: (flyspell-mode)
% End:

%% file: main.tex
\section{Statement of the main result}
\label{sec:main}

\begin{definition}
    \label{def:adm-deform}
    Let $U \subseteq \R^n$ be open. We say that $f : \R^n \to \R^n$ is
    a~\emph{basic deformation in~$U$} if $f$ is of class~$\cnt^1$ and there
    exists a bounded convex open set $V \subseteq U$ such that
    \begin{displaymath}
        f(x) = x \quad \text{for $x \in \R^n \without V$}
        \quad \text{and} \quad
        f \lIm V \rIm  \subseteq V \,.
    \end{displaymath}

    If $f \in \cnt^1(\R^n,\R^n)$ is a composition of a finite number of basic
    deformations, then we say that $f$ is an~\emph{admissible deformation
      in~$U$}. The set of all such deformations shall be denoted $\adm{U}$.
\end{definition}

\begin{remark}
    In most cases the bounded convex set~$V$ shall be a~cube or a~ball.
\end{remark}

\begin{definition}
    Whenever $K \subseteq \R^n$ is compact and $A,B \subseteq \R^n$, we
    define~$\HDK{K}(A,B)$ by
    \begin{displaymath}
        \HDK{K}(A,B) = \max \bigl\{
        \sup \{ \dist(x,A) : x \in K \cap B \} \,,\,
        \sup \{ \dist(x,B) : x \in K \cap A \}
        \bigr\} \,.
    \end{displaymath}
\end{definition}

\begin{definition}
    \label{def:goodclass}
    Let $U \subseteq \R^n$ be an open set. We say that $\mathcal C$ is
    a~\emph{good class in $U$} if
    \begin{enumerate}
    \item
        \label{i:gc:nonempty}
        $\mathcal C \ne \varnothing$;
    \item
        \label{i:gc:compact}
        each $S \in \mathcal C$ is a closed subset of~$\R^n$;
    \item
        \label{i:gc:deformation}
        if $S \in \mathcal C$ and $f \in \adm{U}$, then $f \lIm S \rIm \in
        \mathcal C$;
    \item
        \label{i:gc:hd}
        if $S_i \in \mathcal C$ for $i \in \nat$, and $S \subseteq \R^n$, and
        $\lim_{i \to \infty} \HDK{K}(S_i \cap U, S \cap U) = 0$ for all compact sets
        $K \subseteq U$, then $S \in \mathcal C$.
    \end{enumerate}
\end{definition}

\begin{remark}
    One example of a good class is given in~\ref{le:CechGood}. We expect that
    the methods presented in this article could work also if we assumed that
    admissible deformations are uniform limits of diffeomorphisms (so called
    \emph{monotone maps}) as in~\cite[Definition~1.1]{DDG2016}. In such case,
    the class denoted $\mathcal F(H,\mathcal C)$ defined
    in~\cite[Definition~1.4]{DDG2016} would also be good. However, we had
    trouble checking that all the deformations we use are
    monotone. In~particular, the deformations constructed in~\ref{lem:drb:aux}
    are clearly not monotone and there is no easy way to fix that. We anticipate
    that one could modify the deformation theorem~\ref{thm:deformation} to
    handle the situation from~\ref{lem:drb:aux} and provide an appropriate
    monotone map but, given the overall complexity of the already presented
    material, we chose not to do that. A related result, allowing to approximate
    the cone construction by a sequence of \emph{diffeomorphisms} and to get an
    isoperimetric inequality similar to~\eqref{eq:di:lower}, was obtained
    recently by Pugh~\cite{Pugh2017}.
\end{remark}

\begin{definition}[\protect{cf.~\cite[1.2]{Alm68}}]
    \label{def:Ck-integrand}
    A function $F : \R^n \times \grass{n}{m} \to [0,\infty)$ of
    class~$\cnt^k$ for some non-negative integer~$k \in \nat$ is called
    a~\emph{$\cnt^k$~integrand}.

    If additionally $\inf \im F / \sup \im F \in (0,\infty)$, then we say that
    $F$ is \emph{bounded}.
\end{definition}

\begin{definition}[\protect{cf.~\cite[3.1]{Alm68}}]
    \label{def:pull-back}
    If $\varphi \in \cnt^{1}(\R^n,\R^n)$ and $F$ is an integrand, then 
    the \emph{pull-back} integrand~$\varphi^\#F$ is given by by
    \begin{displaymath}
        \varphi^\#F(x,T) = \left\{
            \begin{aligned}
                &F\bigl( \varphi(x), \uD\varphi(x)\lIm T \rIm \bigr) 
                \| {\textstyle \bigwedge_m} \uD\varphi(x) \circ \project{T} \|
                && \text{if } \dim \uD\varphi(x)\lIm T \rIm = m
                \\
                &0 
                && \text{if } \dim \uD\varphi(x)\lIm T \rIm < m \,.
            \end{aligned}
        \right.
    \end{displaymath}
    If $\varphi$ is a diffeomorphism, then the \emph{push-forward} integrand is
    given by $\varphi_\#F = (\varphi^{-1})^\#F$.
\end{definition}

\begin{definition}[\protect{cf.~\cite[1.2]{Alm68}}]
    \label{def:Fx-integrand}
    If $F$ is a $\cnt^k$~integrand and $x \in \R^n$, then we define
    another $\cnt^k$~integrand $F^x$ by the formula
    \begin{displaymath}
        F^x(y,S) = F(x,S)
        \quad \text{for $y \in \R^n$ and $S \in \grass nm$} \,.
    \end{displaymath}
\end{definition}

\begin{remark}
    Recall that $\boldsymbol{\gamma}_{n,m}$ denotes the canonical
    probability measure on $\grass nm$ invariant under the action of the
    orthogonal group $\orthgroup n$; see~\cite[2.7.16(6)]{Fed69}.
\end{remark}

\begin{definition}[\protect{cf.~\cite[3.5]{All72}}]
    \label{def:var-S}
    Assume $S \subseteq \R^n$ is $\HM^m$~measurable and such that $\HM^m(S \cap
    K) < \infty$ for any compact set $K \subseteq \R^n$. We define $\var{m}(S)
    \in \Var{m}(\R^n)$ in the following way: first decompose $S$ into a~sum $S_u
    \cup S_r$, where $S_u$ is purely $(\HM^m,m)$~unrectifiable and $S_r$ is
    countably $(\HM^m,m)$ rectifiable and Borel (cf.~\cite[3.2.14]{Fed69}); then
    set
    \begin{multline*}
        \var{m}(S)(\alpha) = \int_{S_r} \alpha(x, \Tan^m(\HM^m \restrict S_r, x)) \ud \HM^m(x)
        \\
        + \int_{S_u} \int \alpha(x, T) \ud \boldsymbol{\gamma}_{n,m}(T) \ud \HM^m(x) 
        \quad \text{for $\alpha \in \ccspace{\R^n \times \grass nm}$}
        \,.
    \end{multline*}
\end{definition}

\begin{definition}
    \label{def:Phi-F}
    If $F$ is a $\cnt^k$~integrand, we define the functional $\Phi_F :
    \Var{m}(\R^n) \to [0,\infty]$ by the formula
    \begin{displaymath}
        \Phi_F(V) = \int F(x,S) \ud V(x,S) \,.
    \end{displaymath}
\end{definition}

\begin{remark}
    If $\spt \|V\|$ is compact we have $\Phi_F(V) = V(\gamma F)$, whenever
    $\gamma \in \dspace{\R^n}{\R}$ is such that $\spt \|V\| \subseteq \Int
    \gamma^{-1} \{ 1 \}$.
\end{remark}

\begin{definition}
    \label{def:Phi-F-on-S}
    If $S \subseteq \R^n$ is $\HM^m$~measurable, satisfies $\HM^m(S \cap K) <
    \infty$ for all compact sets $K \subseteq \R^n$, and $S_u \subseteq S$ is
    a~purely $(\HM^m,m)$~unrectifiable part of~$S$, then we set
    \begin{gather*}
        \Phi_F(S) = \Phi_F(\var{m}(S)) \,,
        \\
        \Psi_{F}(S) = \Phi_{F}(S)
        + \int_{S_u} \bigl( \sup \im F^x - \textint{}{} F(x,T) \ud \boldsymbol{\gamma}_{n,m}(T) \bigr)\ud \HM^m(x)
        \,.
    \end{gather*}
    For any other subset $S$ of $\R^n$ we set $\Psi_F(S) = \Phi_F(S) = \infty$.
\end{definition}

\begin{remark}
    \label{rem:pull-back}
    Assume $V \in \Var{m}(\R^n)$, $\varphi : \R^n \to \R^n$ is of class~$\cnt^1$,
    and $F$ is a~$\cnt^0$~integrand. Then
    \begin{displaymath}
        \Phi_{\varphi^{\#}F}(V) = \Phi_F(\varphi_{\#}V) \,.
    \end{displaymath}
    If $S \subseteq \R^n$ is $\HM^m$~measurable and satisfies $\HM^m(S \cap K) <
    \infty$ for all compact sets $K \subseteq \R^n$, then
    \begin{displaymath}
        \varphi_{\#}\var{m}(S) = \var{m}(\varphi \lIm S \rIm) 
    \end{displaymath}
    given $S$ is countably $(\HM^m,m)$~rectifiable or $\varphi = \scale{r}$ for
    some $r \in (0,\infty)$ or $\varphi = \trans{a}$ for some $a \in \R^n$.
\end{remark}

\begin{remark}
    \label{rem:Psi-Phi}
    If $S$ is $\HM^m$~measurable, $\HM^m(S \cap K) < \infty$ for any compact $K
    \subseteq \R^n$, $S = S_u \cup S_r$, where $S_u$ is purely
    $(\HM^m,m)$~unrectifiable and $S_r$ is countably $(\HM^m,m)$~rectifiable,
    $F$ is an integrand, $x \in \R^n$, then
    \begin{displaymath}
        \Psi_{F^x}(S) = \Phi_{F^x}(S_r) + \HM^m(S_u) \sup \im F^x \,.
    \end{displaymath}
\end{remark}

We shall use the following notion of ellipticity based on the definition given
by Almgren; see~\cite[IV.1(7)]{Alm76} and~\cite[1.6(2)]{Alm68}. It can be
understood as a~geometric version of quasi-convexity;
cf.~\cite{Mor1966}. Similar notion for currents can be found
in~\cite[5.1.2]{Fed69}.

\begin{definition}
    \label{def:elliptic}
    An $\cnt^0$~integrand $F$ is called \emph{elliptic} if there exists
    a~continuous function $c : \R^n \to (0,\infty)$ such that for all $T \in
    \grass{n}{m}$ we have
    \begin{displaymath}
        \Psi_{F^x}(S) - \Psi_{F^x}(D) \ge c(x) \bigl( \HM^m(S) - \HM^m(D) \bigr)
    \end{displaymath}
    whenever
    \begin{enumerate}
    \item
        \label{i:ell:D}
        $D = \cball 01 \cap T$ is a unit $m$~dimensional disc in~$T$;
    \item
        \label{i:ell:S}
        $S$ is a compact subset of $\R^n$, $\HM^m(S) < \infty$, and $S$ cannot
        be deformed onto~$R = \Bdry \cball 01 \cap T$ by any Lipschitz map $f :
        \R^n \to \R^n$ satisfying $f(x) = x$ for $x \in R$.
    \end{enumerate}
\end{definition}

\begin{remark} Note the following observations.
    \label{rem:elliptic}
    \begin{itemize}
    \item The \emph{area integrand} $F \equiv 1$ is elliptic.

    \item If $F$ is an integrand and $\varphi \in \cnt^{1}(\R^n,\R^n)$ is
        a~diffeomorphism, then $F$ is elliptic if and only if $\varphi_\#F$ is
        elliptic; cf.~\cite[4.3]{Alm68}.

    \item Any convex combination of elliptic integrands is elliptic.
    \end{itemize}
\end{remark}

\begin{remark}
    \label{rem:elliptic2}
    In the original definition of Almgren one assumes also that~$S$ is
    $(\HM^m,m)$~rectifiable. Since we want to work with possibly unrectifiable
    competitors we need to drop this assumption. For the same reason we used
    $\Psi_{F^x}$ instead of $\Phi_{F^x}$ in~\ref{def:elliptic}.

    Assume $x \in \R^n$ is fixed and $T \in \grass nm$ is such that $F(x,T) = M
    = \sup\{ F(x,P) : P \in \grass nm\}$. Set $E = \int F(x,P) \ud
    \boldsymbol{\gamma}_{n,m}(P)$, $D = \cball 01 \cap T$, and $R = \Bdry \cball
    01 \cap T$. Assume $E < M$ and $S = D \without F \cup W$
    satisfies~\ref{def:elliptic}\ref{i:ell:S}, where $F \subseteq D$ is closed
    and $W$ is purely $(\HM^m,m)$~unrectifiable and closed, and $W \cap (D
    \without F) = \varnothing$. Then
    \begin{displaymath}
        \Phi_{F^x}(S) - \Phi_{F^x}(D) = \Phi_{F^x}(W) - \Phi_{F^x}(F) 
        = E \HM^m(W) - M \HM^m(F) \,.
    \end{displaymath}
    If one could perform this construction ensuring $\HM^m(W) < M/E \HM^m(F)$,
    then the above quantity would become negative. Hence, assuming such
    construction is possible, if we used~$\Phi_{F^x}$ in~place of~$\Psi_{F^x}$
    in~\ref{def:elliptic}, then there would be no elliptic integrands depending
    non-trivially on the second variable. In some directions, filling a hole
    with purely unrectifiable set would always be better then filling it with
    a~flat disc. Nonetheless, we believe this is not possible.

    In a forthcoming article, the second author and Antonio De Rosa show that
    one obtains an equivalent definition of ellipticity if one assumes
    in~\ref{def:elliptic}\ref{i:ell:S} that~$S$ is $(\HM^m,m)$~rectifiable. More
    precisely: if $F$ is a $\cnt^0$ integrand, there exists a continuous
    function $c : \R^n \to (0,1)$, and for all~$S$ and~$D$
    satisfying~\ref{def:elliptic}\ref{i:ell:D}\ref{i:ell:S} such that $S$ is
    $(\HM^m,m)$~rectifiable there holds
    \begin{displaymath}
        \Phi_{F^x}(S) - \Phi_{F^x}(D) \ge c(x) \bigl( \HM^m(S) - \HM^m(D) \bigr) \,,
    \end{displaymath}
    then $F$ is elliptic in the sense of~\ref{def:elliptic}.
\end{remark}

\begin{remark}
    It is not clear whether strict convexity of~$F$ in the second variable is
    enough to ensure ellipticity of~$F$ as is the case in the context of
    currents; see~\cite[5.1.2]{Fed69}. De~Lellis, De~Rosa, and Ghiraldin where
    able to prove their existence result assuming $F$ is uniformly convex in the
    second variable and of class~$\cnt^2$; see~\cite[Definition~2.4]{DDG2016}.
    De~Philippis, De~Rosa, and Ghiraldin defined the so called \emph{atomic
      condition} for the integrand; see~\cite[Definition~1]{DDG2016rect}.
    In~co-dimension one this condition is equivalent to strict convexity of~$F$
    in the second variable; see~\cite[Theorem~1.3]{DDG2016rect}. Moreover,
    varifolds whose first variation with respect to~$F$, satisfying the atomic
    condition, induces a Radon measure are rectifiable;
    see~\cite[Theorem~1.2]{DDG2016rect}.
\end{remark}

Our main theorem reads.

\begin{theorem}
    \label{thm:main}
    Let $U \subset \R^n$ be an open set, $\mathcal C$ be a~good class
    in~$U$, and $F$ be a~bounded elliptic $\cnt^0$~integrand. Set $\mu =
    \inf\bigl\{ \Phi_F(T \cap U) : T \in \mathcal C \bigr\}$.

    If $\mu \in (0,\infty)$, then there exist $S \in \mathcal C$ and a~sequence
    $\{ S_i \in \mathcal C : i \in \nat \}$ such that
    \begin{enumerate}
    \item
        \label{i:main:rect}
        $S \cap U$ is $(\HM^m,m)$ rectifiable. In~particular $\HM^m(S \cap U) < \infty$.
    \item
        \label{i:main:minimal}
        $\lim_{i \to \infty} \Phi_F(S_i \cap U) = \Phi_F(S \cap U) = \mu$.
    \item
        \label{i:main:varifold}
        $\lim_{i \to \infty} \var{m}(S_i \cap U) = \var{m}(S \cap U)$ in $\Var{m}(U)$.
    \item
        \label{i:main:hausdorff}
        $\lim_{i \to \infty} \HDK{K}(S_i \cap U, S \cap U) = 0$ for any compact
        set $K \subseteq U$.
    \end{enumerate}

    Furthermore, if $\R^n \without U$ is compact and there exists a
    $\Phi_F$-minimising sequence in $\mathcal C$ consisting only of compact sets
    (but not necessarily uniformly bounded), then
    \begin{displaymath}
        \diam(\spt\|V\|) < \infty
        \quad \text{and} \quad
        \sup \bigl\{ \diam(S_i \cap U) : i \in \nat \bigr\} < \infty \,.
    \end{displaymath}
\end{theorem}

% Local Variables:
% coding: utf-8
% eval: (ispell-change-dictionary "british")
% eval: (flyspell-mode)
% End:

%  LocalWords:  nm Almgren integrands De Lellis Ghiraldin Philippis

%% file: unrect-mapping.tex
\section{Unrectifiable sets under submersions}

Assume $K \subseteq \R^n$ is purely $(\HM^m,m)$~unrectifiable with
$\HM^m(K) < \infty$ and $f \in \cnt^{1}(\R^n, \R^n)$ is such that
$\uD f(x)$ is of rank at most~$m$ for $x \in \R^n$. We construct an arbitrarily
small $\cnt^1$~perturbation $\tilde f$ of~$f$ such that
$\HM^m(\tilde f\lIm K \rIm)$ becomes very small. Additionally, we ensure that
$\tilde f$ is of the form $\tilde f = f \circ \rho$, where $\rho$ is
a~diffeomorphism of~$\R^n$ such that $\Lip(\rho - \id{\R^n})$ is very small.
This provides a useful feature of $\tilde f$, namely that
$\im \tilde f \subseteq \im f$.

A~similar result was proven recently by Pugh~\cite{Pugh16}. It could be possible
to obtain $\HM^m(\tilde f\lIm K \rIm) = 0$ as was shown by
Ga{\l}\k{e}ski~\cite{Galeski2017} but for our purposes it suffices only to make
the measure small. Also the map constructed in~\cite{Galeski2017} kills only the
measure of the part of~$K$ on which $\dim \im \uD f(x) = m$ and we need to take
care also of the part where the rank of $\uD f$ is strictly less
than~$m$. Finally, we should mention that Almgren alluded that such result
should hold already in~\cite[2.9(b1)]{Alm68}.

In the next preparatory lemma we construct a smooth map
$M : \R \to \orthgroup n$ which continuously rotates a given $m$-plane $S$
onto another given $m$-plane $T$.  We~also derive estimates on~$M'$ as well as
on $\|M(\cdot) - \id{\R^n}\|$ in terms of $\|\project S - \project T\|$.

\begin{lemma}
    \label{lem:rotation}
    Let $n$ and $m$ be positive integers such that $0 < m \le n$. There exists
    $\Gamma \in (0,\infty)$ such that for all $S,T \in \grass nm$
    there exists $M : \R \to \Hom(\R^n,\R^n)$ of class $\cnt^\infty$
    satisfying
    \begin{gather}
        M(0) = \id{\R^n} \,,
        \quad
        M(1)\lIm S \rIm = T \,,
        \quad
        \forall \tau \in \R \quad M(\tau) \in \orthgroup n \,,
        \\
        \label{eq:rot:estimates}
        \forall \tau \in \R \quad
        \|M(\tau) - \id{\R^n}\| \le \Gamma |\tau| \|\project{S} - \project{T}\|
        \quad \text{and} \quad
        \|M'(\tau)\| \le \Gamma \|\project{S} - \project{T}\| \,.
    \end{gather}
\end{lemma}

\begin{proof}
    We shall construct the map $M$ similarly as in~\cite[8.9(3)]{All72}. First
    set
    \begin{gather}
        A = S \cap T \,,
        \quad
        B = S^{\perp} \cap T^{\perp} = (S+T)^{\perp} \,,
        \\
        C = (S^{\perp} \cap T) \oplus (T^{\perp} \cap S) \,,
        \quad
        D = (S+T) \cap A^{\perp} \cap C^{\perp} \,.
    \end{gather}
    Observe that $A$, $B$, $C$, $D$ are pairwise orthogonal and sum up to the
    whole of~$\R^n$, i.e.,
    \begin{displaymath}
        \forall X,Y \in \{ A, B, C, D \} \quad X = Y \text{ or } X \perp Y \,,
        \qquad
        \R^n = A \oplus B \oplus C \oplus D \,.
    \end{displaymath}
    Note also that there exist natural numbers $k$ and $l$ such that
    \begin{gather*}
        k = \dim(S \cap D) = \dim(T \cap D) \,,
        \quad
        \dim(D) = \dim(S \cap D) + \dim(T \cap D) = 2k \,,
        \\
        l = \dim(S^{\perp} \cap T) = \dim(T^{\perp} \cap S) \,,
        \quad
        \dim(C) = 2l \,.
    \end{gather*}
    For our convenience we set $S_0 = S$ and $T_0 = T$. If $k > 0$, then we shall
    construct inductively
    \begin{itemize}
    \item subspaces $S \supseteq S_1 \supseteq S_2 \supseteq \cdots \supseteq S_k$
        and $T \supseteq T_1 \supseteq T_2 \supseteq \cdots \supseteq T_k$
        and $V_1, \ldots, V_k \subseteq S + T$,
    \item and vectors $s_1 \in S_1$, \ldots, $s_k \in S_k$ and $t_1 \in T_1$, \ldots,
        $t_k \in T_k$.
    \end{itemize}
    To start the construction we set $S_1 = S \cap D$ and $T_1 = T \cap D$. Then
    we use~\cite[8,9(3)]{All72} to find $s_1 \in S_1$ so that $|s_1| = 1$ and
    $|\project{(T^{\perp}_1)}s_1| = \|\project{S_1} - \project{T_1}\|$. Note that
    $\|\project{S_1} - \project{T_1}\| < 1$ because the spaces $S_1$ and $T_1$ are
    orthogonal to~$C$. Hence, we may define $t_1 = \project{T_1} s_1
    |\project{T_1} s_1|^{-1}$ and $V_1 = \lin\{ s_1, t_1 \}$. Assuming we have
    constructed $S_1$, \ldots, $S_i$ and $T_1$, \ldots, $T_i$ and $s_1$, \ldots,
    $s_i$ and $t_1$, \ldots, $t_i$ for some $i \in \{1,\ldots,k-1\}$ we proceed by
    requiring
    \begin{gather*}
        S_{i+1} = S_i \cap V_i^{\perp} \,,
        \quad
        T_{i+1} = T_i \cap V_i^{\perp} \,,
        \quad
        s_{i+1} \in S_{i+1} \,,
        \quad
        |s_{i+1}| = 1 \,,
        \\
        |\project{T_{i+1}^{\perp}}s_{i+1}| = \|\project{S_{i+1}} - \project{T_{i+1}}\| \,,
        \quad
        t_{i+1} = \frac{\project{T_{i+1}} s_{i+1}}{|\project{T_{i+1}} s_{i+1}|} \,,
        \quad
        V_{i+1} = \lin\{ s_{i+1} , t_{i+1} \} \,.
    \end{gather*}
    Observe that
    \begin{equation}
        \forall  i \in \{0,1,\ldots,k-1\}
        \ 
        \forall s \in S_{i+1} \subseteq S_i \quad
        \project{(T_i^{\perp})}s = \project{T_{i+1}^{\perp}}s \,;
    \end{equation}
    thus,
    \begin{multline}
        \|\project{S_{i+1}} - \project{T_{i+1}}\|
        = \sup\{ |\project{T_{i+1}}s| : s \in S_{i+1} ,\, |s| = 1 \}
        \\
        = \sup\{ |\project{T_{i}}s| : s \in S_{i+1} ,\, |s| = 1 \}
        \le \sup\{ |\project{T_{i}}s| : s \in S_{i} ,\, |s| = 1 \} \,.
    \end{multline}
    Therefore,
    \begin{equation}
        \forall  i \in \{0,1,\ldots,k\}
        \quad
        \|\project{S_{i}} - \project{T_{i}}\| \le \|\project{S} - \project{T}\| \,.
    \end{equation}
    Clearly $(s_1,\ldots,s_k)$ and $(t_1,\ldots,t_k)$ are orthonormal bases of $S
    \cap D$ and $T \cap D$ respectively. Next, we choose arbitrary orthonormal
    bases $(s_{k+1},\ldots,s_{k+l})$ of $S \cap T^{\perp}$ and
    $(t_{k+1},\ldots,t_{k+l})$ of $T \cap S^{\perp}$ and
    $(e_1,\ldots,e_{n-2(k+l)})$ of $A \oplus B$. We also define
    \begin{displaymath}
        \alpha_i = \arccos(s_i \bullet t_i)
        \quad \text{for $i \in \{1,\ldots,k+l\}$}
    \end{displaymath}
    and note that $0 < \alpha_i \le \pi/2$ by construction.  Now we are in
    position to define $M$. It shall be the identity on $A \oplus B$ and on each
    $V_i = \lin\{s_i,t_i\}$ it will be the rotation sending $s_i$ to $t_i$ for $i
    = 1,2,\ldots,k+l$. More precisely we set
    \begin{displaymath}
        \hat{s}_i = \frac{t_i - (t_i \bullet s_i) s_i}{|t_i - (t_i \bullet s_i) s_i|}
        \quad \text{for $i \in \{1,\ldots,k\}$} \,,
        \quad
        \hat{s}_i = t_i
        \quad \text{for $i \in \{k+1,\ldots,k+l\}$}  \,,
    \end{displaymath}
    and define for $\tau \in \R$
    \begin{gather*}
        M(\tau)s_i = \cos(\tau \alpha_i) s_i + \sin(\tau \alpha_i) \hat{s}_i
        \quad \text{for $i=1,\ldots,k+l$} \,,
        \\
        M(\tau)\hat{s}_i = -\sin(\tau \alpha_i) s_i + \cos(\tau \alpha_i) \hat{s}_i
        \quad \text{for $i=1,\ldots,k+l$} \,,
        \\
        M(\tau) e_i = e_i \quad \text{for $i=1,\ldots,n-2(k+l)$}
    \end{gather*}
    Since
    $\{s_1,\ldots,s_{k+l},\hat{s}_1,\ldots,\hat{s}_{k+l},e_1,\ldots,e_{n-2(k+l)}\}$
    is an orthonormal basis of $\R^n$ we see that $M(\tau) \in \orthgroup{n}$
    for each $\tau \in \R$. It is also immediate that $M(0) = \id{\R^n}$ and
    $M(1) \lIm S \rIm = T$.

    To prove~\eqref{eq:rot:estimates} we first estimate~$\alpha_i$. Recall that $1
    - \cos x = 2 \sin^2(x/2)$ for $x \in \R$ and $|x| \le 2 |\sin x|$ whenever
    $|x| \le \pi/2$; hence, for $i = 1,\ldots,k+l$
    \begin{multline}
        \alpha_i \le 4 \sin(\alpha_i/2)
        = 2 \sqrt 2 \bigl(1 - \cos(\alpha_i)\bigr)^{1/2}
        = 2 \sqrt 2 \bigl(1 - s_i \bullet t_i\bigr)^{1/2}
        \\
        = 2 \sqrt 2 \bigl(1 - \bigl(1 - |\project{T_i^{\perp}}s_i|^2\bigr)^{1/2}\bigr)^{1/2}
        \le 2 \sqrt 2 \bigl(1 - \bigl(1 - \|\project{T} - \project{S}\|^2\bigr)^{1/2}\bigr)^{1/2} \,.
    \end{multline}
    If $\|\project{T} - \project{S}\| < 1/2$, then we use standard estimates
    $\exp(x) \ge 1 + x$ and $\log(1+x) \ge x/(1+x)$ valid for $x > -1$ to derive
    \begin{equation}
        \label{eq:rot:alpha-est}
        \alpha_i
        \le 2 \sqrt 2 \frac{\|\project{T} - \project{S}\|}{\bigl(1 - \|\project{T} - \project{S}\|^2\bigr)^{1/2}}
        \le 8 \|\project{T} - \project{S}\| \,.
    \end{equation}
    If $\|\project{T} - \project{S}\| \ge 1/2$, we have $\bigl(1 - \bigl(1 -
    \|\project{T} - \project{S}\|^2\bigr)^{1/2}\bigr)^{1/2} \le 1 \le 2
    \|\project{T} - \project{S}\|$ so~\eqref{eq:rot:alpha-est} holds also in this
    case.

    Using $|\sin x|\le |x|$ for $x \in \R$ and~\eqref{eq:rot:alpha-est} we
    obtain for $i = 1,2,\ldots,k+l$ and $\tau \in \R$
    \begin{equation}
        |s_i - M(\tau)s_i| = 2 |\sin(\tau \alpha_i/2)|
        \le |\tau| \alpha_i
        \le 8 |\tau| \|\project{T} - \project{S}\| \,.
    \end{equation}
    Thus, whenever $v \in \R^n$ and $|v| = 1$,
    \begin{equation}
        |v - M(\tau)v|^2
        = \sum_{i=1}^{k+l} |\project{V_i}v - M(\tau)(\project{V_i}v)|^2
        = \sum_{i=1}^{k+l} |\project{V_i}v|^2 |s_i - M(\tau)s_i|^2
        \le \bigl( 8 |\tau| \|\project{T} - \project{S}\| \bigr)^2 \,,
    \end{equation}
    which proves the first part of~\eqref{eq:rot:estimates}. By direct computation
    we obtain $|M'(\tau) s_i| = |M'(\tau) \hat{s}_i| = \alpha_i$ for $\tau \in
    \R$ and $i = 1,2,\ldots,k+l$. Therefore, employing~\eqref{eq:rot:alpha-est},
    \begin{equation}
        \|M'(\tau)\| = \max\{ \alpha_i : i = 1,2,\ldots,k+l\}
        \le 8 \|\project{T} - \project{S} \| \,.
        \qedhere
    \end{equation}
\end{proof}

The following technical lemma is a localised and reparameterised version
of~\ref{lem:rotation}. Rough\-ly speaking, we~construct a~diffeomorphism $\rho$~of
$\R^n$ which acts as a rotation inside a~given ball and is the identity
outside some neighbourhood of that ball. To be able to
utilise~\ref{lem:unrect-local} in~\ref{lem:reduce-unrect} we need to perform the
rotation in different coordinates, which is accomplished by passing through
a~diffeomorphism $\varphi$. We use estimates from~\ref{lem:rotation} to bound
$\Lip(\rho - \id{\R^n})$.

\begin{lemma}
    \label{lem:unrect-local}
    Assume
    \begin{gather}
        k \in \nat \,,
        \quad
        U \subseteq \R^n \text{ is open} \,,
        \quad
        q \in \orthproj nm \,,
        \quad
        T = \im q^* \,,
        \quad
        S \in \grass nm \,,
        \\
        a \in U \,,
        \quad
        \tilde r, r \in \R \,,
        \quad
        0 < r < \dist(a, \R^n \without U) \,,
        \quad
        0 < \tilde r < r \,,
        \\
        \varphi \in \cnt^k(U,\R^n) \text{ is a~diffeomorphism onto its image} \,,
        \\
        L = \sup\bigl\{ \max\{ \|\uD\varphi(x)\|, \|\uD\varphi(x)^{-1}\| \} : x \in \cball ar \bigr\} \,,
        \\
        \omega(s) = \sup \bigl\{ \|\uD\varphi(x) - \uD\varphi(y)\| : x,y \in \cball ar ,\, |x-y| \le s \bigr\} 
        \quad \text{for $s \in \R$, $s \ge 0$} \,,
        \\
        \label{eq:un:ST-small}
        \Gamma = \Gamma(L, r, \tilde r) = \bigl( 2 L^2 r / (r - \tilde r) + 1\bigr) \Gamma_{\ref{lem:rotation}} \,,
        \quad 
        \Gamma \|\project S - \project T\| < 1 \,.
    \end{gather}
    Then there exist a diffeomorphism $\rho \in \cnt^k(\R^n,\R^n)$ and
    $p \in \orthproj nm$ such that
    \begin{gather}
        \label{eq:un:rho-id}
        \rho(x) = x \quad \text{for $x \in \R^n \without \oball ar$} \,,
        \quad
        \im p^* = S \,,
        \\
        \label{eq:un:rho-proj}
        q \circ \varphi \circ \rho (x) = p(\varphi(x) - \varphi(a)) + q(\varphi(a))
        \quad \text{for $x \in \cball a{\tilde r}$} \,,
        \\
        \label{eq:un:rho-est}
        \sup\bigl\{ \|\uD\rho(x) - \id{\R^n}\| : x \in \R^n \bigr\} 
        \le 2L \omega(L r \Gamma \|\project S - \project T\|) 
        + L^2 \Gamma \|\project S - \project T\| \,.
    \end{gather}
\end{lemma}

\begin{proof}
    Employ~\ref{lem:rotation} to obtain a smooth map
    $M : \R \to \Hom(\R^n,\R^n)$ such that $M(1)\lIm S \rIm = T$ and
    $M(\tau) \in \orthgroup n$ for each $\tau \in \R$. Let
    $\zeta : \R \to \R$ be of class $\cnt^\infty$ and satisfy $\zeta(t) = 0$
    for $t \le 0$, and $\zeta(t) = 1$ for $t \ge 1$, and $0 \le \zeta'(t) \le 2$
    for $t \in \R$, and $0 \in \Int \zeta^{-1} \{0\}$, and
    $1 \in \Int \zeta^{-1} \{1\}$. Define $\pi : \R^n \to \R^n$, and
    $\eta : \R^n \to \R$, and $p \in \Hom(\R^n,\R^m)$ by requiring
    \begin{gather}
        \eta(x) = (r - |\varphi^{-1}(x) - a|)/(r - \tilde r) 
        \quad \text{if $x \in \varphi \lIm \cball ar \rIm$ } \,,
        \\ 
        \eta(x) = 0
        \quad \text{if $x \in \R^n \without \varphi \lIm \cball ar \rIm$ } \,,
        \quad 
        p = q \circ M(1) \,,
        \\
        \pi(x) = M \circ \zeta \circ \eta(x) (x - \varphi(a)) + \varphi(a)
        \quad \text{for } x \in \R^n \,.
    \end{gather}
    Note that $\eta$ is Lipschitz continuous and $\pi$ is of class $\cnt^k$
    because $0 \in \Int \zeta^{-1} \{0\}$ and $1 \in \Int \zeta^{-1} \{1\}$.
    Moreover, $p \in \orthproj nm$ and
    \begin{gather}
        \label{eq:unrect:pTS}
        \im p^* = M(1)^* \circ q^* \lIm \R^m \rIm = M(1)^{-1}\lIm T \rIm = S \,,
        \\
        \label{eq:unrect:pi-id}
        \pi(x) = x
        \quad \text{whenever $x \in \R^n \without \varphi\lIm \oball ar \rIm$} \,,
        \\
        \label{eq:unrect:pi-proj}
        q \circ \pi (x) = p(x - \varphi(a)) + \varphi(a)
        \quad \text{for $x \in \varphi\lIm \cball a{\tilde r} \rIm$} \,.
    \end{gather}
    Hence, we can set
    \begin{displaymath}
        \rho = \varphi^{-1} \circ \pi \circ \varphi \,.
    \end{displaymath}
    Clearly~\eqref{eq:unrect:pTS}, \eqref{eq:unrect:pi-id},
    \eqref{eq:unrect:pi-proj} imply~\eqref{eq:un:rho-id}
    and~\eqref{eq:un:rho-proj} and we only need to check~\eqref{eq:un:rho-est}.
    Recalling~\eqref{eq:rot:estimates} and~\eqref{eq:un:ST-small} and
    $\Lip(\varphi|_{\cball ar}) \le L$ we estimate for $x \in \cball ar$
    \begin{gather}
        \Lip(\eta) \le \Lip\bigl( (\varphi|_{\cball ar})^{-1} \bigr) / (r - \tilde r)
        \le L / (r - \tilde r) \,,
        \\
        \label{eq:unrect:Dpi-id}
        \begin{aligned}
            \| \uD(\pi - \id{\R^n})(\varphi(x)) \|
            &\le \Lip(\zeta) \Lip(\eta) \|M'(\zeta \circ \eta(\varphi(x)))\| L r
            + \|M(\zeta \circ \eta(\varphi(x))) - \id{\R^n}\|
            \\
            &\le \bigl( 2 L^2 r / (r - \tilde r) + 1\bigr) \Gamma_{\ref{lem:rotation}}\|\project S - \project T\|
            = \Gamma \|\project S - \project T\| < 1 \,.
        \end{aligned}
    \end{gather}
    In particular, using~\eqref{eq:unrect:pi-id} we conclude that
    $\Lip(\pi - \id{\R^n}) < 1$, so $\pi$ and $\rho$ are diffeomorphisms.
    Employing~\eqref{eq:unrect:Dpi-id} we see also that if $x \in \cball ar$,
    then
    \begin{equation}
        \label{eq:unrect:pi-close-id}
        |\pi(\varphi(x)) - \varphi(x)|
        = |(\pi - \id{\R^n})(\varphi(x)) - (\pi - \id{\R^n})(\varphi(y))|
        \le L r \Gamma \|\project S - \project T\| \,,
    \end{equation}
    where $y \in U \without \oball ar$ is any point such that
    $|x-y| = \dist(x, \R^n \without \cball ar) \le r$.
    Utilising~\eqref{eq:unrect:Dpi-id} and~\eqref{eq:unrect:pi-close-id} we see
    that $\Lip(\pi) \le 2$ and for $x \in \cball ar$
    \begin{multline}
        \|\uD\rho(x) - \id{\R^n}\| 
        \le \bigl\| \bigl(\uD\varphi^{-1}(\pi\circ\varphi(x)) - \uD\varphi^{-1}(\varphi(x))\bigr) \circ \uD\pi(\varphi(x)) \bigr\|
        \|\uD\varphi(x)\|
        \\
        + \bigl\| \uD\varphi^{-1}(\varphi(x)) \circ \bigl( \uD\pi(\varphi(x)) - \id{\R^n} \bigr) \bigr\|
        \|\uD\varphi(x)\|
        \\
        \le 2L \omega(|\pi(\varphi(x)) - \varphi(x)|)
        + L^2 \| \uD(\pi - \id{\R^n})(\varphi(x)) \|
        \\
        \le 2L \omega(L r \Gamma \|\project S - \project T\|)
        + L^2 \Gamma \|\project S - \project T\| \,.
        \qedhere
    \end{multline}
\end{proof}

In the next lemma given a purely $(\HM^m,m)$~unrectifiable set~$K$ with
$\HM^m(K) < \infty$ and a map $f \in \cnt^{k}(\R^n, \R^n)$ such that
$\uD f(x)$ is of rank at most~$m$ for $x \in \R^n$ we employ the constant rank
theorem together with a~Vitali covering theorem to get a family of balls in each
of which we apply~\ref{lem:unrect-local} and the Besicovitch--Federer projection
theorem to construct a~diffeomorphism $\rho$ of~$\R^n$ such that
$f \circ \rho \lIm K \rIm$ has significantly less $\HM^m$ measure than $K$
itself. Since, in general, $f$~may map the set where $\uD f(x)$ has rank strictly
less than $m$ into a set of positive $\HM^m$~measure we need to additionally
assume that this does not happen or assume $k \ge n - m + 1$ and employ the
Morse--Sard theorem; see~\ref{rem:Sard}.

\begin{lemma}
    \label{lem:reduce-unrect}
    Let $K \subseteq \R^n$ be purely $(\HM^m,m)$~unrectifiable
    with $\HM^m(K) < \infty$. Let $f : \R^n \to \R^n$ be of
    class $\cnt^{k}$ with $k \ge 1$. Suppose there exists an open set $U
    \subseteq \R^n$ such that
    \begin{gather}
        \label{eq:reduce:rank}
        K \subseteq U 
        \quad \text{and} \quad
        \dim \im \uD f(x) \le m \text{ for all $x \in U$}
        \\
        \label{eq:reduce:sard}
        \text{and} \quad
        \HM^m(f\lIm \{ x \in K : \dim \im \uD f(x) < m \} \rIm) = 0 \,.
    \end{gather}

    Then for any $\varepsilon \in (0,\infty)$ there exists a diffeomorphism
    $\rho_\varepsilon : \R^n \to \R^n$ of class $\cnt^{k}$ such that
    \begin{gather*}
        \HM^m(f \circ \rho_\varepsilon\lIm K \rIm) \le \varepsilon \HM^m(K) \,,
        \quad
        \rho_{\varepsilon}(x) = x
        \quad \text{for } x \in \R^n \without U \,,
        \\
        |x - \rho_\varepsilon(x)| \le \varepsilon
        \quad \text{and} \quad
        \|\id{\R^n} - \uD\rho_\varepsilon(x)\| \le \varepsilon
        \quad \text{for } x \in \R^n 
        \,.
    \end{gather*}
\end{lemma}

\begin{proof}
    Let $\varepsilon \in (0,\infty)$ and let
    $q : \R^m \times \R^{n - m} \to \R^m$ be given by
    $q(x,y) = x$. Set
    \begin{displaymath}
        A = \{ x \in U : \dim \im \uD f(x) = m \} \,.
    \end{displaymath}
    Since $\dim \im \uD f(x) \le m$ for all $x \in U$ we see that
    $A = \{ x \in U :  \bigwedge_{m} \uD f(x) \ne 0 \}$ is open. Hence, for every
    $a \in A$ the constant rank theorem~\cite[3.1.18]{Fed69} ensures the
    existence of open sets $U_a \subseteq U$, $V_a \subseteq \R^n$, maps
    $\varphi_a : U_a \to \R^n$, $\psi_a : V_a \to \R^n$ which are
    diffeomorphisms onto their respective images, and orthogonal projections
    $p_a \in \orthproj{n}{m}$ such that
    \begin{equation}
        \label{eq:f-decomp}
        a \in U_a \,,
        \quad
        f(a) \in V_a \,,
        \quad
        f|_{U_a} = \psi_a^{-1} \circ p_a^\ast \circ q \circ \varphi_a \,.
    \end{equation}
    Applying the Vitali covering theorem (see \cite[2.8.16,\,2.8.18]{Fed69} or
    alternatively \cite[2.8]{MR1333890}) to the measure $\HM^m \restrict K$
    and the family of all the closed balls $\cball ar$ satisfying
    \begin{gather}
        \label{eq:balls1}
        a \in K \cap A \,, \quad
        0 < r < \min\{ 1 , \varepsilon \} \,, \quad
        \cball ar \subseteq U_a \,,
        \\
        \label{eq:balls3}
        \lim_{s \uparrow r} \HM^m \bigl( K \cap \cball ar \without \cball as \bigr) = 0
    \end{gather}
    we obtain a countable disjointed collection $\mathcal B$ of closed balls
    having the properties \eqref{eq:balls1}, \eqref{eq:balls3}, and additionally
    \begin{equation}
        \label{eq:K-cap-A-without-B}
        \HM^m\bigl( (K \cap A) \without \tbcup \mathcal B \bigr) = 0 \,.
    \end{equation}
    Whenever $\cball ar \in \mathcal B$ we set
    \begin{equation}
        \label{eq:phi-psi-lip}
        L_a = \max\left\{
            \Lip(\varphi_a|_{\cball ar}) ,\,
            \Lip\bigl((\varphi_a|_{\cball ar})^{-1}\bigr) ,\,
            \Lip(f|_{\cball ar})
        \right\} \,.
    \end{equation}

    Set
    $I = \{ a \in \R^n : \cball{a}{r} \in \mathcal B \text{ for some } r \in
    \R \}$
    and $T = \im(q^*) = \R^m \times \{0\} \in \grass nm$. Whenever $a \in I$
    define $r_a \in \R$ to be the unique number such that
    $\cball{a}{r_a} \in \mathcal B$. Suppose $a \in I$. Since $\varphi_a$ is
    a~diffeomorphism onto its image, we see that
    $\varphi_a\lIm K \cap \cball{a}{r_a} \rIm$ is purely
    $(\HM^m,m)$~unrectifiable. Hence, the Besicovitch--Federer
    projection theorem (see~\cite[3.3.15]{Fed69} or
    alternatively~\cite[18.1]{MR1333890}) allows us to find a sequence of
    $m$-planes $S_{a,i} \in \grass nm$ such that
    $\|\project{S_{a,i}} - \project{T}\| \to 0$ as $i \to \infty$ and
    \begin{equation}
        \label{eq:ri-meas-zero}
        \HM^m(\project{S_{a,i}}\lIm \varphi_a\lIm K \cap \cball{a}{r_a} \rIm \rIm) = 0
        \quad \text{for all $i \in \nat$} \,.
    \end{equation}
    Using~\eqref{eq:balls3} we find $\tilde{r}_a \in \R$ such that
    $0 < \tilde{r}_a < r_a$ and
    \begin{gather}
        \label{eq:annulus-epsilon}
        \HM^m\bigl( K \cap \cball{a}{r_a} \without \cball a{\tilde{r}_a} \bigr) 
        < (2L_a)^{-m} \varepsilon \HM^m\bigl( K \cap \cball{a}{r_a} \bigr) \,.
    \end{gather}
    Set $\Delta = \Gamma_{\ref{lem:unrect-local}}(L_a, r_a, \tilde{r}_a)$ and
    \begin{displaymath}
        \omega_a(s) = \sup \bigl\{ \|\uD\varphi_a(x) - \uD\varphi_a(y)\| : x,y \in \cball{a}{r_a} ,\, |x-y| \le s \bigr\} 
        \quad \text{for $s \in \R$, $s \ge 0$} \,.
    \end{displaymath}
    Choose $i_a \in \nat$ so big that
    \begin{gather}
        \Delta \|\project{S_{a,i_a}} - \project{T}\| < 1 \,,
        \\
        \label{eq:reduce:Drho-est}
        2L_a \omega_a(L_a r_a \Delta \|\project{S_{a,i_a}} - \project T\|) 
        + L_a^2 \Delta \|\project{S_{a,i_a}} - \project T\| < \min\{ 1, \varepsilon \} \,.
    \end{gather}
    Employ~\ref{lem:unrect-local} with~$S_{a,i_a}$, $\varphi_a$, $r_a$,
    $\tilde{r}_a$ in place of~$S$, $\varphi$, $r$, $\tilde r$ to obtain
    a~diffeomorphism $\rho = \rho_{a} \in \cnt^k(\R^n,\R^n)$ and
    a~projection $p = p_{a} \in \orthproj nm$ satisfying~\eqref{eq:un:rho-id},
    \eqref{eq:un:rho-proj}, \eqref{eq:un:rho-est}.

    To finish the construction, we set
    \begin{displaymath}
        \rho_{\varepsilon}(x) = 
        \left\{
            \begin{aligned}
                &\rho_a(x) && \text{if } x \in \cball{a}{r_a} \in \mathcal B \,,
                \\
                &x && \text{if } x \in \R^n \without \tbcup \mathcal B \,.
            \end{aligned}
        \right.
    \end{displaymath}
    Since $\mathcal B$ is disjointed and each $\rho_a$ is the identity outside
    the corresponding ball $\cball{a}{r_a} \in \mathcal B$, we see that
    $\rho_{\varepsilon}$ is a well defined diffeomorphism of class $\cnt^k$.
    Moreover, using~\eqref{eq:reduce:sard} and~\eqref{eq:K-cap-A-without-B},
    then \eqref{eq:un:rho-proj} and \eqref{eq:un:rho-id} together
    with~\eqref{eq:ri-meas-zero} and finally~\eqref{eq:un:rho-est} combined
    with~\eqref{eq:annulus-epsilon} we obtain
    \begin{multline}
        \HM^m\bigl(f \circ \rho_\varepsilon \lIm K \rIm\bigr)
        \le \HM^m\bigl(f \lIm K \without A \rIm\bigr) 
        + \HM^m\bigl(f \lIm (K \cap A) \without \tbcup \mathcal B \rIm\bigr)
        + \sum_{B \in \mathcal B} \HM^m\bigl(f \circ \rho_{\varepsilon}\lIm K \cap B \rIm\bigr)
        \\
        = \sum_{a \in I} \HM^m\bigl(f \circ \rho_{a}\lIm K \cap \cball{a}{r_a} \rIm\bigr)
        = \sum_{a \in I} \HM^m\bigl(f \circ \rho_{a}\lIm K \cap \cball{a}{r_a} \without \cball a{\tilde{r}_a} \rIm\bigr)
        \\
        \le \sum_{a \in I} (2 L_a)^m \HM^m\bigl(K \cap \cball{a}{r_a} \without \cball a{\tilde{r}_a}\bigr)
        \le \varepsilon \sum_{a \in I} \HM^m\bigl(K \cap \cball{a}{r_a}\bigr)
        \le \varepsilon \HM^m(K) \,.
    \end{multline}
    Recalling~\eqref{eq:un:rho-est} and~\eqref{eq:reduce:Drho-est} we see also
    \begin{displaymath}
        \sup\{ \|\uD\rho_{\varepsilon}(x) - \id{\R^n} \| \}
        = \sup\{ \|\uD\rho_{a}(x) - \id{\R^n}\| : a \in I ,\, x \in \cball{a}{r_a} \}
        \le \varepsilon
    \end{displaymath}
    and
    \begin{multline}
        \sup\{ \|\rho_{\varepsilon}(x) - x\| \}
        = \sup\{ \|\rho_{a}(x) - x\| : a \in I ,\, x \in \cball{a}{r_a} \}
        \\
        \le \sup\{ \Lip(\rho_a - \id{\R^n}) r_a : a \in I \}
        < \varepsilon \sup \{ r_a : a \in I \} \le \varepsilon \,.
        \qedhere
    \end{multline}
\end{proof}

\begin{remark}
    \label{rem:Sard}
    If $k \ge n - m + 1$, then the Morse--Sard
    theorem~\cite[3.4.3]{Fed69} implies that
    $\HM^m(f\lIm \{ x \in K : \dim \im \uD f(x) < m \} \rIm) = 0$ and
    assumption~\eqref{eq:reduce:sard} becomes redundant.
\end{remark}

\begin{corollary}
    \label{cor:perturb-est}
    Set $g = f \circ \rho_{\varepsilon}$ and
    \begin{equation}
        \label{eq:mod-cont-Df}
        \omega(r) = \sup \{ \| \uD f(x) - \uD f(y) \| : x,y \in U \,,\, |x-y| < r \} 
        \quad \text{for $r > 0$} \,.
    \end{equation}
    Then for $x \in \R^n$ we obtain
    \begin{multline}
        \| \uD g(x) - \uD f(x) \| 
        = \| (\uD f(\rho_{\varepsilon}(x)) - \uD f(x)) \circ \uD \rho_{\varepsilon}(x)
        + \uD f(x) \circ ( \uD \rho_{\varepsilon}(x) - \id{\R^n} ) \|
        \\
        \le 2 \omega(\varepsilon) + \|\uD f(x)\| \varepsilon \,.
    \end{multline}
    In particular, for $x \in \R^n$
    \begin{multline}
        \|\uD g(x)\|^m
        \le \bigl( \|\uD f(x)\| + \|\uD g(x) - \uD f(x)\| \bigr)^m
        \le \bigl( (1 + \varepsilon) \|\uD f(x)\| + 2 \omega(\varepsilon) \bigr)^m
        \\
        \le 2^{2m-1} \|\uD f(x)\|^m + 2^{2m-1} \omega(\varepsilon)^m  \,.
    \end{multline}
\end{corollary}

% Local Variables:
% coding: utf-8
% eval: (ispell-change-dictionary "british")
% eval: (flyspell-mode)
% End:

%% file: smooth-retr.tex
\section{Smooth almost retraction of $\R^n$ onto a cube}
\label{sec:retract}

In this section we construct, in~\ref{lem:smooth-retract-on-Q},
a~$\cnt^{\infty}$ function which maps all of $\R^n$ onto the cube~$Q =
[-1,1]^n$. This mapping is \emph{not} a~retraction because it moves points
inside the cube~$Q$. Its main features are that it is smooth and it preserves
all the lower dimensional skeletons of~$Q$ and even the skeletons of the
neighbouring dyadic cubes of side length~$1$. As a corollary
of~\ref{lem:smooth-retract-on-Q} we produce, in~\ref{cor:retr+}, a~function
which maps a small neighbourhood of~$Q$ onto~$Q$ and is the identity outside a
bit larger neighbourhood of~$Q$. We also carefully track the Lipschitz constants
of the mappings.

First we need to introduce some notation to be able to conveniently handle
various faces of the cube~$[-1,1]^n$ and its dyadic neighbours.

\begin{miniremark}
    \label{mrem:cube-setup}
    Let $e_1$, \ldots, $e_n$ be the standard basis of $\R^n$. Set $Q = \{ x
    \in \R^n : x \bullet e_j \le 1 \text{ for } j = 1,2,\ldots,n \} =
    [-1,1]^n$. For $\kappa = (\kappa_1, \ldots, \kappa_n) \in \{-1,0,1\}^n$
    define
    \begin{gather*}
        C_{\kappa} = \left\{ x \in \R^n :
            \begin{array}{c}
                \text{ for } j = 1,\ldots,n 
                \\
                \text{ either }
                \kappa_j \ne 0 \text{ and } (x \bullet e_j) \kappa_j \ge 1
                \\
                \text{ or }
                \kappa_j = 0 \text{ and } |x \bullet e_j| < 1
           \end{array}
        \right\} \,,
        \\
        F_{\kappa} = C_{\kappa} \cap Q \,,
        \quad
        T_{\kappa} = \lin \bigl\{ e_j : \kappa_j = 0 \bigr\} \,,
        \quad
        c_{\kappa} = \sum_{j=1}^n \kappa_j e_j \,.
    \end{gather*}
    Observe that the sets $C_{\kappa}$ for $\kappa \in \{-1,0,1\}^n$ are convex,
    have nonempty interiors, are pairwise disjoint, and form a partition
    of~$\R^n$, i.e.,
    \begin{gather*}
        \bigcup \bigl\{ C_{\kappa} : \kappa \in \{-1,0,1\}^n \bigr\} = \R^n 
        \quad \text{and} \quad
        C_{\kappa} \cap C_{\lambda} = \varnothing \text{ for } \lambda \ne \kappa \,.
    \end{gather*}
    For $\kappa \in \{-1,0,1\}^n$ we have $\dim(T_{\kappa}) = \HM^0(\{ j : \kappa_j =
    0 \})$, and $F_{\kappa}$ is a~$\dim(T_{\kappa})$ dimensional face of~$Q$ lying
    in the affine space $c_{\kappa} + T_{\kappa}$, and $F_{\kappa}$ is relatively
    open in $c_{\kappa} + T_{\kappa}$, and $c_{\kappa}$ is the centre
    of~$F_{\kappa}$. In~particular, $C_{(0,0,\ldots,0)} = F_{(0,0,\ldots,0)} =
    \Int(Q)$.

    For $\lambda \in \{-2,-1,1,2\}^n$ define
    \begin{gather*}
        R_{\lambda} = \left\{ x \in \R^n :
            \begin{array}{c}
                \text{ for } j = 1,\ldots,n 
                \\
                \text{ either }
                |\lambda_j| = 1 \text{ and } 1 \le (x \bullet e_j) \lambda_j \le 2
                \\
                \text{ or }
                |\lambda_j| = 2 \text{ and } 0 \le (x \bullet e_j) \lambda_j \le 2
            \end{array}
        \right\} \,.
    \end{gather*}
    Note that $R_\lambda$ is isometric to~$[0,1]^n$ and if $\lambda \notin
    \{-2,2\}^n$, then $R_{\lambda}$ is one of the neighbouring cubes of~$Q$ with
    side length equal to half the side length of~$Q$.

    Set
    \begin{gather*}
        f_{\kappa}(x) = \eqproject{T_{\kappa}}x + c_{\kappa} 
        \quad
        \text{and} \quad
        f(x) = \sum_{\kappa \in \{-1,0,1\}^n} \CF_{C_{\kappa}}(x) f_{\kappa}(x) 
        \quad \text{for } x \in \R^n \,,
    \end{gather*}
    where $\CF_{C_{\kappa}}$ is the characteristic function of $C_{\kappa}$. 
\end{miniremark}

\begin{remark}
    \label{rem:f-lipschitz}
    Observe that $f$ is simply the nearest point projection from $\R^n$
    onto~$Q$. Since $Q$ is convex it has infinite reach
    and~\cite[4.8(4)(8)]{MR0110078} shows that $f$ is Lipschitz continuous.
    However, we shall need the above decomposition of~$f$ to be able to
    effectively smoothen the~singularities, see~\ref{lem:smooth-retract-on-Q}.
\end{remark}

In the next lemma we construct a smooth mapping from $\R^n$ onto~$[-1,1]^n$.
This is achieved by post-composing the nearest point projection~$f$ with
a~smooth function which has zero derivative exactly in the directions in which
the derivative of~$f$ is undefined.

\begin{lemma}
    \label{lem:smooth-retract-on-Q}
    Let $e_1$, \ldots, $e_n$, $f$, $C_{\kappa}$, $F_{\kappa}$, $Q$ be as
    in~\ref{mrem:cube-setup}. Assume $s : \R \to \R$ is a function of
    class~$\cnt^{\infty}$ such that $\uD^i s(-1) = 0$ and $\uD^i s(1) = 0$ for
    $i \in \nat$. Define $h(x) = \sum_{j=1}^n s(x \bullet e_j) e_j$ for $x \in
    \R^n$. Then
    \begin{enumerate}[label=(\alph*)]
    \item
        \label{i:srQ:smooth}
        $g = h \circ f$ is of class $\cnt^\infty$ with $\uD^i g = \uD^i h \circ
        f$ for $i \in \nat$.

    \item
        \label{i:srQ:preserve}
        If $s$ is monotone increasing and $s(t) = t$ for $t \in \{-2,-1,0,1,2\}$,
        then for each $\kappa \in \{-1,0,1\}^n$ and $\lambda \in
        \{-2,-1,1,2\}^n$ if $C_{\kappa}$, $F_{\kappa}$, $T_{\kappa}$,
        $c_{\kappa}$, $R_{\lambda}$ are as in~\ref{mrem:cube-setup}, then
        \begin{gather*}
            g\lIm C_{\kappa} \rIm = F_{\kappa} \,,
            \quad
            g|_{F_{\kappa}} : F_{\kappa} \to F_{\kappa} \quad \text{is a~homeomorphism} \,,
            \\
            g\lIm T_\kappa \rIm \subseteq T_\kappa \,,
            \quad
            g\lIm R_\lambda \rIm \subseteq R_\lambda \,,
            \quad
            g\lIm c_{\kappa} + T_{\kappa} \rIm \subseteq c_{\kappa} + T_{\kappa} \,.
        \end{gather*}

    \item 
        \label{i:srQ:estimates}
        Let $\varepsilon \in (0,1)$. Assume $s$ satisfies $0 \le s'(t) \le 1 +
        \varepsilon$ and $|s(t) - t| \le \varepsilon$ for $t \in \R$ and
        $s'(t) > 0$ for $t \in \R \without \{-1,1\}$. Then
        \begin{gather*}
            |g(x) - x| \le (1 + \sqrt n) \varepsilon
            \quad \text{for $x \in Q + \cball 0{\varepsilon}$} \,,
            \\
            \Lip(g) = \Lip(h|_{Q}) \le 1 + \varepsilon \,,
            \quad
            \Lip(g - \id{\R^n}) \le 1 \,,
            \\
            \Lip(g|_{K} - \id{K}) < 1
            \quad \text{for each compact set $K \subseteq \Int Q$} \,.
        \end{gather*}
    \end{enumerate}
\end{lemma}

\begin{proof}
    Since $s'(1) = 0 = s'(-1)$ we have
    \begin{gather}
        \label{eq:Dh-in-Cl}
        \uD h(y)u 
        = \sum_{j : \lambda_j = 0} s'(y \bullet e_j) (u \bullet e_j) e_j = \uD h(y)\bigl(\eqproject{T_{\lambda}}u\bigr)
        \in T_{\lambda}  \quad \text{for $y \in C_{\lambda}$} \,.
    \end{gather}

    First, assume $x \in C_{\lambda}$ and $\dim(T_{\lambda}) > 0$. Define $p =
    \eqproject{T_{\lambda}}$ and $q = \eqproject{T_{\lambda}^{\perp}}$. Observe
    that for $j = 1,2,\ldots,n$
    \begin{gather}
        \label{eq:lambda-def}
        \begin{aligned}
            \lambda_j = \sgn(x \bullet e_j) \quad &\text{if and only if} \quad |x \bullet e_j| \ge 1
            \,,
            \\
            \text{and} \quad
            \lambda_j = 0 \quad &\text{if and only if} \quad |x \bullet e_j| < 1
            \,.
        \end{aligned}
    \end{gather}
    Since $\dim(T_{\lambda}) > 0$ there exists $j \in \{1,2,\ldots,n\}$ such
    that $\lambda_j = 0$. Let $u \in \R^n$ be such that
    \begin{gather}
        \label{eq:u-small}
        0 < |u| < \min\bigl\{ 1 - |x \bullet e_j| : \lambda_j = 0 \bigr\} < 1 \,.
    \end{gather}
    Let $\kappa \in \{-1,0,1\}^n$ be defined by
    \begin{gather}
        \label{eq:kappa-def}
        \begin{aligned}
            \kappa_j = \sgn((x+qu) \bullet e_j) \quad &\text{if and only if} \quad |(x+qu) \bullet e_j| \ge 1
            \,,
            \\
            \text{and} \quad
            \kappa_j = 0 \quad &\text{if and only if} \quad |(x+qu) \bullet e_j| < 1
            \,.
        \end{aligned}
    \end{gather}
    Then $x + tqu \in C_{\kappa}$ for all $t \in (0,1)$ because $C_{\kappa}$ is
    convex; hence, $x \in \Clos(C_{\kappa})$. Moreover,
    using~\eqref{eq:lambda-def}, \eqref{eq:u-small}, \eqref{eq:kappa-def},
    we~see that for $j = 1,2,\ldots,n$
    \begin{gather*}
        \lambda_j = 0
        \quad \text{implies} \quad
        |(x+qu) \bullet e_j| = |x \bullet e_j| < 1
        \quad \text{implies} \quad
        \kappa_j = 0 \,,
        \\
        \lambda_j \ne 0
        \quad \text{implies} \quad
        ((x+qu) \bullet e_j) \lambda_j = (x \bullet e_j) \lambda_j + (u \bullet e_j) \lambda_j > 0 \,;
        \\
        \quad \text{hence, either} \quad
        \kappa_j = \lambda_j \ne 0 
        \quad \text{or} \quad
        \kappa_j = 0 \text{ and } \lambda_j \ne 0
        \,.
    \end{gather*}
    Therefore, $T_{\lambda} \subseteq T_{\kappa}$. Recalling $x \in
    \Clos(C_{\kappa}) \cap C_{\lambda}$, we see that if $j \in \{1,2,\ldots,n\}$
    and $\lambda_j \ne \kappa_j$, then $(x \bullet e_j) = \lambda_j$. Thus,
    \begin{gather}
        \label{eq:Tk-Tl}
        \eqproject{T_{\kappa}}x - \eqproject{T_{\lambda}}x 
        = \sum_{j : \kappa_j = 0} (x \bullet e_j) e_j - \sum_{j : \lambda_j = 0} (x \bullet e_j) e_j 
        = \sum_{j : \kappa_j = 0, \lambda_j \ne 0} \lambda_j e_j 
        = c_{\lambda} - c_{\kappa} \,.
    \end{gather}
    Using~\eqref{eq:Tk-Tl} we derive
    \begin{gather*}
        \label{eq:f-perp-inc}
        f(x + qu) - f(x) = f_{\kappa}(x + qu) - f_{\lambda}(x)
        = \eqproject{T_{\kappa}}(qu)
        \in T_{\kappa} \cap T_{\lambda}^{\perp} \,.
    \end{gather*}
    Moreover, since $pu \in T_{\kappa}$ and $x+qu \in C_{\kappa}$ and,
    by~\eqref{eq:u-small}, $x+qu+pu \in C_{\kappa}$ we obtain
    \begin{gather*}
        f(x + qu + pu) - f(x + qu) 
        = \eqproject{T_{\kappa}}(x + qu + pu) - \eqproject{T_{\kappa}}(x + qu)
        = pu \in T_{\lambda} \,.
    \end{gather*}
    Thus,
    \begin{gather*}
        \xi = f(x+u) - f(x)
        = f(x+qu+pu)-f(x+qu) + f(x+qu)-f(x)
        = pu + \eqproject{T_{\kappa}}(qu)
    \end{gather*}
    and, recalling~\eqref{eq:Dh-in-Cl},
    \begin{multline*}
        |g(x+u) - g(x) - \uD h(f(x))u|
        = |h(f(x) + \xi) - h(f(x)) - \uD h(f(x))(pu)|
        \\
        = |h(f(x) + \xi) - h(f(x)) - \uD h(f(x))(\xi)|
    \end{multline*}
    Since $|\xi| \le \Lip(f)|u|$ and $h$ is of class~$\cnt^1$ we obtain
    \begin{gather*}
        \lim_{u \to 0} |u|^{-1} |g(x+u) - g(x) - \uD h(f(x))u| = 0 \,.
    \end{gather*}
    This shows that $\uD g(x) = \uD h(f(x))$ in case $\dim(T_{\lambda}) > 0$.

    Now we shall deal with the case when $x \in C_{\lambda}$ and
    $\dim(T_{\lambda}) = 0$. This means that $f(x) \in F_{\lambda}$ is one of
    the vertexes of~$Q$. Since $h$ is of class~$\cnt^1$ and $\Lip(f) < \infty$
    (see~\ref{rem:f-lipschitz}) and $\uD h(f(x)) = 0$ by~\eqref{eq:Dh-in-Cl}, we
    get
    \begin{gather*}
        |g(x+u) - g(x)| = |h(f(x+u)) - h(f(x)) - \uD h(f(x))(f(x+u) - f(x))| \,.
    \end{gather*}
    Hence, in this case we also get $\uD g(x) = \uD h(f(x))$

    Now we know that $\uD g(x) = \uD h(f(x))$ for all $x \in \R^n$ and since $f$ is
    continuous we see that $g$ is of class~$\cnt^1$. Repeating the whole
    argument with $g = h \circ f$ replaced by $\uD g = \uD h \circ f$ and proceeding
    by induction we see that $g$ is of class~$\cnt^\infty$ so~\ref{i:srQ:smooth}
    is proven.

    Item~\ref{i:srQ:preserve} readily follows from the definition of~$g$.

    Consider now $\varepsilon$ and $s$ as in~\ref{i:srQ:estimates}. For $x \in Q +
    \cball 0{\varepsilon}$ we have
    \begin{multline*}
        |g(x) - x| \le |f(x) - x| + |h(f(x)) - f(x)|
        \\
        \le \varepsilon + \Bigl( \sum_{i=1}^n \bigl( s(f(x) \bullet e_i) - f(x) \bullet e_i \bigr)^2 \Bigr)^{1/2}
        \le (1 + \sqrt n)\varepsilon  \,.
    \end{multline*}
    For $y \in \R^n$ and $u \in \R^n$ with $|u|=1$
    \begin{gather*}
        |\uD h(y)u|^2 = \sum_{i=1}^n s'(y \bullet e_i)^2 (u \bullet e_i)^2 \le 1 + \varepsilon \,;
        \quad \text{hence,} \quad
        \Lip(g) = \Lip(h|_Q) \le 1 + \varepsilon \,.
    \end{gather*}
    For any $y \in Q$ and $u \in \R^n$, recalling $-1 \le s'(t) - 1 \le
    \varepsilon < 1$ for $t \in \R$, we have
    \begin{gather*}
        |\uD h(y)u - u|^2 
        = \sum_{i=1}^n \bigl(s'(y \bullet e_i) - 1\bigr)^2 (u \bullet e_i)^2
        \le 1  \,;
        \quad \text{hence,} \quad
        \Lip(g - \id{\R^n}) \le 1 \,.
    \end{gather*}
    For $K \subseteq \Int Q$ compact and $y \in K$ we have $-1 < s'(y \bullet
    e_i) - 1 < \varepsilon < 1$ so $|\uD h(y)u - u|^2 < 1$ whenever $u \in \R^n$
    satisfies $|u|=1$. Consequently, $\Lip(g|_K - \id{K}) < 1$.
\end{proof}

Next, using~\ref{lem:smooth-retract-on-Q}, we construct another function which
maps some neighbourhood of~$Q=[-1,1]^n$ onto~$Q$ and is the identity a bit
further away from~$Q$. To this end we put~$Q$ inside a~convex open set~$V$ with
smooth boundary and use the distance from~$V$, which is smooth away from the
boundary of~$V$, to interpolate between the mapping constructed
in~\ref{lem:smooth-retract-on-Q} and the identity.

\begin{corollary}
    \label{cor:retr+}
    Let $n \in \nat$. For each $\varepsilon \in (0,1)$ there exists a map $l :
    \R^n \to \R^n$ of class~$\cnt^\infty$ such that if $\Gamma = 16 \sqrt
    n$, then
    \begin{enumerate}[label=(\alph*)]
    \item
        \label{i:rr:id}
        $l(x) = x$ for $x \in \R^n$ satisfying $\dist(x,Q) > \varepsilon$.
    \item 
        \label{i:rr:preserve}
        For each $\kappa \in \{-1,0,1\}^n$ and $\lambda \in \{-2,-1,1,2\}^n$ if
        $C_{\kappa}$, $F_{\kappa}$, $T_{\kappa}$, $c_{\kappa}$, $R_{\lambda}$
        are as in~\ref{mrem:cube-setup}, then
        \begin{gather*}
            l\lIm T_{\kappa} \rIm \subseteq T_{\kappa} \,,
            \quad
            l\lIm F_{\kappa} \rIm \subseteq F_{\kappa} \,,
            \quad
            l\lIm C_{\kappa} \rIm \subseteq C_{\kappa} \,,
            \quad 
            l\lIm c_{\kappa} + T_{\kappa} \rIm \subseteq c_{\kappa} + T_{\kappa} \,,
            \\ 
            l\lIm R_{\lambda} \rIm \subseteq R_{\lambda} \,,
            \quad 
            l\lIm \{ x \in C_{\kappa} : \dist(x, Q) \le \varepsilon/\Gamma \} \rIm \subseteq F_{\kappa} \,.
        \end{gather*}
    \item 
        \label{i:rr:diffeo}
        $l|_{\Int Q} : \Int Q \to \Int Q$ is a diffeomorphism such that for
        each compact set $K \subseteq \Int Q$ we have $\Lip(l|_K - \id{K}) < 1$.
    \item 
        \label{i:rr:lip-on-Q}
        $\Lip(l|_Q - \id{Q}) \le 1$ and $\Lip(l|_Q) \le 1 + \varepsilon$.
    \item
        \label{i:rr:lip-all}
        $\Lip(l) < \Gamma$.
    \item
        \label{i:rr:small}
        $|l(x) - x| \le \varepsilon$ for $x \in \R^n$.
    \item
        \label{i:rr:dist-Q}
        $\dist(l(x), Q) \le \dist(x,Q)$ for $x \in \R^n$.
    \end{enumerate}
\end{corollary}

\begin{proof}
    Let $n \in \nat$ and $\varepsilon \in (0,1)$. Set $\iota = \varepsilon/(2(1
    + \sqrt n))$. Let $\alpha : \R \to \R$ be map of class~$\cnt^\infty$
    such that
    \begin{gather*}
        \alpha(t) = 0 \quad \text{for $t \le 0$} \,,
        \quad
        \alpha(t) = 1 \quad \text{for $t \ge 1$} \,,
        \quad
        0 < \alpha'(t) \le 1 + \iota  \quad \text{for $t \in (0,1)$} \,.
    \end{gather*}
    Let $s : \R \to \R$ be a homeomorphism of class~$\cnt^\infty$ such that
    \begin{gather*}
        0 \le s'(t) \le 1 + \iota 
        \quad \text{and} \quad
        |s(t) - t| \le \iota
        \quad \text{for $t \in \R$} \,,
        \quad 
        s(t) = t \quad \text{for $t \in \{-2,-1,0,1,2\}$} \,,
        \\
        s'(t) > 0  \quad \text{for $t \in \R \without \{-1,1\}$} \,,
        \quad
        \uD^js(-1) = 0 = \uD^js(1)
        \quad \text{for each $j \in \nat$} \,.
    \end{gather*}
    Choose an open convex set $V \subseteq \R^n$ such that $Q + \cball
    0{\iota/4} \subseteq V \subseteq Q + \cball 0{\iota/2}$ and $\Bdry V$ is a
    submanifold of $\R^n$ of class~$\cnt^{\infty}$.  Define $g$ as
    in~\ref{lem:smooth-retract-on-Q} using~$s$ and set
    \begin{gather*}
        \delta(x) = \dist(x,V) \,, \quad
        l(x) = g(x) + (x - g(x))\alpha(2\delta(x)/\iota) \,.
    \end{gather*}
    Since $V$ is convex and $\Bdry V$ is of class~$\cnt^\infty$, we see that $l$
    is of class~$\cnt^\infty$. Clearly $l(x) = x$ whenever $\dist(x,Q) \ge
    \varepsilon \ge \iota$ which establishes~\ref{i:rr:id}.

    \emph{Proof of~\ref{i:rr:preserve}.} Since $T_\kappa$, $F_\kappa$,
    $C_\kappa$, $c_\kappa + T_\kappa$, $R_\lambda$ are convex the inclusions
    $l\lIm T_{\kappa} \rIm \subseteq T_{\kappa}$, $l\lIm F_{\kappa} \rIm
    \subseteq F_{\kappa}$, $l\lIm C_{\kappa} \rIm \subseteq C_{\kappa}$, $l\lIm
    c_{\kappa} + T_{\kappa} \rIm \subseteq c_{\kappa} + T_{\kappa}$, $l\lIm
    R_{\lambda} \rIm \subseteq R_{\lambda}$ readily follow
    from~\ref{lem:smooth-retract-on-Q}\ref{i:srQ:preserve}. For $x \in V$ we
    have $l(x) = g(x)$ so, noting $\iota/4 \ge \varepsilon/(16 \sqrt n)$, we see
    that $l(x) = g(x) \in F_{\kappa}$ whenever $x \in C_{\kappa}$ and $\dist(x,
    Q) \le \varepsilon/\Gamma$.

    Employing~\cite[4.8(3)]{MR0110078}, we see that $\Lip(\delta) = 1$; hence,
    we obtain for $x \in Q + \cball 0\varepsilon$
    \begin{gather*}
        \| \uD l(x) \| = \|\uD g(x)\| + \|\uD g(x) - \id{\R^n}\| 
        + |x - g(x)| (1+\iota) 2/\iota 
    \end{gather*}
    Items~\ref{i:rr:diffeo} and~\ref{i:rr:lip-on-Q} follow immediately
    from~\ref{lem:smooth-retract-on-Q}\ref{i:srQ:estimates} noting that $l(x) =
    g(x)$ for $x \in Q$.
    Recalling~\ref{lem:smooth-retract-on-Q}\ref{i:srQ:estimates} we have
    \begin{gather*}
        \Lip(l) \le 1 + \iota + 1 + 2 (1 + \sqrt n) \iota (1 + \iota) / \iota
        \le 11 \sqrt n
    \end{gather*}
    so~\ref{i:rr:lip-all} holds. For $x \in \R^n$ with $\dist(x,Q) \le \varepsilon$
    we have
    \begin{gather*}
        |l(x) - x| \le |g(x) - x| + |x - g(x)| \le 2(1 + \sqrt n) \iota = \varepsilon \,,
    \end{gather*}
    which proves~\ref{i:rr:small}. To verify~\ref{i:rr:dist-Q} note that $l(x)
    \in \conv\{x,g(x)\}$ and $g(x) \in Q$ for $x \in \R^n$.
\end{proof}

% Local Variables:
% coding: utf-8
% eval: (ispell-change-dictionary "british")
% eval: (flyspell-mode)
% End:

%% file: central-proj.tex
\section{Central projections}
\label{sec:c-proj}

Here we study analytic properties of the central projection from the origin onto
the boundary of a~bounded convex set $V$ containing~$0$.
In~\ref{cor:centr-proj} we derive formulas and estimates for the derivative of
such projection in terms of the position of the origin with respect to the
boundary~$\Bdry V$ and the shape of $\Bdry V$. Then in~\ref{cor:proj+} we
interpolate between a central projection and the identity to get a map which
acts as the central projection inside $V$ and is the identity outside given
neighbourhood of~$V$.

We start by deriving a formula for the derivative of the central projection onto
the boundary of a half-space.

\begin{remark}
    \label{rem:proj-plane}
    Let $\nu, y \in \R^n$ be such that $|\nu| = 1$ and $\nu \bullet y > 0$.
    Define
    \begin{gather*}
        U = \{ z \in \R^n : \nu \bullet z > 0 \} \,,
        \quad
        s : U \to \R \,, \quad
        s(z) = \frac{\nu \bullet y}{\nu \bullet z}
        \quad \text{for $z \in U$} \,,
        \\
        \pi : U \to \R^n \,, \quad
        \pi(z) = s(z) z 
        \quad \text{for $z \in U$} \,.
    \end{gather*}
    Then $s$ and $\pi$ are maps of class $\cnt^{\infty}$ and $\pi$ is the
    central projection from the origin onto the plane $y + T$. A~straightforward
    computation shows also that for $z \in U$, $u \in \R^n$
    \begin{gather*}
        \uD s(z)u = -\frac{(\nu \bullet y) (\nu \bullet u)}{(\nu \bullet z)^2} \,,
        \quad
        \| \uD\pi(z) \| \le \frac{1}{|z|} \biggl(
            \frac{\nu \bullet y}{\nu \bullet \frac{z}{|z|}}
            + \frac{\nu \bullet y}{\bigl(\nu \bullet \frac{z}{|z|}\bigr)^2}
        \biggr) \,.
    \end{gather*}
\end{remark}

\begin{definition}
    \label{def:centr-proj}
    Let $V \subseteq \R^n$ be an open bounded convex set such that $0 \in V$.
    We say that a pair of maps $p : \R^n \without \{0\} \to \R^n$ and $t : \R^n
    \without \{0\} \to (0,\infty)$ defines the \emph{central projection onto
      $\Bdry V$} if
    \begin{displaymath}
        p(x) = t(x)x 
        \quad \text{and} \quad
        t(x) = \sup \bigl\{ t \in (0,\infty) : tx \in V \bigr\} 
        \quad \text{for $x \in \R^n \without \{0\}$} \,.
    \end{displaymath}
\end{definition}

In the next lemma we prove that the derivative at some point~$x$ of the central
projection onto the boundary of a~convex set~$V$ depends only on the affine
tangent plane of~$\Bdry V$ at~$x$ (assuming it exists) and, actually, coincides
with the derivative of the central projection onto that tangent plane.

\begin{lemma}
    \label{lem:convex-plane-comp}
    Let $V \subseteq \R^n$ be an open bounded convex set with $0 \in V$.
    Assume $y \in \Bdry V$, and $\Tan(\Bdry V, y) \in \grass{n}{n-1}$, and $\nu
    \in \Tan(\Bdry V, y)^{\perp}$ is the outward pointing unit normal to~$\Bdry
    V$ at~$y$; in particular $|\nu| = 1$ and $\nu \bullet y > 0$. Suppose $U$,
    $s$, $\pi$ are defined as in~\ref{rem:proj-plane} and $p$, $t$ define the
    central projection onto~$\Bdry V$.

    If $x \in \R^n \without \{0\}$ satisfies $p(x) = y$, then $p$ is differentiable
    at~$x$ and
    \begin{gather*}
        x \in U \,,
        \quad
        \pi(x) = p(x) = y \,,
        \quad
        \uD\pi(x) = \uD p(x) \,.
    \end{gather*}
 \end{lemma}

\begin{proof}
    Fix $x \in \R^n \without \{0\}$ such that $p(x) = y$. Set
    \begin{gather*}
        \eta = \frac{x}{|x|} \,,
        \quad
        S = \lin\{\eta\}^{\perp} \,,
        \quad
        T = \lin\{\nu\}^{\perp} \,,
        \quad
        \theta = \| \project S - \project T \| = \bigl( 1 - (\eta \bullet \nu)^2 \bigr)^{1/2} < 1 \,.
    \end{gather*}
    Let $\delta \in \R$ satisfy $0 < \delta \le 2^{-10} (\eta \bullet
    \nu)^{1/2}$. For $h \in \R^n$ with $|h| \le \delta |x|$ define
    \begin{gather*}
        \gamma_h = \frac{x+h}{|x+h|} \,,
        \quad
        Z_h = \lin\{\gamma_h\} \,,
    \end{gather*}
    and note that
    \begin{gather*}
        \gamma_h \bullet \eta
        = \biggl( 1 + \frac{|\project{S} h|^2}{|\project{S^{\perp}}(x+h)|^2} \biggr)^{-1/2}
        \ge \biggl( 1 + \frac{\delta^2}{(1 - \delta)^2} \biggr)^{-1/2} > 0 \,,
        \\
        \gamma_h \bullet \nu
        \ge \eta \bullet \nu - |\gamma_h - \eta|
        = \eta \bullet \nu - 2\bigl( 1 - \gamma_h \bullet \eta \bigr)
        \ge \eta \bullet \nu - \frac{\delta^2}{(1-\delta)^2} 
        \ge \bigl(1 - 2^{-18}\bigr) \eta \bullet \nu  > 0 \,.
    \end{gather*}
    For $h \in \R^n$ with $|h| \le \delta |x|$ we have
    \begin{gather*}
        \| \project S - \eqproject{Z_h^{\perp}} \|
        = |\project S \gamma_h| 
        = \bigl( 1 - (\eta \bullet \gamma_h)^2 \bigr)^{1/2} < 1 \,,
        \\
        \| \project T - \eqproject{Z_h^{\perp}} \|
        = |\project T \gamma_h| 
        = \bigl( 1 - (\nu \bullet \gamma_h)^2 \bigr)^{1/2} < 1 \,;
    \end{gather*}
    hence, we can define
    \begin{gather*}
        \lambda = \sup\bigl\{ \| \project S - \eqproject{Z_h^{\perp}} \|
        : h \in \R^n \,,\, |h| \le \delta |x| \bigr\} < 1 \,,
        \\
        \varphi = \sup\bigl\{ \| \project T - \eqproject{Z_h^{\perp}} \|
        : h \in \R^n \,,\, |h| \le \delta |x| \bigr\} < 1 \,.
    \end{gather*}
    Next, set
    \begin{gather*}
        \beta(r) = \frac 1r 
        \sup\bigl\{ |\project{T^{\perp}}(z-y)|
        : z \in \Bdry V \cap \cball yr \bigr\} \,.
    \end{gather*}
    Since $T = \Tan(\Bdry V,y) \in \grass{n}{n-1}$ we know that $\beta(r) \to 0$
    as $r \downarrow 0$ and there exists $0 < r_0 < 1$ such that $\beta(r) \le
    \frac 12 (1 - \theta)$ for $0 < r < r_0$.

    Observe that $p$ is continuous at~$x$. If it were not, there would exist
    a~sequence $h_i \in \R^n$ such that $|h_i| \to 0$ as $i \to \infty$ but $y_i
    = p(x+h_i)$ would not converge to~$y = p(x)$. Then $y_i$ would be in the
    cone $\{ tw : t > 0 , |x-w| \le |h_i|\}$ so, since $V$ is bounded, one could
    choose a subsequence of~$y_i$ which would converge to some point $y_0 \in
    S^{\perp}$ and $y_0 \ne y$. Since $V$ is convex, this would imply that $\{
    ty + (1-t)y_0 : 0 \le t \le1 \} \subseteq \Bdry V$. This, in turn, would
    mean that $\eta \in T = \Tan(\Bdry V, y)$ which is impossible because
    $|\project{T^{\perp}} \eta| = \eta \bullet \nu > 0$.

    Knowing that $p$ is continuous we can find $\rho_0 > 0$ such that $|p(x+h) -
    p(x)| \le r_0$ whenever $|h| \le \rho_0$. Fix $h \in \R^n$ with $|h| \le
    \min\{ \delta |x| , \rho_0 \}$ and let $b = p(x+h)$. Set
    \begin{gather*}
        \Gamma = 2 \biggl(1 + \frac{\theta}{1 - \varphi}\biggr)
        \biggl(\frac{|y|}{|x|} + \frac{1}{2|x|}\biggr) \,.
    \end{gather*}
    We shall show that $|b-y| = |p(x+h) - p(x)| \le \Gamma |h|$.

    Let $a,z \in \R^n$ be such that
    \begin{gather*}
        \{z\} = (b + T) \cap S^{\perp} \,,
        \quad
        \{a\} = (z + S) \cap \{ t b : t > 0\} \,;
        \quad \text{then} \quad b \in z + T \,.
    \end{gather*}
    Since $\eqproject{Z^\perp_h} (b-a) = 0$ and $\project{S^\perp} (z-a) = 0$
    and $\project{T^{\perp}}(b-z) = 0$ we obtain
    \begin{multline*}
        |b-a| \le |\project{T}(b-a)| 
        + |\project{T^{\perp}}(b-a)|
        \\
        \le |(\project{T} - \eqproject{Z^{\perp}_h})(b-a)| 
        + |\project{T^{\perp}}(b-z)| 
        + |(\project{T^{\perp}} - \project{S^{\perp}})(z-a)|
        \le \varphi |b-a| + \theta |a-z| \,.
    \end{multline*}
    Thus
    \begin{gather*}
        |b-a| \le \frac{\theta}{1 - \varphi} |a-z|
        \quad \text{and} \quad
        |b-z| \le |b-a| + |a-z| \le \Bigl(1 + \frac{\theta}{1 - \varphi}\Bigr) |a-z| \,.
    \end{gather*}
    Directly from the definition of $a$ it follows that 
    \begin{gather*}
        |a-z| = \frac{|z|}{|x|} |\project{S} h| \le \frac{|z|}{|x|} |h| \,;
        \quad \text{hence,} \quad
        |b-z| \le \Bigl(1 + \frac{\theta}{1 - \varphi}\Bigr) \frac{|z|}{|x|} |h| \,.
    \end{gather*}
    Recall that $|h| \le \rho_0$ so $|b-y| \le r_0 < 1$ so $\beta(|b-y|) \le
    \frac 12 (1-\theta)$ and we can write
    \begin{gather*}
        |z-y| \le |\project{T^{\perp}} (z-y)| + |\project{T} \project{S^{\perp}} (z-y)|
        \le \beta(|b-y|)|b-y| + \theta |z-y| \,;
        \quad \text{so} \quad
        |z| \le |y| + \tfrac 12 \,.
    \end{gather*}
    In consequence
    \begin{gather*}
        |b-y| \le \biggl(1 + \frac{\theta}{1 - \varphi}\biggr)
        \biggl(\frac{|y|}{|x|} + \frac 1{2|x|}\biggr) |h|
        + \tfrac 12 |b-y| 
        \quad \text{so} \quad
        |b-y| \le \Gamma |h| \,.
    \end{gather*}

    Let $h \in \R^n$ be such that $|h| \le \min\{ \delta|x|, \rho_0\}$. Now we
    are ready to estimate $|p(x+h) - \pi(x+h)|$. Set $u = p(x+h) - \pi(x+h)$ and
    observe that
    \begin{gather*}
        |\project{T^{\perp}} (p(x+h) - y)| \le \beta(\Gamma |h|) \Gamma |h| \,,
        \quad
        u \in Z_h \,,
        \quad
        |\project{T} u| = |\project{T} \project{Z} u| \le \varphi |u| \,,
        \\
        |\project{T^\perp} u|
        \le |\project{T^\perp} (p(x+h) - y)| + |\project{T^\perp} (y - \pi(x+h))| 
        = |\project{T^{\perp}} (p(x+h) - y)| \le \beta(\Gamma |h|) \Gamma |h| \,,
        \\
        |u| \le |\project{T} u| + |\project{T^\perp} u| 
        \le \varphi |u| + \beta(\Gamma |h|) \Gamma |h| \,,
        \\
        |p(x+h) - \pi(x+h)| = |u| \le \frac{1}{1 - \varphi} \beta(\Gamma |h|) \Gamma |h| \,.
    \end{gather*}
    It is clear from the definitions that $p(x) = \pi(x) = y$. In consequence
    \begin{multline*}
        \lim_{h \to 0} \frac{| p(x+h) - p(x) - \uD\pi(x)h |}{|h|}
        \\
        \le \lim_{h \to 0} \frac{| \pi(x+h) - \pi(x) - \uD\pi(x)h |}{|h|}
        + \frac{| p(x+h) - \pi(x+h) |}{|h|} = 0 \,,
    \end{multline*}
    which shows that $p$ is differentiable at~$x$ and $\uD p(x) = \uD\pi(x)$.
\end{proof}

\begin{corollary}
    \label{cor:centr-proj}
    Suppose $V \subseteq \R^n$ is an open bounded convex set, and $0 \in V$,
    and $\Bdry V$ is an $(n-1)$~dimensional submanifold of $\R^n$ of
    class~$\cnt^\infty$, and $p$, $t$ define the central projection onto~$\Bdry
    V$, and $\nu(y)$ is the outward pointing unit normal to~$\Bdry V$ at~$y$ for
    $y \in \Bdry V$. Then $p$ and~$t$ are of class~$\cnt^\infty$ and
    \begin{gather*}
        \uD t(x)u 
        = - \frac{\bigl(\nu(p(x)) \bullet p(x)\bigr) \bigl(\nu(p(x)) \bullet u\bigr)}
        {(\nu(p(x)) \bullet x)^2} \,,
        \quad
        \|\uD p(x)\| 
        \le \frac{|p(x)|}{|x|}
        \biggl( 1 +  \frac{1}{\nu(p(x)) \bullet \frac{x}{|x|}} \biggr) \,.
    \end{gather*}
    for $x \in \R^n \without \{0\}$ and $u \in \R^n$. 
\end{corollary}

\begin{proof}
    Since $\Bdry V$ is of class~$\cnt^1$ we can
    apply~\ref{lem:convex-plane-comp} and~\ref{rem:proj-plane} at any single
    point $x \in \R^n \without \{0\}$ to see that
    \begin{gather*}
        \uD p(x)u = \frac{\nu(p(x)) \bullet p(x)}{\nu(p(x)) \bullet x} u 
        - \frac{\bigl(\nu(p(x)) \bullet p(x)\bigr) \bigl(\nu(p(x)) \bullet u\bigr)}{(\nu(p(x)) \bullet x)^2} x 
    \end{gather*}
    for $u \in \R^n$. Noting $t(x) = p(x) \bullet x / (x \bullet x)$ one
    derives the formula for~$\uD t(x)$. This shows that $p$ and $t$ are of
    class~$\cnt^1$. Since $\nu$ is of class $\cnt^{\infty}$, proceeding by
    induction, we see that $p$ and~$t$ are of class $\cnt^\infty$.
\end{proof}

Next, we construct a map which interpolates between the central projection onto
$\Bdry V$ and identity outside some neighbourhood of $V$. 

\begin{corollary}
    \label{cor:proj+}
    Let $\varepsilon \in (0,1)$ and $V \subseteq \R^n$ be open convex with $0
    \in V$ and $p$, $t$ define the central projection onto $\Bdry V$. Assume
    $\Bdry V$ is an $n-1$~dimensional submanifold of $\R^n$ of
    class~$\cnt^\infty$. Then there exists a map $q : \R^n \without \{0\} \to
    \R^n$ of class~$\cnt^\infty$ such that
    \begin{enumerate}[label=(\alph*)]
    \item
        \label{i:pr:id}
        $q(x) = x$ for $x \in \R^n \without V$.
    \item
        \label{i:pr:coincide}
        $q(x) = p(x)$ for $x \in V \without \{0\}$ with $\dist(x, \R^n \without V) \ge \varepsilon$.
    \item
        \label{i:pr:exists-t}
        For each $x \in \R^n \without \{0\}$ there exists $t \in [1,\infty)$ such that $q(x) = tx$.
    \item
        \label{i:pr:conv}
        $q(x) \in \conv\{ x, p(x)\}$ for each $x \in \R^n \without \{0\}$.
    \item
        \label{i:pr:qp-comp}
        $|q(x) - x| \le |p(x) - x|$ whenever $x \in V \without \{0\}$.
    \item
        \label{i:pr:deriv-bound}
        $\|\uD q(x)\| \le 5 |p(x)| |x|^{-1} \Delta$ for $x \in V
        \without \{0\}$, where $\Delta = \inf \bigl\{ \nu(y) \bullet
        \frac{y}{|y|} : y \in \Bdry V \bigr\}^{-1}$.
    \end{enumerate}
\end{corollary}

\begin{proof}
    Set
    \begin{gather*}
        \iota = \min \bigl\{ \varepsilon \,,\, \inf\bigl\{ \dist(x, \R^n \without V) :
        x \in V \,,\,
        t(x) \ge 1 + \varepsilon \bigr\} \bigr\} \,,
        \\
        \delta = \inf\bigl\{ t(x) : x \in V \,,\, \dist(x,\R^n \without V) \ge \iota \bigr\}
        \,.
    \end{gather*}
    Then for $x \in \R^n \without \{0\}$
    \begin{gather}
        \label{eq:t-small-impl}
        1 < t(x) < \delta
        \quad \text{implies} \quad
        \dist(x , \R^n \without V) < \iota \le \varepsilon \,,
        \\
        \label{eq:dist-small-impl}
        \dist(x, \R^n \without V) < \iota 
        \quad \text{implies} \quad
        t(x) - 1 < \varepsilon \,,
        \\
        \label{eq:dist-big-impl}
        \dist(x, \R^n \without V) \ge \iota
        \quad \text{implies} \quad
        t(x) \ge \delta \,.
    \end{gather}
    Choose $\alpha : \R \to \R$ of class~$\cnt^\infty$ such that
    \begin{gather*}
        \alpha(t) = t \quad \text{for $t \ge \delta$} \,,
        \quad
        \alpha(t) \le t \quad \text{for $t \ge 1$} \,,
        \\
        \alpha(t) = 1 \quad \text{for $t \le 1$} \,,
        \quad
        0 \le \alpha'(t) \le 2 \quad \text{for $t \in \R$} \,.
    \end{gather*}
    Set $q(x) = \alpha(t(x)) x$.

    Clearly $q : \R^n \without \{0\} \to \R^n$ is of class~$\cnt^\infty$ and
    $q(x) \in \conv\{x, p(x)\}$ for $x \in \R^n \without \{0\}$ and $q(x) = x$
    for $x \in \R^n \without V$. Moreover, $q(x) = p(x)$ for $x \in V \without
    \{0\}$ satisfying $\dist(x, \R^n \without V) \ge \varepsilon$, which
    follows by~\eqref{eq:dist-big-impl}.  For $x \in V \without \{0\}$ with $0 <
    \dist(x,\R^n \without V) \le \iota$ we have,
    using~\eqref{eq:dist-small-impl},
    \begin{gather*}
        |q(x) - x| = |x| (\alpha(t(x)) - 1) \le |x| (t(x) - 1) = |p(x) - x|
    \end{gather*}
    For $x \in V \without \{0\}$, using the formulas from~\ref{cor:centr-proj}
    and the identity $t(x) = |p(x)| |x|^{-1}$, we get
    \begin{gather*}
        \|\uD p(x)\| \le |p(x)| |x|^{-1} \bigl( 1 + \Delta \bigr) 
        \le 2 |p(x)| |x|^{-1} \Delta  \,,
        \\
        |\uD q(x)u| 
        \le |\alpha'(t(x)) \uD p(x)u| + |\alpha'(t(x))t(x)u| + |\alpha(t(x))u| 
        \quad \text{for $u \in \R^n$, $|u|=1$} \,,
        \\
        \|\uD q(x)\| \le 5 |p(x)| |x|^{-1} \Delta \,.
        \qedhere
    \end{gather*}
\end{proof}

% Local Variables:
% coding: utf-8
% eval: (ispell-change-dictionary "british")
% eval: (flyspell-mode)
% End:

%% file: smooth-ff.tex
\section{Smooth deformation theorem}
\label{sec:deform}

Here we prove a version of Federer--Fleming projection theorem suited for our
purposes. The~proof follows the scheme of~\cite[4.2.6-9]{Fed69}. A~similar
result was also proven in~\cite[Theorem 3.1]{DS2000}. However, we need the
deformation to be smooth, we need to work in Whitney cubes rather than in a grid
of cubes of the same size, and we also need estimates on the measure of the
whole deformation.

To deal with families of dyadic cubes it will be convenient to introduce some
more notation. We shall follow~\cite[1.1--1.9]{Alm1986}.

\begin{definition}
    \label{def:cube}
    Let $k \in \{0,1,\ldots,n\}$ and $Q = [0,1]^k \subseteq \R^k$. We say that
    $R \subseteq \R^n$ is a \emph{cube} if there exist $p \in \orthproj nk$ and
    $o \in \R^n$ and $l \in (0,\infty)$ such that $R = \trans{o} \circ p^* \circ
    \scale{l} \lIm Q \rIm$. We call $\vertex R = o$ the \emph{corner} of~$R$ and
    $\side R = l$ the~\emph{side-length} of~$R$. We~also set
    \begin{itemize}
    \item $\dim(R) = k$ -- the \emph{dimension} of $R$,
    \item $\centre R = \vertex R + \frac 12 \side R (1,1,\ldots,1)$ -- the \emph{centre} of~$R$,
    \item $\cBdry R = \trans{\vertex R} \circ p^* \circ \scale{\side R} \lIm \Bdry Q \rIm$
        -- the \emph{boundary} of~$R$,
    \item $\cInt R = R \without \cBdry R$ -- the \emph{interior} of~$R$.
    \end{itemize}
\end{definition}

\begin{definition}
    \label{def:dyadic-cubes}
    Let $k \in \{0,1,\ldots,n\}$, and $N \in \integers$, and $Q = [0,1]^k
    \subseteq \R^k$, and $e_1$, \ldots, $e_n$ be the standard basis
    of~$\R^n$, and $f_1$, \ldots, $f_k$ be the standard basis of~$\R^k$.

    We~define $\cubes_k(N)$ to be the set of all cubes $R \subseteq \R^n$ of
    the form $R = \trans{v} \circ p^* \circ \scale{2^{-N}} \lIm Q \rIm$, where
    $v \in \scale{2^{-N}}\lIm \integers^n \rIm$ and $p \in \orthproj nk$ is such
    that $p^*(f_i) \in \{ e_1,\ldots,e_n\}$ for $i = 1,2,\ldots,k$.

    We also set
    \begin{gather*}
        \cubes_k = \tbcup\bigl\{ \cubes_k(N) : N \in \integers \bigr\} \,,
        \quad
        \cubes = \cubes_n \,,
        \quad
        \cubes_* = \tbcup\bigl\{ \cubes_k : k \in \{0,1,\ldots,n\} \bigr\} \,.
    \end{gather*}
\end{definition}

\begin{definition}
    \label{def:face}
    Let $k \in \{0,1,\ldots,n\}$, $N \in \integers$, and $K \in \cubes_k(N)$.
    We~say that $L \in \cubes_*$ is a~\emph{face} of~$K$ if and only if $L
    \subseteq K$ and $L \in \cubes_j(N)$ for some $j \in \{0,1,\ldots,k\}$.
\end{definition}

\begin{definition}[\protect{cf.~\cite[1.5]{Alm1986}}]
    \label{def:admissible}
    A family of top-dimensional cubes $\mathcal F \subseteq \cubes$ is said to
    be \emph{admissible} if
    \begin{enumerate}
    \item $K,L \in \mathcal F$ and $K \ne L$ implies $\cInt K \cap \cInt L = \varnothing$,
    \item $K,L \in \mathcal F$ and $K \cap L \ne \varnothing$ implies $\frac 12 \le \side L / \side K \le 2$,
    \item $K \in \mathcal F$ implies $\cBdry K \subseteq \bigcup \{ L \in \mathcal F : L \ne K \}$.
    \end{enumerate}
\end{definition}

\begin{definition}[\protect{cf.~\cite[1.6]{Alm1986}}]
    \label{def:whitney}
    Let $U \subseteq \R^n$ be an open set and $e_1$, \ldots, $e_n$ be the
    standard basis of~$\R^n$. We define the \emph{Whitney family} $\WF(U)$
    corresponding to~$U$ to consist of those cubes $K \in \cubes$ for which
    \begin{itemize}
    \item $\dist_{\infty}(K,\R^n \without U) > 2 \side K$,
    \item if $K \subseteq L \in \cubes$ and $\side L = 2 \side K$, then
        $\dist_{\infty}(L,\R^n \without U) \le 4 \side K$.
    \end{itemize}
    where $\dist_{\infty}(x,y) = \max\{ | (x-y) \bullet e_i| : i = 1,2,\ldots,n
    \}$ for $x,y \in \R^n$ and $\dist_{\infty}(A,B) = \inf\{
    \dist_{\infty}(x,y) : x \in A \,,\, y \in B\}$ for $A,B \subseteq \R^n$.
\end{definition}

\begin{remark}
    If $U \subseteq \R^n$ is open, then the Whitney family $\WF(U)$ is
    admissible.
\end{remark}

\begin{definition}[\protect{cf.~\cite[1.8]{Alm1986}}]
    \label{def:cube-complex}
    Let $\mathcal F \subseteq \cubes$ be admissible. We define the \emph{cubical
      complex $\CX(\mathcal F)$ of $\mathcal F$} to consist of all those cubes
    $K \in \cubes_*$ for which
    \begin{itemize}
    \item $K$ is a face of some cube from $\mathcal F$,
    \item if $\dim(K) > 0$, then $\side K \le \side L$ whenever $L$ is a face of
        some cube in $\mathcal F$ with $\dim(K) = \dim(L)$ and $\cInt K \cap
        \cInt L \ne \varnothing$.
    \end{itemize}
\end{definition}

\begin{remark}
    The second item of~\ref{def:cube-complex} means that whenever two top
    dimensional cubes $P$ and $Q$ from~$\mathcal F$ touch and $P$ is smaller
    than $Q$ and $F$ is a lower dimensional face of $Q$ such that $Q \cap P
    \subseteq F$, then the cubical complex $\CX(\mathcal F)$ does not
    contain~$F$ but rather cubes coming from subdivision of~$F$. This is a key
    property allowing to construct deformations onto skeletons of~$\CX(\mathcal
    F)$.
\end{remark}

Now we are ready to construct a map which is the main building block for the
deformation theorem~\ref{thm:deformation}. Given $a \in (-1,1)^n$ we need to
construct a~$\cnt^\infty$~smooth function $\R^n \to \R^n$ which maps the pointed
cube~$Q \without \{a\} = [-1,1]^n \without \{a\}$ onto $\Bdry Q$, preserves all
the lower dimensional skeletons of~$Q$, preserves the neighbouring dyadic cubes
of side length~$1$, is very close to the identity on~$\Bdry Q$, and is the
identity outside a small neighbourhood of~$Q$. Moreover, if $\dist(a,\Bdry Q)
\ge 1/2$, then we need to control the derivative at~$x \in Q \without \{a\}$ of
this function by a quantity of magnitude $|x-a|^{-1}$. To achieve all this we
proceed as follows. First we apply a diffeomorphism of~$\R^n$ which
preserves~$Q$ and moves~$a$ onto the origin -- this step is necessary to
preserve the neighbouring cubes. Then, we use the ``central projection''
constructed in~\ref{cor:proj+} to map $Q \without \{a\}$ onto the boundary of
some convex set~$V$ with smooth boundary such that $Q \subseteq V$. Finally, we
employ the ``smooth retraction'' produced in~\ref{cor:retr+} to map~$\Bdry V$
onto~$\Bdry Q$.

\begin{lemma}
    \label{lem:smooth-centr-proj}
    Let $Q = [-1,1]^n$ and $\varepsilon \in \bigl(0,\frac 14\bigr)$. There
    exists $\Gamma = \Gamma(n) \in (0,\infty)$ such that for each $a \in
    \Int(Q)$ there exists $\varphi_{a,\varepsilon} : \R^n \without \{a\} \to
    \R^n$ of class~$\cnt^\infty$ such that
    \begin{enumerate}
    \item
        \label{i:scp:id-outside}
        $\varphi_{a,\varepsilon}(x) = x$ for $x \in \R^n$ with $\dist(x,Q) \ge \varepsilon$.
    \item
        \label{i:scp:small-out}
        $|\varphi_{a,\varepsilon}(x) - x| \le \varepsilon$ 
        for $x \in \R^n \without Q$.
    \item
        \label{i:scp:Q-im}
        $\varphi_{a,\varepsilon} \lIm Q \without \{a\} \rIm = \Bdry Q$.
    \item
        \label{i:scp:Fk-im}
        If $\kappa \in \{-1,0,1\}^n \without \{ (0,0,\ldots,0) \}$ and $\lambda
        \in \{-2,-1,1,2\}^n \without \{-2,2\}^n $ and $C_{\kappa}$,
        $F_{\kappa}$, $T_{\kappa}$, $c_{\kappa}$, $R_\lambda$ are defined as
        in~\ref{mrem:cube-setup}, then
        \begin{gather*}
            \varphi_{a,\varepsilon}\lIm F_{\kappa} \rIm \subseteq \Clos F_{\kappa} \,,
            \quad
            \varphi_{a,\varepsilon}\lIm C_{\kappa} \rIm \subseteq \Clos C_{\kappa} \,,
            \quad
            \varphi_{a,\varepsilon}\lIm c_{\kappa} + T_{\kappa} \rIm \subseteq c_{\kappa} + T_{\kappa} \,,
            \\
            \varphi_{a,\varepsilon}\lIm R_{\lambda} \rIm \subseteq R_{\lambda} \,,
            \quad
            \varphi_{a,\varepsilon}\lIm T_{\kappa} \without Q \rIm \subseteq T_{\kappa} \,,
            \quad
            a \in T_{\kappa}
            \quad \text{implies} \quad
            \varphi_{a,\varepsilon}\lIm T_{\kappa} \without \{a\} \rIm \subseteq T_{\kappa} \,.
        \end{gather*}
    \item
        \label{i:scp:deriv-bound}
        For $x \in \Int Q \without \{a\}$ we have
        \begin{displaymath}
            \|\uD\varphi_{a,\varepsilon}(x)\| 
            \le \frac{\Gamma}{2 |x-a| \dist(a, \R^n \without Q)} \,.
        \end{displaymath}
    \item
        \label{i:scp:lip-varphi}
        For $x \in \R^n$ with $\dist(x, \Bdry Q) \le \min\{ \frac 12 \dist(a,
        \R^n \without Q), \frac 14 \}$ there holds
        \begin{displaymath}
            \|\uD \varphi_{a,\varepsilon}(x)\| \le \Gamma \,.
        \end{displaymath}
    \item
        \label{i:scp:small-bdry}
        For $x \in \R^n$ with $\dist(x, \Bdry Q) \le \min\{ \frac 12 \dist(a,
        \R^n \without Q), \frac 14 \}$ we obtain
        \begin{displaymath}
            |\varphi_{a,\varepsilon}(x) - x| \le \Gamma (\dist(x,\Bdry Q) + \varepsilon) \,.
        \end{displaymath}
    \item 
        \label{i:scp:bdry-behav}
        If $x \in \R^n$, $\delta \in \R$, $0 < \delta < \min \{ \frac 12
        \dist(a, \R^n \without Q), \frac 14 \}$, and $\dist(x, \Bdry Q) \le
        \delta$, then
        \begin{displaymath}
            \conv \bigl\{ x, \varphi_{a,\varepsilon}(x) \bigr\} 
            \subseteq \Bdry Q + \cball 0{\delta} \,.
        \end{displaymath}
        In~particular $\dist(\varphi_{a,\varepsilon}(x), \Bdry Q) \le \dist(x,
        \Bdry Q)$ for each $x \in \R^n \without \{a\}$.
    \item 
        \label{i:scp:a-cont}
        Let $X \subseteq \Bdry Q$ be compact and convex, $\kappa \in
        \{-1,0,1\}^n$, $T_{\kappa}$ be defined as in~\ref{mrem:cube-setup}.
        For~$a \in \Int Q$ define $E_a = \varphi_{a,\varepsilon}^{-1}\lIm X
        \rIm$. Then
        \begin{equation}
            \label{eq:scp:HD-conv}
            \lim_{b \to a} \HD \bigl( E_a, E_b \bigr) = 0 
            \quad 
            \text{for $a \in \Int Q$} \,.
        \end{equation}
        Moreover, if~$\mu$ is a~Radon measure over~$\R^n$ and $\dim T_{\kappa} =
        k$, then for $\HM^k$~almost all $a \in \bigl(-\tfrac 12, \tfrac 12
        \bigr)^n \cap T_{\kappa}$
        \begin{equation}
            \label{eq:scp:Radon}
            \mu\bigl(
            {\textstyle \bigcap_{\delta > 0} \bigcup_{b \in \cball{a}{\delta} \cap T_{\kappa}}}
            ((E_a \without E_b) \cup (E_b \without E_a)) \cap T_{\kappa}
            \bigr) = 0 \,.
        \end{equation}
    \end{enumerate}
\end{lemma}

\begin{proof}
    For each $\delta \in [0,\infty)$ let $p_\delta$, $t_\delta$ define the
    central projection onto $\Bdry( Q + \oball 0 \delta )$ as
    in~\ref{def:centr-proj}. Observe that
    \begin{itemize}
    \item $t_\delta$ is continuous for each $\delta \in [0,\infty)$,
    \item $t_\delta(x) \le t_\sigma(x)$ whenever $x \in \R^n \without \{0\}$ and $0 \le \delta \le \sigma < \infty$,
    \item $\lim_{\delta \downarrow 0} t_\delta(x) = t_0(x)$ for $x \in \R^n \without \{0\}$.
    \end{itemize}
    Employing the Dini Theorem (cf.~\cite[7.13]{MR0385023}), we see that
    $t_\delta$ converge uniformly to $t_0$ as $\delta \downarrow 0$ on compact
    subsets of $\R^n \without \{0\}$. Therefore, there exists $\delta_0 \in
    (0,\varepsilon)$ such that
    \begin{gather}
        \label{eq:good-delta}
        \Bdry Q \subseteq \{ x \in \R^n : 1 < t_\delta(x) < 1 + 2^{-10} n^{-1/2} \varepsilon \}
        \quad \text{for $\delta \in (0,\delta_0]$}
        \,.
    \end{gather}

    Set $\iota = \min\{ \delta_0, 2^{-10} \varepsilon /
    \Gamma_{\ref{cor:retr+}}\}$.  Choose an open convex set $V \subseteq \R^n$
    with $\cnt^\infty$ smooth boundary and such that $Q + \cball 0{\iota/4}
    \subseteq V \subseteq Q + \cball 0\iota$. Such $V$ is easily constructed,
    e.g., by~taking first the set $\tilde V = Q + \cball 0{\iota/2}$,
    representing $\Bdry \tilde V$ locally as (rotated) graph of some convex
    function, and then mollifying this function. Employ~\ref{cor:retr+} with
    $2^{-10} \varepsilon$ in place of~$\varepsilon$ to obtain the~map $l \in
    \cnt^{\infty}(\R^n,\R^n)$. Apply~\ref{cor:proj+} with $2^{-10}\iota$, $V$ in
    place of $\varepsilon$, $V$ to construct the map $q \in \cnt^{\infty}(\R^n
    \without \{0\},\R^n)$.

    Fix a symmetric non-negative mollifier $\phi \in \cnt^{\infty}(\R,\R)$ whose
    support equals~$[-1,1]$ and set $\phi_{\rho}(s) = \rho^{-1} \phi(s/\rho)$
    for $s \in \R$ and $\rho > 0$. Let $e_1,\ldots,e_n$ be the standard basis
    of~$\R^n$. For $i = 1,2,\ldots,n$ set $\rho_i = \min \bigl\{ \frac 12, 1 -
    |a \bullet e_i| \bigr\}$ and let $\bar f_{a,i} : \R \to \R$ be
    the~continuous piece-wise affine map satisfying
    \begin{gather}
        \bar f_{a,i}(t) = t  \quad \text{if $t \ge 1 - \tfrac 58 \rho_i$ or $t \le -1 + \tfrac 58 \rho_i$} \,,
        \\
        \bar f_{a,i}(t) = t - a \bullet e_i  \quad \text{if $|t - a \bullet e_i| \le \tfrac 18 \rho_i$} \,,
        \\
        \text{$\bar f_{a,i}$ is affine on $[-1+\tfrac 58 \rho_i, a \bullet e_i - \tfrac 18 \rho_i]$
          and on $[a \bullet e_i + \tfrac 18 \rho_i, 1-\tfrac 58 \rho_i]$} \,.
    \end{gather}
    Next, define $f_{a,i} = \bar f_{a,i} * \phi_{\rho_i/8} \in \cnt^{\infty}(\R,\R)$.
    We obtain
    \begin{itemize}
    \item if $a \bullet e_i = 0$, then $f_{a,i}(t) = t$ for $t \in \R$;
    \item if $|a \bullet e_i| > 0$, then
        \begin{gather*}
            f_{a,i}(a \bullet e_i) = 0 \,,
            \quad
            f_{a,i}(t) = t \quad \text{for $t \in \R$ with $|t| \ge 1 - \tfrac 12 \rho_i$} \,,
            \\
            \frac{1}{2(2-\rho_i)} < f_{a,i}'(t) < \frac{2}{\rho}
            \quad \text{and} \quad
            |f_i(t)| \ge \tfrac 12 |t - a \bullet e_i|
            \quad \text{for $t \in \R$} \,.
        \end{gather*}
    \end{itemize}
    Define the map $f_a : \R^n \to \R^n$ by $f_a(x) = \sum_{i=1}^n f_{a,i}(x
    \bullet e_i) e_i$. Then
    \begin{gather}
        f_a(a) = 0 \,,
        \quad
        |f_a(x)| \ge \tfrac 12 |x-a| \,,
        \quad
        \text{$f_a$ is a diffeomorphism of class~$\cnt^{\infty}$} \,,
        \\
        f_a(x) = x \quad \text{for $x \in \R^n$ with $2 \dist(x, \R^n \without Q) \le \dist(a,\R^n \without Q)$} \,,
        \\
        \label{eq:f-bounds}
        \text{and} \quad
        \frac{1}{2(2 - \dist(a, \R^n \without Q))} \le \|\uD f_a(x)\| \le \frac{2}{\dist(a, \R^n \without Q)} 
        \quad \text{for $x \in Q$} \,.
    \end{gather}
    Define $\varphi_{a,\varepsilon} = l \circ q \circ f_a$. Clearly
    $\varphi_{a,\varepsilon} : \R^n \without \{a\} \to \R^n$ is of
    class~$\cnt^\infty$.

    \emph{Proof of~\ref{i:scp:id-outside}}: If $x \in \R^n$ satisfies $\dist(x, Q) \ge
    \varepsilon$, then $x \in \R^n \without V$, so $q(f_a(x)) = q(x) = x$ and
    $\varphi_{a,\varepsilon}(x) = l(x) = x$.

    \emph{Proof of~\ref{i:scp:small-out}}: Let $x \in \R^n \without Q$. Then $f_a(x) =
    x$. If $x \in \R^n \without V$, then $q(x) = x$ and
    $|\varphi_{a,\varepsilon}(x) - x| = |l(x) - x| \le 2^{-10} \varepsilon$
    by~\ref{cor:retr+}\ref{i:rr:small}. Let $p$, $t$ define the central
    projection onto $\Bdry V$ as in~\ref{def:centr-proj}. If $x \in V \without
    Q$, then by~\ref{cor:proj+}\ref{i:pr:qp-comp} and~\eqref{eq:good-delta} we
    have
    \begin{gather*}
        |q(x) - x|
        \le |p(x) - x| 
        = (t(x) - 1) |x| \le 2^{-10} n^{-1/2} \varepsilon 2 \sqrt n 
        \le 2^{-9} \varepsilon \,.
    \end{gather*}
    Hence, $|\varphi_{a,\varepsilon}(x) - x| \le |l(q(x)) - q(x)| + |q(x) - x|
    \le \varepsilon$ by~\ref{cor:retr+}\ref{i:rr:small}.

    \emph{Proof of~\ref{i:scp:Q-im}}: If $x \in Q \without \{a\}$, then $f_a(x) \in Q
    \without \{0\}$ and $\dist(f_a(x),\R^n \without V) \ge \iota/4$, so $q(f_a(x))
    \in \Bdry V$ by~\ref{cor:proj+}\ref{i:pr:coincide}. Consequently
    $\dist(q(f_a(x)),Q) \le \iota \le 2^{-10}\varepsilon/ \Gamma_{\ref{cor:retr+}}$,
    which implies $\varphi_{a,\varepsilon}(x) = l(q(f_a(x))) \in \Bdry V$
    by~\ref{cor:retr+}\ref{i:rr:preserve}.
   
    \emph{Proof of~\ref{i:scp:Fk-im}}: For $\kappa \in \{-1,0,1\}^n$ let $F_{\kappa}$,
    $C_{\kappa}$, $T_{\kappa}$, and $c_{\kappa}$ be defined as
    in~\ref{mrem:cube-setup}. Fix a~$\kappa \in \{-1,0,1\}^n$ with $\kappa \ne
    (0,\ldots,0)$. Let $x \in F_{\kappa}$ and note that $f_a(x) = x$. Recall
    $\bigcup \{ C_{\lambda} : \lambda \in \{-1,0,1\}^n \without \{(0,\ldots,0)\}
    \} = \R^n \without \Int Q$, so there exists $\lambda \in \{-1,0,1\}^n$
    such that $q(x) \in C_{\lambda}$. Since $q(x) = tx$ for some $t > 1$, it
    follows from the definition of~$C_{\lambda}$ and~$C_{\kappa}$ that
    $\lambda_j = \kappa_j$ whenever $\kappa_j \ne 0$, which implies that
    $F_{\lambda} \subseteq \Clos{F_{\kappa}}$. Since $l\lIm F_{\lambda} \rIm
    \subseteq F_{\lambda}$ we get $\varphi_{a,\varepsilon}(x) = l(q(x)) \in
    \Clos(F_{\kappa})$. Similarly, if $x \in C_{\kappa}$, then $f_a(x) = x$ and
    there exists $\lambda \in \{-1,0,1\}^n$ such that $q(x) \in C_{\lambda}$ and
    $q(x) = tx$ for some $t \ge 1$; hence, $C_{\lambda} \subseteq
    \Clos{C_{\kappa}}$ and $\varphi_{a,\varepsilon}(x) = l(q(x)) \in \Clos(C_{\kappa})$
    because $l\lIm C_{\lambda} \rIm \subseteq C_{\lambda}$.

    For $y \in (c_{\kappa} + T_{\kappa}) \without V$ we have $q(f_a(y)) = q(y) =
    y$ and then $\varphi_{a,\varepsilon}(y) = l(y) \in c_{\kappa} + T_{\kappa}$
    by~\ref{cor:retr+}\ref{i:rr:preserve}. If $y \in (c_{\kappa} + T_{\kappa}) \cap
    V$, then $q(f_a(y)) = q(y) = ty$ for some $t \ge 1$,
    by~\ref{cor:proj+}\ref{i:pr:conv}, and, as before, $q(y) \in C_{\lambda}$
    for some $\lambda \in \{-1,0,1\}^n$ such that $\lambda_j = \kappa_j$
    whenever $\kappa_j \ne 0$. In~this case $\dist(y,Q) \le
    \varepsilon/\Gamma_{\ref{cor:retr+}}$ and we can
    apply~\ref{cor:retr+}\ref{i:rr:preserve} to see that $l(q(y)) \in F_{\lambda}
    \subseteq \Clos{F_{\kappa}} \subseteq c_{\kappa} + T_{\kappa}$.

    Assume now $\lambda \in \{-2,-1,1,2\}^n \without \{-2,2\}^n$ and $x \in
    R_\lambda$. If $x \notin V$, then $\varphi_{a,\varepsilon}(x) = l(x) \in
    R_\lambda$ by~\ref{cor:retr+}\ref{i:rr:preserve}. Thus, assume $x \in
    R_\lambda \cap V$. Let $\kappa \in \{-1,0,1\}^n$ be such that $\kappa_i =
    \lambda_i$ if $|\lambda_i| = 1$ and $\kappa_i = 0$ if $|\lambda_i| = 2$ for
    $i = 1,2,\ldots,n$. We already know that $l(q(x)) \in \Clos{F_\kappa}$ so it
    suffices to show that $(q(x) \bullet e_j) \lambda_j \ge 0$ whenever
    $|\lambda_j| = 2$ for some $j \in \{1,\ldots,n\}$ but this is clear since
    $q(x) = tx$ for some $t > 0$.

    If $a \in T_{\kappa}$, then $f_a\lIm T_{\kappa} \without \{a\} \rIm \subseteq
    T_{\kappa} \without \{0\}$. For $x \in T_{\kappa} \without \{0\}$ we have
    $q(x) = tx$ for some $t \in [1,\infty)$, so $q\lIm T_{\kappa} \without \{0\}
    \rIm \subseteq T_{\kappa}$. Finally, $l\lIm T_{\kappa} \without \{0\} \rIm
    \subseteq T_{\kappa}$; hence, $\varphi_{a,\varepsilon}\lIm T_{\kappa}
    \without \{0\} \rIm \subseteq T_{\kappa}$. If $x \in T_\kappa \without Q$,
    then $f_a(x) = x$ so $\varphi_{a,\varepsilon}\lIm T_{\kappa} \without Q \rIm
    \subseteq T_{\kappa}$ as before.

    \emph{Proof of~\ref{i:scp:deriv-bound}}: Set
    $\Delta = \inf \bigl\{\nu(y) \bullet \tfrac{y}{|y|} : y \in \Bdry V
    \bigr\}^{-1}$.
    Note that $\Delta$ can be bounded by a constant depending only on~$n$ and,
    in~particular, independently of~$\varepsilon$.
    Employ~\ref{cor:proj+}\ref{i:pr:deriv-bound} to get for
    $z \in V \without \{0\}$
    \begin{equation}
        \label{eq:Dqz-bound}
        \|\uD q(z)\| 
        \le  \frac{5 \Delta \sup\{ |y| : y \in \Bdry V \}}{|z|}
        \le \frac{10 \Delta}{|z|} \,.
    \end{equation}
    Hence, combining~\ref{cor:retr+}\ref{i:rr:lip-all}, \eqref{eq:f-bounds},
    and~\eqref{eq:Dqz-bound}, for $x \in Q \without \{a\}$
    \begin{multline}
        \|\uD\varphi_{a,\varepsilon}(x)\| \le \|\uD l(q \circ f_a(x))\| \cdot \|\uD q(f_a(x))\| \cdot \|\uD f_a(x)\|
        \\
        \le \frac{10 \Gamma_{\ref{cor:retr+}} \Delta}{|f_a(x)| \dist(a, \R^n \without Q)} 
        \le \frac{20 \Gamma_{\ref{cor:retr+}} \Delta}{|x-a| \dist(a, \R^n \without Q)}  \,.
    \end{multline}

    \emph{Proof of~\ref{i:scp:lip-varphi}}: If $z \in \R^n \without V$,
    then $\|\uD q(z)\| = 1$ and $|z| > 1$. If $z \in V$ and $\dist(z,\Bdry Q) \le
    \frac 14$, then $|z| \ge \frac 34$, so $\|\uD q(z)\| \le 15 \Delta$.
    Altogether, for $x \in \R^n$ satisfying
    \begin{displaymath}
        \dist(x, \Bdry Q) \le \min\bigl\{ \tfrac 12 \dist(a, \R^n \without Q), \tfrac 14 \bigr\}
    \end{displaymath}
    we have $f_a(x) = x$ and, by~\ref{cor:retr+}\ref{i:rr:lip-all},
    \begin{displaymath}
        \|\uD\varphi_{a,\varepsilon}(x)\| \le 15 \Delta \Gamma_{\ref{cor:retr+}} \,.
    \end{displaymath}

    \emph{Proof of~\ref{i:scp:small-bdry}}: Let $x \in \R^n$ satisfy
    $\dist(x, \Bdry Q) \le \min\{ \frac 12 \dist(a, \R^n \without Q), \frac 14
    \}$. Let $y \in \R^n$ be such that $\dist(y,Q) = \varepsilon$ and $|x-y|
    \le \dist(x,\Bdry Q) + \varepsilon$. Then, $\|\uD\varphi_{a,\varepsilon}(tx +
    (1-t)y)\| \le 15 \Delta \Gamma_{\ref{cor:retr+}}$ for each $t \in [0,1]$ and
    $\varphi_{a,\varepsilon}(y) = y$, so
    \begin{displaymath}
        |\varphi_{a,\varepsilon}(x) - x|
        \le |\varphi_{a,\varepsilon}(x) - \varphi_{a,\varepsilon}(y)| + |y - x|
        \le (15 \Delta \Gamma_{\ref{cor:retr+}} + 1) (\dist(x,\Bdry Q) + \varepsilon) \,.
    \end{displaymath}

    \emph{Proof of~\ref{i:scp:bdry-behav}}: Let $x \in \R^n$, $\delta \in
    \R$, $0 < \delta < \min \{ \frac 12 \dist(a, \R^n \without Q), \frac 14
    \}$, and $\dist(x, \Bdry Q) \le \delta$; then $f_a(x) = x$.

    In case $x \in \R^n \without Q$, if $\kappa \in \{-1,0,1\}^n$ is such that
    $x \in C_{\kappa}$, then $\dist(x, Q) = \dist(x,F_{\kappa})$ and
    $\varphi_{a,\varepsilon}(x) \in \Clos{C_{\kappa}}$. Since $\Clos{C_{\kappa}} \cap
    (\Clos{F_{\kappa}} + \cball 0\delta) \subseteq \Bdry Q + \cball 0\delta$ is
    convex and contains both~$x$ and~$\varphi_{a,\varepsilon}(x)$, we see that
    $\conv\{x,\varphi_{a,\varepsilon}(x)\} \subseteq \Bdry Q + \cball 0\delta$.

    Assume now $x \in Q$. Observe that there exists $\kappa \in \{-1,0,1\}^n$
    such that $\varphi_{a,\varepsilon}(x) \in \Clos{F_{\kappa}}$ and
    $\dist(x,\Bdry Q) = \dist(x, F_{\kappa})$ -- this is because $q$ acts on~$x$
    as central projection with centre at the origin. As before, we see that $x$
    and~$\varphi_{a,\varepsilon}(x)$ both lie in the convex set
    $\Clos{C_{\kappa}} \cap (\Clos{F_{\kappa}} + \cball 0\delta)$ so
    $\conv\{x,\varphi_{a,\varepsilon}(x)\} \subseteq \Bdry Q + \cball 0\delta$.

    \emph{Proof of~\ref{i:scp:a-cont}}: Set $Y = (l \circ q)^{-1} \lIm X \rIm$,
    $P = (-1/2,1/2)^n$, $R = (-3/4,3/4)^n$. Observe that $E_a =
    \varphi_{a,\varepsilon}^{-1} \lIm X \rIm = f_a^{-1} \lIm Y \rIm$ and recall
    that $f_a$ is a~diffeomorphism for each $a \in \Int Q$. It follows from the
    construction that for any compact set $K \in \Int Q$
    \begin{gather}
        \lim_{\delta \to 0} \sup \bigl\{
        | f_a^{-1}(x) - f_b^{-1}(x) | : x \in \R^n ,\, a,b \in K ,\, |a-b| < \delta 
        \bigr\} = 0 \,;
        \\
        \text{thus,}\quad
        \lim_{b \to a} \HD(f_a^{-1}\lIm Y \rIm, f_{b}^{-1}\lIm Y \rIm) = 0 \,.
    \end{gather}

    For $a \in \Int Q$ let $B_a$ be the topological boundary of $E_a \cap
    T_{\kappa}$ relative to $T_{\kappa}$. Let $a \in P \cap T_{\kappa}$. It
    follows from~\eqref{eq:scp:HD-conv} and the construction that
    \begin{equation}
        {\textstyle \bigcap_{\delta > 0} \bigcup_{b \in \cball{a}{\delta} \cap T_{\kappa}}}
        ((E_a \without E_b) \cup (E_b \cap E_a)) \cap T_{\kappa}
        \subseteq R \cap T_{\kappa} \cap B_a \,.
    \end{equation}
    Without loss of generality we may assume $T_{\kappa} = \lin \{e_1, \ldots,
    e_k\}$. Recall the construction of the maps~$l$ and~$q$ to see that~$Y$ is
    a~convex conical cap over $l^{-1}\lIm X \rIm$ with vertex at the origin.
    Let~$B$ be the topological boundary of $Y \cap T_{\kappa}$ relative
    to~$T_{\kappa}$. Define affine lines $L_{a,i} = \{ a + t e_i : t \in \R \}$
    for $a \in \R^n$ and $i \in \{1,2,\ldots,n\}$. Since $B$ is the boundary of
    a~convex set and has empty interior in~$T_{\kappa}$ we see that there exists
    $i \in \{ 1,2,\ldots,k \}$ such that $B \cap L_{a,i}$ contains at most two
    points for each $a \in T_{\kappa}$ and we may decompose $B$ into two
    disjoint sets $B = B_1 \cup B_2$ so that for each $a \in T_{\kappa}$ if $a +
    t_1 e_i \in B_1 \cap L_{i,a}$ and $a + t_2 e_i \in B_2 \cap L_{i,a}$, then
    $t_1 < t_2$. Define $B_{1,a} = f_{a}^{-1} \lIm B_1 \rIm$ and $B_{2,a} =
    f_{a}^{-1} \lIm B_2 \rIm$ for $a \in T_{\kappa}$. Then $B_a$ equals the
    disjoint sum of~$B_{1,a}$ and~$B_{2,a}$ for each $a \in \Int Q \cap
    T_{\kappa}$ because $f_a$ is a~diffeomorphism. Clearly, it suffices to show
    that if $j \in \{1,2\}$, then for $\HM^k$~almost all $a \in P \cap
    T_{\kappa}$ we have
    \begin{equation}
        \mu\bigl( R \cap T_{\kappa} \cap B_{j,a} \bigr) = 0 \,.
    \end{equation}
    Fix $j = \{1,2\}$. Observe that if $a \in P \cap T_{\kappa}$ and $t \in
    (-3/4,3/4)$, then the map $g_{a,t} : L_{a,i} \cap P \to \R$ given by
    $g_{a,t}(b) = f_{b,i}^{-1}(t)$ is strictly increasing.  In~consequence, if
    $b,c \in L_{a,i} \cap P$ and $b \ne c$, then $R \cap T_{\kappa} \cap B_{j,b}
    \cap B_{j,c} = \varnothing$. Since $\mu$ is Radon, there exists at most
    countably many $b \in L_{a,i} \cap P$ for which $\mu\bigl( R \cap T_{\kappa}
    \cap B_{j,b}\bigr) > 0$. In~particular, we obtain for each $a \in P \cap
    T_{\kappa}$
    \begin{equation}
        \HM^1\bigl( \bigl\{ b \in L_{a,i} \cap P :
        \mu\bigl( R \cap T_{\kappa} \cap B_{j,b} \bigr) > 0
        \bigr\} \bigr) = 0 \,.
    \end{equation}
    Since $\HM^k \restrict T_{\kappa}$ coincides with the Lebesgue
    measure~$\LM^k$ on~$T_{\kappa}$ and $\LM^k$ is the product of~$k$ copies of
    the one dimensional Lebesgue measure (cf.~\cite[2.6.5]{Fed69}), we may use
    the Fubini Theorem~\cite[2.6.2(3)]{Fed69} to conclude the proof.
\end{proof}

The next lemma is a counterpart of~\cite[4.2.7]{Fed69}. Given arbitrary Radon
measures $\mu_1, \ldots, \mu_l$, and numbers $m_1$, \ldots, $m_l$, and
a~$k$-plane $T_{\kappa}$ with $\max\{m_1,\ldots,m_l\} < k \le n$ we prove that
there are enough good points~$a \in [-1/2,1/2]^n \cap T_{\kappa}$ for which the
integral $\int_Q \| \uD \varphi_{a,\varepsilon}\|^{m_i} \ud \mu_i$ is controlled
by $\mu_i(Q)$. Later we shall apply this lemma to measures $\mu$ defined as
the~restriction of~$\HM^{m}$ to some $m$~dimensional set $\Sigma \subseteq \R^n$
with density.

\begin{lemma}
    \label{lem:int-Dpia}
    Suppose
    \begin{gather*}
        k,N \in \nat \,, \quad k \le n \,,
        \quad
        \kappa \in \{-1,0,1\}^n \text{ is such that } \HM^0(\{ j : \kappa_j = 0 \}) = k \,,
        \\
        Q \text{ and } T_{\kappa} \text{ are as in~\ref{mrem:cube-setup}}\,,
        \quad
        \varepsilon \in \bigl(0,\tfrac 14\bigr) \,,
        \quad
        A = T_{\kappa} \cap \bigl[ -\tfrac 12, \tfrac 12 \bigr]^n \,,
        \\
        m_1,\ldots,m_l \in (0,k) \,,
        \quad
        \mu_1,\ldots,\mu_l \text{ are Radon measures over~$\R^n$}
        \,.
    \end{gather*}
    For $a \in \Int Q$ let $\varphi_{a,\varepsilon} : \R^n \without \{a\} \to
    \R^n$ be the map constructed in~\ref{lem:smooth-centr-proj} and set
    \begin{gather*}
        \Gamma(k,m) = \Gamma_{\ref{lem:smooth-centr-proj}} \frac{k \unitmeasure{k}}{k-m} k^{(k-m)/2} 
        \quad \text{for $m \in (0,k)$} \,,
        \\
        E = \left\{
            a \in A : 
            \int_{Q} \|\uD \varphi_{a,\varepsilon}\|^{m_i} \ud \mu_i
            \le l \Gamma(k,m_i) \mu_i(Q)
            \text{ for } i \in \{ 1,2,\ldots,l \}
        \right\} \,.
    \end{gather*}

    Then $\LM^k(E) > 0$.
\end{lemma}

\begin{proof}
    Employing~\ref{lem:smooth-centr-proj}\ref{i:scp:deriv-bound} we have for
    $\varepsilon \in (0,1/4)$, $m \in (0,k)$, $x \in Q$, and $y \in A$
    satisfying $|x-y| = \dist(x,A)$
    \begin{multline}
        \int_{A} \| \uD \varphi_{a,\varepsilon}(x) \|^m \ud \LM^k(a)
        \le \Gamma_{\ref{lem:smooth-centr-proj}} \int_{A} |x-a|^{-m} \ud \LM^k(a)
        \le \Gamma_{\ref{lem:smooth-centr-proj}} \int_{A} |y-a|^{-m} \ud \LM^k(a)
        \\
        \le \Gamma_{\ref{lem:smooth-centr-proj}} \int_{T_{\kappa} \cap \cball 0{\sqrt k}} |y|^{-m} \ud \LM^k(y) 
        = \Gamma_{\ref{lem:smooth-centr-proj}} \frac{k \unitmeasure{k}}{k-m} k^{(k-m)/2} = \Gamma(k,m) \,.
    \end{multline}
    Thus, for $i \in \{ 1,2,\ldots,l \}$, using the Fubini
    Theorem~\cite[2.6.2]{Fed69}, we obtain
    \begin{equation}
        \label{eq:dpia:int-a}
        \int_A \int_{Q} \| \uD \varphi_{a,\varepsilon}(x) \|^{m_i} \ud \mu_i^p(x) \ud \LM^k(a)
        \le \Gamma(k,m_i) \mu_i(Q) \,.
    \end{equation}
    Now, we argue by contradiction. If $\LM^k(E)$ was zero, then we would have
    \begin{displaymath}
        1 = \LM^k(A) 
        < \sum_{i=1}^l \frac{1}{l \Gamma(k,m_i) \mu_i(Q)} \int_A \int_Q
        \| \uD \varphi_{a,\varepsilon}(x) \|^{m_i} \ud \mu_i(x) \ud \LM^k(a)
        \le \sum_{i=1}^l \frac 1l = 1 \,.
        \qedhere
    \end{displaymath}
\end{proof}

Now, given a cube~$K \in \cubes_*$ (of arbitrary dimension and size) and sets
$\Sigma_1$, \ldots, $\Sigma_l$ we combine~\ref{lem:int-Dpia}
and~\ref{lem:smooth-centr-proj} to construct a~deformation of~$\R^n$ which
maps $\Sigma_i \cap K$ into $\cBdry K$ for each $i = 1,2,\ldots,l$ and preserves
all the super-cubes of~$K$ (i.e. those which contain $K$) as well as all the
cubes from~$\cubes_*$ which do not touch $\cInt K$ and have side length
at~least~$\frac 12 \side K$. Of~course we also control the derivative.

\begin{lemma}
    \label{lem:one-cube-deform}
    Suppose
    \begin{gather*}
        l \in \nat \,, \quad
        K \in \cubes_* \,,
        \quad
        k = \dim(K) \,,
        \quad
        m_1, \ldots, m_l \in \bigl\{ 1,2, \ldots, k \bigr\} \,,
        \\
        \text{$\nu_1, \ldots, \nu_l$ are Radon measures over $\R^n$} \,,
        \quad
        \Sigma = \tbcup\{ \spt \nu_i : i = 1,2,\ldots,l \} \,,
        \\
        \text{either $\max\{m_1, \ldots, m_l\} \le k - 1$
          and $\HM^{k}(\Sigma \cap K) = 0$}
        \\
        \text{or $m_1 = \cdots = m_l = k$ and $\Sigma \cap K \ne K$} \,.
    \end{gather*}

    Then for each $\varepsilon_0 \in \bigl(0,\frac 14 \side K\bigr)$ there exist
    $\varepsilon \in (0,\varepsilon_0]$, a~neighbourhood~$U$ of~$\Sigma$
    in~$\R^n$, and a~map~$\varphi \in \cnt^{\infty}(\R^n, \R^n)$ such that
    \begin{enumerate}
    \item 
        \label{i:one:admissible}
        $\varphi \in \adm{K + \oball 0{\varepsilon}}$,
    \item
        \label{i:one:id}
        $\varphi(x) = x$ for $x \in \R^n$ satisfying $\dist(x,K) \ge \varepsilon$,
    \item
        \label{i:one:proj}
        $\varphi \lIm U \rIm  \cap K
        = \varphi \lIm U \cap K \rIm
        \subseteq \cBdry K$,
    \item
        \label{i:one:neighb}
        $\varphi \lIm \cBdry K + \cball 0{\varepsilon} \rIm \subseteq \cBdry K + \cball 0{\varepsilon}$,
    \item
        \label{i:one:small}
        $|\varphi(x) - x| \le \varepsilon$ for $x \in \R^n$ satisfying
        $\project T (x - \centre K) \notin \project T \lIm K \rIm$, where $T =
        \Tan(K,\centre K)$,
    \item
        \label{i:one:faces}
        if $L \in \cubes_*$ satisfies either $\side L \ge \side K$ or
        $\side L \ge \frac 12 \side K$ and $L \cap \cInt K = \varnothing$, then
        $\varphi \lIm L \rIm \subseteq L$.
    \item
        \label{i:one:Dout}
        $\|\uD\varphi(x)\| \le \Gamma$ for $x \in \R^n$ with $\dist(x,\cBdry
        K) \le \varepsilon$,
    \item
        \label{i:one:int}
        if $\max\{m_1, \ldots, m_l\} \le k - 1$, then there exists $\Gamma =
        \Gamma(k,l) \in (1,\infty)$ such that
        \begin{displaymath}
            \int_{K} \|\uD \varphi\|^{m_i} \ud \nu_i
            \le \Gamma \nu_i(K)
            \quad  \text{for $i = 1,2,\ldots,l$} \,.
        \end{displaymath}
    \end{enumerate}
\end{lemma}

\begin{proof}
    Let $\varepsilon_0 \in (0, \side{K}/4)$. If $\max\{m_1, \ldots, m_l\} \le k
    - 1$, then set $\varepsilon = \varepsilon_0$. If $m_1 = \cdots = m_l = k$,
    then choose arbitrary $a_0 \in \cInt K \without \Sigma$ and set
    \begin{displaymath}
        \varepsilon = \min \bigl\{ \varepsilon_0, 2^{-8}\dist(a_0,\cBdry K) \bigr\} \,.
    \end{displaymath}
    Translating $\Sigma$ and $K$ by $-\centre K$ we can assume $\centre K =
    0$. Set
    \begin{gather}
        T = \Tan(K,\centre K) \,,
        \quad 
        \iota = \varepsilon / \sqrt 2 \,,
        \quad
        r = \scale{2/\side K} \circ \trans{-\centre K} \,,
        \\
        \mu_i = \bigl( r_{\#}\nu_i \bigr) \restrict T
        \quad \text{ for $i = 1,\ldots,l$} \,.
    \end{gather}
    Note that $r\lIm K \rIm = [-1,1]^n \cap T$. For $a \in K$ let
    $\varphi_{r(a), 2\iota /\side K}$ be the map defined by
    employing~\ref{lem:smooth-centr-proj} with $r(a)$, $2\iota /\side K$ in
    place of $a$, $\varepsilon$ and set
    \begin{displaymath}
        \psi_a = r^{-1} \circ \varphi_{r(a), 2\iota /\side K} \circ r \,.
    \end{displaymath}
    To choose an appropriate $a \in K$, we consider two cases.
    \begin{itemize}
    \item If $\max\{m_1, \ldots, m_l\} \le k - 1$, then we proceed as
        follows. Define $E \subseteq K$ to be the set of all those $a \in K$ for
        which $r(a) \in [-1/2,1/2]^n$ and
        \begin{equation}
            \label{eq:int-dphi}
            \int_{[-1,1]^n}
            \|\uD \varphi_{r(a),2\iota /\side K}\|^{m_i} \ud \mu_i 
            \le l \Gamma_{\ref{lem:int-Dpia}}(k,m_i) \mu_i([-1,1]^n)
            \quad \text{for $i = 1,\ldots,l$}\,.
        \end{equation}
        Apply~\ref{lem:int-Dpia} with $2\iota /\side K$ in place of
        $\varepsilon$ to conclude that $\LM^k(E) > 0$. Since we have
        $\HM^{k}(\Sigma \cap K) = 0$, we may choose $a \in E \without \Sigma$.

    \item If $m_1 = \cdots = m_l = k$, then we set $a = a_0$.
    \end{itemize}
    Since $\Sigma \cap K$ is compact we have
    \begin{displaymath}
        d = \tfrac 12 \min \bigl\{ \iota , \dist(a, \Sigma) \bigr\} > 0 \,.
    \end{displaymath}
    Let $\alpha \in \cnt^{\infty}(\R, \R)$ be such that $\alpha(t) = 1$ for
    $t \ge 7/8$, $\alpha(t) = 0$ for $t \le 1/4$, and $0 < \alpha'(t) < 2$ for
    $t \in (0,1)$. Recall that we assumed $\centre K = 0$; in~particular
    $K \subseteq T$. Set
    \begin{gather*}
        \psi(x) = \left\{
            \begin{array}{ll}
                x & \text{for $x \in \cball a{d/4}$}
                \\
                \alpha(|x-a|/d) \psi_a(x) + \bigl(1 - \alpha(|x-a|/d)\bigr) x
                & \text{for $x \in \R^n \without \cball a{d/4}$}
            \end{array}
        \right.
        \\
        \varphi(x) = \perpproject{T}x + \psi(\project T x) 
        + \bigl(\project T x - \psi(\project T x)\bigr) \alpha(|\perpproject{T} x|/\iota)
        \quad \text{for $x \in \R^n$} \,.
    \end{gather*}
    Clearly $\varphi$ is of class~$\cnt^\infty$. Since $K + \oball
    0{\varepsilon}$ is convex, we see that~\ref{i:one:admissible} is
    satisfied. Set
    \begin{displaymath}
        Q = [-1,1]^n
        \quad \text{and} \quad
        R = r^{-1} \lIm Q \rIm
        \quad \text{and} \quad
        U = \Sigma + \oball{0}{2^{-3}d} \,.
    \end{displaymath}

    \emph{Proof of~\ref{i:one:id}}: If $x \in \R^n$ satisfies $\dist(x,K) \ge
    \varepsilon$, then either $\dist(x, T) \ge \varepsilon/\sqrt 2 \ge \iota$
    and then $\varphi(x) = x$, or $\dist(\project T x, K) \ge \varepsilon/ \sqrt
    2 \ge \iota$ and then $\psi_a(x) = x$,
    by~\ref{lem:smooth-centr-proj}\ref{i:scp:id-outside}, and $\varphi(x) = x$.

    \emph{Proof of~\ref{i:one:proj}}: If $x \in U \cap K$, then
    $\varphi(x) = \psi_a(x)$. Observe that
    $\psi_a(T \without \{a\}) \subseteq T$
    by~\ref{lem:smooth-centr-proj}\ref{i:scp:Fk-im} because $a \in T$. Combining
    this with $\psi_a\lIm R \without \{a\} \rIm \subseteq \Bdry R$, which holds
    due to~\ref{lem:smooth-centr-proj}\ref{i:scp:Q-im}, we see that
    $\varphi \lIm U \cap K \rIm \subseteq \cBdry K$.
    Moreover, by~\ref{lem:smooth-centr-proj}\ref{i:scp:Fk-im} and the definition
    of~$\varphi$ we have
    $\varphi \lIm \R^n \without K \rIm \subseteq \R^n \without K$, so
    $\varphi \lIm U \rIm \cap K = \varphi \lIm U \cap K \rIm$.

    \emph{Proof of~\ref{i:one:neighb}}: Let $x \in \cBdry K + \cball 0{\varepsilon}$.
    Observe that $\perpproject{T}x = \perpproject{T}\varphi(x)$ because
    $\psi_a(\project T x) \in T$ due
    to~\ref{lem:smooth-centr-proj}\ref{i:scp:Fk-im}. Moreover, $\Bdry R \cap (T
    + \cball 0{\varepsilon}) \subseteq \cBdry K + T^{\perp}$, so for any $z \in
    \cBdry K + \cball 0{\varepsilon}$ we have $\dist(z, \cBdry K)^2 =
    \dist(\project T z, \Bdry R)^2 + \dist(z,T)^2$. Noting $|\perpproject{T}x|
    \le \varepsilon$ and
    using~\ref{lem:smooth-centr-proj}\ref{i:scp:Fk-im}\ref{i:scp:bdry-behav} we
    can write
    \begin{multline*}
        \dist(\varphi(x), \cBdry K)^2
        = \dist(\project T \varphi(x), \Bdry R)^2 + \dist(\varphi(x),T)^2
        \\
        = \dist\bigl( (1-t) \psi(\project T x) + t \project T x, \cBdry K\bigr)^2 + |\perpproject{T} x|^2 
        \\
        \le \dist(\project T x, \cBdry K)^2 + |\perpproject{T} x|^2 
        = \dist(x, \cBdry K)^2 \,,
    \end{multline*}
    where $t = \alpha(|\perpproject{T}x|/\iota)$. Thus, $\varphi(x) \in
    \cBdry K + \cball 0{\varepsilon}$.

    \emph{Proof of~\ref{i:one:small}}: Let $x \in \R^n$ be such that $\project T x
    \notin K$. If $\dist(x,K) \ge \varepsilon$, then $\varphi(x) = x$ and there
    is nothing to prove. Assume $\dist(x,K) \le
    \varepsilon$. By~\ref{lem:smooth-centr-proj}\ref{i:scp:small-out} we know
    $|\psi(\project T x) - \project T x| \le \iota$ so for any $t \in [0,1]$ we
    have $|(t \project T x + (1-t) \psi(\project T x)) - \project T x| \le \iota$.
    Setting $t = \alpha(|\perpproject{T}x|/\iota)$ we obtain
    \begin{displaymath}
        |\varphi(x) - x|
        = |\project T \varphi(x) - \project T x|
        \le |(t \project T x + (1-t) \psi(\project T x)) - \project T x| 
        \le \iota \le \varepsilon \,.
    \end{displaymath}

    \emph{Proof of~\ref{i:one:faces}}: Let $L \in \cubes_*$ be such that $\side L \ge
    \frac 12 \side K$ and $L \cap \cInt K = \varnothing$. Observe that if $\side
    L > \frac 12 \side K$, then $L$ is a sum of some cubes from $\cubes_*$ which
    do not intersect $\cInt{K}$ and have side length $\frac 12 \side K$; thus,
    it is enough to prove the claim in case $\side L = \frac 12 \side K$ -- we
    shall assume this holds. Since $\varphi(x) = x$ for $x \in \R^n$ with
    $\dist(x,K) \ge \varepsilon$ and $\varepsilon \le \frac 14 \side K$ we will
    also assume that $L \cap K \ne \varnothing$. For $\kappa \in \{-1,0,1\}^n$
    and $\lambda \in \{-2,-1,1,2\}^n$ let $c_\kappa$, $T_{\kappa}$,
    $R_{\lambda}$, be as in~\ref{mrem:cube-setup}. If $\dim(\project T \lIm L
    \rIm) = \dim(K)$, then $\project T \lIm L \rIm = T \cap r \lIm R_\lambda
    \rIm$ for some $\lambda \in \{-2,-1,1,2\}^n$. If $\dim(\project T \lIm L
    \rIm) < \dim(K)$, then $\project T \lIm L \rIm \subseteq \cBdry K$ and $L$
    is contained in some face of~$R$. In this case let $\kappa \in \{-1,0,1\}^n$
    be such that $L \subseteq r \lIm F_{\kappa} \rIm$ and $F_{\kappa} \subseteq
    F_{\sigma}$ whenever $L \subseteq r \lIm F_{\sigma} \rIm$ for some $\sigma
    \in \{-1,0,1\}^n$. If it happens that $\dim(F_\kappa) > \dim(L)$, then $L$
    must lie inside $T_\sigma$ for some $\sigma \in \{-1,0,1\}^n$. Altogether,
    there exist $\lambda \in \{-2,-1,1,2\}^n \without \{-2,2\}^n$ and $\kappa,
    \sigma \in \{-1,0,1\}^n$ such that
    \begin{displaymath}
        \project T \lIm L \rIm
        = r^{-1} \lIm R_{\lambda} \cap T_\sigma \cap (c_\kappa + T_\kappa) \cap T\rIm  \,.
    \end{displaymath}
    Hence, $\psi\lIm \project T \lIm L \rIm \rIm \subseteq \project T \lIm L
    \rIm$ by~\ref{lem:smooth-centr-proj}\ref{i:scp:Fk-im}. Since $L$ is convex,
    we obtain $\varphi \lIm \project T \lIm L \rIm \rIm \subseteq \project T
    \lIm L \rIm$. Finally, note that $L = \project T \lIm L \rIm +
    \perpproject{T} \lIm L \rIm$ and $\varphi \lIm L \rIm = \varphi \lIm
    \project T \lIm L \rIm \rIm + \perpproject{T} \lIm L \rIm$, which proves
    the claim in case $L \cap \cInt{K} = \varnothing$.

    If $\cInt{K} \cap L \ne \varnothing$ but $\side L \ge \side K$, then $K
    \subseteq \project T \lIm L \rIm$. Clearly $\varphi \lIm K \rIm \subseteq K$
    and $\project T \lIm L \rIm \without K$ is contained in a sum of cubes with
    side length at least $\frac 12 \side K$ which do not intersect~$\cInt
    K$. Hence, $\varphi\lIm \project T \lIm L \rIm \rIm \subseteq \project T
    \lIm L \rIm$ and the claim follows as before.

    \emph{Proof of~\ref{i:one:Dout}}: Assume $x \in \R^n$ satisfies
    $\dist(x, \cBdry K) \le \varepsilon$, then $\dist(\project T x, \cBdry K)
    \le \varepsilon$. Since $d \le \iota/2 \le 2^{-7} \dist(a, \cBdry K)$ we~see
    that $|\project T x - a| \ge (1-2^{-8}) \dist(a, \cBdry K) \ge d$ so
    $\psi(\project T x) = \psi_a(\project T x)$. Recalling $\alpha'(t) \le 2$,
    $\alpha(t) \le 1$ for $t \in \R$, $\iota = \varepsilon/\sqrt 2$,
    and~\ref{lem:smooth-centr-proj}\ref{i:scp:lip-varphi}\ref{i:scp:small-bdry}
    we~get
    \begin{multline*}
        \|\uD\varphi(x)\|
        \le \|\perpproject{T}\| + \|\uD\psi_a(\project T x) \circ T\|
        + \| \project T - \uD\psi_a(\project T x) \circ T \|
        + 2/\iota |\project T x - \psi_a(\project T x)|
        \\
        \le 2 + 10\Gamma_{\ref{lem:smooth-centr-proj}} \,.
    \end{multline*}

    \emph{Proof of~\ref{i:one:int}}: Let us assume $\max\{m_1,\ldots,m_l\}
    \le k-1$; hence, $r(a) \in \bigl[-\frac 12, \frac 12 \bigr]^n \cap T$. Note
    that for $i \in \{1,\ldots,l\}$ and $x \in T \without \cball ad$
    \begin{gather}
        \label{eq:phi-comp}
        \|\uD \varphi(x)\| 
        = \| \perpproject{T} + \uD\psi_a(x) \circ \project T \|
        \le 1 + \|\uD \varphi_{r(a),2\iota /\side K}(r(x))\|
    \end{gather}
    Using~\eqref{eq:phi-comp} and the definition of~$\Sigma$, $U$, and~$\mu_i$
    we get
    \begin{gather}
        \label{eq:phiphi}
        \int_{K} \|\uD \varphi \|^{m_i} \ud \nu_i
        \le 2^{m_i-1} \int_{[-1,1]^n} \|\uD \varphi_{r(a),2\iota /\side K}\|^{m_i} \ud \mu_i + 2^{m_i-1} \nu_i(K)
        \\
        \label{eq:phiphi2}
        \text{and} \quad
        \mu_i([-1,1]^n) = \nu_i(K)
        \quad \text{for $i = 1,\ldots,l$} \,.
    \end{gather}
    Combining~\eqref{eq:int-dphi} with~\eqref{eq:phiphi} and~\eqref{eq:phiphi2}
    yields~\ref{i:one:int}.
\end{proof}

Next, we shall prove our main deformation theorem~\ref{thm:deformation}. Given
a~finite subset $\mathcal A$ of an admissible family $\mathcal F$ of
top~dimensional cubes from~$\cubes$ and some sets $\Sigma_1$, \ldots,
$\Sigma_l$, we deform all these sets onto the $m$~dimensional skeleton
of~$\mathcal A$ using a~smooth deformation of~$\R^n$. Furthermore, we provide
estimates on the measure of the deformed sets (i.e. the images of $\Sigma_i$ for
$i =1,2,\ldots,l$) and, in case $\Sigma_i$ are rectifiable, also on the measure
of the whole deformation (i.e. the images of $[0,1] \times \Sigma_i$ for $i
=1,2,\ldots,l$). The basic idea of the proof is simple: we order all the cubes
of the cubical complex $\CX(\mathcal F)$ which touch the interior of some cube
from~$\mathcal A$ lexicographically with respect to side length and dimension
and then apply~\ref{lem:one-cube-deform} iteratively to each cube. If the
dimensions of $\Sigma_1$, \ldots, $\Sigma_l$ all equal~$m$, then we additionally
ensure that all the $m$-dimensional faces of~$\mathcal A$ are either fully
covered or not covered at all. During this last step we cannot control the
derivative so we actually provide two deformations: one with good estimates
(called ``$g$'') and one without estimates (called ``$f$'') but performing the
last step of cleaning the cubes which are not fully covered.

To be able to estimate the measure of the image of $\Sigma_i$ even if $\Sigma_i$
is not rectifiable, we~need the following simple lemma.

\begin{lemma}
    \label{lem:image-measure}
    Let $S \subseteq \R^n$ be such that $\HM^m(S \cap K) < \infty$ for every
    compact set $K \subseteq \R^n$, $g \in \cnt^1(\R^n,\R^N)$ for some $N \in
    \nat$, and $f \in L^1(\HM^m \restrict g\lIm S \rIm,\R)$ be
    non-negative. Then
    \begin{displaymath}
        \int f \ud \HM^m \restrict g \lIm S \rIm
        \le \int f \circ g \|\uD g\|^m \ud \HM^m \restrict S \,.
    \end{displaymath}
\end{lemma}

\begin{proof}
    Since $S$ can be decomposed into a countable sum of compact sets we can
    assume $S$ is compact. Furthermore, using standard methods of Lebesgue
    integration~\cite[2.3.3, 2.4.8]{Fed69} we can assume $f = \CF_{A}$ for some
    $\HM^m \restrict g \lIm S \rIm$~measurable set $A \subseteq \R^N$. Let
    $\varepsilon > 0$. For each $x \in S$ choose $r_x > 0$ so that
    \begin{displaymath}
        \|\uD g(x)\|^m - \varepsilon \le \Lip(g|_{\cball{x}{r_x}})^m \le \|\uD g(x)\|^m + \varepsilon \,.
    \end{displaymath}
    This is possible since $\uD g$ is continuous. From the family $\{
    \oball{x}{r_x} : x \in S \}$ choose a finite covering~$\mathcal B = \{ B_1,
    \ldots, B_K \}$ of~$S$. For $j = 1,2,\ldots,K$ define $S_j \subseteq S$, and
    $x_j \in \R^n$, and $r_j \in \R$ so that
    \begin{displaymath}
        B_j = \oball{x_j}{r_j}
        \quad \text{and} \quad
        S_j = S \cap B_j \without \tbcup \{ B_i : i = 1,2,\ldots,j-1 \} \,.
    \end{displaymath}
    We obtain
    \begin{multline}
        \label{eq:im:est}
        \int_{g\lIm S \rIm} f \ud \HM^m
        = \HM^m(A \cap g \lIm S \rIm)
        = \HM^m\bigl( \tbcup \{ g \lIm S_j \cap g^{-1} \lIm A \rIm \rIm : j = 1,2,\ldots,K \} \bigr)
        \\
        \le \sum_{j=1}^K \bigl( \|\uD g(x_j)\|^m + \varepsilon \bigr) \HM^m(S_j \cap g^{-1} \lIm A \rIm)
        = \int_S v_{\varepsilon} \ud \HM^m \,,
    \end{multline}
    where 
    \begin{displaymath}
        v_{\varepsilon} = \sum_{j=1}^K \bigl( \|\uD g(x_j)\|^m + \varepsilon \bigr) \CF_{S_j \cap g^{-1} \lIm A \rIm}
        = f \circ g \sum_{j=1}^K \bigl( \|\uD g(x_j)\|^m + \varepsilon \bigr) \CF_{S_j} \,.
    \end{displaymath}
    We obtain the claim by letting $\varepsilon \to 0$ and using the dominated
    convergence theorem; see~\cite[2.4.9]{Fed69}.
\end{proof}

\begin{theorem}
    \label{thm:deformation}
    Suppose
    \begin{gather*}
        \text{$\mathcal F \subseteq \cubes$ is admissible} \,,
        \quad
        \text{$\mathcal A \subseteq \mathcal F$ is finite} \,,
        \quad
        \Sigma_1, \ldots, \Sigma_l \subseteq \R^n \,,
        \quad
        m, m_1, \ldots, m_l \in \nat,
        \\
        \varepsilon_0 = 2^{-4} \min\{ \side R : R \in \mathcal A \}
        \,,
        \quad
        \delta = \max\{ \side R : R \in \mathcal A \} \,,
        \quad
        G_{\varepsilon_0} = \tbcup \mathcal A + \oball{0}{\varepsilon_0} \,,
        \\
        n-1 \ge m = m_1 \ge \ldots \ge m_l \ge 1 \,,
        \quad
        \Sigma = \tbcup \{ \Sigma_i : i = 1,\ldots,l \} \,,
        \quad
        \HM^{m+1}(\Clos{\Sigma} \cap G_{\varepsilon_0}) = 0 \,,
        \\ 
        \text{$\Sigma_i$ is $\HM^{m_i}$~measurable
          and $\HM^{m_i}(\Sigma_i \cap G_{\varepsilon_0}) < \infty$
          for $i = 1,\ldots,l$}
        \,.
    \end{gather*}
    Then for each $\varepsilon \in (0,\varepsilon_0)$, setting $G_{\varepsilon}
    = \tbcup \mathcal A + \oball{0}{\varepsilon}$ and $G_0 = \Int \tbcup
    \mathcal A$, there exist deformations $f,g \in \cnt^{\infty}(\R \times \R^n,
    \R^n$) and a~neighbourhood~$U$ of~$\Sigma \cap G_{\varepsilon}$ in~$\R^n$
    satisfying:
    \begin{enumerate}
    \item
        \label{i:dt:identity}
        $f(t,x) = x$ if either $t = 0$ and $x \in \R^n$ or $t \in \R$ and $x
        \in \R^n \without G_{\varepsilon}$.
    \item
        \label{i:dt:decomp}
        There exist $N,N_0 \in \nat$, $N_0 \le N$, $\varphi_1, \ldots, \varphi_N
        \in \adm{G_{\varepsilon}}$, and $K_1,\ldots,K_N \in \CX(\mathcal F)$
        such that for each $j = 1,\ldots,N$ setting $\psi_0 = \id{\R^n}$ and
        $\psi_j = \varphi_j \circ \psi_{j-1}$ we have
        \begin{gather*}
            \bigl\{ x \in \R^n : \varphi_j(x) \ne x \bigr\} \subseteq K_j + \oball 0{\varepsilon}
            \quad \text{and} \quad
            \varphi_j \lIm \psi_{j-1} \lIm U \rIm \cap K_j \rIm \subseteq \cBdry{K_j} \,.
        \end{gather*}
        Moreover, there exists $s \in \cnt^{\infty}(\R, \R)$ such that $s(0)
        = 0$, and $s(1) = 1$, and $0 \le s'(t) \le 2$ for $t \in \R$, and
        $\uD^{k}s(0) = 0 = \uD^{k}s(1)$ for $k \in \nat$, and
        \begin{displaymath}
            f(t,\cdot) = s(tN-j) \psi_{j+1} + (1 - s(tN-j)) \psi_{j} = g(tN/N_0, x)
        \end{displaymath}
        whenever $j \in \{0,\ldots,N-1\}$ and $t \in \R$ satisfy $j \le tN \le j+1$.
    \item 
        \label{i:dt:admissible}
        $f(t,\cdot) \in \adm{G_{\varepsilon}}$ for each $t \in I$.
    \item
        \label{i:dt:cube-pres}
        If $K \in \CX(\mathcal F)$, then $f(t,\cdot)\lIm K \rIm \subseteq K$ for $t \in I$.
        In particular
        \begin{displaymath}
            f(t,\cdot)\lIm \tbcup \mathcal A \rIm \subseteq \tbcup \mathcal A
            \quad \text{and} \quad
            f(t,\cdot)\lIm \R^n \without \tbcup \mathcal A \rIm \subseteq \R^n \without G_0
            \quad \text{for $t \in I$} \,.
        \end{displaymath}
    \item
        \label{i:dt:m-dim-im}
        $g(1,\cdot)\lIm U \rIm \cap G_0 \subseteq
        \tbcup \bigl\{ K \in \CX(\mathcal F) : K \cap G_0 \ne \varnothing \,,\, \dim(K) = m \bigr\}$.
    \item 
        \label{i:dt:full-faces}
        If $m_1 = \cdots = m_l$, then for each $K \in \CX(\mathcal F) \cap
        \cubes_{m}$ satisfying $K \cap G_0 \ne \varnothing$ there holds
        \begin{displaymath}
            \text{either} \quad
            \cInt{K} \cap f(1,\cdot)[\Sigma \cap \tbcup \mathcal A] = \varnothing
            \quad \text{or} \quad
            K \cap f(1,\cdot)[\Sigma \cap \tbcup \mathcal A] = K \,.
        \end{displaymath}
    \item
        \label{i:dt:estimates}
        There exists $\Gamma = \Gamma(n,m) \in (0,\infty)$ such that for each $t
        \in I$ and $i \in \{1,\ldots,l\}$ and $Q \in \mathcal F$, setting
        $\widetilde Q = \tbcup \{ R \in \mathcal F : R \cap Q \ne \varnothing
        \}$,
        \begin{gather}
            \label{eq:dt:deriv-bound}
            \| \uD f(t,\cdot)(x) \| < \Gamma \quad \text{for $x \in G_{\varepsilon} \without G_0$} \,,
            \\
            \label{eq:dt:im-cube}
            \int_{\Sigma_i \cap Q} \|\uD g(t,\cdot)\|^{m_i} \ud \HM^{m_i}
            < \Gamma \HM^{m_i}\bigl( \Sigma_i \cap (\widetilde Q + \oball{0}{\varepsilon}) \bigr) \,,
            \\
            \label{eq:dt:im-gt}
            \HM^{m_i}(g(t,\cdot) \lIm \Sigma_i \cap G_{\varepsilon} \rIm) 
            \operatorname*{\le}^{\ref{lem:image-measure}}
            \int_{\Sigma_i \cap G_{\varepsilon}} \|\uD g(t,\cdot)\|^{m_i} \ud \HM^{m_i}
            < \Gamma \HM^{m_i}(\Sigma_i \cap G_{\varepsilon}) \,,
            \\
            \label{eq:dt:im-cube-f}
            \HM^{m_i}(f(1,\cdot) \lIm \Sigma_i \rIm \cap Q)
            < \Gamma \HM^{m_i}\bigl( \Sigma_i \cap (\widetilde Q + \oball{0}{\varepsilon}) \bigr) \,,
            \\
            \label{eq:dt:im-f1}
            \HM^{m_i}(f(1,\cdot) \lIm \Sigma_i \cap G_{\varepsilon} \rIm) 
            < \Gamma \HM^{m_i}(\Sigma_i \cap G_{\varepsilon})
            \,;
        \end{gather}
        moreover, if $\Sigma_i$ is $(\HM^{m_i},m_i)$~rectifiable, then
        \begin{equation}
            \label{eq:dt:im-g-full}
            \HM^{m_i+1}(g \lIm I \times (\Sigma_i \cap G_{\varepsilon}) \rIm)
            \le \HM^{m_i+1}(f \lIm I \times (\Sigma_i \cap G_{\varepsilon}) \rIm)
            < \Gamma \delta \HM^{m_i}(\Sigma_i \cap G_{\varepsilon}) \,.
        \end{equation}
    \end{enumerate}
\end{theorem}

\begin{proof}
    Fix $\varepsilon \in (0,\varepsilon_0)$. Without loss of generality we
    assume that $\Sigma_i \subseteq G_{\varepsilon}$ for $i =
    1,2,\ldots,l$. Define
    \begin{displaymath}
        \mathcal C = \bigl\{ K \in \CX(\mathcal F)
        : K \cap G_0 \ne \varnothing \,,\,
        \dim(K) \ge m+1 \bigr\} \,.
    \end{displaymath}
    Let $\Delta = \Delta(n,m) \in \nat$ be so big that
    \begin{displaymath}
        \Delta \ge
        \max\bigl\{
        \HM^0(\{ K \in \CX(\mathcal F) : K \cap L \ne \varnothing \}) 
        : L \in \CX(\mathcal F) \bigr\}  \,.
    \end{displaymath}
    Let $\{ K_1,\ldots,K_{N_0} \} = \mathcal C$ be an enumeration of~$\mathcal
    C$ chosen so that for $1 \le i \le j \le N_0$
    \begin{displaymath}
        \text{either $\dim(K_i) = \dim(K_j)$ and~$\side{K_i} \ge \side{K_j}$}
        \quad
        \text{or~$\dim(K_i) \ge \dim(K_j)$} \,.
    \end{displaymath}
    For each $j = 0,1,\ldots,N_0$ we~shall define inductively maps $\varphi_j$,
    $\psi_j$, $\zeta_{Q,j}$, $\eta_{Q,j}$, measures $\nu_{Q,j,i}$, and open
    sets~$U_j \subseteq \R^n$, where $i \in \{ 1,2,\ldots,l \}$ and $Q \in
    \mathcal F$. First we~set
    \begin{displaymath}
        \psi_0 = \varphi_0 = \eta_{Q,0} = \zeta_{Q,j} = \id{\R^n} \,,
        \quad
        \nu_{Q,0,i} = \HM^{m_i} \restrict \Sigma_i \,,
        \quad 
        U_0 = \R^n \,.
    \end{displaymath}
    Assume that $\varphi_{j-1}$, $\psi_{j-1}$, $\eta_{Q,j-1}$, $U_{j-1}$,
    $\nu_{j-1,1}, \ldots, \nu_{j-1,l}$ are defined for some $j =
    1,2,\ldots,N_0$. We~set
    \begin{equation}
        \label{eq:nu-ji}
        \nu_{Q,j,i} = (\eta_{Q,j-1})_{\#} \bigl( \|\uD\eta_{Q,j-1}\|^{m_i} \HM^{m_i} \restrict \Sigma_i \bigr)
        \quad \text{for $i = 1,2,\ldots,l$ and $Q \in \mathcal F$} \,.
    \end{equation}
    Observe, that all the measures $\nu_{Q,i,j}$ for $i \in \{ 1,2,\ldots,l \}$
    and $Q \in \mathcal F$ are Radon because $\eta_{Q,j-1}$ is smooth and proper
    and $\HM^{m_i} \restrict \Sigma_i$ is finite. Let $\varphi_j$ and~$U_j$ be
    the map and the neighbourhood of~$\Sigma$ constructed by
    employing~\ref{lem:one-cube-deform} with
    \begin{displaymath}
        \begin{aligned}
            K_j \,, \varepsilon/\Delta \,, \HM^0(\{Q \in \mathcal F : Q \cap K_j \ne \varnothing \}) \,,
            &\bigl\{ (m_i, \nu_{Q,i,j}) : i \in \{1,2,\ldots,l\} \,, Q \in \mathcal F \,, Q \cap K_j \ne \varnothing \bigr\}
            \\
            \text{in place of} \quad
            K \,, \varepsilon_0 \,, l \,, &\bigl\{ (m_i,\nu_i) : i \in \{1,2,\ldots,l\} \bigr\} \,.
        \end{aligned}
    \end{displaymath}
    If $Q \cap K_j \ne \varnothing$, we set $\zeta_{Q,j} = \varphi_{j}$, and
    if $Q \cap K_j = \varnothing$, we set $\zeta_{Q,j} = \id{\R^n}$. Next, we
    define
    \begin{equation}
        \psi_j = \varphi_j \circ \psi_{j-i} \,,
        \quad
        \eta_{Q,j} = \zeta_{Q,j} \circ \eta_{Q,j-1} 
        \quad \text{for $Q \in \mathcal F$} \,.
    \end{equation}

    If $m_1 > m_l$, then we set $N = N_0$. If $m = m_1 = \cdots = m_l$, then we
    still have to take care of the cubes of dimension $m$ which are not fully
    covered. In this case we define
    \begin{displaymath}
        \mathcal C' = \left\{
            K \in \CX(\mathcal F) 
            : 
            \begin{gathered}
                K \cap G_0 \ne \varnothing \,,\,
                \dim(K) = m \,,\,
                \\
                K \cap \psi_{N_0}\lIm \Sigma \rIm \ne K \,,\,
                \cInt{K} \cap \psi_{N_0}\lIm \Sigma \rIm \ne \varnothing
            \end{gathered}
        \right\} \,.
    \end{displaymath}
    We enumerate $\mathcal C' = \{ K_{N_0+1}, \ldots, K_{N} \}$ so that for $N_0
    < i \le j \le N$ we have $\side{K_i} \ge \side{K_j}$. For $j = N_0+1,
    \ldots, N$ we define inductively $\varphi_j$, $\psi_j$, $U_j$ similarly as
    before by employing~\ref{lem:one-cube-deform} with $K_j$,
    $\varepsilon/\Delta$, $l$, $\HM^{m} \restrict \psi_{j-1} \lIm \Sigma_1
    \rIm$, \ldots, $\HM^{m} \restrict \psi_{j-1} \lIm \Sigma_l \rIm$, $m$,
    \ldots, $m$ in place of~$K$, $\varepsilon_0$, $l$, $\nu_1$, \ldots, $\nu_l$,
    $m_1$, \ldots, $m_l$ and we set $\psi_j = \varphi_j \circ \psi_{j-i}$.
    
    Let $f(t,\cdot)$ for $t \in I$ be defined as in~\ref{i:dt:decomp}. For $t >
    1$ we set $f(t,\cdot) = f(1,\cdot)$ and for $t < 0$ we set $f(t,\cdot) =
    f(0,\cdot)$. Then we define $g(t,\cdot) = f(tN_0/N,\cdot)$ for all $t \in
    \R$ and we set $U = \tbcap \bigl\{ U_j : j = 1,2,\ldots, N \bigr\}$.
    Clearly~\ref{i:dt:identity} and~\ref{i:dt:decomp} are satisfied.

    \emph{Proof of~\ref{i:dt:admissible}}: First observe that $\varphi_{j} \in
    \adm{K_j + \oball{0}{\varepsilon/\Delta}}$ for each $j = 1,2,\ldots,N$ due
    to~\ref{lem:one-cube-deform}\ref{i:one:admissible}. Since $K_j +
    \cball{0}{\varepsilon/\Delta} \subseteq G_{\varepsilon}$ we see that $\psi_j
    \in \adm{G_{\varepsilon}}$.

    Fix $t \in I$ and choose $j \in \nat$ such that $j \le tN \le j+1$. We have
    \begin{displaymath}
        f(t,\cdot) = \bigl( s(tN-j) \varphi_{j+1} + (1-s(tN-j)) \id{\R^n} \bigr) \circ \psi_{j} \,.
    \end{displaymath}
    Since $\varphi_{j+1} \in \adm{K_j + \oball{0}{\varepsilon/\Delta}}$ and
    $K_{j+1} + \oball{0}{\varepsilon/\Delta}$ is convex we see that
    \begin{displaymath}
        s(tN - j) \varphi_{j+1} + (1-s(tN-j)) \id{\R^n} \in \adm{K_{j+1} + \oball{0}{\varepsilon/\Delta}} \,;
    \end{displaymath}
    hence, $f(t,\cdot) \in \adm{G_{\varepsilon}}$.

    \emph{Proof of~\ref{i:dt:cube-pres}}: Since $\mathcal F$ is admissible
    for $j \in \{1,2,\ldots,N\}$ and $L \in \CX(\mathcal F)$ exactly one of the
    following options holds
    \begin{itemize}
    \item $L \cap K_j = \varnothing$ and then $\varphi_j\lIm L \rIm = L$
        by~\ref{lem:one-cube-deform}\ref{i:one:id};
    \item $L \cap K_j \ne \varnothing$ and $\side{L} \ge \frac 12 \side{K_j}$
        and $L \cap \cInt{K_j} = \varnothing$ and $\varphi_j\lIm L \rIm
        \subseteq L$ by~\ref{lem:one-cube-deform}\ref{i:one:faces};
    \item $L \cap K_j \ne \varnothing$ and $L \cap \cInt{K_j} \ne \varnothing$
        and $\dim(L) > \dim(K_j)$ and $\side{L} \ge \side{K_j}$, because $K_j
        \in \CX(\mathcal F)$, and $\varphi_j\lIm L \rIm \subseteq L$
        by~\ref{def:cube-complex}
        and~\ref{lem:one-cube-deform}\ref{i:one:faces}.
    \end{itemize}
    In consequence, if $L \in \CX(\mathcal F)$, then $\psi_j\lIm L \rIm
    \subseteq L$ for $j \in \{1,\ldots,N\}$ and, since $L$ is convex, we~obtain
    $f(t,\cdot) \lIm L \rIm \subseteq L$ for $t \in I$. In particular
    \begin{equation}
        \label{eq:phi-cube-im}
        \varphi_j\lIm L \rIm \subseteq L
        \quad
        \text{for $j \in \{ 1, 2, \ldots, N \}$ and $L \in \CX(\mathcal F)$} \,.
    \end{equation}

    \emph{Proof of~\ref{i:dt:m-dim-im}}:
    For $j \in \{1,\ldots,N\}$ set
    \begin{displaymath}
        \alpha(j) = \min \bigl\{ i : \dim(K_i) = \dim(K_j) \bigr\}
        \quad \text{and} \quad
        \beta(j) = \max \bigl\{ i : \dim(K_i) = \dim(K_j) \bigr\} \,.
    \end{displaymath}
    For $j = 1, \ldots, N$ set
    \begin{displaymath}
        \mathcal C_j = \bigl\{
        L \in \cubes_* 
        : \exists i \in \nat \quad \alpha(j) \le i \le j
          \text{ and $L$ is a face of $K_i$}
        \bigr\} \,.
    \end{displaymath}
    We shall prove by induction the following claim:
    \begin{equation}
        \label{eq:dt:claim}
        \left|
        \begin{aligned}
            \text{for each $j = 1,2,\ldots,N_0$ if $k = \dim(K_j)$, then}&
            \\
            \psi_j\lIm U \rIm \cap G_0
            \subseteq
            \tbcup \bigl( \mathcal C_j \cap \cubes_{k-1}\bigr) 
            &\cup
            \tbcup \bigl( (\mathcal C_{\beta(j)} \without \mathcal C_j) \cap \cubes_{k}\bigr) \,.
        \end{aligned}
        \right.
    \end{equation}

    If $j=1$, then $k = n$ and $\psi_1 = \varphi_1$ and claim~\eqref{eq:dt:claim}
    follows from~\ref{lem:one-cube-deform}\ref{i:one:proj}\ref{i:one:faces}.

    Assume $j \in \{2,3,\ldots,N_0\}$ and $k = \dim(K_j)$. By inductive
    hypothesis, if $j > \alpha(j)$, then
    \begin{multline*}
        \psi_{j-1}\lIm U \rIm \cap F
        = \tbcup \Bigl(
        \bigl( \mathcal C_{j-1} \cap \cubes_{k-1}\bigr) 
        \cup
        \bigl( (\mathcal C_{\beta(j)} \without \mathcal C_{j-1}) \cap \cubes_{k}\bigr) \Bigr)
        \cap F \cap \psi_{j-1}\lIm U \rIm
        \\
        \subseteq
        \tbcup \bigl( \mathcal C_{j-1} \cap \cubes_{k-1}\bigr) 
        \cup
        \bigl( \psi_{j-1}\lIm U \rIm \cap K_j \bigr)
        \cup
        \bigl(
        \tbcup \bigl( (\mathcal C_{\beta(j)} \without \mathcal C_j) \cap \cubes_{k}\bigr) 
        \cap F \bigr)
    \end{multline*}
    and if $j = \alpha(j)$, then $j-1 = \beta(j-1)$ and
    \begin{multline*}
        \psi_{j-1}\lIm U \rIm \cap F
        = \tbcup \bigl( \mathcal C_{\beta(j-1)} \cap \cubes_k\bigr) 
        \cap F \cap \psi_{j-1}\lIm U \rIm
        = \tbcup \bigl( \mathcal C_{\beta(j)} \cap \cubes_k\bigr) 
        \cap F \cap \psi_{j-1}\lIm U \rIm
        \\
        \subseteq
        \bigl( \psi_{j-1}\lIm U \rIm \cap K_j \bigr)
        \cup
        \bigl(
        \tbcup \bigl( (\mathcal C_{\beta(j)} \without \mathcal C_j) \cap \cubes_{k}\bigr) 
        \cap F \bigr) \,.
    \end{multline*}
    Claim~\eqref{eq:dt:claim} follows now by the following observations:
    \begin{itemize}
    \item if $j > \alpha(j)$ and $Q \in \mathcal C_{j-1} \cap \cubes_{k-1}$,
        then $Q \cap \cInt{K_j} = \varnothing$ so $\varphi_j\lIm Q \rIm
        \subseteq Q$ by~\eqref{eq:phi-cube-im};
    \item if $\alpha(k) \le j \le \beta(j)$, we have $\varphi_j \lIm \psi_{j-1}
        \lIm U \rIm \cap K_j \rIm \subseteq \cBdry{K_j} \subseteq \tbcup
        \bigl( \mathcal C_j \cap \cubes_{k-1}\bigr)$
        by~\ref{lem:one-cube-deform}\ref{i:one:proj};
    \item if $\alpha(k) \le j \le \beta(j)$ and $Q \in (\mathcal C_{\beta(j)}
        \without \mathcal C_j) \cap \cubes_k$, then $\varphi_j\lIm Q \rIm
        \subseteq Q$ by~\eqref{eq:phi-cube-im}.
    \end{itemize}

    \emph{Proof of~\ref{i:dt:full-faces}}: This follows
    by~\ref{i:dt:m-dim-im} and the construction.

    \emph{Proof of~\ref{i:dt:estimates}}: To prove~\eqref{eq:dt:deriv-bound}
    take $x \in G_{\varepsilon} \without G_0$ and note that are at most $\Delta$
    maps amongst $\varphi_1$, \ldots, $\varphi_N$ which move the point $x$ so,
    recalling~\ref{lem:one-cube-deform}\ref{i:one:Dout}, we get
    \begin{equation}
        \label{eq:dt:out-Df}
        \|\uD f(t,\cdot)(x)\| \le \|\uD \psi_{j+1}(x)\| + \|\uD \psi_{j}(x)\| 
        \le 2 \Gamma_{\ref{lem:one-cube-deform}}^{\Delta} 
        \quad \text{whenever $j \le Nt \le j+1$} \,.
    \end{equation}

    To prove~\eqref{eq:dt:im-cube} choose a cube $Q \in \mathcal F$ and $i \in
    \{1,2,\ldots,l\}$. 
    Set
    \begin{displaymath}
        \widetilde Q_{\varepsilon} = \widetilde Q + \oball{0}{\varepsilon}
        = \tbcup \{ R \in \mathcal F : R \cap Q \ne \varnothing \} + \oball{0}{\varepsilon} \,.
    \end{displaymath}
    Observe that for $j \in \{ 1,2,\ldots, N_0 \}$ we have
    \begin{displaymath}
        \eta_{Q,j}^{-1} \lIm \widetilde Q_{\varepsilon} \rIm = \widetilde Q_{\varepsilon}
        \quad \text{and} \quad
        \zeta_{Q,j}^{-1} \lIm \widetilde Q_{\varepsilon} \rIm = \widetilde Q_{\varepsilon} \,.
    \end{displaymath}
    Therefore, we may write
    \begin{multline}
        \label{eq:dt:eta-est}
        \int_{\widetilde Q_{\varepsilon} \cap \Sigma_i}
        \|\uD \eta_{Q,j}\|^{m_i} \ud \HM^{m_i}
        \le \int_{\eta_{Q,j}^{-1} \lIm \widetilde Q_{\varepsilon} \rIm}
        \| (\uD \zeta_{Q,j}) \circ \eta_{Q,j-1} \|^{m_i} \|\uD \eta_{Q,j-1}\|^{m_i}
        \ud \HM^{m_i} \restrict \Sigma_i
        \\
        =  \int_{\zeta_{Q,j}^{-1} \lIm  \widetilde Q_{\varepsilon} \rIm}
        \| \uD \zeta_{Q,j} \|^{m_i} \ud \nu_{Q,j}
        = \int_{\widetilde Q_{\varepsilon}}
        \| \uD \zeta_{Q,j} \|^{m_i} \ud \nu_{Q,j} \,.
    \end{multline}
    If $\zeta_{Q,j} = \id{\R^n}$, then we obtain
    \begin{equation}
        \label{eq:dt:eta-trivial}
        \int_{\widetilde Q_{\varepsilon}}
        \| \uD \zeta_{Q,j} \|^{m_i} \ud \nu_{Q,j}
        = \nu_{Q,j}(\widetilde Q_{\varepsilon})
        = \int_{\widetilde Q_{\varepsilon} \cap \Sigma_i}
        \|\uD \eta_{Q,j-1}\|^{m_i} \ud \HM^{m_i} \,.
    \end{equation}
    If $\zeta_{Q,j} = \varphi_k$ for some $k \in \{1,2,\ldots,N_0\}$, then $K_k
    \subseteq \widetilde Q_{\varepsilon}$ and if $K_k$ is a face of some cube $R
    \in \cubes_*$ with $\dim R > \dim K_k$, then $\spt \nu_{Q,j} \cap \cInt{R} =
    \varnothing$; thus,
    employing~\ref{lem:one-cube-deform}\ref{i:one:id}\ref{i:one:Dout}\ref{i:one:int},
    we get
    \begin{multline}
        \label{eq:dt:eta-non-trivial}
        \int_{\widetilde Q_{\varepsilon}}
        \| \uD \zeta_{Q,j} \|^{m_i} \ud \nu_{Q,j}
        = \int_{\widetilde Q_{\varepsilon} \cap K_k}
        \| \uD \varphi_k \|^{m_i} \ud \nu_{Q,j}
        + \int_{\widetilde Q_{\varepsilon} \without K_k}
        \| \uD \varphi_k \|^{m_i} \ud \nu_{Q,j}
        \\
        \le \Gamma_{\ref{lem:one-cube-deform}} \nu_{Q,j}(\widetilde Q_{\varepsilon})
        = \Gamma_{\ref{lem:one-cube-deform}}  \int_{\widetilde Q_{\varepsilon} \cap \Sigma_i}
        \|\uD \eta_{Q,j-1}\|^{m_i} \ud \HM^{m_i} \,.
    \end{multline}
    Now, the second case (when $\zeta_{Q.j} = \varphi_k$ for some $k \in
    \{1,2,\ldots,N_0\}$) can happen at~most $\Delta$ times so
    combining~\eqref{eq:dt:eta-non-trivial}, \eqref{eq:dt:eta-trivial},
    and~\eqref{eq:dt:eta-est}
    \begin{equation}
        \int_{\widetilde Q_{\varepsilon} \cap \Sigma_i}
        \|\uD \eta_{Q,j}\|^{m_i} \ud \HM^{m_i}
        \le \Gamma_{\ref{lem:one-cube-deform}}^{\Delta} \HM^{m_i}(\Sigma_i \cap \widetilde Q_{\varepsilon}) \,.
    \end{equation}
    We may now finish the proof of~\eqref{eq:dt:im-cube} by writing
    \begin{multline}
        \int_{\Sigma_i \cap Q} \|Dg(t,\cdot)\|^{m_i} \ud \HM^{m_i}
        \le 2^{m_i-1} \int_{\Sigma_i \cap Q} \|\uD \psi_{j+1}(x)\|^{m_i} + \|\uD \psi_{j}(x)\|^{m_i} \ud \HM^{m_i}
        \\
        = 2^{m_i-1} \int_{\Sigma_i \cap Q} \|\uD \eta_{Q,j+1}(x)\|^{m_i} + \|\uD \eta_{Q,j}(x)\|^{m_i} \ud \HM^{m_i}
        \le 2^{m_i} \Gamma_{\ref{lem:one-cube-deform}}^{\Delta} \HM^{m_i}(\Sigma_i \cap \widetilde Q_{\varepsilon}) \,,
    \end{multline}
    whenever $j \in \{1,2,\ldots,N_0\}$ and $t \in \R$ are such that $j \le N_0t \le j+1$.

    To prove~\eqref{eq:dt:im-gt} we write
    \begin{multline}
        \int_{\Sigma_i \cap G_{\varepsilon}} \|Dg(t,\cdot)\|^{m_i} \ud \HM^{m_i}
        \le \sum_{Q \in \mathcal A} \int_{\Sigma_i \cap Q} \|Dg(t,\cdot)\|^{m_i} \ud \HM^{m_i}
        + \int_{\Sigma_i \cap G_{\varepsilon} \without G_0} \|Dg(t,\cdot)\|^{m_i} \ud \HM^{m_i}
        \\
        \le 2^{m_i} \Gamma_{\ref{lem:one-cube-deform}}^{\Delta} \sum_{Q \in \mathcal A} \HM^{m_i}(\Sigma_i \cap \widetilde Q_{\varepsilon})
        + 2^{m_i} \Gamma_{\ref{lem:one-cube-deform}}^{\Delta m_i} \HM^{m_i}(\Sigma_i \cap G_{\varepsilon} \without G_0)
        \\
        \le 2^{m_i} \Delta^2 \Gamma_{\ref{lem:one-cube-deform}}^{\Delta m_i} \HM^{m_i}(\Sigma_i \cap G_{\varepsilon}) \,.
    \end{multline}

    To prove~\eqref{eq:dt:im-cube-f} and~\eqref{eq:dt:im-f1} note that for $Q
    \in \mathcal F$, by construction (in~particular by~\ref{i:dt:full-faces}),
    \begin{displaymath}
        \HM^{m_i}( f(1,\cdot) \lIm \Sigma_i \cap Q \rIm \cap G_0)
        \le \HM^{m_i}( g(1,\cdot) \lIm \Sigma_i \cap Q \rIm \cap G_0)
    \end{displaymath}
    and, recalling~\ref{i:dt:cube-pres} and~\eqref{eq:dt:out-Df},
    \begin{displaymath}
        \HM^{m_i}( f(1,\cdot) \lIm \Sigma_i \cap Q \rIm \without G_0)
        \le 2^{m_i} \Gamma_{\ref{lem:one-cube-deform}}^{\Delta m_i} \HM^{m_i}(\Sigma_i \cap Q \without G_0) \,.
    \end{displaymath}

    Assume now that $\Sigma_i$ is $(\HM^{m_i},m_i)$~rectifiable. Then $I \times
    \Sigma_i$ is $(\HM^{m_i+1},m_i+1)$~rectifiable by~\cite[3.2.23]{Fed69} and
    we may use the area formula~\cite[3.2.20]{Fed69} to
    prove~\eqref{eq:dt:im-g-full}. Define
    \begin{displaymath}
        A_j = G_{\varepsilon} \cap \tbcup
        \left\{
            R \in \mathcal F : R \cap K_{j+1} \ne \varnothing
        \right\} \,.
    \end{displaymath} Let $\tau = (1,0) \in \R \times
    \R^n$ be the ``time direction'' and set $g_t(x) = g(t,x)$ for $(t,x) \in \R
    \times \R^n$. Observe that for $(t,x) \in I \times \Sigma$ and $j \in \nat$
    such that $j \le tN_0 \le j+1$ we have
    \begin{multline}
        \bigl( \HM^{m_i+1} \restrict (I \times \Sigma_i), m_i+1 \bigr) \ap J_{m_i+1} g(t,x)
        \le |\uD g(t,x)\tau| \cdot \|\uD g_t(x)\|^{m_i} 
        \\
        \le 2N_0 |\varphi_{j+1}(\psi_j(x)) - \psi_j(x)| \cdot \|\uD g_t(x)\|^{m_i} \,;
    \end{multline}
    hence, using~\cite[3.2.20, 3.2.23]{Fed69} , \ref{i:dt:decomp},
    \ref{i:dt:cube-pres},
    \ref{lem:one-cube-deform}\ref{i:one:id}\ref{i:one:small}
    \begin{multline}
        \label{eq:dt:full-measure}
        \HM^{m_i+1}\bigl( g \lIm I \times (\Sigma_i \cap G_{\varepsilon}) \rIm \bigr)
        \le \int_0^1 \int_{\Sigma_i \cap G_{\varepsilon}}
        |\uD g(t,x)\tau| \cdot \|\uD g_t(x)\|^{m_i} \ud \HM^{m_i}(x) \ud \LM^1(t)
        \\
        \le \sum_{j=0}^{N_0-1} \int_{j/N_0}^{(j+1)/N_0} \int_{\Sigma_i \cap G_{\varepsilon}} 
        2N_0 |\varphi_{j+1}(\psi_j(x)) - \psi_j(x)| \cdot \|\uD g_t(x)\|^{m_i}  \ud \HM^{m_i}(x) \ud \LM^1(t)
        \\
        \le 2^{m_i} N_0 \sqrt{n} \delta 
        \sum_{j=0}^{N_0-1}
        \int_{j/N_0}^{(j+1)/N_0}  \ud \LM^1(t)
        \int_{\Sigma_i \cap A_j}
        \|\uD\psi_{j+1}(x)\|^{m_i} + \|\uD\psi_{j}(x)\|^{m_i} \ud \HM^{m_i}(x)
        \\
        \le 2^{m_i} \sqrt{n} \delta 2 \Delta^2 \Gamma_{\ref{lem:one-cube-deform}}^{\Delta m_i}
        \HM^{m_i}(\Sigma_i \cap G_{\varepsilon})
        \,.
    \end{multline}
    If $m = m_1 > m_l$, then $N = N_0$ and $f = g$ and there is nothing more to
    prove. Otherwise, we have $m = m_1 = \cdots = m_l$ and $g \lIm I \times
    (\Sigma_i \cap G_{\varepsilon}) \rIm = f \lIm I' \times (\Sigma_i \cap
    G_{\varepsilon}) \rIm$, where $I' = [0,N_0/N]$. Hence, we need to estimate
    $\HM^{m+1}(f \lIm (I \without I') \times (\Sigma_i \cap G_{\varepsilon})
    \rIm)$. Observe, that
    \begin{displaymath}
        f \lIm (I \without I') \times (\Sigma_i \cap G_{0}) \rIm
        \subseteq \tbcup \CX(\mathcal F) \cap \cubes_{m} \,,
    \end{displaymath}
    so $\HM^{m+1}(f \lIm (I \without I') \times (\Sigma_i \cap G_{0}) \rIm) =
    0$. On the other hand $\| \uD f_t(x) \| \le
    \Gamma_{\ref{lem:one-cube-deform}}$ for $x \in G_{\varepsilon} \without G_0$
    and $t \in I \without I'$ so $\HM^{m+1}(f \lIm (I \without I') \times
    (\Sigma_i \cap G_{\varepsilon} \without G_{0}) \rIm)$ can be estimated as
    in~\eqref{eq:dt:full-measure}.
\end{proof}

We finish this section with a small lemma that allows to
apply~\ref{lem:reduce-unrect} to the mapping constructed
in~\ref{thm:deformation}.

\begin{lemma}
    \label{lem:rank-bound}
    If $f \in \cnt^{1}(\R^n,\R^n)$ and $U \subseteq \R^n$ and
    $\dim_{\HM}(f \lIm U \rIm) \le m$, then $\dim \im \uD f(x) \le m$ for $x \in U$.
\end{lemma}

\begin{proof}
    Assume there exists a point $x \in U$ such that $\dim \im \uD f(x) = k>m$.
    Define $L = \im \uD f(x) \in \grass nk$ and set $g = \project{L} \circ
    f$. Observe that
    \begin{displaymath}
        \HM^{k}(g\lIm \cball xr \rIm) \le \HM^k(f \lIm \cball xr \rIm)
        \quad \text{for $r > 0$} \,.
    \end{displaymath}
    Moreover, since $f(y) = f(x) + \uD f(x)(y-x) + o(|x-y|)$ we see that for small
    enough $r > 0$ we have 
    \begin{displaymath}
        f(x) + \uD f(x) \lIm \cball{0}{r/2} \rIm \subseteq g\lIm \cball xr \rIm \,.
    \end{displaymath}
    Hence, for some $r > 0$ we obtain $\HM^k(f \lIm \cball xr \rIm) > 0$ which
    contradicts $\dim_{\HM}(f \lIm U \rIm) \le m$.
\end{proof}

% Local Variables:
% coding: utf-8
% eval: (ispell-change-dictionary "british")
% eval: (flyspell-mode)
% End:

%% file: slicing.tex
\section{Slicing varifolds by continuously differentiable functions}
\label{sec:slices}

We recall the theory developed by Almgren in~\cite[I.3]{Alm76}
and~\cite[\S7]{Almgren:Vari}.

In this sections we shall always assume $U \subseteq \R^n$ is open, $m,n,\nu \in
\integers$ are such that $0 \le \nu \le m < n$, $V \in \Var{m}(U)$, $f \in
\cnt^1(U,\R^{\nu})$ is proper, and $\pi : \R^n \times \grass nm \to \R^n$ is the
projection onto the first factor.
\begin{definition}
    \label{def:Vf-mu}
    Whenever $\beta \in \ccspace{U \times \grass n{m-\nu}}$ and $\varphi \in
    \ccspace{\R^{\nu}}$ we set
    \begin{gather*}
        (V,f)(\beta) = \int_{\{(x,S) \in U \times \grass nm : \|{\bigwedge_{\nu}} \uD f(x) \circ \project S\| > 0\}}
        \beta(x,S \cap \ker \uD f(x))
        \|{\textstyle \bigwedge_{\nu}} \uD f(x) \circ \project S\| \ud V(x,S) \,,
        \\
        \mu_{\beta}(\varphi) = (V,f)(\beta \cdot (\varphi \circ f \circ \pi)) \,.
    \end{gather*}
\end{definition}
It was shown in~\cite[I.3(2)]{Alm76} that $(V,f) \in \Var{{}m - \nu}(U)$ and
$\mu_{\beta}$ is a Radon measure for each $\beta \in \ccspace{U \times \grass
  n{m-\nu}}$.
\begin{definition}
    \label{def:slice}
    The \emph{slice of $V$ with respect to $f$ at $t \in \R^{\nu}$} is the
    varifold $\langle V, f, t \rangle \in \Var{{}m - \nu}(U)$, satisfying
    \begin{displaymath}
        \langle V, f, t \rangle (\beta) 
        = \lim_{r \downarrow 0} \frac{\mu_{\beta}(\cball tr)}{\LM^{\nu}(\cball tr)} \,,
        \quad \text{whenever $\beta \in \ccspace{U \times \grass n{m-\nu}}$} \,.
    \end{displaymath}
\end{definition}

\begin{remark}
    \label{rem:slice:flat-norm}
    By~\cite[I.3(2)]{Alm76}, there exists $\langle V, f, t \rangle \in \Var{{}m -
      \nu}(U)$ and, since $f$ is proper, $\spt \|\langle V, f, t \rangle\|$ is
    compact for $\LM^{\nu}$~almost all $t \in \R^{\nu}$. Next, we view $\{ V \in
    \Var{{}m-\nu}(U) : \spt \|V\| \text{ is compact} \}$ as a subset of the
    vectorspace (cf.~\cite[2.5.19]{Fed69})
    \begin{gather}
        \bigl\{ W \in \ccspace{U \times \grass n{m - \nu}}^{*} : \spt W \text{ is compact} \bigr\} \,,
        \quad \text{with the norm}
        \\
        \vvvert W \vvvert = \sup \bigl\{ W(\beta) 
        : \beta \in \ccspace{U \times \grass n{m - \nu}} ,\, 
        \sup \im |\beta| \le 1 ,\,
        \Lip \beta \le 1
        \bigr\} \,.
    \end{gather}
    Then $\{ W \in \ccspace{U \times \grass n{m - \nu}}^{*} : \spt W \text{ is
      compact} \}$ becomes a separable normed vectorspace such that the norm
    topology coincides with the weak topology.
\end{remark}

\begin{definition}[\protect{cf.~\cite[I.3(3)]{Alm76}}]
    \label{def:Lebesgue-set}
    The~\emph{Lebesgue set} of the slicing operator $\langle V, f, \cdot
    \rangle$ is the set of those $t \in \R^{\nu}$ for which
    \begin{displaymath}
        \lim_{r \downarrow 0} r^{-\nu} \int_{\cball tr}
        \vvvert \langle V, f, s \rangle - \langle V, f, t \rangle \vvvert
        \ud \LM^{\nu}(s) = 0 \,.
    \end{displaymath}
\end{definition}

\begin{remark}
    \label{rem:leb-points}
    Note that $\LM^\nu$ almost all $t \in \R^\nu$ are Lebesgue points of $\langle
    V, f, \cdot \rangle$; see~\cite[I.3(3)]{Alm76} for the proof.
\end{remark}

\begin{remark}
    \label{rem:slice-rect}
    Recalling~\cite[I.3(4)]{Alm76} we see that if $S \subseteq U$ is
    $(\HM^m,m)$~rectifiable and $\HM^m$~measurable and bounded, then
    \begin{displaymath}
        \langle \var{m}(S), f, t \rangle
        = \var{{}m-\nu}(S \cap f^{-1} \{t\} )
        \in \RVar{m-\nu}(U)
        \quad \text{for $\LM^\nu$ almost all $t \in \R^\nu$} \,.
    \end{displaymath}
\end{remark}

Next, we define the product of a varifold with a cube as
in~\cite[I.3(5)]{Alm76}. Products of more general varifolds were described also
in~\cite[\S3]{KM2017}.
\begin{definition}
    \label{def:product}
    If $l \in \integers$, and $l \ge 1$, and $I = \{ t \in \R : 0 \le t \le
    1\}$, and $j_l : \R^l \to \R^l \times \R^n$ and $j_n : \R^n \to
    \R^l \times \R^n$ are injections, then
    \begin{displaymath}
        (\var{l}(I^l) \times V)(\alpha) 
        = \int_{I^l} \int
        \alpha((t,x), j_l \lIm \R^l \rIm + j_n \lIm T \rIm)
        \ud V(x,T) \ud \LM^l(t) \,,
    \end{displaymath}
    for $\alpha \in \ccspace{\R^{l} \times U \times \grass{l+n}{l+m}}$.
\end{definition}

\begin{definition}
    \label{def:aux-fun}
    For $t \in \R$ and $\delta \in (0,1)$ and $\rho : \R^n \to \R$ we
    define the functions
    \begin{displaymath}
        i_t : \R^n \to \R \times \R^{n} \,, 
        \quad
        s_{\delta} : \R \to \R \,,
        \quad
        K_{\rho,t,\delta} : U \to \R \times U \,,
    \end{displaymath}
    by requiring that $s_{\delta}$ is of class~$\cnt^{\infty}$ and
    \begin{gather*}
        s_{\delta}(\tau) = \tau \quad \text{for } \delta \le \tau \le 1 - \delta \,,
        \quad
        s_{\delta}(\tau) = 0 \quad \text{for } \tau \le 0 \,,
        \quad
        s_{\delta}(\tau) = 1 \quad \text{for } \tau \ge 1 \,,
        \\
        0 \le s_{\delta}'(\tau) \le 1+\delta \quad \text{for } \tau \in \R \,,
        \quad
        i_t(x) = (t,x) \quad \text{for $x \in \R^n$} \,,
        \\
        K_{\rho,t,\delta}(x) = \bigl( s_{\delta}((t - \rho(x)) / \delta) , x \bigr) 
        \quad \text{for $x \in \R^n$} \,.
    \end{gather*}
\end{definition}

\begin{lemma}
    \label{lem:slice-blow-up}
    Let $V \in \Var{m}(U)$ and~$\rho \in \cnt^1(U,\R)$ be a proper map, and
    $\iota \in (0,\infty)$, and $t \in \R$ be a~Lebesgue point of~$\langle V,
    \rho, \cdot \rangle$ such that $V \restrict \{ (x,S) \in \R^n \times \grass
    nm : t - \iota \le \rho(x) \le t \} \in \RVar{m}(U)$ and
    \begin{equation}
        \label{eq:sl:limit-cond}
        \lim_{\delta \downarrow 0} \|V\|(\{ x \in U : t - \delta \le \rho(x) < t \}) = 0 \,.
    \end{equation}
    Set $V_0 = V \restrict \{ (x,S) \in U \times \grass nm : \rho(x) \ge t \}$
    and $V_1 = V \restrict \{ (x,S) \in U \times \grass nm : \rho(x) < t \}$.
    Then
    \begin{displaymath}
        \lim_{\delta \downarrow 0} K_{\rho,t,\delta\,\#} V 
        = i_{0\,\#} V_0 + i_{1\,\#} V_1 + \var{1}(I) \times \langle V, \rho, t \rangle 
        \in \Var{m}(\R \times U) \,.
    \end{displaymath}
\end{lemma}

\begin{proof}
    Since $\rho$ and $t$ are fixed we abbreviate $K_{\delta} =
    K_{\rho,t,\delta}$. For $\delta \in (0,1)$ define
    \begin{gather*}
        V_{1,\delta} = V \restrict \{ (x,S) \in U \times \grass nm : \rho(x) < t - \delta \} \,,
        \\
        V_{2,\delta} = V \restrict \{ (x,S) \in U \times \grass nm : t - \delta \le \rho(x) < t \} \,.
    \end{gather*}
    Clearly
    \begin{displaymath}
        K_{\delta\,\#} V = i_{0\,\#} V_0 + i_{1\,\#} V_{1,\delta} + K_{\delta\,\#} V_{2,\delta} 
    \end{displaymath}
    and $\lim_{\delta \downarrow 0} i_{1\,\#} V_{1,\delta} = i_{1\,\#} V_{1}$ so
    it suffices to prove that $\lim_{\delta \downarrow 0} K_{\delta\,\#}
    V_{2,\delta} = \var{1}(I) \times \langle V, \rho, t \rangle$. To~this end it is
    enough to show that $\lim_{\delta \downarrow 0} \vvvert
    K_{\delta\,\#} V_{2,\delta} - \var{1}(I) \times \langle V, \rho, t
    \rangle \vvvert = 0$.

    Let $j_i : \R \to \R \times U$ and $j_n : U \to \R \times U$ be
    injections and let $\pi : U \times \grass{n}{m-1} \to U$ be the projection
    onto the first factor. For $\|V\|$ almost all $x$ we define $T$,
    $R_{\delta}$, $P$, $J_{K,\delta}$, and $J_{\rho}$ by requiring
    \begin{gather*}
        T(x) = \Tan^m(\|V\|, x) \in \grass nm \,,
        \quad
        R_{\delta}(x) = DK_{\delta}(x) \lIm T(x) \rIm \in \grass{n+1}{m} \,,
        \\
        P(x) = j_1\lIm \R \rIm + j_n \lIm T(x) \cap \ker \uD\rho(x) \rIm \in \grass{n+1}{m}\,,
        \\
        J_{K,\delta}(x) = (\|V\|,m) \ap J_m K_{\delta}(x) \in \R \,,
        \quad
        J_{\rho}(x) = (\|V\|,m) \ap J_1\rho(x) \in \R \,.
    \end{gather*}
    Whenever $x \in \dmn T$ and $Q \in \grass n{m-1}$ and $\tau \in [0,1]$ we also set
    \begin{gather*}
        \gamma_{\tau}(x,Q) = \bigl( (\tau, x ) , j_1\lIm \R \rIm + j_n \lIm Q \rIm \bigr)
        \in (\R \times U) \times \grass{n+1}m \,,
        \\
        \psi_{\tau,\delta}(x) = \bigl( (s_{\delta}(\tau),x) , R_{\delta}(x) \bigr)
        \in (\R \times U) \times \grass{n+1}m \,,
        \\
        W_{\delta} = K_{\delta\,\#} V_{2,\delta} - \var{1}(I) \times \langle V,
        \rho, t \rangle \in \ccspace{\R \times U \times \grass{n+1}m}^{*} \,.
    \end{gather*}
    Let $\varepsilon \in (0,1/2)$. If $\vvvert \langle V, \rho, t \rangle
    \vvvert > 0$, then assume additionally that $\varepsilon \le 2^{-5} \vvvert
    \langle V, \rho, t \rangle \vvvert^{-1}$. Find $\delta_0 \in (0,1)$ such
    that for all $\delta \in (0,\delta_0)$
    \begin{gather}
        \label{eq:sl:delta}
        \delta < \min\bigl\{
        2^{-5} \varepsilon \bigl(\vvvert \langle V, \rho, t \rangle \vvvert + 2^{-4} \varepsilon\bigr)^{-1} 
        ,\, \varepsilon^2/2 ,\, \iota \bigr\} \,,
        \\
        \label{eq:sl:L1-close}
        \frac 1{\delta}\int_{t-\delta}^t \vvvert \langle V, \rho, \tau \rangle 
        - \langle V, \rho, t \rangle \vvvert \ud \LM^1(\tau)
        \le 2^{-4} \varepsilon \,,
        \\
        \label{eq:sl:meas-slice}
        \|V\|(\{ x \in U : t - \delta \le \rho(x) < t \}) 
        \le (1+\varepsilon^{-4})^{-1} 2^{-4} \varepsilon 
        \,.
    \end{gather}
    Such $\delta_0 > 0$ exists because $t$ is a~Lebesgue point of $\langle V,
    \rho, \cdot \rangle$ and we assumed~\eqref{eq:sl:limit-cond}. It~follows
    from~\eqref{eq:sl:L1-close}, applied to $[t-\delta^2,t]$ and
    $[t-\delta+\delta^2,t]$ and $[t-\delta,t]$, and from~\eqref{eq:sl:delta}
    that
    \begin{equation}
        \label{eq:sl:reminder}
        \int_{\{ \tau \in I : \tau < \delta \text{ or } \tau > 1 - \delta \}}
        \| \langle V, \rho, t - \delta \tau \rangle \|(U)
        \ud \LM^1(\tau) 
        \le 2 (\vvvert \langle V, \rho, t \rangle \vvvert + 2^{-4} \varepsilon) \delta
        \le 2^{-4} \varepsilon \,.
    \end{equation}

    For any $\alpha \in \ccspace{\R \times U \times \grass{n+1}m}$ such that
    $\sup \im |\alpha| \le 1$ and $\Lip \alpha \le 1$, employing the co-area
    formula~\cite[3.2.22]{Fed69}, we get
    \begin{multline}
        \label{eq:sl:Wdelta-est}
        |W_{\delta}(\alpha)|
        = \biggl| \int_0^1 \langle V, \rho, t - \delta \tau \rangle
        (\alpha \circ \psi_{\delta,\tau} \circ \pi \cdot (\delta J_{K,\delta} / J_{\rho}) \circ \pi)
        - \langle V, \rho, t \rangle (\alpha \circ \gamma_{\tau})
        \ud \LM^1(\tau) \biggr|
        \\
        \le \biggl| \int_0^1 \langle V, \rho, t - \delta \tau \rangle
        (\alpha \circ \psi_{\delta,\tau} \circ \pi \cdot (\delta J_{K,\delta} / J_{\rho}) \circ \pi
        - \alpha \circ \gamma_{\tau})
        \ud \LM^1(\tau) \biggr|
        \\
        + \biggl| \int_0^1 \bigl( \langle V, \rho, t - \delta \tau \rangle - \langle V, \rho, t \rangle \bigr)
        (\alpha \circ \gamma_{\tau})
        \ud \LM^1(\tau) \biggr|
        = B_1(\alpha, \delta) + B_2(\alpha, \delta) \,.
    \end{multline}
    Since $\Lip \gamma_{\tau} = 1$ for $\tau \in [0,1]$, we have by~\eqref{eq:sl:L1-close}
    \begin{equation}
        \label{eq:sl:B2-est}
        B_2(\alpha, \delta) \le \int_0^1
        \vvvert \langle V, \rho, t - \delta \tau \rangle - \langle V, \rho, t \rangle \vvvert
        \ud \LM^1(\tau) \le 2^{-4} \varepsilon \,.
    \end{equation}
    To estimate $B_1(\alpha,\delta)$ we set
    \begin{gather*}
        X = \bigl\{ (x,S) \in \dmn T \times \grass n{m-1} : S = T(x) \cap \ker \uD\rho(x) \in \grass n{m-1} \bigr\} \,,
        \\
        A_1(\delta) = \bigl\{ (x,S) \in X
        : J_{\rho}(x) \le \varepsilon^{-2} \delta
        \text{ and }
        \delta J_{K,\delta}(x) \ge J_{\rho}(x)
        \bigr\} \,,
        \\
        A_2(\delta) = \bigl\{ (x,S) \in X
        : J_{\rho}(x) > \varepsilon^{-2} \delta 
        \text{ and }
        \delta J_{K,\delta}(x) \ge J_{\rho}(x)
        \bigr\} \,,
        \\
        A_3(\delta) = \bigl\{ (x,S) \in X
        : \delta J_{K,\delta}(x) < J_{\rho}(x) \bigr\} \,.
    \end{gather*}
    Clearly $\sum_{i=1}^3 \langle V, \rho, s \rangle \restrict A_i(\delta) =
    \langle V, \rho, s \rangle$ for $s \in [t-\delta,t]$. We estimate first the
    third part. Straightforward computations (cf.~\cite[I.3(1)]{Alm76}) show
    that if $(x,S) \in X$ and $\tau \in [0,1]$ and $\delta \in (0,1)$ and
    $\rho(x) = t - \delta \tau$, then, setting $v = \project{T(x)}(\grad
    \rho(x)) \in S^{\perp} \cap T$,
    \begin{gather}
        R_{\delta}(x) = \lin\bigl\{j_1(1) s_{\delta}'(\tau) |v|/\delta + j_n(v/|v|) \bigr\}
        + j_n\lIm T(x) \cap \ker \uD \rho(x) \rIm \,,
        \quad
        J_{\rho}(x) = |v| \,,
        \\
        \label{eq:sl:RP-dist}
        \delta J_{K,\delta}(x) = 
        \bigl( \delta^2 + s_{\delta}'(\tau)^2 J_{\rho}(x)^2 \bigr)^{1/2} \,,
        \quad
        \|\project{R_{\delta}(x)} - \project{P(x)}\| = J_{K,\delta}(x)^{-1} \,,
        \\
        \label{eq:sl:gamma-expr}
        \gamma_{\tau}(x,S) = ((\tau,x), P(x)) \,.
    \end{gather}
    Hence, recalling~\ref{def:aux-fun} we see that $(t - \rho(x))/\delta \in [0,
    \delta) \cup (1-\delta,1]$ whenever $(x,S) \in A_3(\delta)$ so
    \begin{multline}
        \label{eq:sl:A3-est}
        \biggl| \int_{0}^{1} (\langle V, \rho, t - \delta \tau \rangle \restrict A_3(\delta))
        (\alpha \circ \psi_{\delta,\tau} \circ \pi \cdot (\delta J_{K,\delta} / J_{\rho}) \circ \pi
        - \alpha \circ \gamma_{\tau})
        \ud \LM^1(\tau) \biggr|
        \\
        \le 2 \int_{\{ \tau \in I : \tau < \delta \text{ or } \tau > 1 - \delta \}}
        \| \langle V, \rho, t - \delta \tau \rangle \|(U)
        \ud \LM^1(\tau)
        \le 2^{-3} \varepsilon
        \quad \text{by~\eqref{eq:sl:reminder}} \,.
    \end{multline}
    For $(x,S) \in A_1(\delta)$ we have $J_{K,\delta}(x) \le (1 +
    \varepsilon^{-4} \delta^2)^{1/2} \le 1 + \varepsilon^{-4}$ and $\delta
    J_{K,\delta}(x) / J_{\rho}(x) \ge 1$ so, using the co-area
    formula~\cite[3.2.22]{Fed69} and $\sup \im |\alpha| \le 1$, we obtain
    \begin{multline}
        \label{eq:sl:A1-est}
        \biggl| \int_{0}^{1} (\langle V, \rho, t - \delta \tau \rangle \restrict A_1(\delta))
        (\alpha \circ \psi_{\delta,\tau} \circ \pi \cdot (\delta J_{K,\delta} / J_{\rho}) \circ \pi
        - \alpha \circ \gamma_{\tau})
        \ud \LM^1(\tau) \biggr|
        \\
        \le 2 \sup \im |\alpha| (1+\varepsilon^{-4}) \|V\|(\{ x \in U : t - \delta \le \rho(x) < t \})
        \le 2^{-3} \varepsilon 
        \quad \text{by~\eqref{eq:sl:meas-slice}} \,.
    \end{multline}
    To deal with $A_2(\delta)$ first observe that $\delta
    J_{K,\delta}(x)/J_{\rho}(x) \ge 1$ and $J_{\rho}(x) > \delta
    \varepsilon^{-2}$ and $\varepsilon \le 1/2$ imply
    \begin{displaymath}
        s_{\delta}'(\tau)^2 \ge 1 - \delta^2/J_{\rho}(x)^2 \ge 1 - \varepsilon^4 \ge 1/2 \,,
    \end{displaymath}
    where $\tau = (t - \rho(x))/\delta$. Therefore, by~\eqref{eq:sl:gamma-expr}
    and~\eqref{eq:sl:RP-dist} and~\eqref{eq:sl:delta},
    \begin{multline}
        \label{eq:sl:sup-est}
        M = \sup\{ |\psi_{\delta,\tau}(x) - \gamma_{\tau}(x,S)|
        : (x,S) \in A_2(\delta) ,\, \delta\tau = t - \rho(x) \}
        \\
        = \delta + \sup\{ \|\project{R_{\delta}(x)} - \project{P(x)}\| 
        : (x,S) \in A_2(\delta) ,\, \delta\tau = t - \rho(x) \}
        \\
        \le \varepsilon^2/2 + \delta / \bigl(\delta^2 + 1/2\delta^2/\varepsilon^4 \bigr)^{1/2} 
        \le 2 \varepsilon^2 \,.
    \end{multline}
    If $(x,S) \in A_2(\delta)$, then $\delta \le \varepsilon^2 J_{\rho}(x)$ so
    $\delta J_{K,\delta}(x) / J_{\rho}(x) \le 1 + \varepsilon^4$ and, using
    $\Lip \alpha \le 1$,
    \begin{multline}
        \label{eq:sl:A2-est}
        \biggl| \int_{0}^{1} (\langle V, \rho, t - \delta \tau \rangle \restrict A_2(\delta))
        (\alpha \circ \psi_{\delta,\tau} \circ \pi \cdot (\delta J_{K,\delta} / J_{\rho}) \circ \pi
        - \alpha \circ \gamma_{\tau})
        \ud \LM^1(\tau) \biggr|
        \\
        \le \bigl( \varepsilon^4 + M \bigr)
        \int_0^1 \| \langle V, \rho, t - \delta \tau \rangle\|(U) \ud \LM^1(\tau) 
        \le 4 \varepsilon^2 \int_0^1 \| \langle V, \rho, t - \delta \tau \rangle\|(U) \ud \LM^1(\tau) \,.
    \end{multline}
    Recall that $\varepsilon < 1/2$ and if $\vvvert \langle V, \rho, t \rangle
    \vvvert > 0$, we assumed $\varepsilon \le 2^{-5} \vvvert \langle V, \rho, t
    \rangle \vvvert^{-1}$. Thus, using~\eqref{eq:sl:L1-close}
    \begin{equation}
        \label{eq:sl:eps-int-est}
        4 \varepsilon^2 \int_0^1 \| \langle V, \rho, t - \delta \tau \rangle\|(U) \ud \LM^1(\tau)
        \le \varepsilon (4\varepsilon \vvvert \langle V, \rho, t \rangle \vvvert + 2^{-2}\varepsilon^2)
        \le 2^{-2} \varepsilon \,.
    \end{equation}
    Finally, combining~\eqref{eq:sl:Wdelta-est}, \eqref{eq:sl:B2-est},
    \eqref{eq:sl:A3-est}, \eqref{eq:sl:A1-est}, \eqref{eq:sl:A2-est},
    \eqref{eq:sl:eps-int-est} we see that $|W_{\delta}(\alpha)| \le \varepsilon$
    for $\delta \in (0,\delta_0)$. Since $\delta_0$ was chosen independently
    of~$\alpha$ we have $\vvvert W_{\delta} \vvvert \le \varepsilon$ for $\delta
    \in (0,\delta_0)$.
\end{proof}

\begin{corollary}
    \label{cor:slice-blow-up-ae}
    Assume $a,b \in \R$ are such that $V \restrict \{ (x,S) \in \R^n \times
    \grass nm : a \le \rho(x) \le b \} \in \RVar{m}(U)$. Then for $\LM^1$~almost
    all $t \in (a,b)$ 
    \begin{displaymath}
        \lim_{\delta \downarrow 0} K_{\rho,t,\delta\,\#} V 
        = i_{0\,\#} V_0 + i_{1\,\#} V_1 + \var{1}(I) \times \langle V, \rho, t \rangle 
        \in \Var{m}(\R \times U) \,.
    \end{displaymath}
\end{corollary}

% Local Variables:
% coding: utf-8
% eval: (ispell-change-dictionary "british")
% eval: (flyspell-mode)
% End:

%% file: ludrb.tex
\section{Density ratio bounds}
\label{sec:ludrb}

The main result~\ref{thm:ludrb} of this section gives lower and upper bounds on
the density ratios of~$\|V\|$ for any~$V$ which minimises a bounded
$\cnt^0$~integrand~$F$ (not necessarily elliptic). Our proof follows the ideas
presented in~\cite[2.9(b2)(b3), 3.2(a)(b), 3.4(2) last paragraphs on pp.~347
and~348]{Alm68} as well as in~\cite[7.8, 8.2]{Fle66}.

Let $a \in U \subseteq \R^n$ be fixed, $\rho(x) = |x-a|$, and $M(r) =
\measureball{\|V\|}{\oball ar}$ for $r > 0$. The~key point is to prove the
differential inequalities~\eqref{eq:di:lower} and~\eqref{eq:di:upper}. Assume
$M'(r_0)$ exists and is finite for some $r_0 > 0$ -- this holds for
$\LM^1$~almost all~$r_0 \in (0,\infty)$. Recall that $V$ is a limit of some
sequence of the form $\{ \var{m}(S_i \cap U) : i \in \nat \}$, where $S_i$ belong
to a~good class. Our~strategy is to first choose a compact set $S \subseteq
\R^n$ such that~$\var{m}(S \cap U)$ is weakly close to~$V$, then construct an
admissible deformation~$D$ (using the deformation theorem~\ref{thm:deformation})
of~$S$ such that the $\HM^m(D\lIm S \rIm)$ can be estimated in terms
of~$M'(r_0)$, and finally use minimality of~$V$ to derive estimates on~$M(r_0)$.

More precisely we proceed as follows. We choose a compact set $S \subseteq \R^n$
so that the $4d$-neighbourhood of~$S \cap \rho^{-1}\{r_0\}$ has $\HM^m$~measure
controlled roughly by~$d M'(r_0)$, where $d > 0$ is a~small number. This is
possible because the \emph{mass function} $V \mapsto \|V\|(\R^n)$ is continuous
on the space of varifolds supported in a~fixed compact set. Then, we use the
deformation theorem~\ref{thm:deformation} to ``project'' the part of $S$ lying
inside $2d$-neighbourhood of~$\rho^{-1}\{r_0\}$ onto some $m$~dimensional
cubical complex and we denote the deformed set~$R$. After this step the part
of~$R$ lying in a~$2d$-neighbourhood of~$\rho^{-1}\{r_0\}$
is~$(\HM^m,m)$~rectifiable (as a finite sum of $m$~dimensional cubes) and,
moreover, its measure is still controlled by $d M'(r_0)$ due to the first part
of~\ref{thm:deformation}\ref{i:dt:estimates} which holds for non-rectifiable
sets. Next, we use the co-area formula (valid on rectifiable sets) together with
the Chebyshev inequality and~\ref{rem:leb-points} to find some $r_1 > 0$ in
the~$d$-neighbourhood of~$r_0$ so that the $\HM^{m-1}$~measure of the slice $R
\cap \rho^{-1}\{r_1\}$ is controlled by~$M'(r_0)$ and~$r_1$ is a~Lebesgue point
of the slicing operator $\langle \var{m}(R), \rho, \cdot \rangle$ and $R \cap
\rho^{-1}\{r_1\}$ is $(\HM^{m-1},m-1)$~rectifiable.

To prove~\ref{lem:drb:diff-ineq}\ref{i:di:lower} we cover the slice $R \cap
\rho^{-1}\{r_1\}$ with cubes of equal size $\varepsilon > 0$ and apply the
deformation theorem~\ref{thm:deformation} again to obtain a map $g_2 : I \times
\R^n \to \R^n$. We choose $\varepsilon$ so big that the whole slice $R \cap
\rho^{-1}\{r_1\}$ does not fill, after the deformation, a single $m$-dimensional
cube, which amounts to setting $\varepsilon \approx M'(r_0)^{1/(m-1)}$. This
ensures that $g_2(1,\cdot) \lIm R \cap \rho^{-1}\{r_1\} \rIm$ lies in some
$m-2$~dimensional cubical complex. Since $r_1$ is a Lebesgue point of $\langle
\var{m}(R), \rho, \cdot \rangle$ we may perform a blow-up of the slice
using~\ref{lem:slice-blow-up}. Then we make use of the smoothness of~$g_2$ to
argue that the push-forward $g_{2\#}$ is continuous on the space of varifolds so
that we can compose $g_2$ with the blow-up map~$K_{\delta}$ and pass to the
limit. Next, we estimate the $\HM^m$~measure of the blow-up~limit only by
\emph{one} term, namely the $\HM^{m}$~measure of the image of the whole
deformation of the slice, i.e., $\HM^m(g_2\lIm I \times R \cap \rho^{-1}\{r_1\}
\rIm)$. The other terms drop out because $g_2$ deformed our slice into an
$m-2$~dimensional set. Since $R \cap \rho^{-1}\{r_1\}$ is
$(\HM^{m-1},m-1)$~rectifiable we can use the second part
of~\ref{thm:deformation}\ref{i:dt:estimates} to estimate $\HM^m(g_2\lIm I \times
R \cap \rho^{-1}\{r_1\} \rIm)$ by $\varepsilon M'(r_0) \approx
M'(r_0)^{m/(m-1)}$. Finally, we make use of continuity of the mass to choose one
deformation from the blow-up sequence (without passing to the limit) for which
the desired estimate still holds.

To prove~\ref{lem:drb:diff-ineq}\ref{i:di:lower} we proceed similarly but this
time we choose $\varepsilon \approx \iota > 0$ arbitrarily, we deform the part
of~$R$ lying in the ball $\oball a{r_1}$ rather then just the slice $R \cap
\rho^{-1}\{r_1\}$, and we get two terms in the final estimate. The first term
corresponds to $\iota M'(r_0)$ analogously as before and the second one can be
estimated brutally by the $\HM^m$~measure of the sum of all $m$~dimensional
cubes from $\cubes_m$ touching the ball $\cball a{r_0 + d + 6 \iota \sqrt n}$.
Later, we use scaling to show that this second term depends only on $\iota$ and,
in~\ref{thm:ludrb}, we choose a specific~$\iota$ depending only on $n$, $m$, and
$F$.

We begin with a technical lemma. The estimates
in~\ref{lem:drb:aux}\ref{i:drb:lower-aux}\ref{i:drb:upper-aux} contain an
additional term which includes the parameter~$d$ and which shall be later
absorbed by other terms. The parameter $b$ below shall be set to $M'(r_0)$ in
most cases.

\begin{lemma}
    \label{lem:drb:aux}
    Let $S \subseteq \R^n$, and $a \in \R^n$, and $b,d \in (0,\infty)$, and $r_0
    \in (0,\infty)$. Set $\rho(x) = |x-a|$. Assume $\HM^m(\Clos{S} \cap
      \cball{a}{r_0 + 4d}) < \infty$ and
    \begin{displaymath}
        \label{eq:drb:slice-cond}
        \HM^m( S \cap \oball a{r_0 + 4d} \without \oball a{r_0 - 4d} ) < 9 b d \,.
    \end{displaymath}

    \begin{enumerate}
    \item 
        \label{i:drb:lower-aux}
        There exists a deformation $D \in \cnt^{\infty}(\R^n,\R^n)$ such that
        \begin{gather}
            \label{eq:drb:lower-adm}
            D \in \adm{ \oball a{r_0 + 4d + 40 \sqrt n (\Gamma_{\ref{thm:deformation}}^2 b)^{1/(m-1)}} } \,,
            \\
            \label{eq:drb:lower-est}
            \HM^m( D \lIm S \cap \oball a{r_0 + 4d} \rIm ) 
            \le 9 \Gamma_{\ref{thm:deformation}} b d + 50 \Gamma_{\ref{thm:deformation}}^{2m/(m-1)} b^{m/(m-1)}  \,.
        \end{gather}
    \item 
        \label{i:drb:upper-aux}
        Suppose $\iota \in (0,\infty)$ and $N \in \integers$ satisfy $2^{-N-1} <
        \iota \le 2^{-N}$. There exists a~deformation $F \in
        \cnt^{\infty}(\R^n,\R^n)$ such that setting
        \begin{equation}
            \label{eq:drb:Delta}
            \Delta(\rho,\iota,r_0,d) = 2^{-Nm}
            \HM^0\bigl(
            \bigl\{ K \in \cubes_{m}(N) 
            : K \cap \oball a{r_0 + d + 6 \iota \sqrt n} 
            \ne \varnothing \bigr\}\bigr)
        \end{equation}
        there holds
        \begin{gather}
            \label{eq:drb:upper-adm}
            F \in \adm{ \oball{a}{r_0 + 4d + 8  \iota \sqrt n} }
            \\
            \label{eq:drb:upper-est}
            \HM^m( F \lIm S \cap \oball a{r_0 + 4d} \rIm )
            \le 9 \Gamma_{\ref{thm:deformation}} b d 
            + 10 \Gamma_{\ref{thm:deformation}}^{2m/(m-1)} \iota b
            + \Delta(\rho,\iota,r_0,d) \,.
        \end{gather}
    \end{enumerate}
\end{lemma}

\begin{proof}
    For brevity define $A(d) = \oball{a}{r_0 + d} \without \oball{a}{r_0-d}$ for
    $d \in (0,\infty)$. Set $\varepsilon_1 = (5n)^{-1/2} d$ and find $N_1 \in
    \integers$ such that $2^{-N_1 - 1} < \varepsilon_1 \le 2^{-N_1}$. Define
    \begin{displaymath}
        \mathcal A_1 = \bigl\{
        Q \in \cubes_n(N_1) 
        : Q \cap A(2d) \ne \varnothing
        \bigr\} \,.
    \end{displaymath}
    Note that $\mathcal A_1$ is finite. Apply~\ref{thm:deformation} with $2$,
    $m$, $m$ $S$, $S \cap A(2d)$, $2^{-N_1-4}$, $\cubes_n(N_1)$, $\mathcal A_1$
    in place of $l$, $m_1$, $m_2$, $\Sigma_1$, $\Sigma_2$, $\varepsilon$,
    $\mathcal F$, $\mathcal A$ to obtain the~map $g_1 : I \times \R^n \to \R^n$
    called~``$f$'' there. Observe that
    \begin{gather*}
        \tbcup \mathcal A_1 + \cball 0{2^{-N_1-4}} \subseteq A(4d) 
        \quad \text{and} \quad
        S \cap A(2d) \subseteq A(2d) \subseteq \Int(\tbcup \mathcal A_1) \,.
    \end{gather*}
    In particular, it follows from~\ref{thm:deformation}\ref{i:dt:identity} that
    \begin{equation}
        \label{eq:drb:g1-id}
        g_1(t,x) = x \quad \text{whenever } \rho(x) \ge r_0 + 4d \text{ and } t \in I \,.
    \end{equation}

    Define $R = g_1(1,\cdot) \lIm S \rIm$ and note that $R \cap \Int(\tbcup
    \mathcal A_1)$ is a~finite sum of $m$~dimensional cubes; in~particular it is
    $(\HM^m,m)$~rectifiable. Observe also that if $x \in R \cap A(d)$, then
    there exists an $n$~dimensional cube $K \in \mathcal A_1$ such that $x \in
    K$ and there exists $y \in S$ such that $g_1(y) = x$ and $y \in K$ due
    to~\ref{thm:deformation}\ref{i:dt:cube-pres}. Hence, $|g(y) - y| \le 2
    \varepsilon_1 \sqrt n < d$ and, since $\Lip \rho \le 1$, we get $y \in
    A(2d)$. Therefore,
    \begin{equation}
        \label{eq:drb:d-strip-meas}
        R \cap A(d) \subseteq g \lIm S \cap A(2d) \rIm 
        \quad \text{and} \quad
        \HM^m(R \cap A(d)) < 9 \Gamma_{\ref{thm:deformation}} b d  \,,
    \end{equation}
    by~\ref{thm:deformation}\ref{i:dt:estimates} and~\eqref{eq:drb:slice-cond}.
    Since $R \cap A(d) \subseteq \Int(\bigcup \mathcal A_1)$ is
    $(\HM^m,m)$~rectifiable we may employ the co-area
    formula~\cite[3.2.22]{Fed69} together with $\Lip \rho \le 1$ to obtain
    \begin{displaymath}
        9 \Gamma_{\ref{thm:deformation}} b d 
        > \HM^m(R \cap A(d)) 
        \ge \int_{r_0-d}^{r_0+d} \HM^{m-1}(R \cap \rho^{-1}\{t\}) \ud \LM^1(t) \,.
    \end{displaymath}
    Thus, the Chebyshev inequality gives
    \begin{displaymath}
        \LM^{1}(\{ t \in (r_0 - d, r_0 + d) :
        \HM^{m-1}(R \cap \rho^{-1}\{t\}) \ge 5 \Gamma_{\ref{thm:deformation}} b \})
        \le \tfrac 9{10} 2d \,.
    \end{displaymath}
    Now we see that there exists $r_1 \in (r_0 - d, r_0 + d)$ such that
    \begin{equation}
        \label{eq:drb:r1-slice-meas}
        \HM^{m-1}(R \cap \rho^{-1} \{r_1\}) < 5 \Gamma_{\ref{thm:deformation}} b
    \end{equation}
    and $r_1$ is a Lebesgue point of the slicing operator $\langle \var{m}(R),
    \rho, \cdot \rangle$ (see~\ref{def:Lebesgue-set}) and $\langle \var{m}(R),
    \rho, r_1 \rangle \in \RVar{m-1}(\R^n)$ (see~\ref{rem:slice-rect}). For
    $\delta \in (0,r_1-r_0+d)$ let $K_{\delta} = K_{\rho,r_1,\delta} : \R^n \to
    I \times \R^n$ be defined as in~\ref{def:aux-fun}. Since $R \cap A(d)
    \subseteq \Int(\tbcup \mathcal A_1)$ is a finite sum of $m$~dimensional
    cubes we can apply~\ref{lem:slice-blow-up} to see that
    \begin{multline}
        \label{eq:drb:K-lim}
        \lim_{\delta \downarrow 0} K_{\delta\,\#} \var{m}(R)
        = i_{0\,\#} \var{m}( R \without \oball{a}{r_1} )
        + i_{1\,\#} \var{m}( R \cap \oball{a}{r_1} )
        \\
        + \var{1}([0,1]) \times \langle \var{m}(R), \rho, r_1 \rangle
        \in \Var{m}(\R \times \R^n) \,,
    \end{multline}
    where $i_{0}$ and $i_{1}$ are defined as in~\ref{def:aux-fun}.

    \emph{Proof of~\ref{i:drb:lower-aux}:} For brevity, if $K \in \cubes$, let
    us define $\widehat K$ to be the $n$~dimensional cube with the same centre
    as~$K$ and side length three times as long as $K$. Choose $\varepsilon_2
    \in (0,\infty)$ and $N_2 \in \integers$ so that
    \begin{equation}
        \label{eq:drb:eps2}
        \varepsilon_2^{m-1} 
        = \Gamma_{\ref{thm:deformation}} \HM^{m-1}(R \cap \rho^{-1}\{r_1\}) 
        < 5 \Gamma_{\ref{thm:deformation}}^2 b 
        \quad \text{and} \quad
        2^{-N_2-1} < \varepsilon_2 \le 2^{-N_2} \,.
    \end{equation}
    Define
    \begin{displaymath}
        B = R \cap \rho^{-1}\{r_1\}
        \quad \text{and} \quad
        \mathcal A_2 = \bigl\{ K \in \cubes_n(N_2) : \widehat K \cap \rho^{-1}\{r_1\} \ne \varnothing \bigr\} \,.
    \end{displaymath}
    Apply~\ref{thm:deformation} with $1$, $m-1$, $B$, $2^{-N_2-4}$,
    $\cubes_n(N_2)$, $\mathcal A_2$ in place of $l$, $m_1$, $\Sigma_1$,
    $\varepsilon$, $\mathcal F$, $\mathcal A$ to obtain the~map $g_2 : I \times
    \R^n \to \R^n$ called~``$f$'' there. We define $h : I \times \R^n \to \R^n$
    by setting
    \begin{gather*}
        h(t,x) = g_2(2t,x) \quad \text{for $t \in [0,1/2]$ and $x \in \R^n$} \,,
        \\
        h(t,x) = s_{1/100}(2 - 2t) g_2(1,x) \quad \text{for $t \in (1/2,1]$ and $x \in \R^n$} \,,
    \end{gather*}
    where $s_{1/100}$ is the function defined in~\ref{def:aux-fun}. Due to our
    choice of $\varepsilon_2$ we know,
    from~\ref{thm:deformation}\ref{i:dt:estimates}, that $\HM^{m-1}(g_2\lIm B
    \rIm) < \HM^{m-1}(K)$ for any $(m-1)$~dimensional cube $K \in
    \cubes_{m-1}(N_2)$. We see also that $g_2\lIm B \rIm \subseteq \Int(\tbcup
    \mathcal A_2)$ because $g_2\lIm B \rIm \subseteq \tbcup \{ K \in
    \cubes_n(N_2) : K \cap B \ne \varnothing \}$
    by~\ref{thm:deformation}\ref{i:dt:cube-pres}. Hence,
    by~\ref{thm:deformation}\ref{i:dt:full-faces},
    \begin{displaymath}
        g_2\lIm R \cap \rho^{-1} \{r_1\} \rIm 
        = g_2\lIm B \cap \tbcup \mathcal A_2 \rIm
        \subseteq \tbcup \cubes_{m-2}(N_2) \,,
    \end{displaymath}
    and we obtain
    \begin{gather}
        \label{eq:drb:h0-im}
        h(0,\cdot)_{\#} \var{m}( R \without \oball{a}{r_1} ) 
        = \var{m}( R \without  \oball{a}{r_1} ) \,,
        \\
        \label{eq:drb:h1-im}
        h(1,\cdot)_{\#} \var{m}( R \cap \oball{a}{r_1}) = 0 \,,
        \\
        \label{eq:drb:hI-im}
        h_{\#} \var{m}( I \times R \cap \rho^{-1}\{r_1\} ) 
        = g_{2\,\#} \var{m}( I \times R \cap \rho^{-1}\{r_1\} ) \,.
    \end{gather}
    Since $h$ is of class~$\cnt^{\infty}$ the push-forward $h_{\#}$ is
    continuous on~$\Var{m}(\R \times \R^n)$ so using~\eqref{eq:drb:K-lim}
    together with~\eqref{eq:drb:h0-im}, \eqref{eq:drb:h1-im},
    \eqref{eq:drb:hI-im}
    \begin{multline*}
        \lim_{\delta \downarrow 0} (h \circ K_{\delta})_{\#}
        \var{m}( R \cap \oball{a}{r_0 + 4d} ) 
        = \var{m}(R \cap \oball{a}{r_0 + 4d} \without \oball{a}{r_1} )
        \\
        + g_{2\,\#} \var{m}( I \times R \cap \rho^{-1}\{r_1\} ) \,.
    \end{multline*}
    Since $\rho$ is proper we can find a continuous function $\gamma : \R \times
    \R^n \to \R$ with compact support such that $I \times \oball{a}{r_0 + 4d}
    \subseteq \Int \gamma^{-1}\{1\}$ and use it as a~test function for the weak
    convergence. Thus, recalling that $\langle \var{m}(R), \rho, r_1 \rangle \in
    \RVar{m-1}(\R^n)$ and employing~\ref{thm:deformation}\ref{i:dt:estimates},
    \eqref{eq:drb:r1-slice-meas}, \eqref{eq:drb:d-strip-meas} we obtain
    \begin{multline}
        \label{eq:drb:lim-meas}
        \lim_{\delta \downarrow 0} \HM^m( (h \circ K_{\delta}) \lIm R \cap \oball{a}{r_0 + 4d} \rIm ) 
        \\
        \le \HM^m( R \cap \oball{a}{r_0+4d} \without \oball{a}{r_1} )
        + 2 \varepsilon_2 \Gamma_{\ref{thm:deformation}} \HM^{m-1}(R \cap \rho^{-1}\{r_1\})
        \\
        < 9 \Gamma_{\ref{thm:deformation}} b d + 
        10 \varepsilon_2 \Gamma_{\ref{thm:deformation}}^2 b
        \,.
    \end{multline}
    For $\delta \in (0,r_1 - r_0 + d)$ define $D_{\delta} = h \circ K_{\delta}
    \circ g_1(1,\cdot)$. Using~\eqref{eq:drb:g1-id}, \eqref{eq:drb:eps2},
    \eqref{eq:drb:lim-meas} we can find $\delta_0 \in (0,r_1 - r_0 + d)$ such
    that for all $\delta \in (0,\delta_0]$
    \begin{displaymath}
        \HM^m( D_{\delta} \lIm S \cap \oball{a}{r_0 + 4d} \rIm ) 
        \le 9 \Gamma_{\ref{thm:deformation}} b d + 50 \Gamma_{\ref{thm:deformation}}^{2m/(m-1)} b^{m/(m-1)} \,.
    \end{displaymath}
    Observe that
    \begin{displaymath}
        \tbcup \mathcal A_2 + \cball{0}{2^{-N_2-4}} \subseteq A\bigl(d + 8 \sqrt n (5 \Gamma_{\ref{thm:deformation}}^2 b)^{1/(m-1)}\bigr) \,;
    \end{displaymath}
    hence, $D_{\delta} \in \adm{\conv A(4d + 40 \sqrt n
      (\Gamma_{\ref{thm:deformation}}^2 b)^{1/(m-1)})}$ for $\delta \in
    (0,\delta_0)$ so setting $D = D_{\delta_0}$ finishes the proof
    of~\ref{i:drb:lower-aux}.

    \emph{Proof of~\ref{i:drb:upper-aux}:} Recall that if $K \in \cubes$, then
    $\widehat K$ denotes the $n$~dimensional cube with the same centre as~$K$
    and side length three times as long as $K$. Let $\iota \in (0,\infty)$ and
    $N \in \integers$ satisfy $2^{-N-1} < \iota \le 2^{-N}$. Set
    \begin{displaymath}
        C = R \cap \oball{a}{r_1}
        \quad \text{and} \quad
        \mathcal A_3 = \bigl\{ K \in \cubes_n(N) 
        : \widehat K \cap \oball{a}{r_1} \ne \varnothing \bigr\} \,.
    \end{displaymath}
    Apply~\ref{thm:deformation} with $2$, $m$, $m-1$, $C$, $B$, $2^{-N-4}$,
    $\cubes_n(N)$, $\mathcal A_3$ in place of $l$, $m_1$, $m_2$, $\Sigma_1$,
    $\Sigma_2$, $\varepsilon$, $\mathcal F$, $\mathcal A$ to obtain the~map $g_3
    : I \times \R^n \to \R^n$ called ``$g$'' there. Since $C \subseteq \tbcup
    \bigl\{ K \in \cubes_n(N) : K \cap \oball{a}{r_1} \ne \varnothing \bigr\}$
    we see that $g_3(1,\cdot) \lIm C \rIm \subseteq \Int(\tbcup \mathcal A_3)$,
    by~\ref{thm:deformation}\ref{i:dt:cube-pres}, and conclude
    from~\ref{thm:deformation}\ref{i:dt:m-dim-im} that
    \begin{equation}
        \label{eq:drb:g3-im-meas}
        \HM^m(g_3(1,\cdot) \lIm C \rIm) 
        \le \HM^m(\tbcup \cubes_{m}(N) \cap \tbcup \mathcal A_3) 
        \le \Delta(\rho,\iota,r_0,d) \,.
    \end{equation}
    Therefore, using~\eqref{eq:drb:K-lim}
    and~\ref{thm:deformation}\ref{i:dt:estimates}, \eqref{eq:drb:d-strip-meas},
    \eqref{eq:drb:r1-slice-meas}, \eqref{eq:drb:g3-im-meas} we get
    \begin{multline}
        \label{eq:drb:lim-meas2}
        \lim_{\delta \downarrow 0} \HM^m( (g_3 \circ K_{\delta}) \lIm R \cap \oball{a}{r_0 + 4d} \rIm ) 
        \\
        \le \HM^m( R \cap \oball{a}{r_0+4d} \without \oball{a}{r_1})
        + \HM^m(g_3 \lIm I \times B  \rIm) + \HM^m(g_3(1,\cdot) \lIm C \rIm) 
        \\
        \le 9 \Gamma_{\ref{thm:deformation}} b d + 
        10 \iota \Gamma_{\ref{thm:deformation}}^2 b
        + \Delta(\rho,\iota,r_0,d) \,.
    \end{multline}
    Hence, there exists $\delta_0 \in (0,r_1 - r_0 + d)$ such that $F = g_3
    \circ K_{\delta_0} \circ g_1(1,\cdot)$ satisfies the estimate claimed
    in~\ref{i:drb:upper-aux}. Moreover, we see that
    \begin{displaymath}
        \tbcup \mathcal A_3 + \cball{0}{2^{-N-4}} 
        \subseteq \oball{a}{r_0 + d + 8 \iota \sqrt n} \,;
    \end{displaymath}
    hence, $F \in \adm{ \oball{a}{r_0 + 4d + 8 \iota \sqrt n} }$.
\end{proof}

Now, we can prove the pivotal differential inequalities~\eqref{eq:di:lower}
and~\eqref{eq:di:upper}. There is one technical difficulty that needs to be
taken care of. To be able to employ minimality of~$V$ we need our deformations
to be admissible in an open set~$U \subseteq \R^n$. In~particular, in the proof
of~\eqref{eq:di:lower} we need to perform a~deformation onto cubes of side
length roughly~$M'(r_0)^{1/(m-1)}$ which might be arbitrarily big. This turns
out not to be a problem since big values of the derivative~$M'$ cannot spoil the
lower bound on~$M$. More precisely, we will later use the upper bound on~$M$
proven in~\ref{thm:ludrb}\ref{i:lu:upper} to overcome this difficulty. For the
time being we just assume in~\ref{lem:drb:diff-ineq}\ref{i:di:lower} that some
upper bound on~$M$ holds.

\begin{lemma}
    \label{lem:drb:diff-ineq}
    Assume
    \begin{gather}
        U \subseteq \R^n \text{ is open} \,,
        \quad
        a \in U \,,
        \quad
        \mathcal C \text{ is a good class in $U$} \,,
        \quad
        \{ S_i : i \in \nat \} \subseteq \mathcal C \,,
        \\
        \rho(x) = |x-a| \,,
        \quad 
        r_0, \iota \in (0,\infty) \,,
        \quad
        V = \lim_{i \to \infty} \var{m}(S_i \cap U) \in \Var{m}(U) \,,
        \\
        \label{eq:drb:V-minimal}
        \text{$F$~is a $\cnt^0$ integrand} \,,
        \quad
        \Phi_{F}(V) = \lim_{i \to \infty} \Phi_{F}(S_i \cap U) 
        = \inf\bigl\{ \Phi_{F}(T \cap U) : T \in \mathcal C \bigr\} \,,
        \\
        M(t) = \measureball{\|V\|}{\oball at} \quad \text{for $t \in \R$} \,,
        \quad
        M'(r_0) \text{ exists and is finite} \,,
        \\
        \alpha = \inf F \lIm \cball a{\dist(a,\R^n \without U)} \rIm > 0\,,
        \quad
        \beta = \sup F \lIm \cball a{\dist(a,\R^n \without U)} \rIm < \infty \,,
        \\
        \label{eq:drb:Gamma}
        \Gamma = \Gamma(n,m,F,U,a) = 240 \Gamma_{\ref{thm:deformation}}^{2m/(m-1)} \beta/\alpha \,.
    \end{gather}
    \begin{enumerate}
    \item 
        \label{i:di:lower}
        If $M(r_0) \le \gamma r_0^m$ or $M'(r_0) \le (\gamma/\Gamma)^{1-1/m}
        r_0^{m-1}$ for some $\gamma \in (0,\infty)$,
        \begin{equation}
            \label{eq:di:kappa-low}
            \kappa = \kappa(n,m,\Gamma,\gamma)
            = 1 + (\gamma/\Gamma)^{1/m} \bigl(
            4 \Gamma_{\ref{thm:deformation}}^{(m+1)/(m-1)}
            + 40 \sqrt n \Gamma_{\ref{thm:deformation}}^{2/(m-1)}
            \bigr) \,,
        \end{equation}
        and $\cball a{\kappa r_0} \subseteq U$, then
        \begin{equation}
            \label{eq:di:lower}
            M(r_0) \le \Gamma M'(r_0)^{m/(m-1)} \,.
        \end{equation}

    \item 
        \label{i:di:upper}
        There exists $\gamma = \gamma(n,m,\iota) \in (1,\infty)$ such that setting
        \begin{equation}
            \label{eq:di:kappa-up}
            \kappa = \kappa(n,m,\iota)
            = 1 + \iota \bigl(4 \Gamma_{\ref{thm:deformation}}^{(m+1)/(m-1)} + 8 \sqrt n \bigr)
        \end{equation}
        if $\cball a{\kappa r_0} \subseteq U$, then
        \begin{equation}
            \label{eq:di:upper}
            \frac{M(r_0)}{r_0^m} \le \gamma + \Gamma \iota \frac{M'(r_0)}{r_0^{m-1}} \,.
        \end{equation}
    \end{enumerate}
\end{lemma}

\begin{proof}
    Define $A(d) = \{ x \in \R^n : r_0 - d \le \rho(x) < r_0 + d\}$ for $d \in (0,\infty)$.
    Choose $b,d \in (0,\infty)$ so that
    \begin{gather}
        M'(r_0) \le b \,,
        \quad
        [r_0 - 4d, r_0 + 4d] \subseteq (0,\infty) \,,
        \\
        \label{eq:drb:VA4d}
        \|V\|(A(4d)) = M(r_0 + 4d) - M(r_0 - 4d) < 9 b d  \,,
        \quad
        \|V\|(\Bdry A(4d)) = 0 \,.
        \\
        \label{eq:drb:d-absorb}
        d < \Gamma_{\ref{thm:deformation}}^{(m+1)/(m-1)} \min \bigl\{ b^{1/(m-1)} ,\, \iota r_0  \bigl\} \,.
    \end{gather}
    It follows from~\eqref{eq:drb:VA4d} that $\lim_{i \to \infty}
    \|\var{m}(S_i)\|(A(4d)) = \|V\|(A(4d))$; hence,
    recalling~\eqref{eq:drb:V-minimal}, we~see that there exists $S \in \{ S_i :
    i \in \nat \}$ such that
    \begin{gather}
        \label{eq:drb:S-approx}
        0 \le \Phi_{F}(S \cap U) - \Phi_{F}(V) < \tfrac 14 \alpha M(r_0) \,,
        \quad
        \HM^m(S \cap A(4d)) < 9 b d  \,,
        \\
        \label{eq:drb:Mr0S}
        M(r_0) < 2 \| \var{m}(S) \| (\{x \in \R^n : \rho(x) < r_0 \})
        \,.
    \end{gather}

    \emph{Proof of~\ref{i:di:lower}:} If $M(r_0) \le \gamma r_0^m$ and $M'(r_0)
    > (\gamma/\Gamma)^{1-1/m} r_0^{m-1}$, then~\eqref{eq:di:lower} follows
    trivially. Thus, we may and shall assume that $M'(r_0) <
    (\gamma/\Gamma)^{1-1/m} r_0^{m-1}$. Suppose also $b \le
    (\gamma/\Gamma)^{1-1/m} r_0^{m-1}$ and define $\kappa$ by~\eqref{eq:di:kappa-low}.

    Now, recalling~\eqref{eq:drb:lower-adm} and~\eqref{eq:drb:d-absorb}, we
    apply~\ref{lem:drb:aux}\ref{i:drb:lower-aux} together
    with~\eqref{eq:drb:d-absorb} to obtain the~deformation $D \in \adm{\oball
      a{\kappa r_0}}$ such that
    \begin{displaymath}
        \HM^m( D \lIm \{ x \in S : \rho(x) < r_0 + 4d \} \rIm ) 
        \le 59 \Gamma_{\ref{thm:deformation}}^{2m/(m-1)} b^{m/(m-1)} \,.
    \end{displaymath}
    Since $D \in \adm{U}$ we have $\Phi_{F}(V) \le \Phi_{F}(D \lIm S \rIm \cap
    U)$. Using~\eqref{eq:drb:Mr0S} and~\eqref{eq:drb:S-approx}, and noting that
    $D(x) = x$ whenever $\rho(x) \ge r_0 + 4d$ we see that
    \begin{multline}
        \label{eq:drb:comp}
        \tfrac 14 \alpha M(r_0)
        \le \alpha \HM^m(\{ x \in S : \rho(x) < r_0 + 4d \}) - \tfrac 14 \alpha M(r_0)
        \\
        \le \Phi_{F}(\{ x \in S : \rho(x) < r_0 + 4d \}) + (\Phi_{F}(V) - \Phi_{F}(S \cap U))
        \\
        = \Phi_{F}(V) - \Phi_{F}(D\lIm \{ x \in S \cap U : \rho(x) \ge r_0 + 4d \} \rIm)
        \\
        \le \Phi_{F}(D\lIm \{ x \in S : \rho(x) < r_0 + 4d \} \rIm) 
        \le 59 \beta \Gamma_{\ref{thm:deformation}}^{2m/(m-1)} b^{m/(m-1)} \,.
    \end{multline}
    Recalling~\eqref{eq:drb:Gamma}, the definition of $\Gamma$, we see that
    \begin{displaymath}
        M(r_0) \le \Gamma b^{m/(m-1)} \,.
    \end{displaymath}
    Clearly $M$ is non-decreasing so $M'(r_0) \ge 0$. If $M'(r_0) > 0$, then the
    proof of~\ref{i:di:lower} is finished by setting $b = M'(r_0)$. If~$M'(r_0)
    = 0$, then we may choose $b > 0$ arbitrarily small to obtain $M(r_0) = 0 \le
    \Gamma M'(r_0) = 0$.

    \emph{Proof of~\ref{i:di:upper}:} From~\ref{i:di:lower} we already know that
    if $M'(r_0) = 0$, then $M(r) = 0$ so we may assume $M'(r_0) > 0$ and set $b
    = M'(r_0)$. Define
    \begin{gather*}
        \overline S = \scale{1/r_0} \circ \trans{-a} \lIm S \rIm \,,
        \quad
        \overline M(s) = \|(\scale{1/r_0} \circ \trans{-a})_{\#}V\|(\oball 0s) = r_0^{-m} M(sr_0)
        \text{ for $s \in (0,\infty)$} \,,
        \\
        \overline \rho(x) = \rho \circ \trans{a}(x) = |x| \text{ for $x \in \R^n$} \,,
        \quad
        \overline b = \overline M'(1) = \frac{M'(r_0)}{r_0^{m-1}} \,,
        \quad
        \overline d = \frac{d}{r_0} < \Gamma_{\ref{thm:deformation}}^{(m+1)/(m-1)} \iota \,.
    \end{gather*}
    Apply~\ref{lem:drb:aux}\ref{i:drb:upper-aux} with $\overline S$, $\overline
    d$, $\overline \rho$, $\overline b$, $\iota$ in place of~$S$, $d$, $\rho$,
    $b$, $\iota$ to obtain the~deformation $F \in \adm{\oball 0{\kappa}}$, where
    $\kappa = \kappa(n,m,\iota)$ is defined by~\eqref{eq:di:kappa-up}.
    Combining~\eqref{eq:drb:d-absorb} and~\eqref{eq:drb:upper-est} we see that
    \begin{displaymath}
        \HM^m( F \lIm \overline S \cap \oball 0{1 + 4 \overline d} \rIm ) 
        \le 19 \Gamma_{\ref{thm:deformation}}^{2m/(m-1)} \iota \overline M'(1)
        + \Delta(\iota) \,,
    \end{displaymath}
    where $\Delta(\iota) = \Delta(\overline \rho,\iota,1,\overline d)$ is
    defined by~\eqref{eq:drb:Delta}. Set $L = \trans{a} \circ \scale{r_0} \circ
    F \circ \scale{1/r_0} \circ \trans{-a}$. Since $\cball a{\kappa r_0}
    \subseteq U$ we have $L \in \adm{U}$ so $\Phi_{F}(V) \le \Phi_{F}(L\lIm S
    \rIm)$ and we can compute as in~\eqref{eq:drb:comp} 
    \begin{multline*}
        \frac{\alpha}{4\beta} M(r_0) \le 
        \HM^m(L \lIm S \cap \oball{a}{r_0(1 + 4 \overline d)}\rIm)
        \\
        = r_0^m \HM^m(F \lIm \overline S \cap \oball{0}{1 + 4 \overline d}\rIm)
        \le r_0^m \bigl(
        19 \Gamma_{\ref{thm:deformation}}^{2m/(m-1)} \iota \overline M'(1) + \Delta(\iota) \bigr)
        \\
        = 19 \Gamma_{\ref{thm:deformation}}^{2m/(m-1)} \iota r_0 M'(r_0) + r_0^m \Delta(\iota) \,.
    \end{multline*}
    Hence, we may set $\gamma = 4 \beta / \alpha \Delta(\iota)$.
\end{proof}

\begin{theorem}
    \label{thm:ludrb}
    Assume
    \begin{gather*}
        U \subseteq \R^n \text{ is open} \,,
        \quad
        \mathcal C \text{ is a good class in $U$} \,,
        \quad
        \{ S_i : i \in \nat \} \subseteq \mathcal C \,,
        \\
        V = \lim_{i \to \infty} \var{m}(S_i \cap U) \in \Var{m}(U) \,,
        \quad
        \text{$F$~is a bounded $\cnt^0$ integrand} \,,
        \\
        \Phi_{F}(V) = \lim_{i \to \infty} \Phi_{F}(S_i \cap U) 
        = \inf\bigl\{ \Phi_{F}(T \cap U) : T \in \mathcal C \bigr\} < \infty \,,
        \\
        \Delta = \sup \bigl\{ \Gamma_{\ref{lem:drb:diff-ineq}}(n,m,F,U,x) : x \in \spt \|V\| \bigr\} \,,
        \quad
        \iota = (2 m \Delta)^{-1} \,,
        \\
        \kappa = \kappa(n,m,F,V,U) = \kappa_{\ref{lem:drb:diff-ineq}\ref{i:di:upper}}(n,m,\iota) \,,
        \quad
        a \in \spt \|V\| \subseteq U \,,
        \quad
        r_0 = \dist(a,\R^n \without U) / \kappa \,.
    \end{gather*}
    Then $\Delta, \iota \in (0,\infty)$ and the following statements hold.
    \begin{enumerate}
    \item 
        \label{i:lu:upper}
        There exists $\Gamma = \Gamma(n,m,F,V,U) \in (0,\infty)$ such that for
        all $r \in (0,r_0)$
        \begin{displaymath}
            r^{-m} \measureball{\|V\|}{\oball ar}
            \le \max\bigl\{ \Gamma,  r_0^{-m} \measureball{\|V\|}{\oball a{r_0}} \bigr\} \,.
        \end{displaymath}
    \item
        \label{i:lu:lower}
        Define $\lambda = \lambda(n,m,F,V,U,a) = \max \{ \kappa,
        \kappa_{\ref{lem:drb:diff-ineq}\ref{i:di:lower}}(n,m,\Gamma_{\ref{lem:drb:diff-ineq}}(n,m,F,U,a),\gamma)
        \}$, where
        \begin{displaymath}
            \gamma = \gamma(n,m,F,V,U,a)
            = \max\bigl\{
            \Gamma_{\ref{thm:ludrb}\ref{i:lu:upper}}(n,m,F,V,U) ,\,
            r_0^{-m} \measureball{\|V\|}{\oball a{r_0}}
            \bigr\} \,.
        \end{displaymath}
        For all $r \in (0,\dist(a,\R^n \without U) / \lambda)$ we~have
        \begin{displaymath}
            r^{-m} \measureball{\|V\|}{\oball ar} \ge m^{-m} \Gamma_{\ref{lem:drb:diff-ineq}}(n,m,F,U,a)^{1-m} \,.
        \end{displaymath}
    \end{enumerate}
\end{theorem}

\begin{proof}
    Since $F$ is bounded, it attains its supremum and infimum. Thus, recalling
    the definition of~$\Gamma_{\ref{lem:drb:diff-ineq}}(n,m,F,U,\cdot)$ we see
    that $0 < \Delta < \infty$.

    \emph{Proof of~\ref{i:lu:upper}:} Let $a \in \spt \|V\|$ and $r \in
    (0,r_0)$, where $r_0 = \dist(a,\R^n \without U)/\kappa$. Set $M(s) =
    \measureball{\|V\|}{\oball as}$ for $s \in (0,\infty)$.  Define $\gamma =
    \gamma_{\ref{lem:drb:diff-ineq}\ref{i:di:upper}}(n,m,\iota)$ and $\Gamma =
    2\gamma/\unitmeasure{m}$. For~each $s \in (0,r_0)$ for which $M'(s)$ exists
    and is finite we may apply~\ref{lem:drb:diff-ineq}\ref{i:di:upper} to see
    that
    \begin{equation}
        \label{eq:lu:di}
        s^{-m} M(s)
        \le \gamma + \Delta \iota s^{-(m-1)} M'(s)
        \,.
    \end{equation}
    Now we proceed as in~\cite[8.2]{Fle66}. Choose $\eta > \Gamma$ and assume
    there exists $r_1 \in (0,r_0)$ satisfying $M(r_1) > \eta \unitmeasure{m}
    r_1^m$. Let $r_2 \in [r_1,r_0]$ be the largest number in $[r_1,r_0]$ such
    that $M(s) \ge \eta \unitmeasure{m} s^m$ for $s \in [r_1,r_2]$. Since $M$ is
    non-decreasing we see immediately that $r_2 > r_1$. Using~\eqref{eq:lu:di}
    and the definitions of~$\iota$ and~$\Gamma$, we obtain for $\LM^1$~almost
    all $s \in [r_1,r_2]$
    \begin{displaymath}
        M(s) \le s^m \gamma + \Delta \iota s M'(s)
        < \tfrac 12 M(s) + \tfrac 1{2m} s M'(s) \,.
    \end{displaymath}
    Hence, $m M(s) < s M'(s)$ for $\LM^1$~almost all $s \in [r_1,r_2]$ which
    implies that
    \begin{displaymath}
        \bigl( s^{-m} M(s) \bigr)' > 0
        \quad \text{for $\LM^1$ almost all $s \in [r_1,r_2]$} \,.
    \end{displaymath}
    Using~\cite[2.9.19]{Fed69} for each $s_1,s_2 \in [r_1,r_2]$ with $s_1 <
    s_2$ we obtain
    \begin{displaymath}
        0 < \int_{s_1}^{s_2} \bigl( t^{-m} M(t) \bigr)' \ud \LM^1(t) 
        \le s_2^{-m} M(s_2) - s_1^{-m} M(s_1) \,,
    \end{displaymath}
    which shows that $s^{-m} M(s)$ is increasing for~$s \in [r_1,r_2]$; thus,
    $r_2 = r_0$. Since $\eta > \Gamma$ could be arbitrary the claim is proven.

    \emph{Proof of~\ref{i:lu:lower}:} For each $s \in (0,\dist(a,\R^n \without
    U)/\lambda)$ for which $M'(s)$ exists and is finite we may
    apply~\ref{lem:drb:diff-ineq}\ref{i:di:lower} to see that
    \begin{displaymath}
        (M^{1/m})'(s) \ge m^{-1} \Gamma_{\ref{lem:drb:diff-ineq}}^{(1-m)/m} \,.
    \end{displaymath}
    Employing~\cite[2.9.19]{Fed69} we find out that $M'(s)$ exists and is
    finite for $\LM^1$~almost all $s \in (0,r_0)$ and that
    \begin{displaymath}
        \bigl(\measureball{\|V\|}{\oball ar}\bigr)^{1/m}
        = M(r)^{1/m} \ge \int_{0}^r (M^{1/m})'(s) \ud \LM^1(s)
        \ge m^{-1} \Gamma_{\ref{lem:drb:diff-ineq}}^{(1-m)/m} r \,.
        \qedhere
    \end{displaymath}
\end{proof}

\begin{corollary}
    \label{cor:ludrb}
    Let $F$, $V$, and~$U$ be as in~\ref{thm:ludrb} and $\delta > 0$. There
    exist $\Gamma = \Gamma(n,m,F,V,U,\delta) > 1$ and $\kappa =
    \kappa(n,m,F,V,U,\delta) > 1$ such that for all $x \in \spt \|V\| \subseteq U$
    and $r \in (0,\infty)$ satisfying $r < \dist(x,\R^n \without U)/\kappa$ and
    $\dist(x,\R^n \without U) > \delta$ there holds
    \begin{displaymath}
        \Gamma^{-1} r^m \le \measureball{\|V\|}{\cball xr} \le \Gamma r^m \,.
    \end{displaymath}
    In particular, for all $x \in \spt \|V\| \cap E$ we have
    \begin{displaymath}
        0 < \density_*^m(\|V\|,x) \le \density^{*m}(\|V\|,x) < \infty \,.
    \end{displaymath}
    Using~\cite[2.10.19(1)(3), 2.1.3(5)]{Fed69} and Borel regularity of the
    Hausdorff measure~\cite[2.10.2(1)]{Fed69} we further deduce that there
    exists $C = C(n,m,F,V,U,\delta) > 1$ such that for any Borel set $A
    \subseteq \{ x \in U : \dist(x, \R^n \without U) > \delta \}$ we have
    \begin{displaymath}
        C^{-1} \HM^m(A \cap \spt \|V\|) \le \|V\|(A) \le C \HM^m(A \cap \spt \|V\|) \,.
    \end{displaymath}
\end{corollary}

% Local Variables:
% coding: utf-8
% eval: (ispell-change-dictionary "british")
% eval: (flyspell-mode)
% End:

%% file: rectifiability.tex
\section{Rectifiability of the support of the limit varifold}
\label{sec:rectifiability}

In~\ref{thm:rect} we prove that the support of a~$\Phi_F$~minimising
varifold~$V$ must be $(\HM^m,m)$~rectifiable inside any compact set~$K \subseteq
U$. Using the density ratio bounds~\ref{cor:ludrb} we also conclude,
in~\ref{rem:tangent-cones}, that the approximate tangent cones of~$\|V\|$
coincide with the classical tangent cones of $\spt\|V\|$ for \emph{all} points
$x \in \spt\|V\| \subseteq U$. In consequence, the cones $\Tan(\spt\|V\|,x)$ are
in~fact $m$-planes for $\HM^m$~almost all $x \in \spt\|V\|$.

In the proof of~\ref{thm:rect} we follow the guidelines presented
in~\cite[2.9(b4), p.~341]{Alm68}. We use only boundedness of~$F$ and make
\emph{no use} of ellipticity of~$F$. The proof is done by contradiction. We
assume that $\spt\|V\|$ is not countably $(\HM^m,m)$~rectifiable and we look at
a density point~$x_0 \in U$ of the unrectifiable part of $\spt\|V\|$. We choose
a scale $\rho_1 > 0$ so that the $\HM^m$~measure of the rectifiable part of
$\spt\|V\| \cap \cball{x_0}{\rho_1}$ is negligible in comparison to the
$\HM^m$~measure of the unrectifiable part. Then we use the deformation
theorem~\ref{thm:deformation} to produce a smooth map~$\phi : \R^n \to \R^n$
which deforms~$\spt\|V\| \cap \cball{x_0}{\rho_1}$ onto an~$m$-dimensional
skeleton of some cubical complex. Next, we apply a~perturbation
argument~\ref{lem:reduce-unrect} to obtain a map~$g$ which almost kills the
$\HM^m$~measure of the unrectifiable part of~$\spt\|V\|$ keeping the
$\HM^m$~measure of the rectifiable part negligible.

At this point we know that $\HM^m(g \lIm \spt\|V\| \cap \cball{x_0}{\rho_1}
\rIm)$ is significantly smaller than $\HM^m(\spt\|V\| \cap \cball{x_0}{\rho_1})$
and we want to contradict minimality of~$V$ but we do not know whether
$\var{m}(\spt\|V\|) = V$, i.e., whether $\Phi_F(V) = \Phi_F(\spt\|V\|)$. Thus,
we~look at $g \lIm S_i \rIm$, where $S_i \in \mathcal C$ is an appropriate
minimising sequence (we need to assume it converges in the Hausdorff metric to
some compact set~$S \subseteq \R^n$ such that $\HM^m(S \cap U \without \spt\|V\|) = 0$;
see~\ref{cor:ud:good-ms}). To compare the measures of $g \lIm \spt\|V\| \cap
\cball{x_0}{\rho_1} \rIm$ and $g\lIm S_i \cap \cball{x_0}{\rho_1} \rIm$ we make
use of a~simple observation: these two sets both lie in the $m$-dimensional
skeleton of a fixed cubical complex and are close in Hausdorff metric so their
$\HM^m$~measures must also be close; see~\eqref{eq:cheap-trick}.

\begin{theorem}
    \label{thm:rect}
	Assume
    \begin{gather}
        U \subseteq \R^n \text{ is open} \,,
        \quad
        \mathcal{C} \text{ is a good class in $U$} \,,
        \quad
        \{ S_i: i \in \nat \} \subseteq \mathcal{C} \,,
        \quad
        S \in \mathcal{C} \,,
        \\
        V = \lim_{i \to \infty} \var{m}(S_i \cap U) \in \Var{m}(U) \,,
        \quad
        F \text{ is a bounded $\cnt^0$ integrand} \,,
        \\
        \lim_{i \to \infty} \HDK{K}(S_i \cap U, S \cap U) = 0 
        \quad \text{for any compact set $K \subseteq U$} \,,
        \\
        \HM^m(S \cap U \without \spt \|V\|) = 0 \,,
        \quad
        \Phi_F(V) = \lim_{i\to \infty} \Phi_F(S_i \cap U) 
        = \inf\{ \Phi_F(R \cap U) : R \in \mathcal{C} \} < \infty \,.
    \end{gather}
    Then $\HM^m(\spt\|V\| \cap K) < \infty$ for any compact set $K \subseteq U$
    and $\spt \|V\|$ is a~countably $(\HM^m,m)$~rectifiable subset of~$U$.
\end{theorem}

\begin{proof}
    For $\delta > 0$ set $U_{\delta} = \{ x \in U : \dist(x,\R^n \without U) >
    \delta \}$. Note that for each $\delta > 0$, by~\ref{cor:ludrb}, there
    exists some number $C(\delta) > 1$ such that 
    \begin{gather}
        C(\delta)^{-1} \HM^m \restrict (\spt \|V\| \cap U_{\delta})
        \le \|V\| \restrict U_\delta 
        \le C(\delta) \HM^m \restrict (\spt \|V\| \cap U_{\delta}) \,;
        \\
        \text{hence,} \quad
        \infty > \Phi_F(V) 
        \ge \inf \im F \|V\|(U_{\delta})
        \ge C(\delta)^{-1} \inf \im F \HM^m(\spt \|V\| \cap U_{\delta}) \,.
    \end{gather}
    This proves the first part of~\ref{thm:rect}. We shall prove the second part
    by contradiction.

    Assume that $\spt\|V\|$~is not countably $(\HM^m,m)$~rectifiable.  Then
    there exists $\delta > 0$ such that $E = \spt\|V\| \cap U_{\delta}$ is not
    countably $(\HM^m,m)$~rectifiable. We decompose~$E$ into a disjoint sum $E =
    E_r \cup E_u$, where $E_r$ is $(\HM^m,m)$~rectifiable, and $E_u$ is purely
    $(\HM^m,m)$~unrectifiable, and~$\HM^m(E_u) > 0$.
    Employing~\cite[2.10.19(2)(4)]{Fed69} we choose $x_0 \in E_u$ such that
	\begin{equation}
        \label{eq:bad-point-dens}
		\density^{m}(\HM^m \restrict E_r, x_0) = 0
        \quad \text{and} \quad
        \density^{*m}(\HM^m \restrict E_u,x_0) = b > 0 \,.
	\end{equation}
	Since $F$ is bounded, there exist $0 < C_1 < C_2 < \infty$ such that
	\begin{equation}
        \label{eq:F-bounded}
		C_1 \le F(x,T) \le C_2
        \quad \text{for $(x,T) \in \R^n \times \grass nm$} \,.
	\end{equation}
    From~\ref{cor:ludrb} we see that there exist numbers $0 < C_3 < C_4 <
    \infty$ (depending on $\delta$) such that
	\begin{equation}
        \label{eq:V-bounds}
		C_3 \HM^m(A) \le \|V\|(A) \le C_4 \HM^m(A) 
        \quad \text{for any Borel set $A \subset \spt \|V\| \cap U_{\delta}$} \,.
	\end{equation}
	Next, we fix a small number $\varepsilon>0$ such that
	\begin{equation}
        \label{eq:eps-choice}
        \frac{C_2\bigl( \varepsilon (1+\varepsilon) 
          + \varepsilon (4^m \Gamma_{\ref{thm:deformation}} + \varepsilon) (C_4 + 2)  \bigr)}
        {(1-\varepsilon)^2 C_1 C_3} <  1 \,.
	\end{equation}
    Since $\HM^m(E) < \infty$, we see that $\HM^m \restrict E$ is a Radon measure; hence,
    \begin{equation}
        \label{eq:slice-meas}
		\HM^{m}(E \cap \Bdry \cball{x_0}{\rho} ) > 0
        \quad \text{for at most countably many $\rho > 0$} \,.
    \end{equation}
    Employing~\eqref{eq:bad-point-dens} and~\eqref{eq:slice-meas} we choose $0 <
    \rho_1 < \rho_3 < \rho_2$ such that
	\begin{gather}
        \label{eq:rho-choice}
		\rho_3 = (\rho_1 + \rho_2)/2 \,,
        \quad
        \rho_1 \ge (1 - \varepsilon)^{1/m} \rho_2 \,,
        \quad
        \cball{x_0}{\rho_2} \subseteq U_{\delta} \,,
        \\
        \label{eq:rho-slice}
		\HM^{m}(E \cap \Bdry \cball{x_0}{\rho_i} ) = 0 \quad \text{for $i = 1, 2$} \,,
        \\
        \label{eq:rho-meas}
		(1 - \varepsilon) b \rho_1^m
        \le \HM^m(E_u \cap \cball{x_0}{\rho_1}) 
        \le (1 + \varepsilon) b \rho_1^m 
        \\
        \label{eq:unrect5}
		\HM^m(E_r \cap \cball{x_0}{\rho_1} ) \le \varepsilon b \rho_1^m \,,
        \quad
		\HM^m \bigl( E\cap \cball{x_0}{\rho_2} \without \oball{x_0}{\rho_1} \bigr)
        \le \varepsilon b \rho_2^m.
	\end{gather}
    Choose $k \in \nat$ such that $2^{-k} \le (\rho_2 - \rho_1) / 16 < 2^{-k+1}$.
    Define
    \begin{gather}
        \widetilde Q = \tbcup \bigl\{ R \in \cubes_n(k) : R \cap Q \ne \varnothing \bigr\} 
        \quad \text{for $Q \in \cubes_n(k)$} \,,
        \\
        \mathcal{A} = \bigl\{ Q \in \cubes_n(k) : \widetilde Q \cap \cball{x_0}{\rho_3} \ne \varnothing \bigr\} \,,
        \\
        \Sigma_1 = E_r \cap \cball{x_0}{\rho_1} \,,
        \quad
        \Sigma_2 = S \,,
        \quad
        \Sigma_3 = E \cap \cball{x_0}{\rho_2} \without \oball{x_0}{\rho_1} \,.
    \end{gather}
    Then
    \begin{displaymath}
        \cball{x_0}{\rho_3}
        \subseteq \tbcup \bigl\{ Q \in \mathcal{A} : Q \subseteq \Int \tbcup \mathcal A \bigr\}
        \subseteq \tbcup \mathcal{A} \subseteq \oball{x_0}{\rho_2} \,.
    \end{displaymath}
    Next, apply~\ref{thm:deformation} with $\cubes_n(k)$, $\mathcal{A}$,
    $\Sigma_1$, $\Sigma_2$, $\Sigma_3$, $3$, $m$, $m$, $m$, $2^{-k+8}$ in place
    of $\mathcal F$, $\mathcal A$, $\Sigma_1$, $\Sigma_2$, $\Sigma_3$, $l$,
    $m_1$, $m_2$, $m_3$, $\varepsilon$ to obtain the~map $f : \R \times \R^n \to
    \R^n$ of class~$\cnt^{\infty}$ called ``$g$'' there. Define
    \begin{displaymath}
        A_n = \tbcup \mathcal A \,,
        \quad
        A_m = \tbcup \{ Q \in \cubes_m(k) : Q \subseteq A_n \} \,,
        \quad
        \phi = f(1,\cdot) \,.
    \end{displaymath}
    Recalling~\ref{thm:deformation}\ref{i:dt:identity}\ref{i:dt:m-dim-im}\ref{i:dt:estimates}
    we see that there exists an open set $W \subseteq \R^n$ such that
    \begin{gather}
        \label{eq:phi-W-im}
        S \subseteq W \,,
        \qquad
        \phi \lIm W \cap \cball{x_0}{\rho_3} \rIm \subseteq A_m \,,
        \\
        f(t,x)=x \quad \text{for $x\in \R^n\without \oball{x_0}{\rho_2}$} \,,
        \\
        \int_{\Sigma_i}\|\uD\phi(x)\|^m\ud\HM^m < \Gamma_{\ref{thm:deformation}} \HM^m( \Sigma_i )
        \quad \text{for $i \in \{ 1, 2, 3 \}$} \,.
    \end{gather}
    For each $\iota \in (0,1)$ recall~\ref{lem:rank-bound} and
    apply~\ref{lem:reduce-unrect} with $\oball{x_0}{\rho_3}$, $E_u \cap
    \cball{x_0}{\rho_1}$, $\phi$, $\iota$ in place of $U$, $K$, $f$,
    $\varepsilon$ to get a~diffeomorphism~$\varphi_{\iota}$ of~$\R^n$. Since
    $\phi$ is of~class~$\cnt^{\infty}$, recalling~\ref{cor:perturb-est}, we can
    find $\iota > 0$ such that, setting
    \begin{equation}
        \label{eq:tau-choice}
        \varphi = \varphi_{\iota} \,,
        \quad
        g = \phi \circ \varphi \,,
        \quad
        \tau = \tfrac 14 \min \bigl\{ \varepsilon ,\, \inf\{ |x-y| : x \in S ,\, y \in \R^n \without W \} \bigr\} \,,
    \end{equation}
    we obtain
    \begin{gather}
        \label{eq:Lip-varphi}
        \varphi(x) = x \quad \text{for $x \in \R^n \without \oball{x_0}{\rho_3}$} \,,
        \quad
        \Lip(\varphi - \id{\R^n}) \le \tau \,,
        \\
        \label{eq:varphi-id}
        | \varphi(x) - x | \le \tau
        \quad \text{and} \quad
        \|\uD g(x)\|^m \le 4^{m} \|\uD f(x)\|^m + \tau 
        \quad \text{for $x \in \R^n$} \,,
        \\
        \label{eq:kill-Eu}
        \HM^m\bigl( g[E_u\cap \cball{x_0}{\rho_1}] \bigr)
        \le \tau \HM^m\bigl( E_u \cap \cball{x_0}{\rho_1} \bigr).
    \end{gather}
    Thus, employing \eqref{eq:kill-Eu}, \eqref{eq:varphi-id},
    \eqref{eq:rho-meas}, \eqref{eq:unrect5} we get
	\begin{multline}
        \label{eq:unrect11}
        \HM^m\bigl( g \lIm E \cap \cball{x_0}{\rho_2} \rIm \bigr) 
        \le \tau \HM^m(E_u \cap \cball{x_0}{\rho_1})
        + \int_{\Sigma_1 \cup \Sigma_3} \|\uD g\|^m \ud \HM^m
        \\
        \le \tau (1 + \varepsilon) b \rho_1^m
        + 4^{m} \int_{\Sigma_1 \cup \Sigma_3} \|\uD f\|^m \ud \HM^m
        + \tau \HM^m(\Sigma_1 \cup \Sigma_3)
        \\
        \le \tau (1 + \varepsilon) b \rho_1^m + (4^{m} \Gamma_{\ref{thm:deformation}} + \tau ) \HM^m(\Sigma_1 \cup \Sigma_3)
        \le \tau (1 + \varepsilon) b \rho_1^m + 2 \varepsilon ( 4^{m} \Gamma_{\ref{thm:deformation}} + \tau ) b \rho_2^m \,.
	\end{multline}
    Observe that $\varphi\lIm S \rIm \subseteq W$ by~\eqref{eq:tau-choice},
    \eqref{eq:varphi-id}. Recalling~\eqref{eq:Lip-varphi}, \eqref{eq:phi-W-im},
    we see that $g \lIm S \cap \cball{x_0}{\rho_3} \rIm \subseteq A_m$ so for
    each $\iota > 0$ we can find an~open set $V_{\iota} \subseteq \R^n$ such
    that
	\begin{equation}
        \label{eq:cheap-trick}
        g \lIm S \cap \cball{x_0}{\rho_3} \rIm \subseteq V_{\iota}
        \quad \text{and} \quad
		\HM^m( V_{\iota} \cap A_m) \le \HM^m( g \lIm S \cap \cball{x_0}{\rho_3} \rIm ) + \iota \,.
	\end{equation}
    For $\iota > 0$ set $W_{\iota} = W \cap g^{-1} \lIm V_{\iota} \rIm$ and note
    that $W_{\iota}$ is open and $S \cap \cball{x_0}{\rho_3} \subseteq
    W_{\iota}$. Since $\HDK{K}(S_i \cap U, S \cap U) \to 0$ as $i \to \infty$
    for all compact sets $K \subseteq U$ we see that for $i \in \nat$ large
    enough $S_i \cap \cball{x_0}{\rho_3} \subseteq W_\iota$; thus,
	\begin{equation}
		\limsup_{i \to \infty} \HM^m( g \lIm S_i \cap \cball{x_0}{\rho_3} \rIm )
        \le \HM^m(V_\iota \cap A_m)
        \le \HM^m( g \lIm S \cap \cball{x_0}{\rho_3} \rIm ) + \iota \,.
	\end{equation}
    But $\iota > 0$ can be chosen arbitrarily small, so
	\begin{equation}
        \label{eq:inner-est}
		\limsup_{i \to \infty} \HM^m( g \lIm S_i \cap \cball{x_0}{\rho_3} \rIm )
        \le \HM^m( g \lIm S \cap \cball{x_0}{\rho_3} \rIm ) \,.
	\end{equation}
	Moreover, recalling that~$\var{m}(S_i \cap U) \to V$ as $i \to \infty$ and
    using~\eqref{eq:rho-slice} together with~\cite[2.6(2)(d)]{All72},
    \eqref{eq:V-bounds}, \eqref{eq:varphi-id},
    and~\ref{thm:deformation}\ref{i:dt:estimates} we obtain
	\begin{multline}
        \label{eq:annulus-est}
        \limsup_{i \to \infty} \HM^m( g \lIm S_i \cap \cball{x_0}{\rho_2} \without \oball{x_0}{\rho_1} \rIm )
        \\
        \le \lim_{i \to \infty} \int_{\cball{x_0}{\rho_2} \without \oball{x_0}{\rho_1}} \|\uD g\|^m \ud (\HM^m \restrict S_i)
        = \int_{\cball{x_0}{\rho_2} \without \oball{x_0}{\rho_1}} \|\uD g\|^m \ud \|V\|
        \\
        \le C_4 \int_{\Sigma_3} \|\uD g\|^m \ud \HM^m
        \le 4^m \Gamma_{\ref{thm:deformation}} C_4 \HM^m(\Sigma_3) + C_4 \tau \HM^m(\Sigma_3) \,.
	\end{multline}
    Combining~\eqref{eq:inner-est} and~\eqref{eq:annulus-est} we get
	\begin{equation}
		\limsup_{i\to \infty} \HM^m(g \lIm S_i \cap \cball{x_0}{\rho_2} \rIm)
        \le \HM^m( g \lIm S \cap \cball{x_0}{\rho_3} \rIm )
        + (4^m \Gamma_{\ref{thm:deformation}} + \tau) C_4 \HM^m(\Sigma_3) \,.
	\end{equation}
    Note that $\HM^m(S \cap \cball{x_0}{\rho_2}) = \HM^m(E \cap
    \cball{x_0}{\rho_2})$. In consequence, using~\eqref{eq:F-bounded},
    \eqref{eq:unrect11}, \eqref{eq:unrect5}, \eqref{eq:rho-meas},
    \eqref{eq:rho-choice}, and finally~\eqref{eq:F-bounded},
    \eqref{eq:V-bounds}, \eqref{eq:eps-choice}, we get
	\begin{multline}
        \label{eq:unrect21}
		\limsup_{i\to \infty} \Phi_{F}\left(g \lIm S_i\cap \cball{x_0}{\rho_2} \rIm \right)
        \\
        \le C_2 \HM^m( g \lIm E \cap \cball{x_0}{\rho_2} \rIm )
        + (4^m \Gamma_{\ref{thm:deformation}} + \tau) C_2 C_4 \HM^m(\Sigma_3) 
        \\
        \le \gamma (1 - \varepsilon) C_1 C_3 b \rho_1^m
        \le \gamma C_1 C_3 \HM^m(S \cap \cball{x_0}{\rho_2})
        \le \gamma \Phi_F(V \restrict \cball{x_0}{\rho_2} \times \grass nm) \,,
	\end{multline}
    where
    \begin{equation}
        \label{eq:gamma-le-1}
        \gamma = \frac{C_2\bigl( \tau (1+\varepsilon)
          + \varepsilon (4^m \Gamma_{\ref{thm:deformation}} + \tau) (C_4 + 2)  \bigr)}
        {(1-\varepsilon)^2 C_1 C_3} < 1 \,.
    \end{equation}
	Recall that $g(x) = x$ for $x \in \R^n \without \oball{x_0}{\rho_2}$ and $g
    \lIm \cball{x_0}{\rho_2} \rIm \subseteq \cball{x_0}{\rho_2}$ and $\var{m}(S_i
    \cap U) \to V$ as $i \to \infty$. Hence, using~\eqref{eq:unrect21},
    \eqref{eq:gamma-le-1}, and~\eqref{eq:rho-slice} together
    with~\eqref{eq:V-bounds}, we obtain
    \begin{multline}
        \label{eq:absurd}
        \limsup_{i \to \infty} \Phi_F(g \lIm S_i \rIm \cap U)
        = \lim_{i \to \infty} \Phi_F(S_i \cap U \without \cball{x_0}{\rho_2})
        + \limsup_{i \to \infty} \Phi_F(g \lIm S_i \cap \cball{x_0}{\rho_2} \rIm)
        \\
        < \Phi_F \bigl(V \restrict (U \without \cball{x_0}{\rho_2}) \times \grass nm \bigr)
        + \Phi_F\bigl( V \restrict \cball{x_0}{\rho_2} \times \grass nm \bigr)
        = \Phi_F(V) \,.
    \end{multline}
    Clearly $g \in \adm{U}$ so $g\lIm S_i \rIm \in \mathcal C$ for each $i \in
    \nat$ but~\eqref{eq:absurd} yields
    \begin{displaymath}
        \Phi_F(g \lIm S_i \rIm \cap U) < \Phi_F(V) \quad \text{for large enough $i \in \nat$} \,,
    \end{displaymath}
    which contradicts $\Phi_F(V) = \inf\{ \Phi_F(R \cap U) : R \in \mathcal C \}$.
\end{proof}

\begin{lemma}
    \label{le:ARtan}
    Let $\mu$ be a Radon measure over~$\R^n$ and $a \in \R^n$. Assume there
    exist $C \in (0,\infty)$ and $r_0 \in (0,\infty)$ such that for all $x \in
    \cball a{r_0}$ and~$r \in (0,r_0)$
    \begin{displaymath}
        \mu(\cball xr) \ge Cr^m \,.
    \end{displaymath}
    Then $\Tan^m(\mu,a) = \Tan(\spt \mu, a)$, i.e, the approximate tangent cone
    of~$\mu$ at~$a$ equals the classical tangent cone of the support of~$\mu$
    at~$a$.
\end{lemma}
\begin{proof}
    Following~\cite[3.2.16]{Fed69} if $a \in \R^n$, and $\varepsilon \in (0,1)$,
    and $v \in \R^n$, then we define the cone
    \begin{equation}
        \label{eq:def-E}
        \mathbf{E}(a,v,\varepsilon)
        = \bigl\{ x \in \R^n : \exists r > 0 \ \ |r(x-a) - v| < \varepsilon \bigr\} \,.
    \end{equation}
    Notice that if $|v| < \varepsilon$, then $\mathbf{E}(a,v,\varepsilon) =
    \R^n$ and if $0 < \varepsilon \le |v|$, then we may set
    \begin{equation}
        \label{eq:def-E-r}
        r = \frac{(x-a)}{|x-a|^2} \bullet v
        \quad \text{in~\eqref{eq:def-E}} \,.
    \end{equation}
    Let $S = \spt \mu$. By definition (see~\cite[3.2.16, 3.1.21]{Fed69}) we have
    \begin{gather}
        \label{eq:def-apTan}
        v \in \Tan^m(\mu,a)
        \quad \iff \quad
        \forall \varepsilon > 0 \ \ 
        \density^{*m}(\mu \restrict \mathbf{E}(a,v,\varepsilon), a) > 0
        \\
        \label{eq:def-Tan}
        \text{and} \quad
        v \in \Tan(S,a)
        \quad \iff \quad
        \forall \varepsilon > 0 \ \ 
        S \cap \mathbf{E}(a,v,\varepsilon) \cap \oball a{\varepsilon} \ne \varnothing \,.
    \end{gather}
    Clearly $\Tan^m(\mu,a) \subseteq \Tan(S,a)$ so we only need to show the
    reverse inclusion. Let $v \in \Tan(S,a)$ and $\varepsilon \in (0,1)$ be such
    that $\varepsilon \le |v|$. From~\eqref{eq:def-Tan} we see that there exists
    a~sequence $\{ x_k \in \R^n : k \in \nat \}$ such that
    \begin{equation}
        x_k \in S \cap \mathbf{E}(a,v,1/k) \cap \oball a{1/k} \,.
    \end{equation}
    Let us set $r_k = |x_k - a|$ for $k \in \nat$. Observe that whenever $1/k <
    \min\{ r_0, \varepsilon \} /2$, and $r_k < \varepsilon/2$, and $z \in
    \cball{x_k}{r_k \varepsilon/2}$, then setting $s = (x_k - a) \bullet v |x_k
    - a|^{-2}$ we obtain
    \begin{gather}
        |s(z-a) - v| \le |s(x_k - a) - v| + |s(z - x_k)| < \varepsilon \,;
        \\
        \text{hence,} \qquad
        \cball{x_k}{r_k\varepsilon/2} \subseteq \mathbf{E}(a,v,\varepsilon) \cap \oball{a}{2r_k} \,.
    \end{gather}
    Therefore,
    \begin{multline}
        \density^{*m}(\mu \restrict \mathbf{E}(a,v,\varepsilon),a)
        \ge \lim_{k \to \infty} (2r_k)^{-m} \mu(\mathbf{E}(a,v,\varepsilon) \cap \cball{a}{2r_k})
        \\
        \ge \lim_{k \to \infty} (2r_k)^{-m} \mu(\cball{x_k}{r_k\varepsilon/2})
        \ge  C 2^{-m} (\varepsilon/2)^{m} > 0 \,.
    \end{multline}
    Since $0 < \varepsilon \le |v|$ could be chosen arbitrarily, we see that $v
    \in \Tan^m(\mu,a)$.
\end{proof}

\begin{remark}
    \label{rem:tangent-cones}
	Let $U$ and $V$ be as in~\ref{thm:rect} and set $E = \spt\|V\| \subseteq
    U$. Then for each $a \in E$ one can find $0 < r_0 < \dist(a, \R^n \without
    U)$ such that $\|V\| \restrict \cball{a}{r_0}$ satisfies the conditions
    of~\ref{le:ARtan} at~$a$; thus, for \emph{all} points $a \in E$ we have
    $\Tan^m(\|V\|,a) = \Tan(E,a)$. In~particular, $\Tan(\spt\|V\|,x) \in \grass
    nm$ for $\HM^m$~almost all $x \in \spt\|V\| \subseteq U$;
    see~\cite[3.2.19]{Fed69}.
\end{remark}

\begin{lemma}
    \label{lem:kill-unrect}
    Let $G \subseteq \R^n$ be open and bounded, $S \subseteq \R^n$ be closed
    with $\HM^m(S \cap G) < \infty$ and such that $\HM^m(S \cap \Bdry G) = 0$,
    and let $\varepsilon > 0$. Decompose $S \cap G$ into a disjoint sum $S \cap
    G = S_u \cup S_r$, where $S_u$ is purely $(\HM^m,m)$~unrectifiable and $S_r$
    is $(\HM^m,m)$~rectifiable.

    There exists a number $\Gamma = \Gamma(n,m) \ge 1$ and a
    $\cnt^{\infty}$~smooth map $g \in \adm{G}$ such that
    \begin{displaymath}
        \HM^m(g \lIm S_u \rIm) \le \varepsilon \HM^m(S_u) 
        \quad \text{and} \quad
        \HM^m(g \lIm S_r \rIm) \le \Gamma \HM^m(S_r) \,.
    \end{displaymath}
\end{lemma}

\begin{proof}
    Let $\mathcal F = \WF(G)$ be the Whitney family associated to~$G$. If
    $\HM^m(S_u) > 0$, set $M = \HM^m(S_u)$ and if $\HM^m(S_u) = 0$, set $M =
    1$. Choose $N \in \nat$ so that
    \begin{equation}
        \HM^m(\{ x \in S : \dist(x,\R^n \without G) < 2^{-N} \}) < 2^{-100} \varepsilon M \,.
    \end{equation}
    Define
    \begin{displaymath}
        \mathcal A = \{ Q \in \mathcal F : \side{Q} \ge 2^{-N-10} \,,\, \widetilde Q \cap S \ne \varnothing \} \,,
    \end{displaymath}
    where $\widetilde Q = \tbcup \{ R \in \mathcal F : R \cap Q \ne \varnothing
    \}$ for $Q \in \mathcal F$. In~particular, we obtain
    \begin{gather}
        S \cap G \without \Int \tbcup \mathcal A \subseteq \{ x \in S : \dist(x,\R^n \without G) < 2^{-N} \}
        \\
        \text{and} \quad
        \label{eq:kr:bdry}
        \HM^m(S \cap G \without \Int \tbcup \mathcal A) < 2^{-100} \varepsilon M \,.
    \end{gather}
    Apply the deformation theorem~\ref{thm:deformation} with $\mathcal F$,
    $\mathcal A$, $S_r$, $1$, $m$, $2^{-N-30}$ in place of $\mathcal F$,
    $\mathcal A$, $\Sigma_1$, $l$, $m$, $\varepsilon$ to obtain the map $f \in
    \adm{G}$ of class~$\cnt^{\infty}$ called ``$g$'' there. Let $\omega$ be the
    modulus of continuity of~$\|\uD f\|$ as defined in~\eqref{eq:mod-cont-Df}
    and find $\bar \varepsilon > 0$ such that $\omega(\bar \varepsilon) \le 1$
    and $\bar \varepsilon < 2^{-100} \varepsilon$.  Next,
    recall~\ref{lem:rank-bound} to apply the perturbation
    lemma~\ref{lem:reduce-unrect} with
    \begin{displaymath}
        S_u \,, f \,,
        W = \Int \tbcup \bigl\{ Q \in \mathcal A : Q \subseteq \Int \tbcup \mathcal A \bigr\} \,,
        \bar \varepsilon 
        \quad \text{in place of} \quad
        K \,, f \,, U \,, \varepsilon 
    \end{displaymath}
    and obtain the map $\rho$ called ``$\rho_{\varepsilon}$'' there. Set $g = f
    \circ \rho$ and $A = \tbcup \mathcal A + \cball{0}{2^{-N-30}}$ and $\Gamma =
    2^{2m-1} (\Gamma_{\ref{thm:deformation}} + 1)$. To~estimate $\HM^m(g \lIm
    S_r \rIm)$ we employ~\ref{lem:image-measure} and~\ref{cor:perturb-est}
    \begin{multline}
        \HM^m(g \lIm S_r \rIm) 
        \le \HM^m(S_r \without A) + \int_{S_r \cap A} \|Dg\|^m \ud \HM^m
        \\
        \le \HM^m(S_r \without A) 
        + 2^{2m-1} \int_{S_r \cap A} \|Df\|^m \ud \HM^m 
        + 2^{2m-1} \HM^m(S_r \cap A)
        \\
        \le \HM^m(S_r \without A) 
        + 2^{2m-1} (\Gamma_{\ref{thm:deformation}} + 1) \HM^m(S_r \cap A) 
        \le \Gamma \HM^m(S_r)  \,.
    \end{multline}
    To~estimate $\HM^m(g \lIm S_u \rIm)$ we
    employ~\ref{lem:reduce-unrect} and~\eqref{eq:kr:bdry}
    \begin{multline}
        \HM^m(g \lIm S_u \rIm) 
        \le \HM^m(S_u \without W) 
        + \bar \varepsilon \HM^m(S_u \cap W) 
        \\
        \le 2^{-100} \varepsilon \HM^m(S_u)
        + 2^{-100} \varepsilon \HM^m(S_u \cap W) 
        \le \varepsilon \HM^m(S_u) 
        \qedhere
    \end{multline}
\end{proof}

% Local Variables:
% coding: utf-8
% eval: (ispell-change-dictionary "british")
% eval: (flyspell-mode)
% End:

%% file: density.tex
\section{Unit density of the limit varifold}
\label{sec:ud}

In this section we finish the proof of~\ref{thm:main}. 

We first prove a general ``hair-combing'' lemma~\ref{lem:ud:combing}, which
allows to choose a sequence of sets $\{ S_i : i \in \nat \}$ such that
$\var{m}(S_i \cap U) \to V \in \Var{m}(U)$ and, additionally, $S_i$ converge
locally in Hausdorff metric to some set $S$ such that $\HM^m(S \cap U \without
\spt \|V\|) = 0$. This is achieved by using first the deformation
theorem~\ref{thm:deformation} inside Whitney type cubes covering $U \without
\spt \|V\|$ and then applying the Blaschke Selection Theorem. Since the deformed
sets lie in a fixed grid of cubes we know that the ``hair'' does not accumulate
in the limit.

After that, we basically follow the guidelines presented in~\cite[3.2(d), p. 348
paragraph starting with ``We now verify that\ldots'', p. 349 l. 10--12]{Alm68}
to prove that $\VarTan(V,x) = \{ \var{m}(\Tan(\spt\|V\|,x)) \}$ for each $x \in
\spt \|V\| \subseteq U$ such that $\Tan(\spt\|V\|,x) \in \grass nm$ is
an~$m$-plane and $\density^m(\|V\|,x)$ exits and is finite. That means, we take
a sequence of radii $r_j$ converging to~$0$ and look at the blow-up limit
$(\scale{1/r_j})_{\#}V$. At each scale we use a smooth deformation to project
the part of~$V$ inside $\cball{x}{r_j}$ onto $x+\Tan(\spt\|V\|,x)$. Then we use
ellipticity of~$F$ and minimality of~$V$ to compare $V$ with the deformed~$V$
and conclude that $\density^m(\|V\|,x) = 1$. Of~course we cannot actually work
with~$V$ itself but we need to always look at the minimising sequence, because
$\spt\|V\|$ might not be a member of the good class~$\mathcal C$. After proving
that $\density^m(\|V\|,x) = 1$ we still need to show that $T =
\Tan(\spt\|V\|,x)$ for $V$~almost all $(x,T) \in U \times \grass nm$. To this
end we employ the area formula and a~well known relation between the tilt-excess
and the measure-excess (see~\ref{lem:jacobian-tilt}). Since the measure-excess
vanishes in the limit, so does the tilt-excess and the theorem is proven.

Since in the definition of ellipticity~\ref{def:elliptic}\ref{i:ell:S} we use
more general maps then admissible deformations defined in~\ref{def:adm-deform}
we need to prove that in our case one can replace the former with the latter.
We~do that in~\ref{lem:ell-adm}.

Also because we admit unrectifiable competitors and we used $\Psi_F$ instead of
$\Phi_F$ in~\ref{def:elliptic} but, in the end, we want to minimise $\Phi_F$ so
we need to show that the $\HM^m$~measure of the unrectifiable part of any
minimising sequence vanishes in the limit. This is done
in~\ref{lem:unrect-vanish}.

We emphasis that the proof of~\ref{thm:density}\ref{i:dens:upper-bound} is the
\emph{only} place in the whole paper where we make use of ellipticity of~$F$.

\begin{lemma}
    \label{lem:ud:combing}
    Let $U \subseteq \R^n$ be open. Assume $\{ S_i \subseteq \R^n : i \in \nat
    \}$ is a sequence of closed sets such that $\HM^m(S_i \cap U) < \infty$ for
    $i \in \nat$ and there exists a limit~$V = \lim_{i \to \infty} \var{m}(S_i \cap
    U) \in \Var{m}(U)$.

    Then there exist a closed set $X \subseteq \R^n$, and $g_i \in \adm{U}$, and
    a~subsequence $\{ S_i' : i \in \nat \}$ of $\{ S_i : i \in \nat \}$ such
    that, setting $E = \spt\|V\| \cup (\R^n \without U)$ and $X_i = g_i\lIm S_i'
    \rIm$ for $i \in \nat$, we~obtain
    \begin{gather}
        \lim_{i \to \infty} \HDK{K}(X_i \cap U, X \cap U) = 0 
        \quad \text{for each compact set $K \subseteq U$} \,,
        \\
        \lim_{i \to \infty} \sup \bigl\{ r \in \R : \HM^m(\{ x \in X_i \cap K : \dist(x,E) \ge r \}) > 0 \bigr\} = 0
        \quad \text{for $K \subseteq \R^n$ compact} \,, 
        \\
        \lim_{i \to \infty} \var{m}(X_i \cap U) = V \,,
        \quad
        \HM^m(X \cap U \without \spt \|V\|) = 0 \,.
    \end{gather}

    Furthermore, if $E$ is bounded and $S_i$ is compact for each $i \in \nat$
    and $\sup \{ \HM^m(S_i \cap U) : i \in \nat \} < \infty$, then
    \begin{equation}
        \sup \{ \diam X_i : i \in \nat \} < \infty \,.
    \end{equation}
\end{lemma}

\begin{proof}
    Let $\mathcal F = \{ Q_j : j \in \nat \} = \WF(U \without \spt\|V\|)$ be the
    Whitney family defined in~\ref{def:whitney}.
    For brevity of the notation, if $Q \in \mathcal F$, we define
    \begin{equation}
        \operatorname{N}(Q,0) = \{ Q \} \,,
        \quad
        \operatorname{N}(Q,i) = \bigl\{ R \in \mathcal F : R \cap \tbcup \operatorname{N}(Q,i-1) \ne \varnothing \bigr\} 
        \quad \text{for $i = 1,2,\ldots$}\,,
    \end{equation}
    and we set $\widetilde Q = \tbcup \operatorname{N}(Q,1)$. Since $\var{m}(S_i
    \cap U) \to V$ in $\Var{m}(U)$ as $i \to \infty$,
    using~\cite[2.6.2(c)]{All72}, we~see that
    \begin{equation}
        \label{eq:dens:meas-vanish}
        \limsup_{i \to \infty} \HM^m( S_i \cap \tbcup \operatorname{N}(Q_j,3)) \le \|V\|(\tbcup \operatorname{N}(Q_j,3) ) = 0
        \quad \text{for each $j \in \nat$} \,.
    \end{equation}
    Set $S_i^0 = S_i$ for $i \in \nat$ and define inductively $S_i^j$ for $j \in
    \nat$ by requiring that $\{ S_i^j : i \in \nat \}$ be a~subsequence
    of~$\{S_i^{j-1} : i \in \nat\}$ satisfying
    \begin{equation}
        \label{eq:dens:Sij-subs}
       \HM^m(S_i^j \cap \tbcup \operatorname{N}(Q_j,3) ) < \frac{\side{Q_j}^m}{2^{m+i} \Gamma_{\ref{thm:deformation}}}
       \quad \text{for $i \in \nat$} \,.
    \end{equation}
    For $j \in \nat$ define $P_j = S_j^j$ and $\mathcal A_j \subseteq \mathcal
    F$ to consist of all the cubes $Q \in \mathcal F$ with $Q \subseteq \cball
    0{2^j}$ and satisfying
    \begin{equation}
        \begin{gathered}
            \text{either}  \quad
            P_j \cap \widetilde Q \ne \varnothing
            \text{ and }
            \HM^m(P_j \cap \tbcup \operatorname{N}(Q,3))
            < \frac{\side{Q}^m}{2^{m+j} \Gamma_{\ref{thm:deformation}}}
            \text{ and }
            \side{Q} \ge \frac{1}{2^j}
            \\
            \text{or} \quad
            P_j \cap \widetilde Q \ne \varnothing
            \text{ and }
            \side{Q} \ge 1 \,.
        \end{gathered}
    \end{equation}
    Clearly $\mathcal A_j$ are finite. For each $j \in \nat$
    apply~\ref{thm:deformation} with $\mathcal F$, $\mathcal A_j$, $P_j$, $m$,
    $1$, $2^{-j-8}$ in place of $\mathcal F$, $\mathcal A$, $\Sigma_1$, $m_1$,
    $l$, $\varepsilon$ to obtain the map $\bar f_j \in \cnt^{\infty}(\R \times
    \R^n,\R^n)$ called ``$f$'' there. Set
    \begin{displaymath}
        f_j = \bar f_j(1,\cdot) \in \adm{U}
        \quad \text{and} \quad
        W_j = f_j \lIm P_j \rIm \,.
    \end{displaymath}
    We shall prove that $\lim_{j \to \infty} \var{m}(W_j \cap U) = V$ in $\Var{m}(U)$.

    Let $\varphi \in \ccspace{U \times \grass nm}$ and for $j \in \nat$ let
    $\zeta_j \in \cnt^{\infty}(\R^n,[0,1])$ be such that
    \begin{displaymath}
        \zeta_j(x) = 1 \quad \text{if $\dist(x, \tbcup \mathcal A_j) \le 2^{-j-4}$} \,,
        \quad
        \zeta_j(x) = 0 \quad \text{if $\dist(x, \tbcup \mathcal A_j) \ge 2^{-j-2}$} \,.
    \end{displaymath}
    This choice ensures
    \begin{gather}
        \spt (\CF_{\R^n} - \zeta_j) 
        \subseteq \bigl\{ x \in \R^n : \dist(x,\tbcup \mathcal A_j) > 2^{-j-6} \bigr\} 
        \subseteq \{ x \in \R^n : f_j(x) = x \} \,,
        \\
        \spt \zeta_j \subseteq \tbcup \bigl\{ \widetilde Q : Q \in \mathcal A_j \bigr\} \subseteq U \,.
    \end{gather}
    We set $\bar \zeta_j(x,T) = \zeta_j(x)$ for $(x,T) \in \R^n \times \grass
    nm$ and then
    \begin{multline}
        \var{m}(W_j \cap U)(\varphi) 
        = \var{m}(W_j \cap U)(\varphi \bar \zeta_j) + \var{m}(P_j \cap U)(\varphi (\CF_{\R^n \times \grass nm} - \bar \zeta_j))
        \\
        = \var{m}(P_j \cap U)(\varphi) + \var{m}(W_j \cap U)(\varphi \bar \zeta_j) - \var{m}(P_j \cap U)(\varphi \bar \zeta_j) \,.
    \end{multline}
    Since $\{ P_j : j \in \nat \}$ is a subsequence of $\{ S_j : j \in \nat\}$
    and $\var{m}(S_j \cap U) \to V$ in $\Var{m}(U)$ as $j \to \infty$, we only need
    to show that $\lim_{j \to \infty} \var{m}(W_j \cap U)(\varphi \bar \zeta_j) =
    0$ and $\lim_{j \to \infty} \var{m}(P_j \cap U)(\varphi \bar \zeta_j) = 0$. Set
    \begin{displaymath}
        \mathcal G = \bigl\{ Q \in \mathcal F : \widetilde Q \cap \spt \varphi \ne \varnothing \,,\, \side{Q} < 1 \bigr\} 
        \quad \text{and} \quad
        \mathcal J = \bigl\{ Q \in \mathcal F : \widetilde Q \cap \spt \varphi \ne \varnothing \,,\, \side{Q} \ge 1 \bigr\} \,,
    \end{displaymath}
    then $\mathcal G$ and $\mathcal J$ are finite and do not depend on $j \in
    \nat$. Moreover, $\spt \varphi \cap \spt \zeta_j \subseteq \tbcup \{
    \widetilde Q : Q \in \mathcal A_j \cap \mathcal G \} \cup \tbcup \{
    \widetilde Q : Q \in \mathcal J \}$, and $\dist(\tbcup \{ \widetilde Q : Q
    \in \mathcal J \}, \spt \|V\|) > 0$; hence,
    \begin{multline}
        \label{eq:dens:vPj}
        \lim_{j \to \infty} \var{m}(P_j \cap U)(\varphi \bar \zeta_j)
        \le \lim_{j \to \infty} \sup \im |\varphi| \HM^m(\spt \varphi \cap \spt \zeta_j \cap P_j)
        \\
        \le \sup \im |\varphi| \lim_{j \to \infty} \biggl(
            \sum_{Q \in \mathcal G \cap \mathcal A_j} \HM^m(P_j \cap \widetilde Q)
            + \sum_{Q \in \mathcal J} \HM^m(P_j \cap \widetilde Q)
        \biggr)
        \\
        \le \sup \im |\varphi| \biggl( \lim_{j \to \infty} \HM^0(\mathcal G) \Gamma_{\ref{thm:deformation}}^{-1} 2^{-j}
        + \sum_{Q \in \mathcal J} \lim_{j \to \infty} \HM^m(P_j \cap \widetilde Q) \biggr) = 0 \,.
    \end{multline}
    To deal with $\var{m}(W_j \cap U)(\varphi \bar \zeta_j)$ we first note that whenever
    $Q \in \mathcal F$, then,
    recalling~\ref{thm:deformation}\ref{i:dt:estimates},
    \begin{multline}
        \label{eq:dens:fjPcapQ}
        \HM^m(f_j \lIm P_j \rIm \cap \widetilde Q)
        \le 
        \sum_{R \in \operatorname{N}(Q,1)} \HM^m(f_j \lIm P_j \rIm \cap R)
        \\
        \le \sum_{R \in \operatorname{N}(Q,1)} \Gamma_{\ref{thm:deformation}} \HM^m\bigl( P_j \cap ( \widetilde R + \oball{0}{2^{-j-8}}) \bigr) 
        \le \Delta^2 \Gamma_{\ref{thm:deformation}} \HM^m(P_j \cap \tbcup \operatorname{N}(Q,3)) \,,
    \end{multline}
    where $\Delta = 4^n \ge \sup \{ \HM^0(\operatorname{N}(T,1)) : T \in
    \mathcal F \}$. Now, we can estimate as in~\eqref{eq:dens:vPj}
    using~\eqref{eq:dens:fjPcapQ}
    \begin{multline}
        \lim_{j \to \infty} \var{m}(W_j \cap U)(\varphi \bar \zeta_j)
        \le \sup \im |\varphi| \lim_{j \to \infty} \biggl(
            \sum_{Q \in \mathcal G \cap \mathcal A_j} \HM^m(f_j \lIm P_j \rIm \cap \widetilde Q)
            + \sum_{Q \in \mathcal J} \HM^m(f_j \lIm P_j \rIm \cap \widetilde Q)
        \biggr)
        \\
        \le \sup \im |\varphi| \biggl( \lim_{j \to \infty} \HM^0(\mathcal G) \Delta^2 2^{-j}
        + \Delta^2 \Gamma_{\ref{thm:deformation}} \sum_{Q \in \mathcal J} \lim_{j \to \infty} \HM^m(P_j \cap \tbcup \operatorname{N}(Q,3)) \biggr) = 0 \,.
    \end{multline}
    This finishes the proof that $\lim_{j \to \infty} \var{m}(W_j \cap U) = V$.

    Let $\{ K_i \subseteq U : i \in \nat \}$ be a sequence of compacts sets such
    that $\tbcup \{ K_i \subseteq U : i \in \nat \} = U$ and $K_i \subseteq
    K_{i+1}$ for $j \in \nat$. We set $W^0_j = W_j$ for $j \in \nat$. For $i \in
    \nat$ we define $\{ W^i_j : j \in \nat \}$ to be a subsequence of $\{
    W^{i-1}_j : j \in \nat \}$ such that the sequence $\{ W^i_j \cap K_i : j \in
    \nat \}$ converges in the Hausdorff metric to some compact set $F_i$ -- this
    can be done employing the Blaschke Selection Theorem;
    see~\cite{Pri40}. Finally, we set $R_i = W^i_i$ for $i \in \nat$ and $R =
    \tbcup \{ F_i : i \in \nat \}$. We see that for any compact set $K \subseteq
    U$ we have $\HDK{K}(R_i \cap U, R \cap U) \to 0$ as $i \to \infty$.

    Now, we need to show that $\HM^m(R \cap U \without \spt \|V|) = 0$. It will
    be enough to prove that $\HM^m(R \cap Q) = 0$ for each $Q \in \mathcal F$.
    Recall that $R \cap Q$ is the limit, in the Hausdorff metric, of some
    subsequence of $\{ f_j \lIm P_j \rIm \cap Q : j \in \nat \}$. We claim that
    for big enough $j \in \nat$ the set $f_j \lIm P_j \rIm \cap Q$ lies in the
    $m-1$-dimensional skeleton of $\mathcal F$, i.e., $f_j \lIm P_j \rIm \cap Q
    \subseteq \tbcup \cubes_{m-1} \cap \CX(\mathcal F)$, which implies that $R
    \cap Q \subseteq \tbcup \cubes_{m-1} \cap \CX(\mathcal F)$ and $\HM^m(R \cap
    Q) = 0$. To prove our claim let $j_0, j_1, j_2 \in \integers$ be such that
    $Q = Q_{j_0}$ and $\side{Q} = 2^{-j_1}$ and $Q \subseteq \cball 0{2^{j_2}}$,
    and assume $j > \max\{j_0,j_1,j_2\}$.

    In case $\widetilde Q \cap P_j = \varnothing$ we have $f_j \lIm P_j \rIm
    \cap Q = \varnothing$, by~\ref{thm:deformation}\ref{i:dt:cube-pres}, and
    there is nothing to prove. Thus, we assume that $\widetilde Q \cap P_j \ne
    \varnothing$ which implies that $Q \in \mathcal A_j$ due
    to~\eqref{eq:dens:Sij-subs} and $j > \max\{ j_0, j_1, j_2 \}$. For $R \in
    \operatorname{N}(Q,1)$ such that $R \cap P_j \ne \varnothing$ we~estimate
    using~\ref{thm:deformation}\ref{i:dt:estimates}
    \begin{multline}
        \label{eq:dens:fjPjcapR}
        \HM^m(f_j \lIm P_j \rIm \cap R)
        \le \Gamma_{\ref{thm:deformation}} \HM^m\bigl( P_j \cap ( \widetilde R + \oball{0}{2^{-j-8}}) \bigr)
        \\
        \le \Gamma_{\ref{thm:deformation}} \HM^m\bigl( P_j \cap \tbcup \operatorname{N}(Q,3)\bigr)
        < 2^{-m} \side{Q}^m \le \side{R}^m \,.
    \end{multline}
    If $R \cap P_j \ne \varnothing$ and $R \in \operatorname{N}(Q,1)$, then it
    follows that $R \subseteq \Int \tbcup \mathcal A_j$; hence,
    combining~\eqref{eq:dens:fjPjcapR}
    with~\ref{thm:deformation}\ref{i:dt:full-faces}, we see that $f_j \lIm P_j
    \rIm \cap R \subseteq \tbcup \cubes_{m-1} \cap \CX(\mathcal F)$. Next,
    observe that
    \begin{equation}
        \label{eq:dens:fjPjm-1}
        f_j \lIm P_j \rIm \cap Q \subseteq
        \tbcup \{ f_j \lIm R \rIm : R \in \operatorname{N}(Q,1) ,\, R \cap P_j \ne \varnothing \}
        \subseteq \tbcup \cubes_{m-1} \cap \CX(\mathcal F) \,.
    \end{equation}

    Let $K \subseteq \R^n$ be compact. Observe that for each $k \in \nat$, there
    are only finitely many cubes in $\mathcal F$ which touch~$K$ and have side
    length at least $2^{-k}$. If $j_0 \in \nat$ is the maximal index of such
    cube and $j_1 \in \nat$ is such that $Q_{j_0} \subseteq \cball 0{2^{j_1}}$,
    then for $j > \max\{j_0,j_1,k\}$ the estimate~\eqref{eq:dens:fjPjcapR} holds
    for any $R \in \operatorname{N}(Q,1)$ whenever $Q \in \mathcal F$ satisfies
    $\side{Q} \ge 2^{-k}$ and $Q \cap K \ne \varnothing$. In~consequence, as
    in~\eqref{eq:dens:fjPjm-1},
    \begin{displaymath}
        f_j\lIm P_j \rIm \cap \tbcup \bigl\{ Q \in \mathcal F : \side{Q} \ge 2^{-k} \,,\, Q \cap K \ne \varnothing \bigr\}
        \subseteq \tbcup \CX(\mathcal F) \cap \cubes_{m-1} \,.
    \end{displaymath}
    Recalling that $\mathcal F = \WF(U \without \spt\|V\|)$ was the Whitney
    family associated to $\R^n \without E$, we see that there exists $\{ r_j \in
    (0,\infty) : j \in \nat \}$ such that $r_j \downarrow 0$ as $j \to \infty$
    and
    \begin{displaymath}
        \HM^m(\{ x \in f_j\lIm P_j \rIm \cap K : \dist(x, E) \ge r_j \}) = 0 
        \quad \text{for $j \in \nat$} \,.
    \end{displaymath}

    Observe that for each $i \in \nat$ there exists $j(i) \in \nat$ such that
    $R_i = f_{j(i)}\lIm P_{j(i)} \rIm$, so setting $S'_i = P_{j(i)}$ for $i \in
    \nat$ we obtain a subsequence of $\{S_i : i \in \nat\}$. Hence, we may
    finish the proof of the first part of~\ref{lem:ud:combing} by setting $X =
    R$ and $g_i = f_{j(i)}$ and $X_i = R_i$ for $i \in \nat$.

    Assume now that $E$ is bounded and $S_i$ is compact for $i \in \nat$ and
    $\sup \{ \HM^m(S_i \cap U) : i \in \nat \} < \infty$. In this case we need
    to further modify the sets $R_i$ to ensure that all the resulting sets $X_i$
    are bounded. To this end we shall simply project the sets $R_j$ onto the
    cube $[-M_1,M_1]^n$. Since our definition of admissible mappings allows only
    for deformations inside convex sets we need to perform the projection in
    several steps. Using the fact that outside $[-M_1,M_1]^n$ the sets $R_j$ lie
    in the $m-1$ dimensional skeleton of $\mathcal F$ we~will deduce that the
    projected sets $X_j$ give rise to the same varifolds as $R_j$, i.e.,
    $\var{m}(R_j \cap U) = \var{m}(X_j \cap U)$.

    Suppose there exists $M_0 > 1$ such that
    \begin{displaymath}
        E = \spt\|V\| \cup (\R^n \without U) \subseteq [-M_0,M_0]^n
        \quad \text{and} \quad
        \sup \{ \HM^m(S_i \cap U) : i \in \nat \} < M_0 \,.
    \end{displaymath}
    Then $\sup \{ \HM^m(R_i \cap U) : i \in \nat \} < M_0$ and we can find $M_1 >
    2^{10} M_0$ such that
    \begin{displaymath}
        \tbcup \{ Q \in \mathcal F : \side{Q}^m \le \Gamma_{\ref{thm:deformation}} M_0 \} \subseteq [-2^{-10} M_1, 2^{-10} M_1]^n \,.
    \end{displaymath}

    Let $e_1,\ldots,e_n$ be the standard basis of~$\R^n$. For $i \in \{ -n,
    \ldots, -1, 1, \ldots, n\}$ and $j \in \nat$ we define
    \begin{itemize}
    \item $L_i = \{ x \in \R^n : x \bullet e_{|i|} = \tfrac{i}{|i|} M_1 \}$,
    \item $H_i = \{ x \in \R^n : x \bullet e_{|i|} \tfrac{i}{|i|} \ge M_1 \}$,
    \item $p_i$ to be the affine orthogonal projection onto the affine plane $L_i$,
    \item $Y_{j,i} \subseteq H_i$ to be a a large cube containing $[-M_1,M_1]^n
        \cap L_i$ and such that
        \begin{displaymath}
            R_j \cap H_i \subseteq Y_{j,i} \,,
        \end{displaymath}
    \item $\varphi_{i,j} \in \cnt^{\infty}(\R^n)$ to be a cut-off function such
        that $0 \le \varphi_{i.j}(x) \le 1$ for $x \in \R^n$ and
        \begin{displaymath}
            Y_{j,i} \subseteq \varphi_{j,i}^{-1} \{1\}
            \quad \text{and} \quad
            \spt \varphi_{j,i} \subseteq \R^n \without [-2^{-1}M_1, 2^{-1}M_1]^n \text{ is compact and convex} \,,
        \end{displaymath}
    \item $h_{j,i} = p_i \varphi_{j,i} + (\CF_{\R^n} - \varphi_{j,i}) \id{\R^n} \in \adm{U}$,
    \item $h_j = h_{j,-n} \circ \cdots \circ h_{j,-1} \circ h_{j,1} \circ \cdots \circ h_{j,n}  \in \adm{U}$.
    \end{itemize}
    We set $X_i = h_i \lIm R_i \rIm$ and $g_i = h_i \circ f_{j(i)}$ for $i \in
    \nat$. Clearly $X_j \subseteq [-M_1,M_1]^n$ for each $j \in \nat$ so,
    employing the Blaschke Selection Theorem, we can assume that $X_j$ converges
    in the Hausdorff metric to some compact set~$X$.  Observe that if $i,j \in
    \nat$ and $R_i = f_j\lIm P_j \rIm$ and $Q \in \mathcal F$ is such that $Q
    \cap P_j \ne \varnothing$ and $Q \cap \R^n \without [-2^{-1} M_1, 2^{-1}
    M_1]^n \ne \varnothing$, then $\side{Q}^m > \Gamma_{\ref{thm:deformation}}
    M_0 > 1$ and $Q \in \mathcal A_j$ and $\Gamma_{\ref{thm:deformation}}
    \HM^m(P_j \cap \tbcup \operatorname{N}(Q,3)) \le
    \Gamma_{\ref{thm:deformation}} M_0 < \side{Q}^m$; hence, $f_j \lIm P_j \rIm
    \cap Q \subseteq \CX(\mathcal F) \cap \cubes_{m-1}$,
    by~\ref{thm:deformation}\ref{i:dt:full-faces}, and $\HM^m(f_j \lIm P_j \rIm
    \cap Q) = 0$ so $\HM^m(h_i \lIm R_i \rIm \cap Q) = 0$. Since $X_i \cap
    [-2^{-1} M_1, 2^{-1} M_1]^n = R_i \cap [-2^{-1} M_1, 2^{-1} M_1]^n$ for $i
    \in \nat$, we see that $\var{m}(R_i \cap U) = \var{m}(X_i \cap U)$ for $i \in
    \nat$ and $\var{m}(X_i \cap U) \to V \in \Var{m}(U)$ as $i \to \infty$.
\end{proof}

\begin{corollary}
    \label{cor:ud:good-ms}
    Assume
    \begin{gather}
        U \subseteq \R^n \quad \text{is open} \,,
        \quad
        \text{$F$ is a bounded $\cnt^0$ integrand} \,,
        \\
        \text{$\mathcal{C}$ is a good class in~$U$} \,,
        \quad
        \mu = \inf\{ \Phi_F(R \cap U) : R \in \mathcal{C} \} \in (0, \infty) \,.
    \end{gather}

    Then there exist $\{ S_i : i \in \nat \} \subseteq \mathcal{C}$, $S \in
    \mathcal C$, $V \in \Var{m}(U)$, and $E \subseteq \R^n$ such that
    \begin{gather}
        E = \spt \|V\| \cup (\R^n \without U) \,,
        \quad
        \HM^m(S \cap U \without \spt \|V\|) = 0 \,,
        \\
        \lim_{i \to \infty} \HDK{K}(S_i \cap U, S \cap U) = 0 
        \quad \text{for $K \subseteq U$ compact} \,,
        \\
        \lim_{i \to \infty} \sup \bigl\{ r \in \R : \HM^m(\{ x \in S_i \cap K : \dist(x,E) \ge r \}) > 0 \bigr\} = 0
        \quad \text{for $K \subseteq \R^n$ compact} \,, 
        \\
        \sup \{\HM^m(S_i \cap U) : i \in \nat \} < \infty \,,
        \quad
        \lim_{i \to \infty} \var{m}(S_i \cap U) = V \,,
        \quad
        \lim_{i \to \infty} \Phi_F(S_i \cap U) = \mu \,.
    \end{gather}

    Furthermore, if $B = \R^n \without U$ is compact and there exists
    a~minimising sequence consisting of compact elements of~$\mathcal C$, then
    \begin{displaymath}
        \spt \|V\| \quad \text{is bounded}
        \quad \text{and} \quad
        \sup \{\diam S_i : i \in \nat \} < \infty \,,
    \end{displaymath}
\end{corollary}

\begin{proof}
    Let $\{ R_i : i \in \nat \}$ be any minimising sequence in $\mathcal C$, i.e.,
    \begin{displaymath}
        \Phi_F(R_i \cap U) \to \inf\{ \Phi_F(R \cap U) : R \in \mathcal{C} \} = \mu
        \quad \text{as $i \to \infty$} \,.
    \end{displaymath}
    Observe that since $F$ is bounded and $\mu$ is finite we have $\HM^m(R_i
    \cap U) < 2 (\inf \im F)^{-1} \mu$ for all but finitely many $i \in \nat$ --
    we~shall assume it holds for all $i \in \nat$. In~consequence, we can choose
    a~subsequence $\{ R_i' : i \in \nat \}$ of $\{ R_i : i \in \nat \}$ such
    that $\var{m}(R_i' \cap U)$ converges as $i \to \infty$ to some $V \in
    \Var{m}(U)$.  If $B = \R^n \without U$ is compact, then we
    use~\ref{cor:ludrb} too see that $\spt \|V\|$ must be bounded; hence, $E$ is
    bounded. Next, we~apply~\ref{lem:ud:combing} to $\{ R_i' : i \in \nat \}$ to
    obtain a~subsequence $\{ P_i : i \in \nat \}$ of $\{ R_i' : i \in \nat \}$
    and maps $\{ g_i \in \adm{U} : i \in \nat \}$ and $S \in \mathcal
    C$. Finally, we~set $S_i = g_i\lIm P_i \rIm$.
\end{proof}

\begin{remark}
    \label{rem:summary}
    Let us summarise what we know so far.

    Under the assumptions of~\ref{cor:ud:good-ms} we obtain a minimising
    sequence $\{ S_i : i \in \nat\} \subseteq \mathcal C$, and $V \in
    \Var{m}(U)$, and $S \in \mathcal C$ satisfying
    \begin{gather}
        \lim_{i \to \infty} \Phi_F(S_i \cap U) = \mu
        \quad \text{and} \quad
        V = \lim_{i \to \infty} \var{m}(S_i \cap U) 
        \quad \text{and} \quad
        \HM^m(S \cap U \without \spt\|V\|) = 0
        \\
        \text{and} \quad
        \lim_{i \to \infty} \HDK{K}(S \cap U, S_i \cap U) = 0
        \quad \text{for $K \subseteq U$ compact} \,.
    \end{gather}
    Next, employing~\ref{thm:rect} we see that $\spt\|V\|$ is countably
    $(\HM^m,m)$~rectifiable and has locally finite $\HM^m$~measure.
    In~particular, $\density^m(\HM^m \restrict \spt\|V\|,x) = 1$ for
    $\HM^m$~almost all $x \in \spt \|V\| \subseteq U$ by~\cite[3.2.19]{Fed69}
    so, using~\ref{cor:ludrb} and~\cite[2.8.18, 2.9.5]{Fed69}, also the density
    $\density^m(\|V\|,x)$ exists and is finite for $\|V\|$~almost all~$x \in U$.
    Recalling~\ref{rem:tangent-cones} we see that for $\HM^m$~almost all $x \in
    \spt \|V\|$ the \emph{classical} tangent cone $\Tan(\spt \|V\|,x)$ is an
    $m$-dimensional subspace of~$\R^n$.
\end{remark}

In the next lemma we relate the \emph{$L^2$~tilt-excess} to the
\emph{measure-excess}; see~\eqref{eq:ud:tm-excess}. Similar upper bound was also
proven in~\cite[3.13]{snulmenn.poincare}.

\begin{lemma}
    \label{lem:jacobian-tilt}
    Let $P,Q \in \grass nm$. Then
    \begin{equation}
        \tfrac 12 \|\project P - \project Q\|^2  
        \le 1 - \| \tbwedge_m \project P \circ \project Q \| 
        \le 2^{2m+3} \|\project P - \project Q\|^2 \,.
    \end{equation}
\end{lemma}

\begin{proof}
    Employ \cite[8.9(3)]{All72} to find $u \in Q$ such that $|u| = 1$ and
    $\|\project P - \project Q\| = |\perpproject P u|$. Let $u_1, \ldots, u_m$
    be an orthonormal basis of~$Q$ such that $u_1 = u$. We have
    \begin{equation}
        \| \tbwedge_m \project P \circ \project Q \|
        = | \project P u_1 \wedge \cdots \wedge \project P u_m |
        \le | \project P u_1 | = \bigl( 1 - \|\project P - \project Q\|^2 \bigr)^{1/2} \,.
    \end{equation}
    In consequence, since $(1-x)^{1/2} \le 1 - \frac 12 x$ for $x \in [0,1]$, we obtain
    \begin{equation}
        1 - \| \tbwedge_m \project P \circ \project Q \| 
        \ge 1 - \bigl( 1 - \|\project P - \project Q\|^2 \bigr)^{1/2} 
        \ge \tfrac 12 \|\project P - \project Q\|^2 \,.
    \end{equation}

    Next, we shall derive the upper estimate. Let $q \in \orthproj nm$ be such
    that $\im q^* = Q$. Since $\project P \circ \project Q = \project Q -
    \perpproject P \circ \project Q$ and $q^* \circ q = \project Q$ and
    $(\perpproject P)^* = \perpproject P$ we obtain, using~\cite[1.7.6. 1.7.9,
    1.4.5]{Fed69},
    \begin{multline}
        \| \tbwedge_m \project P \circ \project Q \|^2
        =| \tbwedge_m \project P \circ \project Q |^2
        = \bigl| \tbwedge_m \bigl( q^* - \perpproject P \circ q^* \bigr) \circ q \bigr|^2
        = \bigl| \tbwedge_m \bigl( q^* - \perpproject P \circ q^* \bigr) \bigr|^2
        \\
        = \trace \bigl( \tbwedge_m \bigl( q^* - \perpproject P \circ q^* \bigr)^* \circ \bigl( q^* - \perpproject P \circ q^* \bigr) \bigr)
        = \det \bigl( \id{\R^m} - q \circ \perpproject P \circ q^* \bigr)
        \\
        = \sum_{j=0}^m (-1)^j \trace \bigl( \tbwedge_j q \circ \perpproject P \circ q^* \bigr)
        = \sum_{j=0}^m (-1)^j \trace \bigl( \tbwedge_j \bigl( \perpproject P \circ q^* \bigr)^* \circ \bigl( \perpproject P \circ q^* \bigr) \bigr)
        \\
        = \sum_{j=0}^m (-1)^j \bigl| \tbwedge_j \bigl( \perpproject P \circ q^* \bigr) \bigr|^2
        = 1 - \bigl| \perpproject P \circ \project Q \bigr|^2 + E \,,
    \end{multline}
    where $E = \sum_{j=2}^m (-1)^j \bigl| \tbwedge_j \bigl( \perpproject P \circ
    q^* \bigr) \bigr|^2$. Note that $1 - x \le (1-x)^{1/2}$ for $x \in [0,1]$;
    hence,
    \begin{equation}
        \label{eq:et:two-side-with-E}
        1 - \| \tbwedge_m \project P \circ \project Q \|
        = 1 - \bigl( 1 - \bigl| \perpproject P \circ \project Q \bigr|^2 + E  \bigr)^{1/2}
        \le \bigl| \perpproject P \circ \project Q \bigr|^2 - E  \,.
    \end{equation}
    Employing~\cite[1.7.6, 1.7.9, 1.3.2]{Fed69} together
    with~\cite[8.9(3)]{All72} we get
    \begin{gather}
        \label{eq:et:norms2}
        \bigl| \tbwedge_j \bigl( \perpproject P \circ q^* \bigr) \bigr|^2
        \le \binom mj \| \project P - \project Q\|^{2j} \quad \text{for $j = 0,1,\ldots,m$} \,.
    \end{gather}
    If $\|\project P - \project Q\|^2 \le 2^{-(m+2)}$, then
    \begin{multline}
        \label{eq:et:small}
        |E| \le \sum_{j=2}^m \binom mj \| \project P - \project Q\|^{2j}
        \le 2^{m} \sum_{j=2}^m \| \project P - \project Q\|^{2j}
        \\
        \le 2^{m} \frac{\| \project P - \project Q\|^{4} - \| \project P - \project Q\|^{2m}}{1 - \| \project P - \project Q\|^{2}}
        \le \tfrac 12 \| \project P - \project Q\|^{2} \,.
    \end{multline}
    If $2^{-(m+2)} < \|\project P - \project Q\|^2 = \| \perpproject P \circ \project Q\|^2 \le 1$, then
    \begin{equation}
        \label{eq:et:big}
        |E| \le \sum_{j=2}^m \binom mj \| \project P - \project Q\|^{2j}
        \le 2^{m}
        \le 2^{2m+2} \| \project P - \project Q\|^{2} \,.
    \end{equation}
    Combining~\eqref{eq:et:two-side-with-E}, \eqref{eq:et:norms2},
    \eqref{eq:et:small}, and~\eqref{eq:et:big} we obtain
    \begin{equation}
        1 - \| \tbwedge_m \project P \circ \project Q \| \le 2^{m+3} \|\project P - \project Q\|^2 \,.
        \qedhere
    \end{equation}
\end{proof}

The following lemma relates Lipschitz maps from the definition of an elliptic
integrand~\ref{def:elliptic} to admissible maps defined in~\ref{def:adm-deform}.
Given~$S$ and~$D$ as in~\ref{def:elliptic} and a Lipschitz deformation which
maps~$S$ onto the relative boundary of~$D$ the lemma provides conditions on the
set~$S$ under which one can construct an \emph{admissible} map deforming~$S$
onto the relative boundary of~$D$.

\begin{lemma}
    \label{lem:ell-adm}
    Assume
    \begin{gather*}
        T \in \grass nm \,,
        \quad
        S \subseteq \cball 01 \text{ is closed} \,,
        \quad
        S \cap \Bdry \cball 01 \subseteq T \cap \Bdry \cball 01 = R \,,
        \\
        \delta = \tfrac 12 \sup\left\{ \rho \in [0,1] : 
            \begin{gathered}
                S \cap (R + \cball 0{\rho}) \subseteq T \,,\, 
                \\
                (S + \cball 0{\rho}) \cap (\R^n \without \oball 01) \without (R + \cball 0{\rho}) = \varnothing
            \end{gathered}
        \right\} > 0 \,.
    \end{gather*}
    Suppose there exists a map $f : \R^n \to \R^n$ satisfying
    \begin{displaymath}
        \Lip f < \infty \,,
        \quad
        f(x) = x \quad \text{for $x \in R$} \,,
        \quad 
        f \lIm S \rIm \subseteq R \,.
    \end{displaymath}
    Then there exist $\Gamma = \Gamma(\Lip f,\delta) \in (0,\infty)$ and a map
    $g : \R^n \to \R^n$ such that
    \begin{displaymath}
        \Lip g < \Gamma \,,
        \quad
        g \lIm S \rIm \subseteq R \,,
        \quad
        g \lIm \cball 01 \rIm \subseteq \cball 01 \,,
        \quad
        g(x) = x \quad \text{for $x \in \R^n \without \oball 01$} \,.
    \end{displaymath}
    Moreover for each $\varepsilon > 0$ there exists $h \in \adm{\oball
      0{1+\varepsilon}}$ such that
    \begin{displaymath}
        \Lip h \le \Gamma 
        \quad \text{and} \quad
        h\lIm S \rIm \subseteq R \,.
    \end{displaymath}
\end{lemma}

\begin{proof}
    Define a map $\bar g : \R^n \without \oball 01 \cup S \to \R^n$ by setting
    \begin{displaymath}
        \bar g(x) = x \quad \text{for $x \in \R^n \without \oball 01$}
        \quad \text{and} \quad
        \bar g(x) = f(x) \quad \text{for $x \in S$} \,.
    \end{displaymath}
    We shall check that $\Lip \bar g < \infty$. Since $S \without (R + \cball
    0{\delta})$ does not touch $\Bdry \cball 01$ we see that $\bar g$ is
    Lipschitz continuous on each of the sets
    \begin{displaymath}
        S \,,
        \quad
        \R^n \without \oball 01 \,,
        \quad
        (\R^n \without \oball 01) \cup S \without (R + \cball 0{\delta}) \,.
    \end{displaymath}
    Note, however, that the Lipschitz constant of $\bar g$ on the last set
    depends on~$\delta$. Now, it suffices to estimate $|\bar g(x) - \bar g(y)|$
    for $x \in S \cap (R + \cball 0{\delta}) \subseteq T$ and $y \in \R^n
    \without \oball 01$. Note that $x/|x| \in R$, so $f(x/|x|) = x/|x|$ and
    \begin{displaymath}
        \bigl| \bar g(x) - \bar g(y) \bigr| \le \bigl| f(x) - f(x/|x|) \bigr| + \bigl| x/|x| - y \bigr|
        \le \Lip f \bigl| x - x/|x| \bigr| + \bigl| x - y \bigr| + \bigl| x/|x| - x \bigr| \,.
    \end{displaymath}
    Clearly $|x - x/|x|| \le 2 |x-y|$, so $\Lip \bar g < \infty$.

    Next, we extend $\bar g$ to a Lipschitz map $\tilde g : \R^n \to \R^n$ using
    a~standard procedure; see~\cite[3.1.1]{EvansGariepy1992}. Let $L = \Lip
    \tilde g \in [1,\infty)$. Since $\tilde g(x) = x$ for $x \in \Bdry \cball
    01$ we know $\tilde g \lIm \cball 01 \rIm \subseteq \cball 0{2L + 1}$.
    Define the map $l : \R^n \to \R^n$ by setting
    \begin{displaymath}
        l(x) = \left\{
            \begin{aligned}
                &x &&\text{for $x \in (\R^n \without \cball 01) \cup \cball 0{1/(7L)}$} \,,
                \\
                &\frac{x}{|x|} &&\text{for $x \in \cball 01 \without \cball 0{1/(4L)}$} \,,
                \\
                &\sigma(|x|)\frac{x}{|x|} + (1-\sigma(|x|)) x &&\text{for $x \in \cball 0{1/(4L)} \without \cball 0{1/(7L)}$} \,,
            \end{aligned}
            \right.
    \end{displaymath}
    where $\sigma(t) = (28 L t - 4)/3$. Finally set
    \begin{displaymath}
        g(x) = \left\{
            \begin{aligned}
                &l \circ \scale{1/(3L)} \circ \tilde g(x) &&\text{for $x \in \cball 01$} \,,\\
                &x &&\text{for $x \in \R^n \without \cball 01$} \,.
            \end{aligned}
            \right.
    \end{displaymath}

    To prove the second part of the lemma assume $\varepsilon \in (0,2^{-7})$.
    Choose $\iota \in (0,2^{-12}\varepsilon)$ so small that $g\lIm \cball
    x{\iota}\rIm \subseteq \oball {g(x)}{2^{-12}\varepsilon}$ for all $x \in
    S$. Let $\varphi : \R^n \to \R$ be a mollifier such that $\spt \varphi \in
    \oball 0{\iota}$ and $\varphi(x) = \bar \varphi(|x|)$ for some smooth
    function $\bar \varphi : \R \to [0,\infty)$. Let $\bar h = \varphi *
    g$. Clearly $\bar h \in \adm{\oball 0{1+\iota}}$ and $\bar h \lIm S \rIm
    \subseteq R + \cball 0{2^{-12}\varepsilon}$. To map $S$ onto $R$ we shall
    compose $\bar h$ with the following map
    \begin{displaymath}
        k(x) = \lambda(\dist(x,R)) \frac{\project T x}{|\project T x|}
        + \left(1 - \lambda(\dist(x,R))\right) x \,,
    \end{displaymath}
    where $\lambda : \R \to \R$ is a $\cnt^\infty$~smooth map such that
    \begin{displaymath}
        \lambda(t) = 1 \quad \text{for $t < 2^{-10}\varepsilon$} \,,
        \quad 
        \lambda(t) = 0 \quad \text{for $t > 2^{-7}\varepsilon$} \,,
        \quad \text{and} \quad
        -2^{8}/\varepsilon \le \lambda'(t) \le 0
        \,.
    \end{displaymath}
    Noting that $|(\project T x)/|\project T x| - x| = \dist(x,R)$ we see that
    $\Lip k$ \emph{does not} depend on~$\varepsilon$. Therefore we may set $h =
    k \circ \bar h$.
\end{proof}

\begin{remark}
    \label{rem:ell-adm}
    Note that if $S$ was allowed to approach $\Bdry \cball 01$ tangentially,
    i.e., if we did not assume $S \cap (R + \cball 0{\delta}) \subseteq T$, then
    the auxiliary map $\bar g$ constructed in the proof above might have not
    been Lipschitz continuous.
\end{remark}

In the next lemma we basically prove that the $\HM^m$~measure of the
unrectifiable part of elements of any minimising sequence must vanish in the
limit.

\begin{lemma}
    \label{lem:unrect-vanish}
    Assume
    \begin{gather}
        U \subseteq \R^n \quad \text{is open} \,,
        \quad
        \text{$F$ is a bounded $\cnt^0$ integrand} \,,
        \quad
        \text{$\mathcal{C}$ is a good class in~$U$} \,,
        \\
        \{ S_i : i \in \nat \} \subseteq \mathcal C \,,
        \quad
        V = \lim_{i \to \infty} \var{m}(S_i \cap U) \,,
        \quad
        \mu = \Phi_F(V) = \inf\{ \Phi_F(R \cap U) : R \in \mathcal{C} \} \in (0, \infty) \,.
    \end{gather}
    Denoting the purely $(\HM^m,m)$~unrectifiable part of~$S_i \cap U$ by $\bar
    S_i$ we~have
    \begin{gather}
        \label{eq:uv:density}
        \lim_{r \downarrow 0} \lim_{i \to \infty} \frac{\HM^m(\bar S_i \cap \cball xr)}{\unitmeasure{m} r^m} = 0 
        \quad \text{for $\|V\|$ almost all~$x$} \,,
        \\
        \label{eq:uv:Phi-Psi}
        \lim_{i \to \infty} \HM^m(\bar S_i) = 0 \,;
        \quad \text{hence,} \quad
        \lim_{i \to \infty} \Phi_F(S_i \cap U) = \lim_{i \to \infty} \Psi_F(S_i \cap U) \,.
    \end{gather}
\end{lemma}

\begin{proof}
    Let $x_0 \in U$ be such that $\density^m(\|V\|,x_0) \in (0,\infty)$ -
    from~\ref{rem:summary} we know that $\|V\|$ almost all~$x_0$ satisfy this
    condition. Without loss of generality we shall assume $x_0 = 0$.
    Suppose~\eqref{eq:uv:density} is not true at~$x_0$. Then there exists
    a~subsequence $\{ R_i : i \in \nat\}$ of $\{ S_i : i \in \nat\}$, a sequence
    of radii $\{ r_j : j \in \nat \}$ with $r_j \downarrow 0$ as $j \to \infty$,
    and $\delta \in (0,\infty)$ such that denoting $R_{j,i} = \scale{1/r_j}\lIm
    R_i \rIm \cap \cball 01$ we get
    \begin{equation}
        \label{eq:ud:contr}
        \forall j \in \nat \ \exists i_0(j) \ \forall i \ge i_0 \quad
        \delta < \HM^m(\bar R_{j,i}) < 2 \delta \,,
    \end{equation}
    where $\bar R_{j,i}$ denotes the purely $(\HM^m,m)$~unrectifiable part
    of~$R_{j,i}$. Set $\hat R_{j,i} = R_{j,i} \without \bar R_{j,i}$.
    Using~\cite[2.9.11, 2.8.18]{Fed69} we may and shall assume that
    \begin{displaymath}
        \lim_{t \downarrow 0} \frac{\HM^m(\bar R_{j,i} \cap \cball xt)}{\HM^m(R_{j,i} \cap \cball xt)} = 1
        \quad \text{for \emph{all} $x \in \bar R_{j,i}$} \,.
    \end{displaymath}
    Set $\xi_1 = \inf \im F$ and $\xi_2 = \sup \im F$ and let $\BFconst(n)$ be
    the optimal constant from the Besicovitch-Federer covering
    theorem~\cite[2.8.14]{Fed69}. Choose $\iota \in (0,2^{-144})$ so that
    \begin{equation}
        \label{eq:ud:iota}
        2 \iota \unitmeasure{m} \density^m(\|V\|,0) \xi_2 \Gamma_{\ref{lem:kill-unrect}} \BFconst(n) < 2^{-144} \xi_1 \delta \,.
    \end{equation}
    Define
    \begin{displaymath}
        \bar{\mathcal B}_{j,i} =
        \left\{
            \cball xt :
            \begin{gathered}
                t \in (0,\iota) \,,\,
                x \in \bar R_{j,i} \,,\,
                \HM^m(R_{j,i} \cap \Bdry \cball xt) = 0 \,,\,
                \\
                \HM^m(\hat R_{j,i} \cap \cball xt) \le \iota \HM^m(R_{j,i} \cap \cball xt) \,,\,
                \\
                \HM^m(\bar R_{j,i} \cap \cball xt) \ge (1 - \iota) \HM^m(R_{j,i} \cap \cball xt)
            \end{gathered}
        \right\} \,.
    \end{displaymath}
    For each $i,j \in \nat$ with $i \ge i_0(j)$ employ the Besicovitch-Federer
    covering theorem~\cite[2.8.14]{Fed69} to obtain at most countable subfamily
    $\mathcal B_{j,i}$ of $\bar{\mathcal B}_{j,i}$ such that
    \begin{displaymath}
        \bar R_{j,i} \subseteq \tbcup \mathcal B_{j,i} \cap \cball 01
        \quad \text{and} \quad
        \CF_{\bigcup \mathcal B_{j,i}} \le \sum_{B \in \mathcal B_{j,i}} \CF_{B} \le \BFconst(n) \CF_{\bigcup \mathcal B_{j,i}} \,.
    \end{displaymath}
    Define $E_{j,i} = \tbcup \mathcal B_{j,i} \cap \cball 01$ and observe that
    \begin{equation}
        \HM^m(\hat R_{j,i} \cap E_{j,i})
        \le \iota \sum_{B \in \mathcal B_{j,i}} \HM^m(R_{j,i} \cap B \cap \cball 01)
        \le \iota \BFconst(n) \HM^m(R_{j,i} \cap E_{j,i}) \,.
    \end{equation}
    Employ~\ref{lem:kill-unrect} with $\iota$ in place of $\varepsilon$ to
    obtain the map $f_{j,i} \in \adm{\Int E_{j,i}}$ of class~$\cnt^{\infty}$
    such that
    \begin{gather}
        \HM^m(f_{j,i} \lIm E_{j,i} \cap \bar R_{j,i} \rIm) 
        \le \iota \HM^m(E_{j,i} \cap \bar R_{j,i}) = \iota \HM^m(\bar R_{j,i}) \le 2 \iota \delta 
        \\
        \text{and} \quad
        \HM^m(f_{j,i} \lIm E_{j,i} \cap \hat R_{j,i} \rIm)
        \le \Gamma_{\ref{lem:kill-unrect}} \HM^m(E_{j,i} \cap \hat R_{j,i}) 
        \le \iota \Gamma_{\ref{lem:kill-unrect}} \BFconst(n) \HM^m(R_{j,i} \cap \cball 01) \,.
    \end{gather}
    Observe that we have
    \begin{equation}
        \label{eq:ud:ud}
        \lim_{j \to \infty} \lim_{i \to \infty} \HM^m(R_{j,i} \cap \cball 01)
        = \unitmeasure{m} \density^m(\|V\|,0) \in (0,\infty) \,;
    \end{equation}
    hence, we choose $j,i_1 \in \nat$ such that $i_i \ge i_0(j)$ and
    \begin{displaymath}
        \HM^m(R_{j,i} \cap \cball 01) < 2 \unitmeasure{m} \density^m(\|V\|,0)
        \quad \text{for all $i \ge i_1$} \,.
    \end{displaymath}
    In~consequence, recalling~\eqref{eq:ud:iota} and
    assuming~\eqref{eq:ud:contr} we obtain for all $j \in \nat$
    \begin{multline}
        \limsup_{i \to \infty} \Phi_{F}(\scale{r_j} \lIm f_{j,i} \lIm R_{j,i} \rIm \rIm \cap U ) 
        = \limsup_{i \to \infty} \Phi_{F_j}(f_{j,i} \lIm R_{j,i} \rIm \cap \scale{1/r_j} \lIm U \rIm)
        \\
        \le \limsup_{i \to \infty} \bigl( \Phi_{F}(S_{i} \cap U) - \Phi_{F_j}(\bar R_{j,i} \cap E_{j,i})
        + \Phi_{F_j}(f_{j,i} \lIm \hat R_{j,i} \cap E_{j,i} \rIm) + \Phi_{F_j}(f_{j,i} \lIm \bar R_{j,i} \cap E_{j,i} \rIm) \bigr)
        \\
        \le \limsup_{i \to \infty} \bigl( \Phi_{F}(S_i \cap U) + r_j^m \bigl( - \xi_1 \delta 
        + 2 \iota \unitmeasure{m} \density^m(\|V\|,0) \xi_2 \Gamma_{\ref{lem:kill-unrect}} \BFconst(n)
        + 2 \iota \delta \bigr) \bigr) 
        \\
        < \lim_{i \to \infty} \Phi_{F}(S_i \cap U) = \mu \,.
    \end{multline}
    This contradicts the definition of~$\mu$ so~\eqref{eq:ud:contr} cannot hold
    and~\eqref{eq:uv:density} is proven.

    We will now show the second claim of the lemma. Assume~\eqref{eq:uv:Phi-Psi}
    is not true. Then there exists a subsequence $\{ Q_i : i \in \nat\}$ of
    $\{S_i : i \in \nat\}$ and $\vartheta \in (0,\infty)$ such that
    \begin{equation}
        \label{eq:ud:contr2}
        \vartheta < \HM^m(\bar Q_i) < 2 \vartheta
        \quad \text{for all $i \in \nat$} \,,
    \end{equation}
    where $\bar Q_i$ denotes the purely $(\HM^m,m)$~unrectifiable part of $Q_i
    \cap U$. Let $\{\nu_i : i \in \nat\}$ be a convergent subsequence of $\{
    \HM^m \restrict \bar Q_i : i \in \nat\}$. Observe that~\eqref{eq:ud:contr2}
    implies that $\nu = \lim_{i \to \infty} \nu_i$ is a finite non-zero Radon
    measure over~$U$. Recalling~\eqref{eq:uv:density} we see that
    \begin{displaymath}
        \density^m(\nu,x) = 0 \quad \text{for $\|V\|$ almost all~$x$} \,.
    \end{displaymath}
    Moreover, $\nu$ is absolutely continuous with respect to~$\|V\|$ which,
    employing~\ref{cor:ludrb}, is absolutely continuous with respect to $\sigma
    = \HM^m \restrict \spt \|V\|$. Thus, recalling~\eqref{eq:RNder}
    and~\ref{rem:summary}, and using~\cite[2.9.7, 2.8.18, 2.9.5]{Fed69},
    we~deduce that
    \begin{displaymath}
        0 < \nu(U) 
        = \int_U \mathbf{D}(\nu, \sigma, x) \ud \sigma(x)
        = \int_U \density^m(\nu,x) \ud \sigma(x) = 0 \,.
    \end{displaymath}
    Hence, $\nu$ could not be non-zero and~\eqref{eq:uv:Phi-Psi} is proven.
\end{proof}

\begin{theorem}
    \label{thm:density}
    Assume $U$, $F$, $\mathcal C$, $\{ S_i : i \in \nat \} \subseteq \mathcal C$, $V$, $\mu$ are as in~\ref{cor:ud:good-ms} and
    \begin{displaymath}
        x_0 \in \spt \|V\| \subseteq U \,,
        \quad
        T = \Tan(\spt \|V\|, x_0) \in \grass nm \,,
        \quad
        \density^m(\|V\|,x_0) \in (0,\infty) \,.
    \end{displaymath}
    Then
    \begin{enumerate}
    \item
        \label{i:dens:lower-bound}
        $\density^m(\|V\|,x_0) \ge 1$.
    \end{enumerate}
    Moreover, if $F$ is elliptic, then
    \begin{enumerate}[resume]
    \item
        \label{i:dens:upper-bound}
         $\density^m(\|V\|,x_0) = 1$;
     \item
         \label{i:dens:VarTan}
         $\VarTan(V,x_0) = \{ \var{m}(T) \}$.
    \end{enumerate}
\end{theorem}

\begin{proof}
    \textit{Proof of~\ref{i:dens:lower-bound}.}
    Define $E = \spt \|V\|$ and $B = \R^n \without U$. Without loss of
    generality, we assume $x_0 = 0$. Employing~\ref{cor:ud:good-ms} we shall
    also assume that $\{S_i : i \in \nat\}$ satisfies all the conclusions
    of~\ref{cor:ud:good-ms}. In~particular, for some $S \in \mathcal C$ and all
    compacts sets $K \subseteq U$
    \begin{gather}
        \lim_{i \to \infty} \HDK{K}(S_i \cap U, S \cap U) = 0  \,,
        \\
        \label{eq:ud:HMmfar}
        \lim_{i \to \infty} \sup \bigl\{ r \in \R : \HM^m(\{ x \in S_i \cap K : \dist(x,E \cup B) \ge r \}) > 0 \bigr\} = 0 \,.
    \end{gather}
    Define
    \begin{displaymath}
        \delta(r) = \sup\left\{
            \frac{\dist(x,T)}{|x|} : x \in E \cap \oball{x}{2r} \without \{ 0 \}
        \right\}
        \quad \text{for $r \in (0,\infty)$} \,.
    \end{displaymath}
    Recall~\cite[3.4(1)]{All72} and $\density^m(\|V\|,x_0) \in (0,\infty)$ to
    see that $\VarTan(V,0)$ is compact and nonempty so we can choose $C \in
    \VarTan(V,0)$ and $\{ r_j \in \R : j \in \mathscr{P}\}$ such that $r_j
    \downarrow 0$ as $j \to \infty$, and $\delta(r_1) < 1$, and
    $\oball{0}{3r_1}\subseteq U$, and
    \begin{equation}
        \label{eq:ud:rj-choice}
        C = \lim_{j \to \infty} (\scale{1/r_j})_{\#} V \,,
        \quad \text{and} \quad
        \|V\| ( \Bdry \cball{0}{r_j} ) = 0
        \quad \text{for $j \in \nat$} \,.
    \end{equation}
    Set $\delta_j = \delta(r_j)$ and $\varepsilon_j = 10 \delta_j^{1/2}$. For
    $j \in \nat$ let $f_j, g_j, h_j \in \cnt^{\infty}(\R,\R)$ be such that
    \begin{gather}
        f_j(t) = 1 \quad \text{for $t \le (1-\varepsilon_j)r_j$} \,,
        \quad
        f_j(t) = 0 \quad \text{for $t \ge (1-\varepsilon_j/2)r_j$} \,,
        \\
        0 \le f_j(t) \le 1
        \quad \text{and} \quad
        |f_j'(t)| \le 4/(\varepsilon_jr_j) \quad \text{for $t \in \R$} \,,
        \\
        g_j(t) = 0
        \quad \text{for $t \le (1-3\varepsilon_j)r_j$ or $t \ge (1-\varepsilon_j/2)r_j$} \,,
        \\
        g_j(t) = 1
        \quad \text{for $(1-2\varepsilon_j)r_j \le t \le (1-\varepsilon_j)r_j$} \,,
        \\
        0 \le g_j(t) \le 1 
        \quad \text{and} \quad
        |g_j'(t)| \le 4/(\varepsilon_jr_j)
        \quad \text{for $t \in \R$} \,,
        \\
        h_j(t) = 1 \quad \text{for $t \le 2 \delta_j r_j$} \,,
        \quad
        h_j(t) = 0 \quad \text{for $t \ge 3 \delta_j r_j$} \,,
        \\
        0 \le h_j(t) \le 1
        \quad \text{and} \quad
        |h_j'(t)| \le 2/(\delta_j r_j) \,.
    \end{gather}
    We define $p_j \in \cnt^{\infty}(\R^n,\R^n)$ and $q_j \in
    \cnt^{\infty}(\R^n,\R^n)$ so that
    \begin{gather}
        p_j(x) = \project{T}(x) + \bigl( 1- f_j(|\project{T}(x)|) h_j(|\perpproject{T}(x)|) \bigr) \perpproject{T}(x)
        \\
        \text{and} \quad
        q_j(x) = \project{T}(x) + \bigl( 1- g_j(|\project{T}(x)|) h_j(|\perpproject{T}(x)|) \bigr) \perpproject{T}(x) 
        \quad \text{for $x \in \R^n$}\,.
    \end{gather}
    We set $B_j = \oball{0}{2r_j} \cap \bigl( T + \cball{0}{2\delta_jr_j} \bigr)$. Then
    \begin{equation}
        \label{eq:ud:Lip-pq}
        \Lip \bigl( p_j|_{B_j} \bigr)
        \le 6 + \frac{8\delta_j}{\varepsilon_j} \le 6 +  \delta_j^{1/2}
        \quad \text{and} \quad
        \Lip \bigl( q_j\vert_{B_j} \bigr)
        \le 6 + \frac{8\delta_j}{\varepsilon_j} \le 6 + \delta_j^{1/2} \,.
    \end{equation}
    Using~\eqref{eq:ud:HMmfar} and possibly passing to a subsequence of~$\{ S_i
    : i \in \nat\}$ we shall further assume that for $i,j \in \nat$ with $i \ge
    j$ there holds
    \begin{equation}
        \label{eq:ud:far-no-Hm}
        \HM^m\bigl( S_i \cap \oball{0}{3r_j/2} \without \bigl(T + \cball{0}{2\delta_jr_j} \bigr) \bigr) = 0 \,.
    \end{equation}
    For $j \in \nat$ the map $p_j$ is clearly admissible so for $i \in \nat$ we
    have $p_j \lIm S_i \rIm \in \mathcal C$ and
    \begin{equation}
        \label{eq:densityliminf}
        \Phi_F\big( p_j \lIm S_i \rIm \cap U \big) \ge \mu
        \quad \text{and} \quad
        \liminf_{i \to \infty} \left( \Phi_F \big( p_j \lIm S_i \rIm \cap U \big) - \Phi_F(S_i \cap U) \right) \ge 0 \,.
    \end{equation}
    Define $A_j = \cball{0}{r_j} \without \oball{0}{(1-3\varepsilon_j)r_j}$ for
    $j \in \nat$ and $\xi_1 = \inf \im F > 0$ and $\xi_2 = \sup \im F < \infty$.
    For $i,j \in \nat$ with $i \ge j$, recalling~\eqref{eq:ud:Lip-pq}, we have
    \begin{align}
      q_j \lIm S_i \rIm 
      &= \big(S_i \without \cball{0}{r_j}\big) 
        \cup \big(q_j \lIm S_i \cap A_j \rIm \big)
        \cup \big(S_i \cap \oball{0}{(1-3\varepsilon_j}r_j) \big) \,;
      \\
      \text{thus,} \quad
      \Phi_F\big( q_j \lIm S_i \rIm \cap U \big)
      &= \Phi_F(S_i \cap U) + \Phi_F\big( q_j \lIm S_i \cap A_j \rIm \big) - \Phi_F(S_i\cap A_j)
      \\
      &\le \Phi_F(S_i \cap U) + \kappa_j\HM^m(S_i\cap A_j) \,.
    \end{align}
    where $\kappa_j = \xi_2 \big( 6 + \delta_j^{1/2} \big)^m - \xi_1$ and
    $\kappa_{\infty} = 6^m\xi_2 - \xi_1$. In~consequence
    \begin{displaymath}
        \limsup_{i \to \infty}\left( \Phi_F \big( q_j \lIm S_i \rIm \cap U \big) - \Phi_F(S_i \cap U) \right) 
        \le \kappa_j \limsup_{i \to \infty} \HM^m(S_i\cap A_j) \,.
    \end{displaymath}
    Since $\var{m}(S_i \cap U) \to V \in \Var{m}(U)$ as $i \to \infty$ and $A_j$
    is compact, we have
    \begin{gather}
        \label{eq:appsman}
        \limsup_{i \to \infty} \HM^m (S_i \cap A_j) \le \|V\|(A_j) \,;
        \\
        \label{eq:densitylimsup}
        \text{thus,} \quad
        \limsup_{i \to \infty} \left( \Phi_F \big( q_j \lIm S_i \rIm \cap U \big) - \Phi_F(S_i \cap U) \right)
        \le \kappa_j \|V\|(A_j) \,.
    \end{gather}
    Since $(1-\varepsilon_j/2)r_j+3\delta_jr_j< r_j$, we get
    \begin{gather}
        q_j|_{\R^n \without \oball{0}{r_j}} =
        p_j|_{\R^n \without \oball{0}{r_j}} = \id{\R^n\without\oball{0}{r_j}}
        \\
        \text{and} \quad
        q_j \lIm \oball{0}{r_j} \rIm \subseteq \oball{0}{r_j} \,,
        \quad
        p_j \lIm \oball{0}{r_j} \rIm \subseteq \oball{0}{r_j} \,,
    \end{gather}
    so we obtain for $i,j \in \nat$ with $i \ge j$
    \begin{gather}
        q_j \lIm S_i \rIm \cap (\R^n\without \oball{0}{r_j})
        = S_i \cap (\R^n \without \oball{0}{r_j}) = p_j\lIm S_i \rIm \cap (\R^n\without \oball{0}{r_j})
        \\
        \text{and} \quad
        q_j \lIm S_i \rIm \cap \oball{0}{r_j} = q_j \lIm S_i\cap \oball{0}{r_j} \rIm \,,
        \quad
        p_j \lIm S_i \rIm \cap \oball{0}{r_j} = p_j \lIm S_i\cap \oball{0}{r_j} \rIm \,.
    \end{gather}
    Hence, recalling~\eqref{eq:densityliminf} and~\eqref{eq:densitylimsup}, we
    see that for $j \in \nat$
    \begin{equation}
        \label{eq:DeCo}
        \limsup_{i\to\infty} \left(
            \Phi_F \big( q_j \lIm S_i \cap \oball{0}{r_j} \rIm \big)
            - \Phi_F \big( p_j \lIm S_i \cap \oball{0}{r_j} \rIm \big)
        \right)
        \le \kappa_j \|V\|(A_j) \,.
    \end{equation}
    For $i,j \in \nat$ define, recalling~\ref{def:pull-back},
    \begin{gather}
        \widetilde{p}_j = \scale{1/r_j} \circ p_j \circ\scale{r_j} \,,
        \quad
        \widetilde{q}_j = \scale{1/r_j} \circ q_j \circ \scale{r_j} \,,
        \quad
        S_{j,i} = \scale{1/r_j} \lIm S_i \rIm \,,
        \\
        \widetilde{A}_j = \cball 01 \without \oball{0}{1-3\varepsilon_j} \,,
        \quad
        W_{j,i} = \scale{1/r_j} \circ q_j \lIm S_i \rIm \,,
        \quad
        Z_{j,i} = \scale{1/r_j} \circ p_j \lIm S_i \rIm \,,
        \\
        F_j = \scale{r_j}^{\#}F \,, 
        \quad \text{i.e.,} \quad
        F_j(x,T) = r_j^{m} F(r_j x,T) \quad \text{for $(x,T) \in \R^n \times \grass nm$} \,.
    \end{gather}
    Then 
    \begin{gather}
        \scale{1/r_j} \big\lIm p_j \lIm S_i \rIm \cap \oball{0}{r_j} \big\rIm
        = Z_{j,i} \cap \oball 01 \,,
        \quad
        \scale{1/r_j} \big\lIm q_j \lIm S_i \rIm \cap \oball{0}{r_j} \big\rIm
        = W_{j,i}\cap \oball 01 \,,
        \\
        \Phi_{F_j}(X) = \Phi_F\bigl( (\scale{r_j})_{\#} X \bigr)
        \quad \text{for $X \in \Var{m}(\R^n)$} \,.
    \end{gather}
    Since $\var{m}(\scale{1/r_j} \lIm S_i \rIm) = (\scale{1/r_j})_{\#}
    \var{m}(S_i)$ by~\ref{rem:pull-back}, we~get for $j \in \nat$
    using~\eqref{eq:DeCo}
    \begin{multline}
        \limsup_{i \to \infty} r_j^{-m} \bigl(
            \Phi_{F_j} \left( W_{j,i} \cap \oball 01 \right)
            - \Phi_{F_j} \left( Z_{j,i}\cap \oball 01 \right)
        \bigr)
        \le r_j^{-m} \kappa_j \|V\| (A_j)
        \\
        = \kappa_j \bigl\| (\scale{1/r_j})_{\#}V \bigr\| ( \widetilde{A}_j ) \,;
    \end{multline}
    hence, recalling~$\density^m(\|V\|,x_0) \in (0,\infty)$
    and~\cite[3.4(2)]{All72},
    \begin{multline}
        \limsup_{j \to \infty} \limsup_{i \to \infty}
        r_j^{-m} \bigl(
            \Phi_{F_j} \left( W_{j,i} \cap \oball 01 \right)
            - \Phi_{F_j} \left( Z_{j,i} \cap \oball 01 \right)
        \bigr)
        \\
        \le \kappa_{\infty} \|C\|(\Bdry \oball 01) = 0 \,.
    \end{multline}
    Employing~\eqref{eq:ud:far-no-Hm} we obtain for $i,j \in \nat$ with $i \ge j$
    \begin{equation}
        \label{eq:EsHauCo}
        \HM^m \bigl( S_{j,i} \cap \oball 0{3/2} \without (T + \cball{0}{2\delta_j}) \bigr) = 0 \,.
    \end{equation}

    For $i,j \in \nat$ define
    \begin{equation}
        \label{eq:def:Yji}
        \rho_j = 1 - 3\varepsilon_j/2
        \quad \text{and} \quad
        Y_{j,i} = \cball 0{\rho_j} \cap \spt (\HM^m \restrict W_{j,i}) \,.
    \end{equation}
    Since $W_{j,i}$ is closed we have $Y_{j,i} \subseteq W_{j,i} \cap \cball
    0{\rho_j}$ and $\HM^m(\cball 0{\rho_j} \cap W_{j,i} \without Y_{j,i}) = 0$.
    Roughly speaking, $Y_{j,i}$ equals $W_{j,i} \cap \cball 0{\rho_j}$ with
    removed ``hair''. We will now check that for $i,j \in \nat$ big enough with
    $i \ge j$, one cannot deform $Y_{j,i}$ onto $R_j = T \cap \Bdry \cball
    0{\rho_j}$ by any Lipschitz continuous map $\R^n \to \R^n$ which
    fixes~$R_j$. Assume, by~contradiction, that there exists such a~retraction
    $\bar \phi_j$. Observe that whenever $\gamma \in (0,\varepsilon_j/2)$
    \begin{multline*}
        r_j^{-m} \HM^m(q_j \lIm S_i \rIm \cap \cball 0{r_j(\rho_j + \gamma)} - \oball 0{r_j\rho_j})
        \\
        \le r_j^{-m} \HM^m(T \cap \cball 0{r_j(\rho_j + \gamma)} - \oball 0{r_j\rho_j})
        = \unitmeasure{m} ( (\rho_j + \gamma)^m - \rho_j^m ) \xrightarrow{\gamma \downarrow 0} 0 \,.
    \end{multline*}
    For each $i,j \in \nat$ we choose $\gamma_{j,i} \in (0,\varepsilon_j/2)$ such that
    \begin{multline}
        \label{eq:gamma-j}
        \xi_2 \Gamma_{\ref{lem:ell-adm}}(\Lip \bar \phi_j,\varepsilon_j/2)^m
        \HM^m(q_j \lIm S_i \rIm \cap \cball 0{r_j(\rho_j + \gamma_{j,i})} - \oball 0{r_j\rho_j})
        \\
        < \Phi_{F}(q_j \lIm S_i \rIm \cap \oball 0{r_j \rho_j} \without \oball 0{(1-3\varepsilon_j)r_j}) \,.
    \end{multline}
    Recalling~\eqref{eq:ud:far-no-Hm} and that $\varepsilon_j/2 = 5
    \delta_j^{1/2}$ we see that
    \begin{gather*}
        Y_{j,i} \cap \Bdry \cball 0{\rho_j} \subseteq R_j \,,
        \quad 
        Y_{j,i} \cap (R_j + \cball 0{\varepsilon_j/2}) \subseteq T \,,
        \\
        \text{and} \quad
        (Y_{j,i} + \cball 0{\varepsilon_j/4}) \cap (\R^n \without \oball 0{\rho_j})
        \without (R_j + \cball 0{\varepsilon_j/4}) = \varnothing \,.
    \end{gather*}
    Hence, by~lemma~\ref{lem:ell-adm}, there exists $\phi_j \in \adm{\oball
      0{\rho_j + \gamma_{j,i}}}$ such that $\phi_j \lIm Y_{j,i} \rIm \subseteq
    R_j$ and $\Lip \phi_j \le \Gamma_{\ref{lem:ell-adm}}(\Lip \bar \phi_j,
    \varepsilon_j/2)$. Clearly $\scale{r_j} \circ \phi_j[W_{j,i}] \in
    \mathcal{C}$ so using~\eqref{eq:gamma-j} we get
    \begin{multline}
        \label{eq:ud:cut-out}
        \mu \le \Phi_{F}( \scale{r_j} \circ \phi_j \lIm W_{j,i} \rIm  \cap U)
        = \Phi_{F}( \scale{r_j} \lIm W_{j,i} \rIm  \cap U ) 
        \\
        - \Phi_{F}( \scale{r_j} \lIm W_{j,i} \cap \oball{0}{\rho_j + \gamma_{j,i}} \rIm )
        + \Phi_{F}( \scale{r_j} \circ \phi_j \lIm W_{j,i} \cap \oball{0}{\rho_j + \gamma_{j,i}} \without \oball 0{\rho_j} \rIm )
        \\
        \le \Phi_{F}( q_j \lIm S_i \rIm \cap U ) - \Phi_{F}( q_j \lIm S_i \rIm \cap \oball{0}{(1-3\varepsilon_j)r_j} )
        \\
        = \Phi_{F}(S_i \cap U) - \left(
            \Phi_{F}( S_i \cap \oball{0}{(1-3\varepsilon_j)r_j} ) 
            - \bigl( \Phi_{F}( q_j \lIm S_i \rIm \cap U ) - \Phi_{F}(S_i \cap U) \bigr)
        \right)  \,.
    \end{multline}
    We choose $j_0 \in \nat$ so big that for $j \ge j_0$ we have
    \begin{equation}
        \label{eq:ud:j0-choice}
        r_j^{-m} \xi_1 \measureball{\|V\|}{\oball{0}{(1-3\varepsilon_j)r_j}}
        - r_j^{-m} 2 \kappa_j \|V\|(A_j)
        >  2^{-4} \xi_1 \unitmeasure{m} \density^m(\|V\|,0) \,,
    \end{equation}
    which is possible because
    \begin{displaymath}
        \lim_{j \to \infty} r_j^{-m} \kappa_j \|V\|(A_j) = 0
        \quad \text{and} \quad
        \lim_{j \to \infty} r_j^{-m} \measureball{\|V\|}{\oball{0}{(1-3\varepsilon_j)r_j}} = \unitmeasure{m} \density^m(\|V\|,0) > 0 \,.
    \end{displaymath}
    For each $j \ge j_0$ we select $i_0 = i_0(j) \in \nat$ such that $i_0 \ge j$
    and for $i \ge i_0$
    \begin{equation}
        \label{eq:ud:i0-j0}
        \begin{gathered}
            \Phi_{F}(S_i \cap U) - \mu < 2^{-7} r_j^m \xi_1 \unitmeasure{m} \density^m(\|V\|,0)
            \\ \text{and} \quad
            \Phi_F(q_j \lIm S_i \rIm \cap U) - \Phi_F(S_i \cap U) \le 2 \kappa_j \|V\|(A_j) \,,
        \end{gathered}
    \end{equation}
    which is possible because $\{ S_i : i \in \nat \}$ is a minimising sequence
    and due to~\eqref{eq:densitylimsup}. Combining~\eqref{eq:ud:cut-out},
    \eqref{eq:ud:j0-choice}, and~\eqref{eq:ud:i0-j0} we get for $j \ge j_0$ and
    $i \ge i_0(j)$ the following contradictory estimate
    \begin{displaymath}
        \mu \le \mu + r_j^m \xi_1 \unitmeasure{m} \density^m(\|V\|,0) \bigl( 2^{-7} - 2^{-4} \bigr) < \mu \,.
    \end{displaymath}

    We now know that $W_{j,i}$ cannot be deformed onto $R_j = T \cap \Bdry
    \cball{0}{\rho_j}$ by any Lipschitz continuous map $\R^n \to \R^n$
    fixing~$R_j$. In~consequence we get
    \begin{equation}
        \label{eq:ud:full-disc}
        \project T \lIm W_{j,i} \cap \cball 01 \rIm \cap \cball{0}{\rho_j}
        = Z_{j,i} \cap \cball{0}{\rho_j}
        = \widetilde p_j \lIm W_{j,i} \rIm \cap \cball{0}{\rho_j}
        = R_j
    \end{equation}
    because otherwise we could deform $W_{j,i}$ onto $R_j$. In~particular,
    \begin{equation}
        \label{eq:us:lower-measure}
        \liminf_{j \to \infty} \liminf_{i \to \infty} \HM^m(W_{j,i} \cap \cball 01) 
        \ge \lim_{j \to \infty} \HM^m(T \cap \cball{0}{\rho_j})
        = \unitmeasure{m} \,,
    \end{equation}
    so~\ref{i:dens:lower-bound} is now proven.

    \medskip

    \textit{Proof of~\ref{i:dens:upper-bound}.} Define $X_j = \oball
    01 \without \cball{0}{\rho_j}$ and let $F^0$ be defined as
    in~\ref{def:Fx-integrand}. Recall~\eqref{eq:def:Yji}
    and~\eqref{eq:us:lower-measure}. Note that whenever $A \subseteq R^n$ is
    closed and satisfies $\HM^m(A \cap K) < \infty$ for all compact $K \subseteq
    \R^n$, and $j \in \nat$, then
    \begin{displaymath}
        \Psi_{F_j^0}(A) = \Psi_{F^0}\bigl( \scale{r_j} \lIm A \rIm \bigr) = r_j^m \Psi_{F^0}(A) \,.
    \end{displaymath}
    Since $F$ is elliptic and $\scale{1/\rho_j} \lIm Y_{j,i} \rIm$
    satisfies~\ref{def:elliptic}\ref{i:ell:S} employing~\eqref{eq:ud:far-no-Hm}
    we can find $\xi_3 > 0$ such that for $i,j \in \nat$ with $j \ge j_0$ and $i
    \ge i_0(j)$
    \begin{multline}
        0 \le \HM^m( W_{j,i} \cap \oball{0}{\rho_j} ) 
        - \HM^m( T \cap \oball{0}{\rho_j} )
        \\
        \le \xi_3 r_j^{-m} \bigl(
        \Psi_{F^0}( \scale{r_j} \lIm W_{j,i} \cap \oball{0}{\rho_j} \rIm )
        - \Psi_{F^0}( \scale{r_j} \lIm T \cap \oball{0}{\rho_j} \rIm )
        \bigr)
        \\
        \le \xi_3 r_j^{-m} \bigl(
        \Psi_{F_j^0}( W_{j,i} \cap \oball 01 )
        - \Psi_{F_j^0}( Z_{j,i} \cap \oball 01 )
        \bigr)
        + r_j^{-m} \xi_3 \Psi_{F_j^0}( Z_{j,i} \cap X_j ) \,.
    \end{multline}
    Since $\widetilde{q}_j( \oball 01 ) \subseteq \oball 01$ and
    $\widetilde{q}_j(x) = x$ for $x \in \oball{0}{1-3\varepsilon_j} \cup
    (\R^n \without \oball 01)$, we see that
    \begin{gather}
        W_{j,i} \cap \oball 01
        \supseteq \big(S_{j,i}\cap \cball{0}{1-3\varepsilon_j}\big) \cup
        \big(W_{j,i}\cap \widetilde{A}_j\big) \,;
        \\
        \text{thus,} \quad
        \HM^m \bigl( W_{j,i} \cap \oball 01 \bigr)
        \ge \HM^m \bigl( S_{j,i} \cap \oball 01 \bigr)
        - \HM^m \bigl( S_{j,i} \cap \widetilde{A}_j \bigr)
        + \HM^m \bigl( W_{j,i} \cap \widetilde{A}_j \bigr) \,.
    \end{gather}
    Hence, we get 
    \begin{multline}
        \label{eq:ud:step1}
        \bigl| \HM^m \bigl( S_{j,i} \cap \oball 01 \bigr) - \HM^m(T\cap \oball 01) \bigr|
        \\
        \le \bigl| \HM^m\bigl( W_{j,i}\cap \oball 01 \bigr) - \HM^m(T\cap \oball 01) \bigr|
        + \HM^m\bigl( S_{j,i} \cap \widetilde{A}_j \bigr)
        + \HM^m\bigl( W_{j,i} \cap \widetilde{A}_j \bigr)
        \\
        \le \bigl( \HM^m\bigl( W_{j,i} \cap \oball{0}{\rho_j} \bigr) - \HM^m( T \cap \oball{0}{\rho_j} ) \bigr)
        + \HM^m\bigl( S_{j,i} \cap \widetilde{A}_j \bigr)
        + 2 \HM^m\bigl( W_{j,i} \cap \widetilde{A}_j \bigr)
        \\
        + \HM^m( T \cap \oball 01\without \oball{0}{\rho_j} )
        \\
        \le \xi_3 r_j^{-m} \bigl( \Psi_{F_j^0} \bigl( W_{j,i} \cap \oball 01 \bigr)
        - \Psi_{F_j^0}( Z_{j,i} \cap \oball 01) \bigr)
        + r_j^{-m} \xi_3 \Psi_{F_j^0} \bigl( Z_{j,i} \cap X_j \bigr)
        \\
        + 2 \HM^m \bigl( W_{j,i}\cap \widetilde{A}_j \bigr)
        + \HM^m \bigl( S_{j,i}\cap \widetilde{A}_j \bigr)
        + \HM^m( T \cap \oball 01\without \oball{0}{\rho_j} )
        \,.
    \end{multline}
    Recalling~\eqref{eq:ud:Lip-pq} and~\ref{rem:Psi-Phi} we see that
    \begin{gather}
        \label{eq:ud:error-terms1}
        r_j^{-m} \xi_3 \Psi_{F_j^0} \bigl( Z_{j,i} \cap X_j \bigr)
        \le \xi_3 \xi_2 (6 + \delta_j^{1/2})^m \HM^m(S_{j,i} \cap \widetilde A_j) \,,
        \\
        \label{eq:ud:error-terms2}
        \HM^m \bigl( W_{j,i}\cap \widetilde{A}_j \bigr) 
        \le (6 + \delta_j^{1/2})^m \HM^m(S_{j,i} \cap \widetilde A_j) \,.
    \end{gather}
    We observe that
    \begin{multline}
        \label{eq:ud:errors}
        \limsup_{j \to \infty} \limsup_{i \to \infty} \HM^m(S_{j,i} \cap \widetilde A_j)
        \le \limsup_{j \to \infty} r_j^{-m} \|V\| (\widetilde A_j)
        \\
        = \limsup_{j \to \infty} \bigl\| (\scale{1/r_j})_{\#}V \bigr\| ( \widetilde A_j )
        \le \|C\|(\Bdry \cball 01) = 0 \,.
    \end{multline}
    Let us define $\omega : (0,\infty) \to \R$ by the formula
    \begin{equation}
        \omega(r) = \sup \bigl\{ |F(0,T) - F(x,T)| + |\sup \im F^0 - \sup \im F^x | : x \in \cball 0r \,,\, T \in \grass nm \bigr\} \,.
    \end{equation}
    Then, we may write
    \begin{multline}
        \label{eq:ud:F0-to-F}
        r_j^{-m} \bigl( \Psi_{F_j^0} \bigl( W_{j,i} \cap \oball 01 \bigr)
        - \Psi_{F_j^0}( Z_{j,i} \cap \oball 01) \bigr)
        \\
        \le r_j^{-m}  \bigl( \Psi_{F_j} \bigl( W_{j,i} \cap \oball 01 \bigr)
        - \Psi_{F_j}( Z_{j,i} \cap \oball 01) \bigr)
        \\
        + \omega(r_j) \bigl( \HM^m(W_{j,i} \cap \oball 01) + \HM^m(Z_{j,i} \cap \oball 01) \bigr) \,.
    \end{multline}
    Using again~\eqref{eq:ud:Lip-pq} we have
    \begin{multline}
        \label{eq:ud:Wji-bound}
        \limsup_{j \to \infty} \limsup_{i \to \infty} \HM^m(W_{j,i} \cap \oball 01) + \HM^m(Z_{j,i} \cap \oball 01)
        \\
        \le \limsup_{j \to \infty} \limsup_{i \to \infty} 2 (6 + \delta_j^{1/2})^m \HM^m(S_{j,i} \cap \oball 01)
        \le 12 \|C\| \cball 01 < \infty \,.
    \end{multline}
    Since $F$ is continuous and $\grass nm$ is compact we see that $\lim_{r \to
      0} \omega(r) = 0$; hence, combining~\eqref{eq:ud:step1}
    with~\eqref{eq:ud:error-terms1}, \eqref{eq:ud:error-terms2},
    \eqref{eq:ud:errors}, \eqref{eq:ud:F0-to-F}, and~\eqref{eq:ud:Wji-bound} we
    obtain
    \begin{multline}
        \label{eq:ud:step2}
        \limsup_{j \to \infty} \limsup_{i \to \infty}
        \bigl| \HM^m \bigl( S_{j,i} \cap \oball 01 \bigr) - \HM^m(T\cap \oball 01) \bigr|
        \\
        \le \limsup_{j \to \infty} \limsup_{i \to \infty}
        \xi_3 r_j^{-m} \bigl( \Psi_{F_j} \bigl( W_{j,i} \cap \oball 01 \bigr)
        - \Psi_{F_j}( Z_{j,i} \cap \oball 01) \bigr) \,.
    \end{multline}
    For any $i,j \in \nat$ we have
    \begin{multline}
        \label{eq:ud:Zji-Si}
        \Psi_F(\scale{r_j}\lIm Z_{j,i} \rIm \cap U)
        = \Psi_F(S_i \cap U)
        + \bigl( \Psi_{F_j}(Z_{j,i} \cap \oball 01) - \Psi_{F_j}(W_{j,i} \cap \oball 01)  \bigr)
        \\
        + \Psi_{F_j}(W_{j,i} \cap \widetilde A_j) - \Psi_{F_j}(S_{j,i} \cap \widetilde A_j) \,.
    \end{multline}
    Since $V = \lim_{i \to \infty} \var{m}(S_i \cap U)$ is minimising, we obtain
    \begin{displaymath}
        \mu = \Phi_F(V) \le \Phi_F(\scale{r_j}\lIm Z_{j,i} \rIm \cap U) \le \Psi_F(\scale{r_j}\lIm Z_{j,i} \rIm \cap U)
        \quad \text{ for each $i,j \in \nat$} \,;
    \end{displaymath}
    hence, transforming~\eqref{eq:ud:Zji-Si} we get
    \begin{multline}
        r_j^{-m} \bigl( \Psi_{F_j}(W_{j,i} \cap \oball 01) - \Psi_{F_j}(Z_{j,i} \cap \oball 01) \bigr)
        \le r_j^{-m} \bigl( \Psi_F(S_i \cap U) - \Phi_F(V) \bigr)
        \\
        + r_j^{-m} \Psi_{F_j}(W_{j,i} \cap \widetilde A_j) 
        + r_j^{-m} \Psi_{F_j}(S_{j,i} \cap \widetilde A_j) \,.
    \end{multline}
    Estimating as in~\eqref{eq:ud:error-terms1}, \eqref{eq:ud:error-terms2},
    \eqref{eq:ud:errors} and using~\eqref{eq:uv:Phi-Psi} we reach the conclusion
    \begin{equation}
        \limsup_{j \to \infty} \limsup_{i \to \infty}
        r_j^{-m} \bigl( \Psi_{F_j}(W_{j,i} \cap \oball 01) - \Psi_{F_j}(Z_{j,i} \cap \oball 01) \bigr)
        = 0 \,.
    \end{equation}
    Plugging this into~\eqref{eq:ud:step2} we obtain
    \begin{multline}
        0 = \limsup_{j \to \infty} \limsup_{i \to \infty}
        \bigl| \HM^m \bigl( S_{j,i} \cap \oball 01 \bigr) - \HM^m(T\cap \oball 01) \bigr| 
        \\
        = \lim_{j \to \infty} r_j^{-m} \measureball{\|V\|}{\oball{0}{r_j}} - \unitmeasure{m} 
        = \unitmeasure{m} \bigl( \density^m(\|V\|,0) - 1 \bigr) \,.
    \end{multline}
    Hence, $\|C\| = \HM^m \restrict T$ by~\cite[3.4(2)]{All72}. To have $C =
    \var{m}(T)$ we still need to show that $P = T$ for $C$ almost all $(x,P)$.

    \medskip

    \textit{Proof of~\ref{i:dens:VarTan}.} Let $\hat S_{j,i}$ denote
    the $(\HM^m,m)$~rectifiable part of $S_{j,i} \cap \cball 01$.
    From~\ref{lem:unrect-vanish} and~\ref{i:dens:upper-bound} we know that
    \begin{equation}
        \label{eq:ud:ud-rect}
        \lim_{j \to \infty} \lim_{i \to \infty} \HM^m(\hat S_{j,i})
        = \unitmeasure{m} \density^m(\|V\|,0) = \unitmeasure{m} \,.
    \end{equation}
    Set $p_{j,i} = \project{T}|_{\hat S_{j,i}}$. Since $p_{j.i}(x) = \widetilde
    p_j(x)$ for $x \in S_{j,i} \cap \cball 01 \without \widetilde A_j$ and
    recalling~\eqref{eq:ud:errors} and~\eqref{eq:ud:full-disc} we see that
    \begin{displaymath}
        \lim_{j \to \infty} \lim_{i \to \infty} \HM^m(p_{j,i} \lIm \hat S_{j,i} \rIm) = \unitmeasure{m} \,.
    \end{displaymath}
    Clearly $\Lip p_{j,i} \le 1$ so $1 - \ap J_m p_{j,i}(x) \ge 0$ for $\HM^m$
    almost all $x \in \hat S_{j,i}$. Employing the area
    formula~\cite[3.2.20]{Fed69} we have
    \begin{multline}
        \label{eq:ud:tm-excess}
        0 \le \lim_{j \to \infty} \lim_{i \to \infty} \int_{\hat S_{j,i}} 1 - \ap J_m p_{j,i} \ud \HM^m
        \\
        = \lim_{j \to \infty} \lim_{i \to \infty} 
        \HM^m(\hat S_{j,i}) 
        - \int_{p_{j,i} \lIm \hat S_{j,i} \rIm} \HM^0(p_{j,i}^{-1} \{y\}) \ud \HM^m(y)
        \\
        \le \lim_{j \to \infty} \lim_{i \to \infty} \bigl( \HM^m(\hat S_{j,i}) 
        - \HM^m(p_{j,i} \lIm \hat S_{j,i} \rIm) \bigr) = 0 \,.
    \end{multline}
    Hence, employing~\ref{lem:jacobian-tilt}, we get
    \begin{equation}
        \lim_{j \to \infty} \lim_{i \to \infty} \int_{\cball 01} \| \project P - \project T \|^2 \ud \var{m}(\hat S_{j,i})(x,P) = 0 \,.
    \end{equation}
    Define $\varphi(x,P) = \CF_{\cball 01}(x) \| \project P - \project T \|^2$
    for $(x,P) \in \R^n \times \grass nm$. Recalling~\ref{lem:unrect-vanish} and
    noting that $\|C\|(\Bdry \cball 01) = 0$, we see that
    \begin{equation}
        0 = \lim_{j \to \infty} \lim_{i \to \infty} \var{m}(S_{j,i})(\varphi) 
        = C(\varphi) 
        = \int_{\cball 01} \| \project P - \project T \|^2 \ud C(x,P) \,.
    \end{equation}
    Therefore, $P = T$ for $C$ almost all $(x,P) \in \cball 01 \times \grass nm$.

    Now, since $C \in \VarTan(V,0)$ was chosen arbitrarily we obtain for
    \emph{all} $C \in \VarTan(V,0)$
    \begin{equation}
        \label{eq:ud:C-on-ball}
        C(\psi) = \int_T \psi(x,T) \ud \HM^m(x)
        \quad \text{whenever $\psi \in \ccspace{\oball 01 \times \grass nm}$} \,.
    \end{equation}
    In particular, if $C \in \VarTan(V,0)$ and $\rho \in (0,1)$, then $C' =
    (\scale{\rho})_{\#}C \in \VarTan(V,0)$ also satisfies~\eqref{eq:ud:C-on-ball};
    hence, for all $C \in \VarTan(V,0)$
    \begin{equation}
        C(\psi) = \int_T \psi(x,T) \ud \HM^m(x)
        \quad \text{whenever $\psi \in \ccspace{\R^n \times \grass nm}$} \,.
        \qedhere
    \end{equation}
\end{proof}

Now we have all the ingredients to prove our main theorem.

\begin{proof}[\protect{Proof of~\ref{thm:main}}]
    First we recall~\ref{rem:summary} too obtain $S$ and $\{ S_i : i \in \nat
    \}$ and too see that we may apply~\ref{thm:density} at $\|V\|$ almost
    all~$x_0$. We then get $\density^m(\|V\|,x) = 1$ and $\Tan(\spt\|V\|,x) = T$
    for $V$~almost all $(x,T) \in U \times \grass nm$. We know $\HM^m(S \cap U
    \without \spt\|V\|) = 0$, which means that $V = \var{m}(S \cap U)$ and that
    $S$ and $\{ S_i : i \in \nat \}$ satisfy all the conditions
    of~\ref{thm:main}.
\end{proof}

% Local Variables:
% coding: utf-8
% eval: (ispell-change-dictionary "british")
% eval: (flyspell-mode)
% End:

%% file: spans.tex
\section{An example of a good class: homological spanning}
\label{sec:cech-spanning}

Let us fix an abelian coefficient group~$G$. We shall work in the category
$\mathcal A_1$ of~\emph{arbitrary pairs and their maps} as defined
in~\cite[I,1]{MR0050886}. This means that $(X,A)$ is an object in~$\mathcal A_1$
if $X$ is a topological space and $A \subseteq X$ is an arbitrary subset with
the relative topology. Morphisms in $\mathcal A_1$ between $(X,A)$ and $(Y,B)$
are continuous functions $f : X \to Y$ such that $f \lIm A \rIm \subseteq
B$. If~$(X,A)$, $(Y,B)$ are objects in~$\mathcal A_1$ and $f : (X,A) \to (Y,B)$
is a morphism in~$\mathcal A_1$, then the symbol $\Ch_k(X,A;G)$ shall denote the
$k^{\mathrm{th}}$~\v{C}ech homology group~\cite[IX,3.3]{MR0050886} of the
pair~$(X,A)$ with coefficients in~$G$ and $\Ch_k(f) : \Ch_k(X,A;G) \to
\Ch_k(Y,B;G)$ the induced morphism of abelian groups. In~case $A = \varnothing$,
we write $\Ch_k(X;G) = \Ch_k(X,\varnothing;G)$. For any sets $X \subseteq Y
\subseteq \R^n$, we~will denote by~$i_{X,Y}$ the inclusion map $X
\hookrightarrow Y$.

\begin{definition}
	\label{def:spans}
	Let $B$ be a closed subset of~$\R^{n}$, $L$ a subgroup of
    $\Ch_{m-1}(B;G)$. We say that a closed set $E \subseteq \R^n$ spans~$L$ if
    $L \subseteq \ker\bigl( \Ch_{m-1}(i_{B, B\cup E}) \bigr)$. In~other words,
    $E$ spans~$L$ if the composition
    \begin{displaymath}
        L \hookrightarrow \Ch_{m-1}(B;G) \xrightarrow{\Ch_{m-1}(i_{B, B\cup E})} \Ch_{m-1}(B \cup E;G)
    \end{displaymath}
    is zero.

    We denote by $\check{\mathcal{C}}(B,L,G)$ the collection of all closed
    subsets of~$\R^n$ which span~$L$.
\end{definition}

We shall prove that $\check{\mathcal{C}}(B,L,G)$ is a good class in the sense of
definition~\ref{def:goodclass}.

\begin{lemma}\label{le:DLS}
	Let $B$ be a closed subset of~$\R^n$, $L$ a subgroup of
    $\Ch_{m-1}(B;G)$. Let $\{ E_{k} \subseteq \R^n : k \in \nat \}$ be a
    decreasing sequence of closed sets. If $B \subseteq E_{k+1} \subseteq E_k$
    and $L \subseteq \ker \bigl( \Ch_{m-1}(i_{B,E_{k}}) \bigr)$ for all $k \geq
    1$, then, by setting $E = \bigcap_{k \geq 1} E_k$, we have $L \subseteq
    \ker\bigl( \Ch_{m-1}(i_{B,E}) \bigr)$.
\end{lemma}

\begin{proof}
	Since $E = \bigcap_{k\geq 1} E_k$, we have that $E=\varprojlim E_k$, see for
	example \cite[Theorem 2.5 on p.~260]{MR0050886}. We let
	\begin{displaymath}
		\varphi:\Ch_{m-1}(E;G)\to\varprojlim\Ch_{m-1}(E_k;G)
	\end{displaymath}
	be the natural isomorphism, and let 
	\begin{displaymath}
		\pi_k:\varprojlim\Ch_{m-1}(E_k;G)\to \Ch_{m-1}(E_k;G)
	\end{displaymath}
	be the natural projections. Then
	\begin{displaymath}
		\Ch_{m-1}(i_{E,E_k})=\varphi\circ \pi_k.
	\end{displaymath}

	Since 
	\begin{displaymath}
		\Ch_{m-1}(i_{B,E_j})=\Ch_{m-1}(i_{E_k,E_j}) \circ\Ch_{m-1}(i_{B,E}),
	\end{displaymath}
	by the universal property of inverse limit, there exist a homomorphism 
	\begin{displaymath}
		\psi:\Ch_{m-1}(B;G)\to \varprojlim\Ch_{m-1}(E_k;G)
	\end{displaymath}
	such that 
	\begin{displaymath}
		\Ch_{m-1}(i_{B,E_k})=\pi_k\circ\psi.
	\end{displaymath}
	Then 
	\begin{displaymath}
		\pi_k\circ\psi(L)=\Ch_{m-1}(i_{B,E_k})(L)=0,
	\end{displaymath}
	thus $\psi(L)=0$.
	We see that 
	\begin{displaymath}
		\Ch_{m-1}(i_{B,E})=\varphi^{-1}\circ \psi,
	\end{displaymath}
	thus
	\begin{displaymath}
		\Ch_{m-1}(i_{B,E})(L)=\varphi^{-1}\circ\psi(L)=0.
        \qedhere
	\end{displaymath}
\end{proof}

\begin{remark}
	\label{re:defclo}
	For any continuous map $\varphi:\R^{n}\to\R^{n}$ with $\varphi\vert_{B}=\id
    B$, we have 
	\begin{displaymath}
		\varphi[S]\in \check{\mathcal{C}}(B,L,G)
        \quad \text{whenever $S\in \check{\mathcal{C}}(B,L,G)$} \,.
	\end{displaymath}

    In particular, $\check{\mathcal{C}}(B,L,G)$ satisfies
    condition \ref{i:gc:deformation} of definition~\ref{def:goodclass}.
\end{remark}

\begin{lemma}
    \label{le:CechGood}
	Let $B$, $L$ and $\check{\mathcal{C}}(B,L,G)$ be given as in Definition
    \ref{def:spans}. Then $\check{\mathcal{C}}(B,L,G)$ is a good class in the
    sense of~\ref{def:goodclass}.
\end{lemma}

\begin{proof}
    Observe that $\R^n \in \check{\mathcal{C}}(B,L,G)$ so
    $\check{\mathcal{C}}(B,L,G)$ is non-empty and contains only closed sets by
    definition. Recalling~\ref{re:defclo} we only need to check
    condition~\ref{i:gc:hd} of definition~\ref{def:goodclass}.

    If $\{ S_k:k\in \nat \}\in \check{\mathcal{C}}(B,L,G)$ is a sequence such
    that $S_i \to S$ locally in Hausdorff distance for some closed set $S$, then
    by putting
    \begin{displaymath}
        E_k = \Clos{(B \cup \textstyle{ \bigcup_{i \geq k}} S_k)} 
        = B \cup S \cup \textstyle{ \bigcup_{i \geq k}} S_k \,,
    \end{displaymath}
    we have that  
    \begin{displaymath}
        L\subseteq \ker \Ch_{m-1}\left( i_{B,E_k} \right).
    \end{displaymath}
    By Lemma \ref{le:DLS}, we get that $L\subseteq \ker \Ch_{m-1}\left(i_{B,B\cup S} \right)$.
\end{proof}

\begin{remark}
	Replacing {\v C}ech homology group with {\v C}ech cohomology group, we let
    $L$ be a~subgroup of $\Ch^{m-1}(B;G)$, $\mathcal{C}$ a collection of closed
    sets $E$ such that the composition
	\begin{displaymath}
        \Ch^{m-1}(B \cup E;G) \xrightarrow{\Ch^{m-1}(i_{B,B \cup E})} \Ch^{m-1}(B;G) \twoheadrightarrow \Ch^{m-1}(B;G) / L
	\end{displaymath}
    is zero.  By continuity of {\v C}ech cohomology theory, see for example
    \cite[Theorem 3.1 in p.~261]{MR0050886}, we get also that $\mathcal{C}$ is
    a~good class. Indeed, we have a similar result as Lemma~\ref{le:DLS}, but
    with {\v C}ech cohomology groups instead of {\v C}ech homology groups. For
    the proof we refer the reader to~\cite[Proposition (2.7)]{MR879567}.
\end{remark}

% Local Variables:
% coding: utf-8
% eval: (ispell-change-dictionary "british")
% eval: (flyspell-mode)
% End: